\documentclass[11 pt, twoside, letterpaper]{scrartcl}

\usepackage[headsepline,plainheadsepline]{scrpage2}
\usepackage{amsfonts, amsmath, amssymb, amsthm, amsopn, latexsym, graphicx, mathrsfs, verbatim, enumerate, url, stmaryrd, enumitem} 
\usepackage{tikz-cd}

\usepackage[pagebackref]{hyperref}

\usepackage[all]{hypcap} 
\usepackage{pst-all} 
\usepackage[margin=1in]{geometry} 
\usepackage{Style} 
\usepackage{subcaption}
\usepackage[charter]{mathdesign}

\DeclareMathAlphabet{\mathbbm}{U}{bbm}{m}{n}

\graphicspath{{}}

\newcommand{\executeiffilenewer}[3]{%
\ifnum\pdfs

trcmp{\pdffilemoddate{#1}}%
{\pdffilemoddate{#2}}>0%
{\immediate\write18{#3}}\fi%
}

\makeatletter
\newcommand{\subjclass}[2][1991]{%
  \let\@oldtitle\@title%
  \gdef\@title{\@oldtitle\footnotetext{#1 \emph{Mathematics subject classification.} #2.}}%
}
\newcommand{\keywords}[1]{%
  \let\@@oldtitle\@title%
  \gdef\@title{\@@oldtitle\footnotetext{\emph{Key words and phrases.} #1.}}%
}
\makeatother

\setkomafont{title}{\normalfont\bfseries}
\setkomafont{section}{\normalfont\bfseries\center\large}
\setkomafont{subsection}{\normalfont\bfseries}
\setkomafont{subsubsection}{\normalfont\itshape}
\addtokomafont{disposition}{\rmfamily}

\hypersetup{
colorlinks=true,
linkcolor=blue,
citecolor=red
}

\newtheorem{introthm}{Theorem}

\newtheorem{thm}{Theorem}[section]
\newtheorem{lem}[thm]{Lemma}
\newtheorem{cor}[thm]{Corollary}
\newtheorem{prop}[thm]{Proposition}
\newtheorem{quest}[thm]{Question}

\theoremstyle{definition}
\newtheorem{rmk}[thm]{Remark}
\newtheorem{defi}[thm]{Definition}
\newtheorem{ex}[thm]{Example}

\newtheorem*{Not}{Notation}

\theoremstyle{definition}

\newcommand{\fps}[1]{{\left\llbracket {#1} \right\rrbracket}}
\newcommand{\Jacdiv}[1]{{R_{#1}}}
\newcommand{\kkk}[3]{{k_{{#1}\stackrel{#3}{\to}{#2}}}}
\newcommand{\kk}[2]{{k_{{#1}{\to}{#2}}}}
\newcommand{\eee}[3]{{e_{{#1}\stackrel{#3}{\to}{#2}}}}
\newcommand{\ee}[2]{{e_{{#1}{\to}{#2}}}}

\newcommand{\divi}{\mathrm{div}}
\newcommand{\inte}{\mathrm{int}}
\newcommand{\subrelcpct}{\Subset}
\newcommand{\segment}[3]{[{#1},{#2}]_{#3}}
\newcommand{\eskel}[1]{\mc{S}_{\hat{#1}}}
\newcommand{\skel}[1]{\mc{S}_{{#1}}}
\newcommand{\Div}{\on{Div}}
\newcommand{\Ediv}{\mc{E}}
\newcommand{\Weil}{\mc{D}}
\newcommand{\Exc}{\mc{E}}
\newcommand{\Rest}{\mc{R}}
\newcommand{\CC}[1]{\textrm{Cc}_{#1}}
\newcommand{\VC}[1]{\mc{C}_{#1}}
\newcommand{\rees}{\mc{R}}
\renewcommand{\pr}{\on{pr}}

\title{Local dynamics of non-invertible maps near normal surface singularities\keywords{aofin}}

\author{William Gignac \and Matteo Ruggiero\thanks{Institut de Math\'{e}matiques de Jussieu, Universit\'e Paris Diderot \hfill \href{mailto:matteo.ruggiero@imj-prg.fr}{matteo.ruggiero@imj-prg.fr}}}

\date{}

\newcommand{\refprop}[1]{\hyperref[prop:#1]{Proposition~\ref*{prop:#1}}}
\newcommand{\refthm}[1]{\hyperref[thm:#1]{Theorem~\ref*{thm:#1}}}
\newcommand{\refcor}[1]{\hyperref[cor:#1]{Corollary~\ref*{cor:#1}}}
\newcommand{\refdef}[1]{\hyperref[def:#1]{Definition~\ref*{def:#1}}}
\newcommand{\refrmk}[1]{\hyperref[rmk:#1]{Remark~\ref*{rmk:#1}}}
\newcommand{\reflem}[1]{\hyperref[lem:#1]{Lemma~\ref*{lem:#1}}}
\newcommand{\refsec}[1]{\hyperref[sec:#1]{Section~\ref*{sec:#1}}}
\newcommand{\reffig}[1]{\hyperref[fig:#1]{Figure~\ref*{fig:#1}}}
\newcommand{\refex}[1]{\hyperref[ex:#1]{Example~\ref*{ex:#1}}}
\newcommand{\refsssec}[1]{\hyperref[sssec:#1]{Subsection~\ref*{sssec:#1}}}
\newenvironment{change}{}{}

\clearscrheadfoot
\lohead{\normalfont Local dynamics of non-invertible maps near normal surface singularities\hfill \pagemark}

\lehead{\normalfont\pagemark\hfill William Gignac and Matteo Ruggiero}
\begin{document}

\maketitle
\thispagestyle{empty}

\begin{abstract}\noindent  We study the problem of finding algebraically stable models for non-invertible holomorphic fixed point germs $f\colon (X,x_0)\to (X,x_0)$, where $X$ is a complex surface having $x_0$ as a normal singularity. We prove that as long as $x_0$ is not a cusp singularity of $X$, then it is possible to find arbitrarily high modifications $\pi\colon X_\pi\to (X,x_0)$ such that the dynamics of $f$ (or more precisely of $f^N$ for $N$ big enough) on $X_\pi$ is algebraically stable. This result is proved by understanding the dynamics induced by $f$ on a space of valuations associated to $X$; in fact, we are able to give a strong classification of all the possible dynamical behaviors of $f$ on this valuation space. We also deduce a precise description of the behavior of the sequence of attraction rates for the iterates of $f$. Finally, we prove that in this setting the first dynamical degree is always a quadratic integer.
\end{abstract}

\tableofcontents



\section*{Introduction}
\addcontentsline{toc}{section}{Introduction}

Let $(X,x_0)$ be an irreducible germ of a complex surface $X$ at a point $x_0\in X$. 
In this article, we study the local dynamics of dominant non-invertible germs $f\colon (X,x_0)\to (X,x_0)$ of holomorphic maps that have $x_0$ as a fixed point. 
We are specifically interested in the case when $x_0$ is a singular point of the surface, though in this case we will always assume $X$ is normal at $x_0$.
Normality in particular implies that $x_0$ is an isolated singularity.

Understanding local dynamics around fixed points is often essential to understanding the behavior of global holomorphic dynamical systems, so a great deal of research has been devoted to the topic, with most of it focusing on the smooth case of holomorphic  germs $f\colon (\C^n,0)\to (\C^n,0)$. In dimension $n = 1$, classical arguments have given an essentially complete understanding of local dynamics; in particular, for \emph{non-invertible} germs $f\colon (\C,0)\to (\C,0)$, it is always possible to make a local holomorphic change of coordinates  so that $f$ takes the simple form $f(z) = z^c$ for some integer $c>1$ \cite{carleson-gamelin:cplxdyn, milnor:dyn1cplxvar}. An analogous statement holds if we replace $\C$ by any field of characteristic $0$. In positive characteristic normal forms are more involved (see \cite{ruggiero:superattrdim1charp}), but still computable.
The situation in dimension $n \geq 2$ is far more complicated.
A non-invertible \emph{planar} holomorphic germ $f\colon(\C^2,0)\to (\C^2,0)$ comes in two flavors: either the differential $d\!f_0$ of $f$ at $0$ is nilpotent, in which case $0$ is said to be a \emph{superattracting} fixed point of $f$, or else $d\!f_0$ has exactly one non-zero eigenvalue, and $0$ is said to be \emph{semi-superattracting}. 
In these cases, one might hope to analyze the local dynamics as was done in dimension $1$, by finding suitable local coordinates in which $f$ has a simple form, called a \emph{normal form}, making it easier to compute and study the iterates of $f$. 
Work in this direction for \emph{planar} semi-superattracting germs has been done in \cite{ruggiero:rigidification}, but in the superattracting case this strategy has been successfully carried out for only special classes of germs, see e.g. \cite{favre:rigidgerms, abate-raissy:renormalization} for results in dimension $2$, and \cite{ruggiero:rigidgerms, buff-epstein-koch:bottchercoordinates} for results in higher dimensions.
Unfortunately, it seems as if a complete and useful description of normal forms for general superattracting germs is infeasible (see \cite{hubbard-papadopol:supfixpnt}), owing largely to the richness of analytic geometry in dimensions $2$ or higher, a central obstruction being that it is unclear how to deal with the growth of the geometric complexity of the critical set of the iterates of $f$. 

Recently, another approach to studying the dynamics of planar non-invertible germs $f\colon (\C^2,0)\to (\C^2,0)$ has been successfully employed. In this approach, one investigates the dynamics of $f$ on birationally equivalent models of $(\C^2,0)$. Given a modification $\pi\colon X_\pi\to (\C^2,0)$, i.e., a proper holomorphic map which is an isomorphism over $\C^2\setminus\{0\}$, one can lift $f$ to a meromorphic map $f_\pi\colon X_\pi\dashrightarrow X_\pi$ and study the dynamics of $f_\pi$ on the exceptional set $\pi^{-1}(0)$ of $\pi$. 
Understanding the orbits of the \emph{exceptional divisors} of $\pi$, that is, the divisors of $X_\pi$ supported in $\pi^{-1}(0)$, gives very concrete information about the local dynamics of $f$, see \cite{favre-jonsson:eigenval, gignac-ruggiero:attractionrates, ruggiero:rigidification, gignac:conjarnold}. 
Unfortunately, for any given modification $\pi$, the action of $f_\pi$ on exceptional divisors can be difficult to control due to the presence of indeterminacy points of $f_\pi$. 
One can imagine, for instance, a situation where $f_\pi^n$ contracts an exceptional prime divisor $E$ of $\pi$ to an indeterminacy point of $f_\pi$ for infinitely many $n$. 
If this were to occur, the best hope would be that one could avoid such behavior by passing to a higher modification. 
One of our main results is that this is almost always possible. 

\begin{introthm}\label{thm:geometricstability}
Let $(X,x_0)$ be an irreducible germ of a normal complex surface at a point $x_0\in X$, and let $f\colon (X,x_0)\to (X,x_0)$ be a dominant non-invertible holomorphic map.
Assume we are not in the situation of a finite germ at a cusp singularity.
Then for any modification $\pi\colon X_{\pi}\to (X,x_0)$, one can find a modification $\pi'\colon X_{\pi'}\to (X,x_0)$ dominating $\pi$ with the following property: if $E$ is any exceptional prime divisor of $\pi'$ and $f_{\pi'}\colon X_{\pi'}\dashrightarrow X_{\pi'}$ is the (meromorphic) lift of $f$, then $f_{\pi'}^n(E)$ is an indeterminacy point of $f_{\pi'}$ for at most finitely many $n$. 
Moreover the space $X_{\pi'}$ may be taken to have at most cyclic quotient singularities.
\end{introthm}
Recall that $f\colon (X,x_0)\to (X,x_0)$ is \emph{finite} near $x_0$ when $f$ does not contract any curve $C\subset X$ passing through $x_0$ to $x_0$.
In the smooth case when the surface germ is $(\C^2,0)$,  \refthm{geometricstability} was essentially proved in \cite{gignac-ruggiero:attractionrates} for superattracting fixed point germs $f\colon (\C^2,0)\to (\C^2,0)$. 
To be precise, it was shown that there exists \emph{some} non-trivial model $\pi\colon X_\pi\to (\C^2,0)$ where the orbits of exceptional prime divisors are well behaved, rather than arbitrarily high models, but the full theorem follows from a minor modification of the argument given there, see \S\ref{sec:algebraic_stability}. 
In the semi-superattracting case, \refthm{geometricstability} follows similarly from the work in \cite{ruggiero:rigidification}.
The results of these two articles will play a prominent role in this paper.

In both of these papers, the broad strategy to prove \refthm{geometricstability} is the same: instead of studying the dynamics of $f_\pi\colon X_\pi\dashrightarrow X_\pi$ on any model $\pi\colon  X_\pi\to (X,x_0)$ of $(X,x_0)$ individually, one analyzes a dynamical system $f_\bullet\colon \mc{V}_X\to \mc{V}_X$ on a space $\mc{V}_X$ that encodes all bimeromorphic models $\pi\colon X_\pi\to (X,x_0)$ simultaneously. 
Results about the dynamics of $f_\bullet$ can then often be translated into statements about the dynamics of $f$ on bimeromorphically equivalent models of $(X,x_0)$. 
This $\mc{V}_X$ is the space of (suitably normalized) centered, rank one semivaluations on the local ring $\mc{O}_{X, x_0}$ of $X$ at $x_0$ which are trivial on $\C$, see \S\ref{sec:valuation_spaces} for definitions. 
If $X$ is algebraic, $\mc{V}_X$ can be viewed as a subset of the Berkovich analytification $X^\an$ of $X$, where the ground field $\C$ is equipped with the trivial absolute value. 
This broad strategy was first employed by Favre-Jonsson \cite{favre-jonsson:eigenval} in the case when $X = \C^2$, and they have since successfully used it in a variety of settings \cite{favre-jonsson:valanalplanarplurisubharfunct, favre-jonsson:dynamicalcompactifications, favre:holoselfmapssingratsurf, favre-jonsson:valmultideals, boucksom-favre-jonsson:valuationsplurisubharsing}.

In this article, we follow the same strategy, and our central result (from which \refthm{geometricstability} will follow) is a classification of the possible dynamics of $f_\bullet\colon \mc{V}_X\to \mc{V}_X$, the simplest possible summary of which is given below. A finer classification can (and will) be given if one has more information about the nature of the singularity $(X, x_0)$ and the map $f$, but to avoid technicalities in the introduction the following less precise statement will suffice.

\begin{introthm}\label{thm:valdynamics} Let $(X,x_0)$ be an irreducible germ of a normal complex surface at a point $x_0\in X$, and let $f\colon (X,x_0)\to (X,x_0)$ be a dominant non-invertible holomorphic map. Denote by $f_\bullet\colon \mc{V}_X\to \mc{V}_X$ the induced dynamical system on the space $\mc{V}_X$ of normalized, centered, rank one semivaluations on $\mc{O}_{X, x_0}$. Then there is a subgraph $S \subset \mc{V}_X$, homeomorphic to either a point, a closed interval, or a circle, such that $f_\bullet (S) = S$ and for any quasimonomial valuation $\nu\in \mc{V}_X$ one has that $f^n_\bullet \nu\to S$ as $n\to \infty$.
\end{introthm}

A special situation in the previous theorem is given when $S$ is a circle, $f$ is finite, and $f_\bullet$ acts on $S$ as a rotation of infinite order (irrational rotation).
In this situation, the conclusion of \refthm{geometricstability} is simply false, see \S\ref{ssec:finitecusp}.

As applications of \hyperref[thm:geometricstability]{Theorems~\ref*{thm:geometricstability}} and \ref{thm:valdynamics} we will derive three additional interesting results, stated below. The first is related to the notion of \emph{algebraic stability} of holomorphic dynamical systems, an important and actively studied topic in complex dynamics, see for instance \cite{diller-favre:dynbimeromapsurf, favre:monomial, bedford-diller:birationalsurfacemaps, dinh-sibony:regbiratmaps, dinh-sibony:superpotentials, favre-jonsson:dynamicalcompactifications, jonsson-wulcan:monomial, lin:monomialmaps}. In our local setting, the classical  notion of algebraic stability has meaning only when the germ $f\colon (X,x_0)\to (X,x_0)$ is \emph{finite} near $x_0$.
In this case, given any modification $\pi\colon X_\pi\to (X,x_0)$, the lift $f_\pi\colon X_\pi\dashrightarrow X_\pi$ induces a $\Z$-linear pull-back operation $f^*\colon \Ediv(\pi)\to \Ediv(\pi)$ on the group $\Ediv(\pi)$ of exceptional divisors of $\pi$, see \S\ref{ssec:action_dual_divisors}. However, due to the presence of indeterminacy points of $f_\pi$, it may happen  that the pull-back operation is not functorial, that is, possibly $(f_\pi^n)^*\neq (f_\pi^*)^n$ for some or even all $n > 1$. On the other hand, the control on the orbits of exceptional prime divisors given by \refthm{geometricstability} allows us to obtain a sort of ``asymptotic'' functoriality.


\begin{introthm}\label{thm:algebraicstability} Let $(X,x_0)$ be an irreducible germ of a normal complex surface at a point $x_0\in X$, and let $f\colon (X,x_0)\to (X,x_0)$ be a finite non-invertible holomorphic map. Assume we are not in the case of $f$ inducing an irrational rotation at a cusp singularity $x_0$ of $X$. Then for any modification $\pi\colon X_{\pi}\to (X,x_0)$ there exists a dominating modification $\pi'\colon X_{\pi'}\to (X,x_0)$ and an integer $N = N(\pi')\geq 1$ such that $(f_{\pi'}^{N+n})^* = (f_{\pi'}^N)^*(f_{\pi'}^*)^n$ for all $n\geq 0$. Moreover, $X_{\pi'}$ may be taken to have at worst cyclic quotient singularities. 
\end{introthm}

The second application is related to the sequence of so-called \emph{attraction rates} $c(f^n, \nu)$ of the iterates of $f^n$ along a given valuation $\nu\in \mc{V}_X$. The attraction rate $c(f^n, \nu)$ is by definition \[
c(f^n, \nu) := \min_{\phi\in f^{n*}\mf{m}} \nu(\phi),\] where $\mf{m}$ is the maximal ideal of $\mc{O}_{X,x_0}$. The dynamical relevance of these attraction rates can be seen in the fact that understanding the values of $c(f^n, \nu)$ as $\nu\in \mc{V}_X$ varies allows one to recover the ideals $f^{n*}\mf{m}$. Our last theorem describes the structure of the sequence $\{c(f^n,\nu)\}_{n\geq 1}$ for some $\nu\in \mc{V}_X$.

\begin{introthm}\label{thm:recursion}
Let $(X,x_0)$ be an irreducible germ of a normal complex surface at a point $x_0\in X$, and let $f\colon (X,x_0)\to (X,x_0)$ be a dominant non-invertible holomorphic function.
Assume we are not in the case of a finite germ $f$ at a cusp singularity $(X,x_0)$ inducing an irrational rotation. Then for any quasimonomial valuation $\nu\in \mc{V}_X$, the sequence of attraction rates $c_n := c(f^n, \nu)$ eventually satisfies an integral linear recursion relation. More precisely, there exist integers $a$, $b$, $N$, and $m$ with $N,m\geq 1$ such that $c_{n+2m} = ac_{n+m} + bc_n$ for all $n\geq N$.
\end{introthm}

Again, the conclusions of \hyperref[thm:geometricstability]{Theorems~\ref*{thm:algebraicstability}} and \ref{thm:recursion} do not hold when $f$ is a finite germ at a cusp singularity inducing an irrational rotation.
We prove that this situation may happen for every cusp singularity, and in this case, no linear recursion relations are satisfied for the sequence of attraction rates. We deduce the non-existence of algebraically stable models in this case.

A coarser invariant than the sequence of attraction rates is given by the \emph{first dynamical degree}, which measures the exponential growth rate of this sequence.
It is defined by $c_\infty(f, \nu)=\lim_n (c(f^n,\nu))^{1/n}$ for any quasimonomial valuation $\nu$.
In fact it can be easily shown that this limit exists, and it does not depend on $\nu$, so we simply denote it by $c_\infty(f)$.
The arithmetical properties of the first dynamical degrees have been explored in several contexts, see e.g. \cite{diller-favre:dynbimeromapsurf, favre-jonsson:eigenval, blanc-cantat:dyndegbiratprojsurf}; in the global setting, its logarithm gives an upper bound to the topological entropy, see \cite{dinh-sibony:bornesupentropietop}.
For superattracting germs at a smooth point $(\nC^2,0)$, Favre and Jonsson proved (see \cite[Theorem A]{favre-jonsson:eigenval}) that the first dynamical degree is always a quadratic integer. This can also be deduced directly from the study of the recursion relation of the sequence of attraction rates, see \cite{favre-jonsson:dynamicalcompactifications, gignac-ruggiero:attractionrates}.
We have seen that in the singular setting, such recursion relation is not always present. Nevertheless, the first dynamical degree still shares the same arithmetic properties.

\begin{introthm}\label{thm:quadratic_integer}
Let $(X,x_0)$ be an irreducible germ of a normal complex surface at a point $x_0\in X$, and let $f\colon (X,x_0)\to (X,x_0)$ be a dominant non-invertible holomorphic function.
Then the first dynamical degree $c_\infty(f)$ is a quadratic integer.
\end{introthm}

Each of the theorems stated above were shown for polynomial endomorphisms in \cite{favre-jonsson:eigenval, favre-jonsson:dynamicalcompactifications}, and for superattracting germs $f\colon (\nC^2,0)\to (\nC^2,0)$ in \cite{favre-jonsson:eigenval, gignac-ruggiero:attractionrates}.
The new content in the theorems is in the case when $x_0$ is a singular point of $X$. It is remarkable that the results from the smooth setting carry over so seamlessly to the singular setting. Moreover, the generality of these results is especially noteworthy given that very little is known about local dynamics near singularities.
Because the techniques used in this article are valuative rather than complex analytic, at no point will we use in an essential way that we are working over the complex numbers, and in fact all of the theorems stated above hold equally well if $\nC$ is replaced by any field of characteristic $0$.
Some of the results 
carries over fields of positive characteristic as well (see \hyperref[rmk:weak_convergence_charp]{Remarks~\ref*{rmk:weak_convergence_charp}}, \ref{rmk:jacobian_charp} and \S\ref{ssec:positive_char}).
Nonetheless, we will continue to work exclusively over $\nC$, as this is probably the most familiar setting to the audience.

The core of the paper is the proof of \refthm{valdynamics}.
The main technical tool is the construction of a suitable distance $\rho$ on the set of normalized valuations $\mc{V}_X$ associated to a normal surface singularity, that we call \emph{angular distance}.
We study the non-expanding properties of the action $f_\bullet$ induced by non-invertible germs $f\colon (X,x_0) \to (Y,y_0)$ between normal surface singularities with respect to the angular distance (see \refthm{strong_contraction}).
Understanding the behavior of the angular distance and related objects helps also to describe the geometry of surface singularities, see \cite{garciabarroso-gonzalezperez-popescupampu:ultrametricspacesarborescentsing}.
When $f$ is not finite, we show that the action of $f_\bullet$ strictly decreases the angular distance.
We deduce the existence and uniqueness of a fixed point for $f_\bullet$ in this case.
When $f$ is finite, a theorem of Wahl \cite{wahl:charnumlinksurfsing} tells us that $(X,x_0)$ is necessarily log canonical.
We use the classification of log canonical surface singularities, and continue the analysis of finite germs developed in \cite{favre:holoselfmapssingratsurf}.
In both cases, we need to control the behavior of the attraction rates of the iterates of $f$. We deduce some restrictions on the dynamics of $f_\bullet$ for semi-superattracting germs (see \refprop{semisuper_nonfinite} and \refthm{superattracting_finite}).

The structure of this article is as follows.
In \S\ref{sec:resolutions} we recall a few facts on resolution of singularities and intersection theory.
In \S\ref{sec:valuation_spaces} we describe several properties of the valuation spaces associated to normal surface singularities, and construct the angular distance.
In \S\ref{sec:log_disc} we recall the definition of log discrepancy of a valuation, study its properties on the valuation spaces, and recall the classification of log canonical surface singularities.
In \S\ref{sec:dynamics_valspaces} we define the action $f_\bullet$ induced by $f$ on valuative spaces, and we study several properties. In particular we establish the contracting properties with respect to the angular distance, and state a more precise version of \refthm{valdynamics}, namely \refthm{classification}.
Sections \S\ref{sec:dynamics_nonfinite} and \S\ref{sec:dynamics_finite} are devoted to the proof of \refthm{valdynamics} in the non-finite and finite cases.
In \S\ref{sec:algebraic_stability} we prove \hyperref[thm:geometricstability]{Theorems~\ref*{thm:geometricstability}} and \ref{thm:algebraicstability}.
In \S\ref{sec:attraction_rates} we prove \hyperref[thm:recursion]{Theorems~\ref*{thm:recursion}} and \ref{thm:quadratic_integer}.
Finally, \S\ref{sec:examples} contains worked-out examples and remarks.
We conclude with an appendix \S\ref{sec:cusps} where we recall the arithmetical construction of cusp singularities, and the construction of finite endomorphisms on them given by \cite{favre:holoselfmapssingratsurf}.
\bigskip

\noindent\textbf{Acknowledgements.}
The authors would like to thank Charles Favre, Mattias Jonsson, and Matt Baker for interesting and fruitful discussions about valuative dynamics and its applications, and Patrick Popescu-Pampu for interesting discussions on the geometry of cusp singularities and useful remarks on a preliminary version of the manuscript. They would also like to thank Hussein Mourtada for discussions about resolution of singularities, and Nguyen-Bac Dang and Alberto Branciari for remarks on fixed point theorems.
This work was supported by the ERC-Starting grant "Nonarcomp" no. 307856.

\section{Normal surface singularities, resolutions, and intersection theory}\label{sec:resolutions}

In this section we recall some ideas and constructions from the theory of surface singularities and, in particular, the resolution of such singularities. These ideas will be essential to understanding the spaces of valuations we begin studying in the next section, the central objects of this article.
Throughout this paper, we let $(X,x_0)$ be an irreducible germ of a normal complex surface at a point $x_0\in X$.
We refer to $(X,x_0)$ as a \emph{normal surface singularity}.
Let $R=R_X$ be the completed local ring $\hat{\mc{O}}_{X,x_0}$ and let $\mf{m}=\mf{m}_X$ be its maximal ideal.

\begin{defi} A \emph{modification} of $(X,x_0)$ is a proper holomorphic map $\pi\colon X_\pi\to (X,x_0)$, where $X_\pi$ is a normal complex surface  and $\pi$ is a biholomorphism over $X \setminus \{x_0\}$. A modification is a \emph{resolution} of $(X,x_0)$ if in addition $X_\pi$ is regular. Such a resolution is \emph{good} if the exceptional locus $\pi^{-1}(x_0)$ is a divisor whose support has simple normal crossings.
\end{defi}

It is a fundamental fact that good resolutions of $(X,x_0)$ always  exist, and moreover for any $\mf{m}$-primary ideal $\mf{a}\subset R$, one can find good resolutions of $(X,x_0)$ which resolve $\mf{a}$ in the following sense.

\begin{defi}
If $\mf{a}\subset R$ is an $\mf{m}$-primary ideal, that is, a proper ideal containing some power of $\mf{m}$, then a \emph{log resolution} of $\mf{a}$ is a good resolution $\pi$ of $(X,x_0)$ such that the ideal sheaf $\pi^*\mf{a}$ is locally principal.
\end{defi}
%

In the following, we will also need to solve the singularities of a curve inside $(X,x_0)$

\begin{defi}
Let $(C,x_0) \subset (X,x_0)$ be the germ of a (reduced) curve at $x_0$ sitting in a normal surface singularity $(X,x_0)$. Then a \emph{embedded resolution} of $(C,x_0)$ in $(X,x_0)$ is a good resolution $\pi\colon X_\pi \to (X,x_0)$ so that $\pi^{-1}(C)$ has simple normal crossings.
\end{defi}

A modification $\pi'$ of $(X,x_0)$ is said to \emph{dominate} another modification $\pi$ if there is a holomorphic map $\eta\colon X_{\pi'}\to X_\pi$ such that $\pi' = \pi\circ \eta$. Because this situation will appear frequently, it will be convenient to set once and for all the following notation.

\begin{Not} If $\pi$ and $\pi'$ are two modifications of $(X,x_0)$, we write $\pi'\geq \pi$ to mean that $\pi'$ dominates $\pi$, and in this case we will always denote by $\eta_{\pi\pi'}\colon X_{\pi'}\to X_\pi$ the holomorphic map $\eta_{\pi\pi'} = \pi^{-1}\circ \pi'$. 
\end{Not}

Modifications of $(X,x_0)$ form a category, the morphisms being the maps $\eta_{\pi\pi'}$. Any two modifications of $(X,x_0)$ are dominated by some good resolution, making good resolutions an inverse system within this category. A good resolution $\pi_0$ is said to be \emph{minimal} if $\eta_{\pi\pi_0}$ is a biholomorphism whenever $\pi_0$ dominates another good resolution $\pi$. It is an important theorem that minimal good resolutions of $(X,x_0)$ always exist, and that every good resolution of $(X,x_0)$ dominates some minimal good resolution. More precisely, every good resolution of $(X,x_0)$ is obtained from a minimal good resolution by a composition of point blowups \cite[Chapter 5]{laufer:normal2dimsing}.

\begin{rmk}\label{rmk:mingoodres_notunique}
For normal surface singularities, there is a unique minimal good resolution, see \cite[Theorem 5.12]{laufer:normal2dimsing}.
Notice that two exceptional primes in a minimal good resolution may intersect in more than one point.
This may happen for example for cusp singularities like $X=\{x^2+y^4+z^6=xyz\}$, where the exceptional divisor of the minimal good resolution consists of a cycle of two rational curves (which hence intersect transversely in two points $p$ and $q$), see \S\ref{ssec:example_cusp_42}.
A classical object associated to a good resolution $\pi \colon X_\pi \to (X,x_0)$ is the dual graph $\varGamma$, which is a simplicial graph that has as vertices the irreducible components of $\pi^{-1}(x_0)$ (also called exceptional primes), and an edge between two exceptional primes $E \neq F$ for each point in $E \cap F$.
In particular, there could be several edges connecting two vertices $E$ and $F$.
By further blow-ups, we may also assume that any two (different) exceptional primes in a good resolution intersect in at most one point.
In the cusp example above, this can be obtained by either blowing-up $p$ or $q$. In particular, minimal good resolutions satisfying this stronger property are not unique in general.
For such good resolutions, the edge between two vertices is uniquely determined by its endpoints, and notations can be eased.
This stronger concept will be also useful to construct the dual graph as embedded in the vector space of exceptional $\nR$-divisors, see \S \ref{ssec:dualgraphs}.

One could equivalently work with good resolutions in this stronger sense throughout all the paper, since the uniqueness of the minimal good resolution will not play any role.
\end{rmk}




\begin{rmk}
The modifications considered in this paper will nearly always be good resolutions. On the other hand, in order to prove \refthm{geometricstability} we will be forced to consider some modifications $\pi$ for which $X_\pi$ has some mild singularities, namely cyclic quotient singularities, see \S\ref{ssec:lc_and_lt} for definitions. These spaces $X_\pi$ still have very controlled geometry; for instance, they are $\Q$-factorial \cite{artalbartolo-martinmorales-ortigasgalindo:cartierweilonquotientsingularities}.
They can also be interpreted V-manifolds, first introduced by Satake \cite{satake:gaussbonnetVmflds}, and then renamed \emph{orbifolds} by Thurston \cite{thurston:geomtopo3mfld}.
\end{rmk}

\subsection{The intersection theory of good resolutions} \label{ssec:intersectiontheory}

Let $\pi\colon X_\pi\to (X,x_0)$ be a good resolution. We now turn our attention to the intersection theory of $X_\pi$. Throughout this article, we will denote by $\Ediv(\pi)$ the group of \emph{exceptional divisors} of $X_\pi$, that is, the group of Weil divisors with support lying in the exceptional locus $\pi^{-1}(x_0)$. The prime divisors with support in $\pi^{-1}(x_0)$ are called the \emph{exceptional primes} of $\pi$; they form a $\Z$-basis of $\Ediv(\pi)$. It will also be convenient to consider the vector spaces of exceptional $\R$- and $\Q$-divisors of $\pi$, which by definition are the elements of $\Ediv(\pi)_\R := \R\otimes_\Z \Ediv(\pi)$ and $\Ediv(\pi)_\Q := \Q\otimes_\Z \Ediv(\pi)$, respectively.

The intersection product on $X_\pi$ induces a symmetric bilinear form $\Ediv(\pi)\times \Ediv(\pi)\to \Z$ that, crucially, is negative definite \cite{grauert:ubermodifikationen}, thus giving rise to a  canonical identification of $\Ediv(\pi)_\R$ with its dual  $\Ediv(\pi)_\R^*$. If $E_1,\ldots, E_n$ are the exceptional primes of $\pi$, then we will always denote by $\check{E}_1,\ldots, \check{E}_n\in \Ediv(\pi)_\R$ the basis dual to $E_1,\ldots, E_n$. Concretely, $\check{E}_i$ is the unique $\R$-divisor with the property that $\check{E}_i \cdot E_j = \delta_{i\!j}$ for each $j$. This dual basis will appear prominently in later sections. One of its useful properties, an easy consequence of the projection formula, is that if $\pi'$ is a good resolution dominating $\pi$ and if $F_i$ is the strict transform in $X_{\pi'}$ of $E_i$, then $\eta_{\pi\pi'}^*\check{E}_i = \check{F}_i$. Another important property is the following.

\begin{prop}\label{prop:anti-effective}
The divisors $\check{E}_i$ are (strictly) anti-effective, that is, $\check{E}_i\cdot \check{E}_j < 0$ for each $i$ and $j$.
\end{prop}

\begin{rmk}\label{rmk:checkcoordinates}
In fact, we can easily write $\check{E}_i$ in coordinates with respect to the basis $\{E_j\}$, as
$$
\check{E}_i = \sum_j (\check{E}_i \cdot \check{E}_j) E_j.
$$
This formula can be easily proved by intersecting each side of the formula with $\check{E_j}$ for all $j$.
By duality, we also get
$$
E_i = \sum_j (E_i \cdot E_j) \check{E}_j.
$$
\end{rmk}

\refprop{anti-effective} is a consequence of a rather useful linear algebra lemma.

\begin{lem}\label{lem:linear-algebra} Let $V$ be a finite dimensional real inner product space and suppose that $e_1,\ldots, e_n$ is a basis of $V$ with the property that $(e_i\cdot e_j) \leq 0$ for each $i\neq j$. Let $e_1^*,\ldots, e_n^*\in V$ be the corresponding dual basis with respect to the inner product. Then the following statements hold. \begin{enumerate}[itemsep=-1ex]
\item For all $i$ and $j$, one has $(e_i^*\cdot e_j^*)\geq 0$.
\item Let $\varGamma$ be the graph with vertices $e_1,\ldots, e_n$ and an edge connecting two vertices $e_i$ and $e_j$ if and only if $(e_i\cdot e_j) < 0$. Then $(e_i^*\cdot e_j^*) > 0$ if and only if the vertices $e_i$ and $e_j$ belong to the same connected component in $\varGamma$.
\end{enumerate}
\end{lem}
\begin{proof} To start, we assume that $\varGamma$ is connected. We wish to prove that $(e_i^*\cdot e_j^*)>0$ for arbitrary $i$ and $j$, but after re-indexing, we may assume without loss of generality that $j = 1$, and that the sequence of vertices $e_1,\ldots, e_i$ is a path in $\varGamma$ from $e_1$ to $e_i$. If we perform the Gram-Schmidt orthogonalization process on the basis $e_1,\ldots, e_n$, we obtain an orthogonal basis $u_1,\ldots, u_n$ defined recursively by $u_1 = e_1$ and \[
u_j = e_j -\sum_{k < j}\frac{(e_j\cdot u_k)}{|u_k|^2}u_k\] for $j>1$. Writing the $u_j$ out in terms of the basis $e_k$ then gives $u_j =  \sum_{k\leq j}\lambda_{jk}e_k$ for some $\lambda_{jk}\in \R$. By a straightforward induction on $j$, which we leave to the reader, we in fact have (a) that $\lambda_{jk}\geq 0$ for all $k\leq j$, and (b) that if $j\leq i$, then $\lambda_{jk} > 0$ for all $k\leq j$. Note (a) is a consequence of the assumption that $(e_k\cdot e_l)\leq 0$ whenever $k\neq l$, and (b) needs the assumption that $e_1,\ldots, e_i$ is a path in $\varGamma$. Continuing the proof, if we now write the $e_k^*$ in terms of the orthogonal basis $u_j$, we obtain\[
e_k^* = \sum_j \frac{(e_k^*\cdot u_j)}{|u_j|^2}u_j = \sum_{j\geq k} \frac{\lambda_{jk}}{|u_j|^2}u_j.\] Using this and statements (a) and (b) above, one computes that \[
(e_i^*\cdot e_1^*) = \sum_{j\geq i}\frac{\lambda_{ji}\lambda_{j1}}{|u_j|^2}\geq \frac{\lambda_{ii}\lambda_{i1}}{|u_i|^2} > 0.\] This completes the proof in the case when $\varGamma$ is connected.

 Now, let us consider general $\varGamma$. Up to re-indexing, assume that the basis $e_1,\ldots, e_n$ is ordered in blocks corresponding to the connected components of $\varGamma$. Then the matrix $M$ whose $i\!j$-th entry is $(e_i\cdot e_j)$ is a block diagonal matrix, with each diagonal block corresponding to a connected component of $\varGamma$. The inverse $M^{-1}$ is exactly the matrix whose $i\!j$-th entry is $(e_i^*\cdot e_j^*)$. Since $M$ was block diagonal, $M^{-1}$ will also be, and the diagonal blocks of $M^{-1}$ will be the inverses of the diagonal blocks of $M$. By the connected case, each of the diagonal blocks of $M^{-1}$ will have strictly positive entries, completing the proof.
\end{proof}

\begin{proof}[Proof of {\refprop{anti-effective}}] We apply \reflem{linear-algebra} in the case when $V = \Ediv(\pi)_\R$, the given basis  $e_1,\ldots, e_n$ is the set of exceptional primes $E_1,\ldots, E_n$, and the inner product is the negative of the intersection product. In this case, the graph $\varGamma=\varGamma_\pi$ is what is typically referred to as the \emph{dual graph} of $\pi$; it is the graph with vertices $E_1,\ldots, E_n$ and an edge connecting $E_i$ and $E_j$ if and only if the exceptional primes $E_i$ and $E_j$ intersect in $X_\pi$ (see also \S\ref{ssec:dualgraphs}). Because the dual graph of any good resolution $\pi$ is connected, we conclude that $\check{E}_i\cdot \check{E}_j < 0$ for each $i$ and $j$.
\end{proof}

Important for us are notions of (relative) positivity for exceptional divisors, particularly the notions of relative ampleness and nefness, a good reference for which is \cite[\S1.7]{lazarsfeld:positivityalggeo1}. The relative ampleness of a divisor $D\in \Ediv(\pi)$ is defined in terms of projective embeddings, but can be characterized numerically by having $D\cdot E_i > 0$ for each of the exceptional primes $E_i$, a characterization which extends to exceptional $\R$-divisors \cite[Theorem B]{felgueiras:amplecone}. Similarly, an $\R$-divisor $D\in \Ediv(\pi)_\R$ is \emph{relatively nef} if $D\cdot E_i \geq 0$ for each $i$. Evidently, an exceptional $\R$-divisor $D$ is relatively nef (resp.\ relatively ample) if and only if $D=\sum_i \lambda_i \check{E}_i$ with $\lambda_i\geq0$ (resp.\ $>0$). The sets $\mathrm{Nef}(\pi)$ and $\mathrm{Amp}(\pi)$ of relatively nef and ample $\R$-divisors are therefore cones; $\mathrm{Nef}(\pi)$ is a strict convex polyhedral cone, $\mathrm{Amp}(\pi)$ is its (nonempty) interior, 
and $\mathrm{Nef}(\pi)$ is the closure $\overline{\mathrm{Amp}(\pi)}$.  It is then easy to construct bases of $\Ediv(\pi)_\R$ consisting of integral relatively ample divisors, which after scaling can even be taken to be relatively very ample. 

As an immediate consequence of \refprop{anti-effective}, we obtain the following.

\begin{cor}\label{cor:anti-effective2}
If $D_1$ and $D_2$ are nonzero relatively nef divisors, then $D_1\cdot D_2  < 0$. In particular, nonzero relatively nef divisors are strictly anti-effective.
\end{cor}

We conclude this section with another extremely useful consequence of \reflem{linear-algebra}.
\begin{prop}\label{prop:positivity_EFH}
Let $\pi\colon X_\pi \to (X,x_0)$ be a good resolution, and let $E$, $F$ and $H$ be exceptional prime divisors of $\pi$. Then one has the inequality
\begin{equation}\label{eqn:positivity_EFH}
(\check{E} \cdot \check{F})(\check{E} \cdot \check{H}) \leq (\check{E} \cdot \check{E})(\check{F} \cdot \check{H}),
\end{equation}
with equality if and only if $E=H$, $E = F$ or $F$ and $H$ do not belong to the same connected component of $\varGamma_\pi \setminus \{E\}$, where $\varGamma_\pi$ is the dual graph of $\pi$.
\end{prop}
\begin{proof}
If $E=F$, then \eqref{eqn:positivity_EFH} trivially holds as an equality, so we may assume that $E$ and $F$ are distinct.
Let $D_1, \ldots, D_{n-2}$ be the exceptional primes of $\pi$ that are not equal to $E$ or $F$, and set
$$
Q=\check{F} - \frac{\check{E} \cdot \check{F}}{\check{E} \cdot \check{E}}\check{E}.
$$
Notice that \eqref{eqn:positivity_EFH} is equivalent to $Q \cdot \check{H} \leq 0$.

Consider the basis $D_1, \ldots, D_{n-2}, F, -\check{E}$ of $\Ediv(\pi)_\nR$, which for simplicity we denote by $u_1, \ldots, u_n$. This basis has the property that $u_i \cdot u_j \geq 0$ for all $i \neq j$ (while $u_i \cdot u_i < 0$ for all $i$).
Moreover we have $Q \cdot u_i = 0$ for all $i \neq n-1$, and $Q \cdot u_{n-1}= 1$.

If we write $Q=a_1 u_1 + \cdots + a_n u_n$, then
\begin{equation}\label{eqn:positivity_linsys}
\left(\begin{array}{ccc}
u_1\cdot u_1 & \cdots & u_n\cdot u_1\\
\vdots & \ddots & \vdots\\
u_1\cdot u_n & \cdots & u_n\cdot u_n
\end{array}\right)
\left(\begin{array}{c} a_1\\\vdots\\a_n\end{array}\right)
= \left(\begin{array}{c} Q\cdot u_1\\ \vdots \\ Q\cdot u_n\end{array}\right).
\end{equation}
By solving this linear system, we get $a_i = (u_{i}^* \cdot u_{n-1}^*)$.
By \reflem{linear-algebra} applied to the basis $u_1, \ldots, u_n$, and the inner product given by the negative of the intersection product, the matrix $(u_{i}^* \cdot u_j^*)_{ij}$ has non-positive entries.
We conclude that $a_i\leq 0$ for each $i$. 
Because each of the divisors $u_i$ is effective, it follows that $Q = \sum a_i u_i$ is anti-effective. That is to say, $Q\cdot \check{H}\leq 0$ for each exceptional prime $H$ of $\pi$, and the inequality \eqref{eqn:positivity_EFH} is proved.

To determine when we have an equality, we need to be more careful.
First notice that $a_n = 0$.
In fact $u_j \cdot u_n = 0$ for all $j = 1, \ldots, n-1$, while $u_n \cdot u_n = \check{E} \cdot \check{E} < 0$. We deduce that $a_n=0$ from the last line of \eqref{eqn:positivity_linsys}, and the fact that $Q \cdot u_n = 0$.

If we consider the first $n-1$ equations of the linear system \eqref{eqn:positivity_linsys} (for the variables $a_i$, $i=1, \ldots, n-1$) the associated matrix $U_E$ is the intersection matrix of $\ol{\pi^{-1}(x_0) \setminus E}$.
By \reflem{linear-algebra}, $a_i = (u_{i}^* \cdot u_{n-1}^*) < 0$ for $i=1, \ldots, n-1$ if and only if $u_i$ and $F$ belong to the same connected component of $\ol{\pi^{-1}(x_0) \setminus E}$.
We conclude by noticing that $Q \cdot \check{H} = a_i$ if $H=u_i$. 
\end{proof}

\subsection{Log resolutions and divisors}\label{ssec:logresolutions}

Suppose now that $\mf{a}\subset R$ is an $\mf{m}$-primary ideal, and that $\pi\colon X_\pi\to (X, x_0)$ is a log resolution of $\mf{a}$.  Then $\pi^*\mf{a}$ is the ideal sheaf in $\mc{O}_{X_\pi}$ of some exceptional divisor $D\in\Ediv(\pi)$, that is, $\pi^*\mf{a} = \mc{O}_{X_\pi}(-D)$. If, as before, we let $E_1,\ldots, E_n$ be the exceptional primes of $\pi$, the divisor $D$ is given explicitly as $D = \sum_i b_iE_i$, where $b_i$ is the minimum order of vanishing along $E_i$ of all the functions $\phi\circ \pi$ as $\phi$ ranges over $\mf{a}$, i.e., \begin{equation}\label{generic_multiplicities}
b_i = \min_{\phi\in \mf{a}}\,\,\ord_{E_i}(\phi\circ \pi).\end{equation} Necessarily $b_i\geq 1$ for each $i$, since $\mf{a}\subseteq \mf{m}$. 

The divisor $-D$ has positivity properties owing to the fact that, essentially by the definition of being a log resolution, the line bundle $L := \mc{O}_{X_\pi}(-D)$ is \emph{relatively base point free}, meaning that $L$ has sections which do not vanish at any prescribed point $x\in \pi^{-1}(x_0)$. In particular, the restrictions $L|_{E_i}$  are base point free for each $i$. With $E_i$ being a smooth curve, this implies $\deg L|_{E_i}\geq 0$ for each $i$, or equivalently in the language of intersections $-D\cdot E_i\geq 0$ for each $i$. In other words, the divisor $-D$ is relatively nef. We will use the following notation for the divisor $-D$, in keeping with \cite{jonsson:berkovich}.

\begin{defi} If $\pi$ is a log resolution of an $\mf{m}$-primary ideal $\mf{a}$, then $Z_\pi(\mf{a})\in \Ediv(\pi)$ will denote the divisor for which $\pi^*\mf{a} = \mc{O}_{X_\pi}(Z_\pi(\mf{a}))$.
\end{defi}

Conversely, if $\pi$ is a good resolution of $(X,x_0)$ and $Z\in \Ediv(\pi)$ is a divisor for which the line bundle $\mc{O}_{X_\pi}(Z)$ is relatively base point free, then $\pi$ will  be  a log resolution of the ideal $\mf{a} := \pi_*\mc{O}_{X_\pi}(Z)$, moreover with $Z = Z_\pi(\mf{a})$. Note, the condition that $\mc{O}_{X_\pi}(Z)$ be relatively base point free is satisfied by definition when $Z$ is relatively very ample, and thus every relatively very ample divisor $Z\in \Ediv(\pi)$ is of the form $Z = Z_\pi(\mf{a})$ for some $\mf{m}$-primary ideal $\mf{a}$. As we noted at the end of \S\ref{ssec:intersectiontheory}, one can find a basis of $\Ediv(\pi)_\R$ consisting of relatively very ample divisors, thus proving the following.

\begin{prop}\label{prop:basis-of-ideals} For any good resolution $\pi$ of $(X,x_0)$, there exist $\mf{m}$-primary ideals $\mf{a}_1,\ldots, \mf{a}_n$ for which $\pi$ is a log resolution and for which the divisors $Z_\pi(\mf{a}_i)$ form a basis of $\Ediv(\pi)_\R$.
\end{prop}

Finally, let us point out that if $\pi$ is a log resolution of $\mf{a}$ and  $\pi'$ is  any good resolution dominating $\pi$, then $\pi'$ is again a log resolution of $\mf{a}$, and moreover $Z_{\pi'}(\mf{a}) = \eta_{\pi\pi'}^* Z_\pi(\mf{a})$.

\section{Normal surface singularities and their valuation spaces}\label{sec:valuation_spaces}

Let $(X,x_0)$ be a normal surface singularity. In this section we will  introduce and discuss the basic structure of an associated space $\mc{V}_X$ of (semi-)valuations; it is by studying dynamics on this space that we will be able to deduce the main theorems of the paper. In the case when $X = \C^2$, the space $\mc{V}_{\C^2}=\mc{V}$ was introduced and described in detail by Favre-Jonsson, who called it the \emph{valuative tree} \cite{favre-jonsson:valtree, favre-jonsson:eigenval, jonsson:berkovich}. Unfortunately in our singular setting the literature is less extensive. The spaces $\mc{V}_X$ were briefly analyzed by Favre in \cite{favre:holoselfmapssingratsurf}, and they have appeared in a somewhat different vein in the works \cite{fantini:normalizedlinks, thuillier:homotopy, defelipe:topspacesvalgeomsing}. Our aim in this section is to give a fairly detailed treatment of the anatomy of $\mc{V}_X$, in the spirit of \cite[\S7]{jonsson:berkovich}. We will also introduce in \S\ref{angular_metric} a new tool we 
call the \emph{angular metric}, which will prove to be quite useful in the context of dynamics.
As in the previous section, we denote by $R$ the completed local ring $\hat{\mc{O}}_{X,x_0}$ and by $\mf{m}$ its unique maximal ideal.

\begin{defi} A \emph{rank one semivaluation} on $R$ is a function $\nu\colon R\to \R\cup\{+\infty\}$ satisfying $\nu(\phi\psi) = \nu(\phi) + \nu(\psi)$ and $\nu(\phi + \psi)\geq \min\{\nu(\phi), \nu(\psi)\}$ for all $\phi, \psi\in R$. Such a semivaluation is said to be \emph{centered} if in addition $\nu(\phi)\geq 0$ for all $\phi\in R$ and $\nu(\phi)>0$ if and only if $\phi\in \mf{m}$. The collection of all rank one centered semivaluations on $R$ will be denoted $\hat{\mc{V}}_X$.
\end{defi}


Equivalently, one may define rank one centered semivaluations on $R$ not as functions on $R$ itself, but on the set $\ms{I}$ of ideals of $R$. In this formulation, the elements $\nu\in \hat{\mc{V}}_X$ are functions $\nu\colon \ms{I}\to \R\cup\{+\infty\}$ which satisfy $\nu(R) = 0$, $\nu(\mf{m})>0$, $\nu(\mf{a})\geq \nu(\mf{b})$ whenever $\mf{a}\subseteq \mf{b}$, $\nu(\mf{a} + \mf{b}) = \min\{\nu(\mf{a}), \nu(\mf{b})\}$, and $\nu(\mf{a}\mf{b}) = \nu(\mf{a}) + \nu(\mf{b})$. The equivalence between these two formulations is established by defining $\nu(\phi) = \nu(\phi R)$ for any $\phi\in R$ and exploiting the fact that ideals of $R$ are finitely generated. In the following, we will go back and forth between these two perspectives without further comment.

\begin{defi} A semivaluation $\nu\in \hat{\mc{V}}_X$ is \emph{finite} if $\nu(\mf{m})<+\infty$ and \emph{normalized} if $\nu(\mf{m}) = 1$. We let $\hat{\mc{V}}_X^*$ and $\mc{V}_X$ denote the subsets of $\hat{\mc{V}}_X$ consisting of all finite and normalized semivaluations, respectively. Note, there is exactly one semivaluation in $\hat{\mc{V}}_X$ which is not finite, which we denote by $\triv_{x_0}$.
\end{defi}

\begin{rmk}
The interest in centered valuations is quite natural since we study the singularity $(X,x_0)$ locally at $x_0$. Nevertheless, the normalization with respect to the value on the maximal ideal, although very natural, in only one of the possible normalizations we could consider. Other normalizations appear in literature, and we could also work with different normalizations, for example with respect to other $\mf{m}$-primary ideals, or with respect to irreducible curves. See \S\ref{ssec:different_normalizations} for further remarks.
\end{rmk}

The natural topology to put on $\hat{\mc{V}}_X$ is the so-called \emph{weak topology}, which is the weakest topology such that for each $\phi\in R$ the evaluation map $\nu\in \hat{\mc{V}}_X\mapsto \nu(\phi)\in [0,+\infty]$ is continuous (or, equivalently, one may require evaluation on ideals $\mf{a}\subseteq R$ to be continuous). It is not metrizable, but nonetheless possesses very useful properties: in the weak topology, $\mc{V}_X$ is compact Hausdorff, sequentially compact, path connected, and locally contractible. 

Note that any positive scalar multiple of a semivaluation $\nu\in \hat{\mc{V}}_X$ is again a semivaluation in $\hat{\mc{V}}_X$, and consequently any finite semivaluation $\nu$ can be scaled to obtain a normalized semivaluation  $\nu(\mf{m})^{-1}\nu\in \mc{V}_X$. In fact, scalar multiplication gives a homeomorphism $(0,+\infty)\times \mc{V}_X\to \hat{\mc{V}}_X^*$, so  topologically $\hat{\mc{V}}_X^*$ is a cylinder over $\mc{V}_X$. Similarly, $\hat{\mc{V}}_X$ is  a cone over $\mc{V}_X$ with vertex $\triv_{x_0}$.

\subsection{Classification of finite semivaluations}\label{ssec:classification}

In his foundational work on resolution of singularities in dimension two \cite{zariski:resolutionofsingularities}, Zariski classified centered semivaluations $\nu\in\hat{\mc{V}}_X^*$ into four types according to certain associated algebraic invariants of $\nu$ (the \emph{rational rank} and \emph{transcendence degree}), and then characterized these types geometrically. While we will not care about the algebraic invariants here, the geometric characterizations will be crucial. In the terminology of Favre-Jonsson, any $\nu\in \hat{\mc{V}}_X^*$ is either a divisorial valuation, an irrational valuation, an infinitely singular valuation, or a curve semivaluation.

\smallskip

A \emph{divisorial} valuation $\nu\in \hat{\mc{V}}_X^*$ is a valuation which is proportional to the order of vanishing along some exceptional prime divisor in some modification of $(X,x_0)$. Explicitly, $\nu$ is divisorial if there exists some modification $\pi\colon X_\pi\to (X,x_0)$, some exceptional prime divisor $E$ of $X_\pi$, and some constant $\lambda>0$ such that $\nu(\phi) = \lambda\ord_E(\phi\circ \pi)$ for all $\phi\in R$. Note, the order of vanishing $\ord_E$ along $E$ makes sense as we assume $X_\pi$ is normal, and hence $\mc{O}_{X_\pi, E}$ is a DVR for each exceptional prime $E$ of $X_\pi$. Observe also that if $\pi'$ is a modification dominating $\pi$, and if $F$ is the strict transform of $E$ in $X_{\pi'}$, then $\nu(\phi) = \lambda \ord_F(\phi\circ \pi')$ as well, so there is no harm in assuming for instance that $\pi$ is a good resolution of $(X,x_0)$.
Given an exceptional prime $E$, we denote by $\divi_E$ the divisorial valuation defined by $\divi_E(\phi)=\pi_*\ord_E(\phi):=\ord_E(\phi \circ \pi)$, and by $\nu_E$ the associated normalized valuation $\nu_E = b_E^{-1}\divi_E$, where $b_E = \divi_E(\mf{m})$ is called the \emph{generic multiplicity} of $E$ (see also \S\ref{ssec:dualgraphs}).
In the following, we say that a divisorial valuation $\lambda \nu_E$ is \emph{realized} by $\pi$ (or in $X_\pi$), if $\pi\colon X_\pi \to (X,x_0)$ is a modification and $E$ is an exceptional prime of $X_\pi$.

\smallskip

As divisorial valuations are associated to exceptional primes of good resolutions of $(X,x_0)$, \emph{irrational} valuations are associated to intersections of exceptional primes of good resolutions.
Explicitly, suppose that $\pi$ is a good resolution of $(X,x_0)$, and that $E_1$ and $E_2$ are two exceptional primes of $\pi$ that intersect at a point $p\in X_\pi$.
There then exist local holomorphic coordinates $(z_1, z_2)$ of $X_\pi$ around $p$ for which $E_i$ has defining equation $z_i = 0$. For any real $r,s\geq 0$ not both zero, one obtains a valuation $\nu_{r,s}\in \hat{\mc{V}}_X^*$ in the following manner.
For any $\psi = \sum_{i\!j} a_{i\!j}z_1^iz_2^j \in \nC\fps{z_1, z_2}$, we consider the monomial valuation at $p$ defined by
$
\mu_{r,s}(\psi) := \min\{ir + js : a_{i\!j}\neq 0\}
$.
Then $\nu_{r,s} = \pi_* \mu_{r,s}$.

If $r$ and $s$ are rationally independent, then $\nu_{r,s}$ is called an irrational valuation. On the other hand, if $r$ and $s$ are rationally dependent, then in fact $\nu_{r,s}$ is actually a divisorial valuation, though possibly not associated to any exceptional prime of $\pi$.
In either case, we call $\nu_{r,s}$ the \emph{monomial valuation} at $p$ with \emph{weights} $r$ and $s$.
Notice that in general $\nu_{r,s}$ is not normalized. In fact, if $\pi$ is a log resolution of $\mf{m}$, then $\nu_{r,s}(\mf{m}) = b_{E_1}r + b_{E_2}s$.
\begin{defi}
Let $\pi\colon X_\pi \to (X,x_0)$ be a log resolution of $\mf{m}$, and $E,F$ two exceptional primes of $X_\pi$ intersecting at a point $p$.
We denote by $\segment{\nu_E}{\nu_F}{p}$ (or simply $\segment{\nu_E}{\nu_F}{}$) the set of normalized monomial valuations $\nu_{r,s}$ at $p$.
The \emph{monomial parameterization} of $\segment{\nu_E}{\nu_F}{p}$ is the map $w \colon [0,1] \to \segment{\nu_E}{\nu_F}{p}$, defined for all $t \in [0,1]$ as $w(t)=\nu_{r(t),s(t)}$, where $r(t)=(1-t)b_E^{-1}$ and $s(t)=tb_F^{-1}$.
\end{defi}
In the terminology of Favre-Jonsson, divisorial and irrational valuations make up the class of \emph{quasimonomial} valuations, a term reflecting the geometry just described. They are also commonly referred to as \emph{Abhyankar} valuations when thinking of them in terms of algebraic invariants, a name  which alludes to the fact that the sum of the rational rank and transcendence degree is maximal for these valuations. As our perspective is geometric here, we will call them quasimonomial valuations.

\smallskip

A \emph{curve semivaluation} is a semivaluation associated to an irreducible formal curve germ $(C, x_0)$ in $X$ passing through $x_0$. It is well known that any such formal curve germ can be uniformized, that is, there exists a formal parameterization $h\colon (\C, 0)\to (C,x_0)$. Explicitly, a curve semivaluation associated to $C$ is any semivaluation of the form $\nu(\phi) = \lambda\ord_0(\phi|_C\circ h)$ for $\lambda>0$. Observe that curve semivaluations are \emph{not} valuations; in fact their kernel $\nu^{-1}(+\infty)$ is the prime ideal $\mf{p}$ corresponding to the curve germ $(C,x_0)$ in $X$.
We denote by $\inte_C$ the curve semivaluation obtained setting $\lambda = 1$, and by $\nu_C$ the associated normalized valuation $\nu_C = m(C)^{-1}\inte_C$, where $m(C) = \inte_C(\mf{m})$ is the \emph{multiplicity} of $C$.

\smallskip

The only remaining $\nu\in \hat{\mc{V}}_X^*$ are the \emph{infinitely singular} valuations. These have an interpretation in terms of Puiseux series which we will not need and thus omit. We point out, however, they are actual valuations instead of simply semivaluations. See \cite{favre-jonsson:valtree} or \cite[pp. 648-9]{zariski:resolutionofsingularities} for details.

\subsection{Dual divisors associated to valuations and b-divisors}\label{ssec:dualdivisors}

Let $\nu\in \hat{\mc{V}}_X^*$, and let $\pi\colon X_\pi\to (X,x_0)$ be a modification of $(X,x_0)$. By the valuative criterion of properness, $\nu$ has a unique \emph{center} in $X_\pi$, that is, a unique scheme-theoretic point $\xi\in X_\pi$ with the property that $\nu$ takes nonnegative values on the local ring $\mc{O}_{X_\pi, \xi}$ and strictly positive values exactly on its maximal ideal $\mf{m}_\xi$. We will write this $\xi = \cen_\pi(\nu)$. Note, the map $\cen_\pi\colon \hat{\mc{V}}_X^*\to X_\pi$ is typically called the \emph{reduction map} in the context of non-archimedean analytic geometry; it is anti-continuous in the sense that the inverse image of a Zariski open set is closed. If $\pi'$ is a modification dominating $\pi$, then one easily sees that $\cen_\pi = \eta_{\pi\pi'}\circ \cen_{\pi'}$.

Suppose now that $\pi\colon X_\pi\to (X,x_0)$ is a good resolution of $(X,x_0)$. Then $\nu$ can be evaluated on the group $\Ediv(\pi)$ of exceptional divisors by setting $\nu(D) := \nu(\phi)$, where $\phi\in \mc{O}_{X_\pi, \xi}$ is a defining equation of $D\in \Ediv(\pi)$ around $\xi = \cen_\pi(\nu)$. Clearly this evaluation $\nu\colon \Ediv(\pi)\to \R$ is $\Z$-linear, and hence extends linearly to an evaluation map $\nu\colon \Ediv(\pi)_\R\to \R$ on $\R$-divisors. Recalling that the intersection product on $X_\pi$ gives  a natural identification of $\Ediv(\pi)_\R^*$ with $\Ediv(\pi)_\R$, we conclude the following.

\begin{prop}\label{prop:defZnu}
For each $\nu\in \hat{\mc{V}}_X^*$ and each good resolution $\pi\colon X_\pi\to (X,x_0)$, there is a unique $\R$-divisor $Z_\pi(\nu)\in \Ediv(\pi)_\R$ such that $\nu(D) = Z_\pi(\nu)\cdot D$ for each $D\in \Ediv(\pi)_\R$.
\end{prop}

If the center $\xi = \cen_\pi(\nu)$ lies within exactly one exceptional prime $E$ of $\pi$, then  $Z_\pi(\nu) = \lambda \check{E}$ for some $\lambda>0$. Otherwise, $\xi$ must be the intersection point of two exceptional primes, say $E$ and $F$, in which case $Z_\pi(\nu) = r\check{E} + s\check{F}$ for some $r,s>0$. Either way, we see that $Z_\pi(\nu)$ is relatively nef. 

Given $\nu\in \hat{\mc{V}}_X^*$ one can compute $Z_\pi(\nu)$ for any good resolution $\pi$ of $(X,x_0)$, but conversely, knowing the divisors $Z_\pi(\nu)$ for every good resolution $\pi$ allows one to recover the semivaluation $\nu$. In fact, one can evaluate semivaluations on ideals using intersections in the following way. 

\begin{prop}\label{prop:evaluate-with-intersections} Let $\nu\in \hat{\mc{V}}_X^*$ and let $\mf{a}\subset R$ be an $\mf{m}$-primary ideal. Let $\pi\colon X_\pi\to (X,x_0)$ be a log resolution of $\mf{a}$. Then $\nu(\mf{a}) = -Z_\pi(\nu)\cdot Z_\pi(\mf{a})$.
\end{prop}
\begin{proof} Let $\xi$ be the center of $\nu$ in $X_\pi$, so that $\nu$ defines a centered semivaluation on the local ring $\mc{O}_{X_\pi, \xi}$. Tautologically, the value of $\nu$ on $\mf{a}\subset R$ agrees with the value of $\nu$ on the ideal $(\pi^*\mf{a})_\xi\subset \mc{O}_{X_\pi,\xi}$. Because $\pi$ is a log resolution of $\mf{a}$, the ideal sheaf $\pi^*\mf{a}$ is the ideal sheaf of the divisor $-Z_\pi(\mf{a})$, and thus in particular $\nu((\pi^*\mf{a})_\xi) = \nu(-Z_\pi(\mf{a})) = -Z_\pi(\nu)\cdot Z_\pi(\mf{a})$, completing the proof.
\end{proof}

\begin{cor}\label{cor:eval-compatibility} Let $\pi$ and $\pi'$ be good resolutions of $(X,x_0)$, with $\pi'$ dominating $\pi$. Then $(\eta_{\pi\pi'})_*Z_{\pi'}(\nu) = Z_\pi(\nu)$.
\end{cor}
\begin{proof} It suffices to check that $(\eta_{\pi\pi'})_*Z_{\pi'}(\nu)\cdot Z = Z_\pi(\nu)\cdot Z$ for all divisors $Z$ in a given basis $Z_1,\ldots, Z_n$ of $\Ediv(\pi)_\R$. Using \refprop{basis-of-ideals}, we can take the $Z_i$ to be of the form $Z_i = Z_\pi(\mf{a}_i)$ for some $\mf{m}$-primary ideals $\mf{a}_i$ of $R$. Then applying the projection formula and \refprop{evaluate-with-intersections} twice yields \[(\eta_{\pi\pi'})_*Z_{\pi'}(\nu)\cdot Z_\pi(\mf{a}_i) = Z_{\pi'}(\nu)\cdot \eta_{\pi\pi'}^*Z_\pi(\mf{a}_i) = Z_{\pi'}(\nu)\cdot Z_{\pi'}(\mf{a}_i) = -\nu(\mf{a}_i) = Z_\pi(\nu)\cdot Z_\pi(\mf{a}_i),
\] as desired.
\end{proof}

Note, it general it is not true that $Z_{\pi'}(\nu) = \eta_{\pi\pi'}^*Z_\pi(\nu)$. There is, on the other hand, one case in which this equality does hold, namely when $\nu$ is a divisorial valuation associated to an exceptional prime $E$ of $\pi$. In this case, $\nu = \lambda \divi_E$ for some $\lambda>0$, and $Z_\pi(\nu) = \lambda\check{E}$. If $F$ is the strict transform of $E$ in $X_{\pi'}$, then similarly $\nu = \lambda \divi_F$ and $Z_{\pi'}(\nu)= \lambda\check{F}$. The equality $Z_{\pi'}(\nu) = \eta_{\pi\pi'}^*Z_\pi(\nu)$ thus follows from the fact that $\eta_{\pi\pi'}^*\check{E} = \check{F}$, which is in turn a consequence of the projection formula. 

We may define $Z_\pi=Z_\pi(\nu)$ or $Z_\pi=Z_\pi(\mf{a})$ for all modifications $\pi$, by taking any high enough good resolution $\pi'$ dominating $\pi$ where $Z_{\pi'}$ is defined, and setting $Z_\pi=(\eta_{\pi\pi'})_* Z_{\pi'}$. As a consequence of \refcor{eval-compatibility}, $Z_\pi$ is well defined, as it does not depend on the choice of $\pi'$.
The family of divisors $Z=(Z_\pi)_\pi$ is sometimes called a \emph{b-divisor} in the sense of Shokurov \cite{shokurov:prelimitingflips} (see also \cite{boucksom-favre-jonsson:degreegrowthmeromorphicsurfmaps,favre:holoselfmapssingratsurf}).

\begin{defi}
Let $(X,x_0)$ be a normal surface singularity. A \emph{(Weil) (exceptional) b-divisor} $Z$ on $(X,x_0)$ is a collection $Z=(Z_\pi)_\pi$ of Weil $\nR$-divisors $Z_\pi \in \Ediv(\pi)_\nR$ for any modification $\pi \colon X_\pi \to (X,x_0)$, satisfying the relation $(\eta_{\pi\pi'})_*Z_{\pi'}=Z_\pi$ for all modifications $\pi, \pi'$ with $\pi'$ dominating $\pi$.
The divisor $Z_\pi$ is called the \emph{incarnation} of $Z$ in the model $\pi$.

If moreover there exists a modification $\pi$ so that $Z_{\pi'}=\eta_{\pi\pi'}^*Z_{\pi}$ for all modifications $\pi'$ dominating $\pi$, we say that $Z$ is a \emph{Cartier} b-divisor, \emph{determined} by the model $\pi$.

We denote by $\Ediv(X)$ the set of $b$-divisors, and by $\Ediv_C(X)$ its subset of \emph{Cartier} b-divisors.

Finally, we say that a b-divisor $Z=(Z_\pi)_\pi$ is \emph{nef} if $Z_\pi$ is nef for all modifications $\pi$.
\end{defi}

Notice that, to make sense of the definition of Cartier divisors, we have to be able to pull back divisors. Since $(X,x_0)$ is a surface, we can do it numerically in the sense of Mumford.
Alternatively, we may notice that if $\pi$ is a modification which is a determination of a Cartier b-divisor $Z$, then any other modification $\pi'$ dominating $\pi$ is a determination of $Z$. We may hence assume that $\pi$ is a resolution of $(X,x_0)$ in the definition of Cartier b-divisors.
In the following (see \S\ref{ssec:action_bdivisors}), we will consider a more general class of b-divisors. The b-divisors defined here will be the subset of \emph{exceptional} b-divisors, which have the property of having every incarnation exceptional.

\begin{rmk}\label{rmk:exceptional_bdivisors}
We describe here some properties of b-divisors, see \cite{boucksom-favre-jonsson:degreegrowthmeromorphicsurfmaps} for proofs.
The set $\Ediv_C(X)$ of Cartier b-divisor has a natural structure of infinite dimensional $\nR$-vector space.
Denote by $\varGamma_X^*$ the set of exceptional primes over $X$, which is in bijection with the set of divisorial valuations $\mc{V}_X^{\divv}$, through the identification $E \mapsto \nu_E$.
A basis for $\Ediv_C(X)$ is naturally given by dual divisors $\check{E}$, with $E$ varying in $\varGamma_X^*$. Analogously, one can consider the basis given by $Z(\nu_E)$, where $\nu_E$ varies among normalized divisorial valuations in $\mc{V}_X^{\text{div}}$.
Another natural basis for $\Ediv_C(X)$ can be constructed as follows.
Since $X$ is a surface, for any exceptional prime $E \in \varGamma_X^*$, there is a minimal resolution $\pi \colon X_\pi \to (X,x_0)$ realizing $\nu_E$. We set $Z(E)$ to be the b-divisor determined by $E$ in the model $\pi$.
The set $\{Z(E)\ |\ E \in \varGamma_X^*\}$ forms a basis of $\Ediv_C(X)$.
This latter basis has the property of being orthogonal with respect to the intersection form \emph{when $(X,x_0)$ is smooth}.
This no longer holds for singularities, since $Z(E) \cdot Z(F) \neq 0$ if $E$ and $F$ are two exceptional primes realized and intersecting in the minimal resolution of $(X,x_0)$.
Nevertheless, it allows to describe the space of Weil (exceptional) b-divisors $\Ediv(X)$, which is isomorphic to the vector space of functions $\varGamma_X^* \to \nR$. In fact, given a b-divisor $Z$, we may associate to it the function sending $E$ to $\check{E} \cdot Z := \ord_E(Z_\pi)$, where $Z_\pi$ is the incarnation of $Z$ in any model $\pi$ realizing $\nu_E$.

Finally, we may endow $\Ediv(X)$ with a topology, by stating that a sequence of b-divisors $Z_n$ converges to $Z$ if for any $E \in \varGamma_X^*$ we have $\check{E} \cdot Z_n \to \check{E} \cdot Z$.

With respect to this topology, and the weak topology on the valuative space, the map $\hat{\mc{V}}_X^* \to \Ediv(X)$ given by $\nu \mapsto Z(\nu)$ is continuous (see \refprop{evemb}).
\end{rmk}
\begin{rmk}\label{rmk:defZanonprimary}
\refprop{defZnu} can be used to define the b-divisor associated to a non-necessarily $\mf{m}_X$-primary ideal $\mf{a}$.
Let $\pi \colon X_\pi \to (X,x_0)$ be any good resolution.
We define $Z_\pi(\mf{a})$ as the unique divisor supported on $\pi^{-1}(0)$ so that $\ord_{E}(\pi^*\mf{a})=-\check{E} \cdot Z_\pi(\mf{a})$ for all exceptional primes $E$ of $\pi$.
The family $Z(\mf{a})=(Z_\pi(\mf{a}))_\pi$ defines a nef b-divisor which is Cartier if and only if $\mf{a}$ is $\mf{m}_X$-primary. 
\end{rmk}

\subsection{Intersection theory and skewness}

Having the family of divisors $Z_\pi(\nu)$ associated to semivaluations $\nu$ allows us to study semivaluations using the geometry of the vector spaces $\Ediv(\pi)_\R$ in which the $Z_\pi(\nu)$ reside. For instance, given two semivaluations $\nu$ and $\mu$ it is natural to consider the quantities \begin{equation}\label{skewnesses}
\alpha_\pi(\nu) := -Z_\pi(\nu)\cdot Z_\pi(\nu) \hspace{1 cm}\mbox{and}\hspace{1 cm}\beta_\pi(\nu \mid \mu) := \frac{Z_\pi(\nu)\cdot Z_\pi(\nu)}{Z_\pi(\nu)\cdot Z_\pi(\mu)}.
\end{equation} These are both positive. The former is of course the norm squared of $Z_\pi(\nu)$ with respect to the inner product given by the negative of the intersection product, and thus gives some measure of the size of $\nu$. The latter measures the component of $Z_\pi(\mu)$ in the direction of $Z_\pi(\nu)$, giving a means of comparing $\nu$ and $\mu$. 

\begin{prop}\label{prop:monotonic} Let $\nu, \mu\in \hat{\mc{V}}_X^*$. Then for high enough good resolutions $\pi$ of $(X,x_0)$,  the quantities $\alpha_\pi(\nu)$ and $\beta_\pi(\nu\mid \mu)$ increase monotonically to some values $\alpha(\nu)$ and $\beta(\nu\mid\mu)$ belonging to $(0, +\infty]$.
\end{prop}
\begin{proof} First we consider the quantities $\beta_\pi(\nu\mid\mu)$. If $\mu = \lambda\nu$ for some $\lambda>0$, then   $\beta_\pi(\nu\mid \mu) = 1/\lambda$ for all good resolutions $\pi$, so the convergence is clear. Assume therefore that $\mu$ is not a scalar multiple of $\nu$. Then necessarily for some good resolution $\pi_0$ of $(X,x_0)$ the  centers $\xi = \cen_{\pi_0}(\nu)$ and $\zeta = \cen_{\pi_0}(\mu)$ will be distinct. Suppose that $\pi$ is a good resolution dominating $\pi_0$. Then $\eta_{\pi_0\pi}$ factors as $\eta_{\pi_0\pi} = \eta_2\circ \eta_1$, where $\eta_1\colon X_{\pi_1}\to X_{\pi_0}$ is a sequence of point blowups over $\xi$ and $\eta_2\colon X_\pi\to X_{\pi_1}$ is a sequence of point blowups over $\zeta$. Since $\eta_1$ is an isomorphism over a neighborhood of $\zeta$, we have that $Z_{\pi_1}(\mu) = \eta_1^*Z_{\pi_0}(\mu)$ and thus by the projection formula \[
Z_{\pi_1}(\nu)\cdot Z_{\pi_1}(\mu) = Z_{\pi_1}(\nu)\cdot \eta_1^*Z_{\pi_0}(\mu) = \eta_{1*}Z_{\pi_1}(\nu)\cdot Z_{\pi_0}(\mu) = Z_{\pi_0}(\nu)\cdot Z_{\pi_0}(\mu).\] The exact same argument used for $\eta_2$ allows us to conclude that  \[
Z_{\pi}(\nu)\cdot Z_\pi(\mu) = Z_{\pi_1}(\nu)\cdot Z_{\pi_1}(\mu) = Z_{\pi_0}(\nu)\cdot Z_{\pi_0}(\mu).\] We have therefore shown that the quantity $Z_\pi(\nu)\cdot Z_\pi(\mu)$ is constant in $\pi$, so long as $\pi$ dominates $\pi_0$. As a consequence, the denominator in the definition \eqref{skewnesses} of $\beta_\pi(\nu\mid\mu)$ is constant for high enough $\pi$, and the proposition will be proved if we can show that $\alpha_\pi(\nu)$ increases monotonically with $\pi$. But this follows from \refcor{eval-compatibility} and the next lemma. 
\end{proof}

\begin{lem}\label{lem:decreasing} Let $\pi$ and $\pi'$ be good resolutions of $(X,x_0)$, with $\pi'$ dominating $\pi$. For any  $D\in \Ediv(\pi')_\R$, one has $D\cdot D \leq (\eta_{\pi\pi'})_*D\cdot (\eta_{\pi\pi'})_*D$.
\end{lem}
\begin{proof} To ease notation, let $\eta = \eta_{\pi\pi'}$. The point now is that by the projection formula $\eta^*\eta_*D$ is the orthogonal projection of $D$ onto the subspace $\eta^*\Ediv(\pi)_\R\subset \Ediv(\pi')_\R$. Thus $D = \eta^*\eta_*D + R$, where $R$ is orthogonal to $\eta^*\Ediv(\pi)_\R$, so \[
D\cdot D = \eta^*\eta_*D\cdot \eta^*\eta_*D + R\cdot R = \eta_*D \cdot \eta_*D + R\cdot R\leq \eta_*D\cdot \eta_*D,\] with the second equality by the projection formula and the inequality because the intersection product is negative definite.
\end{proof}

\refprop{monotonic} can also interpreted in terms of b-divisors. In fact, given two \emph{nef} b-divisors $Z,W \in \Ediv(X)$, we may define their intersection as $Z \cdot W = \lim_\pi Z_\pi \cdot W_\pi \in [-\infty,0)$. The nef hypothesis assures that the value $Z_\pi \cdot W_\pi$ is non-increasing with respect to $\pi$, and allows the limit to exist.
Notice that by the projection formula, $Z \cdot W$ is always finite as far as at least one between $Z$ and $W$ is Cartier. In this case, $Z \cdot W = Z_\pi \cdot W_\pi$ for any determination $\pi$ of the Cartier divisor.
More generally, from the proof of \refprop{monotonic}, we deduce the following property.
\begin{prop}\label{prop:intersection_differentcenters}
Let $(X,x_0)$ be a normal surface singularity, and $\nu, \mu \in \hat{\mc{V}}^*$ be two semivaluations.
If $\nu$ and $\mu$ have different centers in $X_\pi$ for a good resolution $\pi \colon X_\pi \to (X,x_0)$, then
$$
Z(\nu) \cdot Z(\mu)=Z_\pi(\nu) \cdot Z_\pi(\mu).
$$
\end{prop}

Analogously, \refprop{positivity_EFH} can be stated in terms of b-divisors. To ease statements, we first introduce a definition.
\begin{defi}
Let $(X,x_0)$ be a normal surface singularity, and let $\nu, \mu_1, \mu_2 \in \mc{V}_X$ be three semivaluations.
We say that \emph{$\nu$ disconnects $\mu_1$ and $\mu_2$} if either $\nu=\mu_1$, $\nu=\mu_2$, or $\mu_1$ and $\mu_2$ belong to different connected components of $\mc{V}_X \setminus \{\nu\}$.
\end{defi}

\begin{prop}\label{prop:positivity_bdivisors}
Let $(X,x_0)$ be a normal surface singularity, and $\nu, \mu_1, \mu_2 \in \hat{\mc{V}}_X^*$ be three semivaluations. Set $Z=Z(\nu)$ and $W_j=Z(\mu_j)$ for $j=1,2$.
Then we have
\begin{equation}\label{eqn:positivity_bdivisors}
(Z \cdot W_1)(Z \cdot W_2) \leq (Z \cdot Z)(W_1 \cdot W_2).
\end{equation}
Moreover (taking $\nu, \mu_1, \mu_2$ normalized) the equality holds if and only if $\nu$ disconnects $\mu_1$ and $\mu_2$.
\end{prop}
\begin{proof}
First notice that the statement is homogeneous on $\nu$, $\mu_1$ and $\mu_2$. If the three valuations are divisorial, the statement is just a rephrased version of \refprop{positivity_EFH}.
The general case is obtained taking the limit over the modifications $\pi$, and using the structure of $\mc{V}_X$ as a real graph.
\end{proof}
Notice that \eqref{eqn:positivity_bdivisors} still holds if we replace $W_1$ and $W_2$ with any \emph{nef} b-divisors. In fact any nef b-divisor is a non-negative linear combination of divisors of the form $\check{E}$ for some exceptional primes $E$.

We now come back to the quantity $\alpha(\nu)$ and $\beta(\nu\mid \mu)$. The former is called the \emph{skewness} of $\nu$ in \cite{favre-jonsson:valtree, favre-jonsson:eigenval, favre-jonsson:valmultideals, favre-jonsson:dynamicalcompactifications, jonsson:berkovich}, though it should be noted that in some of these works the definition of skewness differs from ours by a sign. We will call $\beta(\nu\mid \mu)$ the \emph{relative skewness} of $\nu$ relative to $\mu$. It is worth pointing out that the proof of \refprop{monotonic} gives a way of relating skewness and relative skewness.

\begin{cor}\label{cor:skewness_comparison}
Let $\nu, \mu\in \hat{\mc{V}}_X^*$. If $\pi\colon X_\pi\to (X,x_0)$ is a good resolution of $(X,x_0)$ for which $\nu$ and $\mu$ have distinct centers in $X_\pi$, then
\begin{equation}\label{skewness_comparison}
\alpha(\nu) = \beta(\nu\mid\mu)[-Z_\pi(\nu)\cdot Z_\pi(\mu)].
\end{equation}
In particular, if $\nu$ and $\mu$ are not proportional (so that such a good resolution $\pi$ exists) then $\alpha(\nu)$ is finite if and only if $\beta(\nu\mid\mu)$ is finite. 
\end{cor}

Notice that for non-divisorial valuations $\nu\in \hat{\mc{V}}_X^*$, the skewness $\alpha(\nu)$ may be infinite. It turns out that the skewness is always finite for quasimonomial valuations. In fact, one can give a very explicit formula for $\alpha(\nu)$ when $\nu$ is quasimonomial using the next proposition.

\begin{prop}\label{prop:skewness_formula}
Let $\pi\colon X_\pi\to (X,x_0)$ be a good resolution of $(X, x_0)$, and suppose that $E_1$ and $E_2$ are two exceptional primes of $\pi$ which intersect in a point $p\in X_\pi$. Let $\nu_1$ and $\nu_2$ be the monomial valuations at $p$ with weights $r_1, s_1$ and $r_2, s_2$, respectively, as defined in \S\ref{ssec:classification}.
Then
\[
\lim_{\pi'\geq \pi} Z_{\pi'}(\nu_1)\cdot Z_{\pi'}(\nu_2) = Z_{\pi}(\nu_1)\cdot Z_\pi(\nu_2) - \min\{r_1s_2, r_2s_1\},\] the limit taken over all good resolutions $\pi'\geq \pi$. 
\end{prop}
\begin{proof} After reversing the roles of $E_1$ and $E_2$ and/or the roles of $\nu_1$ and $\nu_2$, we may assume we are in one of the following two cases: \begin{enumerate}[itemsep=-1ex]
\item $r_1\geq s_1$ and $s_2 > r_2$, or
\item $r_1\geq s_1$ and $r_2\geq s_2$.
\end{enumerate}
Let $\pi_1$ be the good resolution of $(X,x_0)$ obtained by blowing up $p\in X_\pi$, and let $F$ be the new exceptional prime thus obtained.
Abusing notation, let $E_1$ and $E_2$ also denote the strict transforms of $E_1$ and $E_2$ in $X_{\pi_1}$. Let $p_i$ be the point of intersection of $F$ with $E_i$ for $i = 1,2$. 

Assume first we are in case 1.
Then one easily checks that $\nu_1$ is the monomial valuation at $p_1$ giving weight $r_1-s_1$ to $E_1$ and $s_1$ to $F$, whereas $\nu_2$ is the monomial valuation at $p_2$ giving weight $s_2-r_2$ to $E_2$ and $r_2$ to $F$. In particular, $\nu_1$ and $\nu_2$ have different centers in $\pi_1$, and by \refprop{intersection_differentcenters} we get
\[
\lim_{\pi'\geq \pi} Z_{\pi'}(\nu_1)\cdot Z_{\pi'}(\nu_2) = Z_{\pi_1}(\nu_1)\cdot Z_{\pi_1}(\nu_2) = [(r_1-s_1)\check{E}_1 + s_1\check{F}]\cdot [(s_2-r_2)\check{E}_2 + r_2\check{F}].\] Using the (easy) relation $\check{F} = \check{E}_1 + \check{E}_2 - F$, the right hand side of this equation becomes \begin{align*}
[r_1\check{E}_1 + s_1\check{E}_2 - s_1F]\cdot[r_2\check{E}_1 + s_2\check{E}_2 - r_2 F] & = [r_1\check{E}_1 + s_1\check{E}_2]\cdot[r_2\check{E}_1 + s_2\check{E}_2] - s_1r_2\\
& = Z_\pi(\nu_1)\cdot Z_\pi(\nu_2) - s_1r_2,
\end{align*}
the last equality by the projection formula. This completes the proof in case 1. 

Assume next we are in case 2. Now $\nu_1$ and $\nu_2$ are both monomial valuations at $p_1$ giving weights $r_{11} := r_1 - s_1$ and $r_{21} := r_2 - s_2$ to $E_1$, respectively, and weights $s_{11} := s_1$ and $s_{21} := s_2$ to $F$, respectively.
A similar computation to the one made in case 1 shows that
\[
Z_{\pi_1}(\nu_1)\cdot Z_{\pi_1}(\nu_2) = Z_{\pi}(\nu_1)\cdot Z_{\pi}(\nu_2) - s_1s_2,
\]
from which it follows that
\[
Z_{\pi_1}(\nu_1)\cdot Z_{\pi_1}(\nu_2) - \min\{r_{11}s_{21}, r_{21}s_{11}\} = Z_{\pi}(\nu_1)\cdot Z_{\pi}(\nu_2) - \min\{r_1s_2, r_2s_1\}.
\]
It therefore suffices to prove the proposition for $\pi_1$ instead of $\pi$. 

By iterating the above argument, one obtains a (possibly finite) sequence of good resolutions $\pi_n\geq \pi$ such that for each $n$ the valuations $\nu_1$ and $\nu_2$ are monomial valuations at some point $p_n\in X_{\pi_n}$ with weights $r_{1n}, s_{1n}$ and $r_{2n}, s_{2n}$, respectively, and such that $\pi_{n+1}$ is obtained from $\pi_n$ by blowing up $p_n$. Moreover, we have inductively that
\begin{equation}\label{inductive_step}
Z_{\pi_n}(\nu_1)\cdot Z_{\pi_n}(\nu_2) - \min\{r_{1n}s_{2n}, r_{2n}s_{1n}\} = Z_{\pi}(\nu_1)\cdot Z_\pi(\nu_2) - \min\{r_1s_2, r_2s_1\}
\end{equation}
for each $n$.
The sequence $\pi_n$ terminates at some $\pi_N$ if $\nu_1$ and $\nu_2$ have different centers after blowing up $p_N$; in this case the proposition is proved by case 1 above.
If the sequence $\pi_n$ does not terminate, then $\nu_1$ and $\nu_2$ have the same center in every good resolution of $(X,x_0)$, and hence are proportional. We may therefore assume with no loss in generality that $\nu_1 = \nu_2$, and thus that $r_{1n} = r_{2n}$ and $s_{1n}=s_{2n}$ for every $n$. The identity \eqref{inductive_step} then becomes
\[
Z_{\pi_n}(\nu_1)\cdot Z_{\pi_n}(\nu_2) - r_{1n}s_{1n} = Z_{\pi}(\nu_1)\cdot Z_{\pi}(\nu_2) - r_1s_1
\]
for all $n\geq 1$.
Since $r_{1n}s_{1n}\to 0$ as $n\to \infty$, we obtain in the limit that
\[
\lim_{\pi'\geq \pi} Z_{\pi'}(\nu_1)\cdot Z_{\pi'}(\nu_2) = \lim_{n\to \infty} Z_{\pi_n}(\nu_1)\cdot Z_{\pi_n}(\nu_2)  = Z_{\pi}(\nu_1)\cdot Z_{\pi}(\nu_2) - r_1s_1,\] completing the proof.
\end{proof}

\begin{cor}\label{cor:skewness_formula}
With the same setup as \refprop{skewness_formula}, we have the identities
\begin{equation}\label{eqn:skewness_formula}
\alpha(\nu_i) = -Z_\pi(\nu_i)\cdot Z_\pi(\nu_i) + r_is_i
\end{equation}
and
\begin{equation}\label{eqn:relative_skewness_formula}
\beta(\nu_1\mid\nu_2) = \frac{-Z_\pi(\nu_1)\cdot Z_\pi(\nu_1) + r_1s_1}{-Z_\pi(\nu_1)\cdot Z_\pi(\nu_2) + \min\{r_1s_2, s_1r_2\}}.
\end{equation}
\end{cor}

We denote by $\mc{V}_X^\alpha$ the set of normalized valuations with finite skewness. As a consequence of \refcor{skewness_formula}, $\mc{V}_X^\alpha$ contains the set of quasimonomial valuations $\mc{V}_X^{\qm}$.

We conclude with an important alternate characterization of the 
relative skewness $\beta(\nu\mid \mu)$, as what is sometimes called the 
\emph{relative Izumi constant} of $\nu$ with respect to $\mu$.

%
%

\begin{prop}\label{prop:famous_equality}
Let $\nu, \mu\in \hat{\mc{V}}_X^*$.
Then 
$\beta(\nu\mid \mu)$ can also be computed as
\[
\beta(\nu\mid\mu) = \sup_{\mf{a}} \frac{\nu(\mf{a})}{\mu(\mf{a})},\] the supremum being taken over all $\mf{m}$-primary ideals $\mf{a}$ of $R$.
\end{prop}
\begin{proof} Suppose that $\mf{a}$ is an $\mf{m}$-primary ideal, and let $\pi$ be a log resolution of $\mf{a}$.
By \refprop{evaluate-with-intersections},
\[
\frac{\nu(\mf{a})}{\mu(\mf{a})} = \frac{Z_{\pi}(\nu)\cdot Z_{\pi}(\mf{a})}{Z_\pi(\mu)\cdot Z_{\pi}(\mf{a})}.
\]
Consider the hyperplane $H\subset \Ediv(\pi)_\R$ consisting of all divisors $D$ such that $Z_\pi(\mu)\cdot D = Z_\pi(\mu)\cdot Z_\pi(\mf{a})$. It follows from \refcor{anti-effective2}, that $H\cap \mathrm{Nef}(\pi)$ is a cross-section of the nef cone in $\Ediv(\pi)_\R$. The map \[
D\in H\cap \mathrm{Nef}(\pi) \mapsto \frac{Z_{\pi}(\nu)\cdot D}{Z_{\pi}(\mu)\cdot D} = \frac{ Z_{\pi}(\nu) \cdot D}{Z_{\pi}(\mu)\cdot Z_{\pi}(\mf{a})}\] is maximized by Cauchy-Schwarz precisely when $D$ is proportional to  $Z_{\pi}(\nu)$, and in this case \[
\frac{Z_\pi(\nu)\cdot D}{Z_{\pi}(\mu)\cdot D} = \frac{Z_\pi(\nu)\cdot Z_{\pi}(\nu)}{Z_\pi(\mu)\cdot Z_\pi(\nu)} = \beta_\pi(\nu\mid \mu).\] Thus proves that $\nu(\mf{a})/\mu(\mf{a})\leq \beta_\pi(\nu\mid \mu)$, and we can conclude that $\sup_\mf{a} \nu(\mf{a})/\mu(\mf{a}) \leq \beta(\nu\mid \mu)$. To show the opposite inequality, fix a good resolution $\pi$ of $(X,x_0)$. Since relatively ample $\Q$-divisors are dense in the nef cone of $\pi$, there exist relatively ample $\Q$-divisors $D_n$ with $D_n\to Z_{\pi}(\nu)$ as $n\to \infty$. For each $n$, let $k_n\in \N$ be a large enough integer that $k_nD_n$ is an integral relatively very ample divisor, and hence $k_nD_n = Z_\pi(\mf{a}_n)$ for some $\mf{m}$-primary ideal $\mf{a}_n$. Then \[
\frac{\nu(\mf{a}_n)}{\mu(\mf{a}_n)} = \frac{Z_\pi(\nu)\cdot Z_\pi(\mf{a}_n)}{Z_\pi(\mu)\cdot Z_\pi(\mf{a}_n)} = \frac{Z_\pi(\nu)\cdot D_n}{Z_\pi(\mu)\cdot D_n}\to \frac{Z_\pi(\nu)\cdot Z_\pi(\nu)}{Z_\pi(\mu)\cdot Z_\pi(\nu)} = \beta_\pi(\nu\mid \mu). \] Since $\beta(\nu\mid \mu) = \sup_\pi \beta_\pi(\nu\mid \mu)$, it follows immediately that $\beta(\nu\mid\mu) \leq \sup_\mf{a} \nu(\mf{a})/\mu(\mf{a})$. 
\end{proof}

\begin{rmk}
A similar argument can be deployed to prove an analogous statement for the skewness $\alpha$. In fact, the intersection $-Z_\pi(\mf{m}) \cdot Z_\pi(\mf{a})$ does not depend on the chosen log resolution $\pi$ of $\mf{m}$. One can define $m(\mf{a})$ as such value.
Then one can prove 
$$
\frac{\alpha(\nu)}{\nu(\mf{m})}=\sup_\mf{a}\frac{\nu(\mf{a})}{m(\mf{a})}.
$$ 
It suffices to replace $Z_\pi(\mu)$ by $Z_\pi(\mf{m})$ in the proof of \refprop{famous_equality}, and recall that $\nu(\mf{m})=-Z_\pi(\mf{m}) \cdot Z_\pi(\nu)$ for any log resolution $\pi$ of $\mf{m}$.
One can also replace the supremum over $\mf{m}$-primary ideals $\mf{a}$ with the supremum over elements $\phi$ in $\mf{m}$, and accordingly $\nu(\mf{a})$ with $\nu(\phi)$.
In fact, by \refprop{basis-of-ideals}, for any good desingularization $\pi$ one can find a $\mf{m}$-primary ideal $\mf{a}_\pi$ so that $\pi^*\phi = \mc{O}_{X_\pi}(Z_\pi(\mf{a}_\pi))$. Moreover, if $\pi'$ dominates $\pi$, then $\eta_{\pi\pi'}^*Z_\pi(\mf{a}_\pi) - Z_{\pi'}(\mf{a}_{\pi'})$ is effective (see also \refprop{order_effective}).
Alternatively, one can notice that $\nu(\phi)$ can be computed as the limit (the supremum) over good resolutions $\pi$ of $-Z_\pi(\nu)\cdot Z_\pi(\mf{a}_\pi)$, where $\mf{a}_\pi = (\phi) + \mf{m}^n$ for $n$ big enough (depending on $\pi$). 

Finally, if $\nu$ is a curve semivaluation and $\mu$ is another semivaluation, then $\beta(\nu\mid \mu) = +\infty$; indeed, if $\mf{p}\subset R$ is the kernel of $\nu$, one need only note that for the primary ideals $\mf{a}_n = \mf{p} + \mf{m}^n$ one has $\nu(\mf{a}_n)/\mu(\mf{a}_n)\to +\infty$ as $n\to \infty$. Since $\beta(\nu\mid \mu)$ is finite if and only if $\alpha(\nu)$ is finite, we have proved that curve semivaluations always have infinite skewness. On the other hand, infinitely singular valuations can have either finite or infinite skewness. In particular, having finite skewness is not equivalent to being quasimonomial.
\end{rmk}

\subsection{Weak topology and tangent vectors}\label{ssec:tangent_vectors}

We now focus our attention to the space of normalized semivaluations $\mc{V}_X$. The utility of this space lies in its topological properties (it is compact Hausdorff) as well as the fact that its structure reflects the combinatorics of good resolutions of $(X,x_0)$.
We start by describing in more details the weak topology on $\mc{V}_X$ introduced above.
\begin{defi}
Let $(X,x_0)$ be a normal surface singularity, $\pi \colon X_\pi \to (X,x_0)$ be a proper birational map, and $p \in \pi^{-1}(x_0)$ a point in the exceptional divisor of $\pi$. We set
$$
U_\pi(p)=\{\nu \in \mc{V}_X\ |\ \cen_\pi(\nu)=p\}.
$$
\end{defi}
Notice that the point $p$ does not need to be a smooth point of $X_\pi$. 
The family of sets $U_\pi(p)$ where $\pi$ varies among good resolutions and $p$ among points in $\pi^{-1}(x_0)$ form a prebasis for the weak topology on $\mc{V}_X$.
A (non-empty) connected weak open subset of $\mc{V}_X$ is the connected component of the complement of a finite set in $\mc{V}_X$.
We will need a notation for such connected weak open sets.
In what follows, a \emph{connected subgraph} of $\mc{V}_X$ is a weakly closed connected subset of $\mc{V}_X$ whose set of endpoints $\partial S$ is finite.
\begin{change}
An \emph{endpoint} of $S$ is a point $s \in S$ that does not disconnect any connected subset $U$ of $S$, i.e., for any connected subset $U \subseteq S$, $U \setminus \{s\}$ remains connected.
\end{change}
\begin{defi}
Let $(X,x_0)$ be a normal surface singularity, and let $S \subset \mc{V}_X$ be a connected subgraph.
We denote by $U(S)$ the connected weak open subset of $\mc{V}_X$ given by the connected component of $\mc{V}_X \setminus \partial S$ containing $S \setminus \partial S$.
\end{defi}
We will use extensively this notation when $I$ is an interval with divisorial endpoints. In this case, $U(I)=U_\pi(p)$, where $\pi\colon X_\pi \to (X,x_0)$ is a modification obtained from any good resolution $\pi' \colon X_\pi' \to (X,x_0)$ that realizes the endpoints of $I$, and contracting the chain of exceptional primes $E$ of $X_{\pi'}$ so that $\nu_E$ belongs to the interior part of $I$. In this case $p$ is the image through $\eta_{\pi\pi'}$ of the contracted chain of exceptional primes.

This description of the weak topology is related to the construction of tangent vectors attached to (divisorial) valuations, see \cite[Section 1.6]{favre-jonsson:eigenval} for the analogous construction in the smooth case.
Let $\nu \in \mc{V}_X$ be any normalized semivaluation. Let $U$ be any connected weak open neighborhood of $\nu$, and let $T_\nu(U)$ be the set of connected components of $U \setminus \{\nu\}$.
Given another connected open neighborhood $V \subset U$, there is a natural map $j_{V,U}\colon T_\nu(V) \to T_\nu(U)$ which sends a connected component $W$ of $V\setminus\{\nu\}$ to the unique connected component $W'$ of $U \setminus \{\nu\}$ which contains $W$.
The map $j_{V,U}$ is always surjective but not necessarily injective, 
the obstruction 
given by the presence of circles in $\mc{V}_X$.
The maps $j_{V,U}$ while $U,V$ vary among nested connected weakly open neighborhoods of $\nu$ form an inverse system, which allows the following definition.
\begin{defi}
Let $(X,x_0)$ be a normal surface singularity, and let $\mc{V}_X$ be its space of normalized semivaluations.
The \emph{tangent space} $T_\nu \mc{V}$ of $\mc{V}_X$ at $\nu$ is given by the projective limit
$$
T_\nu \mc{V}_X = \lim_{\substack{\longleftarrow\\U}} T_\nu(U).
$$
Its elements are called \emph{tangent vectors}.
\end{defi}

\begin{rmk}
The definition of tangent vector given above works in more general settings. For valuation spaces of normal surface singularities, the tangent space can be constructed also without taking projective limits.
This follows from the fact that the maps $j_{V,U}$ stabilize in the projective limit, meaning that there exists $U_0$ connected weakly open neighborhood of $\nu$ for which $j_{V,U}$ are bijections for every $V \subseteq U \subseteq U_0$.
Such open $U_0$ can be constructed by taking any good resolution $\pi\colon X_\pi \to (X,x_0)$, and by letting $U_0$ be the connected component of $(\mc{V}_X \setminus \mc{S}_\pi^*) \cup \{\nu\}$ containing $\nu$.
Here $\mc{S}_\pi^*$ denotes the set of divisorial valuations $\nu_E$ associated to exceptional primes $E$ of $X_\pi$. 
\end{rmk}

One can distinguish the type of valuation $\nu$ in $\mc{V}_X$ (non quasimonomial, irrational and divisorial respectively) according to the number of elements in $T_\nu \mc{V}_X$ (one, two, or infinitely many respectively).
In particular, for a divisorial valuation $\nu=\nu_E$, there is a natural $1$-to-$1$ correspondence between points in $E$ and tangent vectors in $T_\nu\mc{V}_X$. Explicitly, the point $p$ corresponds to the tangent vector determined by the set of valuations $U_\pi(p)$ whose center in $X_\pi$ is $p$.

\subsection{Dual graphs and the structure of $\mc{V}_X$}\label{ssec:dualgraphs}

We investigate here the relations between the dual graphs associated to good resolutions and the structure of valuation spaces.
Let $\pi\colon X_\pi\to (X,x_0)$ be any log resolution of $\mf{m}$. From \eqref{generic_multiplicities}, we know $Z_\pi(\mf{m}) = -\sum b_{E}E$, where the sum is taken over all exceptional primes $E$ of $\pi$ and where $b_{E} = \divi_{E}(\mf{m})\geq 1$ is the minimal order of vanishing of any function in $\mf{m}$ along $E$. This positive integer is called the \emph{generic multiplicity} of $E$.
Recall that the dual graph $\varGamma_\pi$ of $\pi$ is defined abstractly as the graph whose vertices are the exceptional primes of $\pi$ and with an edge connecting two distinct vertices $E$ and $F$ if and only if $E$ and $F$ intersect in $X_\pi$.
For us, it will be more convenient to realize the dual graph explicitly as a subset $\Gamma_\pi\subset \Ediv(\pi)_\R$ in the following manner.
In this section, we will assume the good resolutions $\pi \colon X_\pi \to (X,x_0)$ to satisfy the stronger property of having any two exceptional primes intersecting in at most one point (see \refrmk{mingoodres_notunique}).

\begin{defi}
The \emph{dual graph} $\Gamma_\pi$ of $\pi$ is the subset of $\Ediv(\pi)_\R$ consisting of the divisors $b_{E}^{-1}\check{E}$ for each of the exceptional primes $E$, as well as the straight line segment connecting $b_{E}^{-1}\check{E}$ and $b_{F}^{-1}\check{F}$ if $E$ and $F$ intersect in $X_\pi$. 
\end{defi}

\begin{prop}\label{prop:evemb}
The map $\ev_\pi\colon \mc{V}_X\to \Ediv(\pi)_\R$ which takes $\nu\in \mc{V}_X$ to its associated divisor $Z_\pi(\nu)$ is continuous and has $\Gamma_\pi$ as its image. Moreover, there is a continuous embedding $\emb_\pi\colon \Gamma_\pi\to \mc{V}_X$ such that $\ev_\pi\circ \emb_\pi$ is the identity on $\Gamma_\pi$.
\end{prop} 

The notation $\ev_\pi$ and $\emb_\pi$ is chosen to agree with that of \cite{jonsson:berkovich},  where $\ev_\pi$ is called \emph{evaluation}.

\begin{proof} To see that $\ev_\pi$ is continuous is suffices to check that $Z_\pi(\nu)\cdot Z$ varies continuously with $\nu$ for all divisors $Z$ in a basis $Z_1,\ldots, Z_n\in \Ediv(\pi)_\R$. By \refprop{basis-of-ideals}, we may choose such a basis consisting of divisors of the form $Z_i  = Z_\pi(\mf{a}_i)$ for $\mf{m}$-primary ideals $\mf{a}_i\subset R$. Applying \refprop{evaluate-with-intersections}, we have $Z_\pi(\nu)\cdot Z_i = -\nu(\mf{a}_i)$, which varies continuously in $\nu$ by the definition of the weak topology. Therefore $\ev_\pi$ is continuous.

Let $\nu\in \mc{V}_X$ and let $\xi = \cen_\pi(\nu)$. Assume first that $\xi$ lies within a unique exceptional prime $E$ of $\pi$. In this case we have seen $Z_\pi(\nu) = \lambda \check{E}$ for some $\lambda>0$. However, because $\nu$ is normalized, we must have $1 = \nu(\mf{m}) = -Z_\pi(\nu)\cdot Z_\pi(\mf{m}) = \lambda b_E$, proving that $Z_\pi(\nu) = b_E^{-1}\check{E}\in \Gamma_\pi$.
Assume next that $\xi$ is the intersection point of two exceptional primes $E$ and $F$ of $\pi$. Then $Z_\pi(\nu) = r\check{E} + s\check{F}$ for some $r,s>0$ satisfying the normalization condition $1 = \nu(\mf{m}) =  -Z_\pi(\nu)\cdot Z_\pi(\mf{m}) = rb_E + sb_F$. This says exactly that $Z_\pi(\nu)$ lies on the straight line segment between $b_E^{-1}\check{E}$ and $b_F^{-1}\check{F}$, and thus lies within $\Gamma_\pi$. To complete the proof, we need only construct the embedding $\emb_\pi$. This is done as follows. First, the vertex points $b_E^{-1}\check{E}$ in $\Gamma_\pi$ are mapped by $\emb_\pi$ to the normalized divisorial valuation $b_E^{-1}\divi_E$. Then, the edge points $r\check{E} + s\check{F}$ are mapped by $\emb_\pi$ to the monomial valuation $\nu_{r,s}$ at the point $E\cap F$ that was defined in \S\ref{ssec:classification}. It is trivial to check that $\emb_\pi$ has the desired properties.
\end{proof}

Observe that if $\pi'$ is a good resolution dominating $\pi$, then $\ev_{\pi} = (\eta_{\pi\pi'})_*\circ \ev_{\pi'}$ by \refcor{eval-compatibility}. In particular, $(\eta_{\pi\pi'})_*$ maps $\Gamma_{\pi'}$ onto $\Gamma_\pi$. In this way, the dual graphs $\Gamma_\pi$ and the maps $(\eta_{\pi\pi'})_*$ form an inverse system, allowing us to speak of the inverse limit $\varprojlim \Gamma_\pi$.
Notice that by construction $\varprojlim \Gamma_\pi$ naturally embeds in the set of (exceptional) $\nR$-b-divisors $\Ediv(X)$.
Moreover, by the universal property of the inverse limit, the  maps $\ev_\pi$ induce a continuous map $\ev\colon \mc{V}_X\to \varprojlim \Gamma_\pi$. The structure of $\mc{V}_X$ is then given by the following theorem. 

\begin{thm}\label{thm:structure} The map $\ev\colon \mc{V}_X\to \varprojlim \Gamma_\pi$ is a homeomorphism.
\end{thm}

The proof of this theorem can be found in \cite[Theorem 7.9]{jonsson:berkovich} for $X = \C^2$, but in fact the argument works equally well in our singular setting. As a consequence of the theorem, the map $r_\pi := \emb_\pi\circ \ev_\pi$ is a retraction of $\mc{V}_X$ onto the \emph{embedded dual graph} $\mc{S}_\pi := \emb_\pi(\Gamma_\pi) \subset \mc{V}_X$.
More difficult to see is that in fact $\mc{V}_X$ \emph{deformation retracts} onto $\mc{S}_\pi$ (see \cite{berkovich:book, thuillier:homotopy}) and thus in particular has the homotopy type of the finite graph $\mc{S}_\pi\cong \Gamma_\pi$. 

Notice that the map $\ev$ coincides with the restriction of the map $\hat{ev}\colon \hat{\mc{V}}_X \to \Ediv(X)$ considered at the end of \S \ref{ssec:dualdivisors} to the set of normalized semivaluations $\mc{V}_X$.

\begin{rmk}
If $\pi_0$ is a minimal good resolution of $(X,x_0)$ dominated by $\pi$, then, using the fact that $\pi$ is obtained from $\pi_0$ by a composition of point blowups, one sees that $\Gamma_\pi$ deformation retracts onto a subset homeomorphic to the dual graph of $\pi_0$, and thus in fact $\mc{V}_X$ has the homotopy type of the dual graph of any minimal good resolution. Note, any two minimal good resolutions of $(X,x_0)$ have homeomorphic dual graphs, an easy consequence of \cite[Theorem 5.12]{laufer:normal2dimsing}.
Moreover, minimal good resolutions are obtained from the unique good resolution in the sense of Laufer by a sequence of blowups of \emph{satellite points} (see \refrmk{mingoodres_notunique}). Hence the image through $\emb_\pi$ of the dual graph $\Gamma_{\pi_0}$ of any minimal good resolution $\pi_0$ does not depend on $\pi$. We denote such image as $\skel{X}$, and refer to it as the \emph{skeleton} of $\mc{V}_X$.
\end{rmk}


For any log resolution $\pi$ of $\mf{m}$, we make $\Gamma_\pi$ into a \emph{metric} graph by specifying the lengths of its edges: if $E$ and $F$ are distinct intersecting exceptional primes of $\pi$, the length of the edge from $b_E^{-1}\check{E}$ to $b_F^{-1}\check{F}$ is set to be $1/(b_Eb_F)$. The metric has the very useful property that if $\pi'$ is a good resolution dominating $\pi$, then the continuous map $\ev_{\pi'}\circ \emb_\pi$ is an isometric embedding $i_{\pi'\pi}\colon \Gamma_\pi\to \Gamma_\pi'$. One can prove this easily when $\pi'$ is obtained from $\pi$ by a point blowup; the general case then follows by induction. Also by induction one sees that the inclusions $i_{\pi'\pi}$ are compatible in  that $\emb_\pi = \emb_{\pi'}\circ i_{\pi'\pi}$. It therefore makes sense to speak of a direct limit $\varinjlim \Gamma_\pi$ and an induced continuous map $\emb\colon \varinjlim \Gamma_\pi\to \mc{V}_X$. Note that by definition a valuation $\nu\in \mc{V}_X$ is quasimonomial if and only if it lies in the 
image $\mc{S}_\pi$ of $\emb_\pi$ for some $\pi$, so the image of $\emb\colon \varinjlim\Gamma_\pi\to \mc{V}_X$ is precisely the set $\mc{V}_X^\qm$ of quasimonomial valuations. It follows from this and \refthm{structure} that quasimonomial valuations are dense in $\mc{V}_X$. In fact, one can show that all four types of semivaluations are dense in $\mc{V}_X$.

\begin{thm} The map $\emb\colon \varinjlim \Gamma_\pi\to \mc{V}_X^\qm$ is a continuous bijection, but is not a homeomorphism.
\end{thm}

Of course, the map $\emb$ cannot be a homeomorphism, since $\varinjlim \Gamma_\pi$ is by construction a metric space, whereas the weak topology is not metrizable. However we may push forward the metric through $\emb$ and obtain a new, strictly stronger topology on the space of quasimonomial valuations.

\begin{defi} Throughout this article, we denote by $d$ the metric on $\mc{V}_X^\qm$ constructed in this manner. Its induced topology will be called the \emph{strong topology}.
\end{defi}

\begin{prop} The skewness  $\alpha\colon \mc{V}_X\to (0,+\infty]$  and relative skewness  $\beta\colon \mc{V}_X\times \mc{V}_X\to (0, +\infty]$ functions are lower semicontinuous. When $\alpha$ is restricted to $\mc{V}_X^\qm$ and $\beta$ is restricted to $\mc{V}_X^\qm\times \mc{V}_X^\qm$, they are both continuous in the strong topology. 
\end{prop}
\begin{proof} For each log resolution $\pi$ of $\mf{m}$, the evaluation map $\ev_\pi$ is continuous, and hence so are the maps $\alpha_\pi$ and $\beta_\pi$ defined in \eqref{skewnesses}. Since $\alpha$ and $\beta$ are the supremum of $\alpha_\pi$ and $\beta_\pi$, respectively, they are lower semicontinuous. The continuity in the strong topology on quasimonomial valuations is immediate from \eqref{eqn:skewness_formula} for $\alpha$ and \eqref{eqn:relative_skewness_formula} for $\beta$. 
\end{proof}

\subsection{Partial order, trees, and parameterizations}\label{ssec:partial_order}

The space $\mc{V}_X$ carries a natural partial ordering, namely where $\nu\geq \mu$ if and only if $\nu(\mf{a})\geq \mu(\mf{a})$ for all $\mf{m}$-primary ideals $\mf{a}$ of $R$. One can give an equivalent geometric interpretation of this partial order in the following manner.

\begin{prop}\label{prop:order_effective}
For semivaluations $\nu, \mu\in \mc{V}_X$, one has $\nu\geq \mu$ if and only if the divisor $Z_\pi(\mu) - Z_\pi(\nu)$ is effective for each good resolution $\pi$ of $(X,x_0)$. 
\end{prop}
\begin{proof} Suppose first that $Z_\pi(\mu) - Z_\pi(\nu)$ is effective for each $\pi$. Let $\mf{a}$ be any $\mf{m}$-primary ideal of $R$, and let $\pi$ be a log resolution of $\mf{a}$. Then by \refprop{evaluate-with-intersections} we see that $\nu(\mf{a}) - \mu(\mf{a}) = [Z_\pi(\mu) - Z_\pi(\nu)]\cdot Z_\pi(\mf{a})$. Since $Z_\pi(\mf{a})$ is relatively nef and $Z_\pi(\mu) - Z_\pi(\nu)$ is effective, this quantity is $\geq 0$, as desired. 

Conversely, assume that $\nu\geq \mu$, and let $\pi$ be any good resolution of $(X,x_0)$. Let $E$ be an exceptional prime of $\pi$. Since relatively ample $\Q$-divisors are dense in $\mathrm{Nef}(\pi)$, we can find a sequence of relatively ample divisors $D_n\in \Ediv(\pi)_\Q$ such that $D_n\to\check{E}$ as $n\to \infty$. Let $k_n\geq 1$ be an integer for which $k_nD_n$ is an integral relatively very ample  divisor, and thus such that $k_nD_n = Z_\pi(\mf{a}_n)$ for some $\mf{m}$-primary ideal $\mf{a}_n$. We then have by \refprop{evaluate-with-intersections} that \[1\geq \frac{\mu(\mf{a}_n)}{\nu(\mf{a}_n)} = \frac{Z_\pi(\mu)\cdot Z_\pi(\mf{a_n})}{Z_\pi(\nu)\cdot Z_\pi(\mf{a}_n)} = \frac{Z_\pi(\mu)\cdot D_n}{Z_\pi(\nu)\cdot D_n}\to \frac{Z_\pi(\mu)\cdot \check{E}}{Z_\pi(\nu)\cdot \check{E}},\] proving that $[Z_\pi(\mu) - Z_\pi(\nu)]\cdot \check{E} \geq 0$. As $E$ was an arbitrary exceptional prime of $\pi$, we conclude that $Z_\pi(\mu)  - Z_\pi(\nu)$ is effective. 
\end{proof}

Given this geometric characterization of the partial order, it makes sense to expect a tight relationship between the partial ordering and the behavior of the skewness and relative skewness functions $\alpha$ and $\beta$, which were defined in terms of the geometry of the dual divisors $Z_\pi(\nu)$. We investigate this interaction now.

\begin{lem}\label{lem:monotonicity}
Let $\nu, \mu\in \mc{V}_X$ be semivaluations. Then:
\begin{enumerate}[itemsep=-1ex]
\item The relative skewness $\beta$ satisfies $\beta(\mu\mid\nu)\geq 1$ with equality if and only if $\nu\geq \mu$. 
\item If $\nu > \mu$, then $\alpha(\mu) = -Z_\pi(\nu)\cdot Z_\pi(\mu)$ for any good resolution $\pi\colon X_\pi\to (X,x_0)$ such that $\mu$ and $\nu$ have distinct centers in $X_\pi$. 
\item The skewness function is monotonic in that $\alpha(\nu) > \alpha(\mu)$ if $\nu > \mu$. 
\end{enumerate}
\end{lem}
\begin{proof} The first statement  follows immediately from  \refprop{famous_equality}. Note, the inequality $\beta(\mu\mid \nu) \geq 1$ is using the assumption that $\mu$ and $\nu$ are normalized, so that $\mu(\mf{m})/\nu(\mf{m}) = 1$.  The second statement is then a consequence of the first and \refcor{skewness_comparison}. For the third, take a good resolution $\pi\colon X_\pi\to (X,x_0)$ such that $\nu$ and $\mu$ have different centers in $X_\pi$ and apply \refcor{skewness_comparison} again to derive
\[
\alpha(\nu) = \beta(\nu\mid \mu)[-Z_\pi(\nu)\cdot Z_\pi(\mu)] = \beta(\nu\mid \mu)\alpha(\mu) > \alpha(\mu),
\]
completing the proof. 
\end{proof}

We focus our attention to skewness and relative skewness of normalized valuations which are monomial with respect to some fixed infinitely near point $p$.
The general situation can be described as follows.

\begin{prop}\label{prop:skewness_parameter}
Let $\pi\colon X_\pi\to (X,x_0)$ be a log resolution of $\mf{m}$, and suppose that $E$ and $F$ are two exceptional primes of $\pi$ which intersect at a point $p\in X_\pi$. Let $w\colon [0,1] \to \segment{\nu_E}{\nu_F}{p}$ be the monomial parameterization at $p$, and let $\mu \in \mc{V}_X$ be any valuation of finite skewness.
\begin{enumerate}
\item The function $[0,1] \ni t \mapsto \alpha(w(t))$ is a polynomial map of degree one or two.
\item The function $[0,1] \ni t \mapsto \beta(w(t) \mid \mu)$ is a piecewise rational map of degree $\leq 2$.
\item The function $[0,1] \ni t \mapsto \beta(\mu \mid w(t))$ is a piecewise rational map of degree $\leq 1$.
\end{enumerate}
Moreover, stationary points for these maps correspond to rational parameters.
\end{prop}
\begin{proof}
Throughout the proof, we will let $\Delta$ denote the constant
\[
\Delta := \frac{\check{E}^2}{b_E^2} - \frac{\check{E}\cdot \check{F}}{b_Eb_F}.
\]
Observe that $\Delta \leq 0$ with equality if and only if $\nu_E \leq  \nu_F$; indeed, this is simply another way of writing the fact that $\beta(\nu_E\mid \nu_F) \geq 1$, with equality if and only if $\nu_E \leq \nu_F$. We now rewrite \eqref{eqn:skewness_formula} and \eqref{eqn:relative_skewness_formula} for valuations in $\segment{\nu_E}{\nu_F}{p}$ in terms of $\Delta$:
\begin{equation}\label{eqn:skewness_formula2}
\alpha(w(t)) = \left[\alpha(\nu_F) - \alpha(\nu_E) - 2\Delta - \frac{1}{b_Eb_F}\right]t^2 + \left[2\Delta + \frac{1}{b_Eb_F}\right] t + \alpha(\nu_E)
\end{equation}
and
\begin{equation}\label{eqn:relative_skewness_formula2}
\beta(w(t) \mid w(s)) = \frac{\left[\alpha(\nu_F) - \alpha(\nu_E) - 2\Delta - \frac{1}{b_Eb_F}\right]t^2 + \left[2\Delta + \frac{1}{b_Eb_F}\right]t + \alpha(\nu_E)}{\left[\alpha(\nu_F) - \alpha(\nu_E) - 2\Delta - \frac{1}{b_Eb_F}\right]st + \left[\Delta s + \Delta t + \frac{\min\{t,s\}}{b_Eb_F}\right] + \alpha(\nu_E)}.
\end{equation} 
The first assertion is a direct consequence of \eqref{eqn:skewness_formula2}, while the last two assertions directly follow from \eqref{eqn:relative_skewness_formula2} if $\mu=w(s)$ for some $s\in[0,1]$.

If it is not the case, then up to taking another good resolution $\pi'$ dominating $\pi$, we may assume that $\mu$ and $w(t)$ have distinct centers for any $t \in [0,1]$. Then the result follows from \refcor{skewness_comparison}.

\end{proof}

Once the general behavior of skewness and relative skewness is understood, we now investigate what happens when the monomial parameterization is monotonic. In this case, the situation is quite simpler, and more similar to what happens in the smooth setting.

\begin{prop}\label{prop:affine_linear}
Let $\pi\colon X_\pi\to (X,x_0)$ be a log resolution of $\mf{m}$, and suppose that $E$ and $F$ are exceptional primes of $\pi$ which intersect in a point $p$. Let $w\colon [0,1]\to \segment{\nu_E}{\nu_F}{p}$ be the monomial parameterization at $p$. The following are equivalent.
\begin{enumerate}[itemsep=-1ex]
\item The map $w(t)$ is monotonic increasing with respect to the partial order on $\mc{V}_X$.  
\item There exist $t_1, t_2\in [0,1]$ with $0<t_1 < t_2$ such that $w(t_1) < w(t_2)$. 
\item The function $\alpha(w(t))$ is affine linear and increasing with $t$. 
\item The distance $d(\nu_E, \nu_F):=1/b_Eb_F$ is equal to $\alpha(\nu_F) - \alpha(\nu_E)$. 
\end{enumerate}
\end{prop}
\begin{proof}
We use here the same notations as in the proof of \refprop{skewness_parameter}.
The implication (1 $\Rightarrow$ 2) is trivial.

(2 $\Rightarrow$ 3) Let $f\colon [t_1,1]\to \R$ be the function $f(s) = \beta(w(t_1)\mid w(s))$. By \eqref{eqn:relative_skewness_formula2}, this a rational function of degree at most one. On the other hand, $f$ takes the value $1$ twice: $f(t_1) = 1$ by definition and $f(t_2) = 1$ since we have assumed $w(t_1) < w(t_2)$. Thus $f\equiv 1$, proving that $w(t_1) < w(s)$ for all $s>t_1$. Now let $g\colon [0,1]\to \R$ be the function $g(t) = \beta(w(t)\mid \nu_F)$. By what we have just shown it takes the value $1$ at two distinct $t$, namely at $t = t_1$ and $t = 1$. On the other hand, since $t_1$ lies in the interior of $[0,1]$ and $g(t)\geq 1$ everywhere, it must be that $g'(t_1) = 0$. We have therefore found (with multiplicity) at least three solutions to the equation $g(t) = 1$. By \eqref{eqn:relative_skewness_formula2}, however, $g$ is a rational function of degree at most $2$, so this is possible only if $g\equiv 1$. This proves that $\nu_F \geq w(t)$ for all $t$. By \hyperref[lem:monotonicity]{Lemma~\ref*{lem:monotonicity}(2)}, \[
\alpha(w(t)) = -Z_\pi(w(t))\cdot Z_\pi(\nu_F)= -(1-t)\frac{\check{E}\cdot \check{F}}{b_Eb_F} - t\frac{\check{F}^2}{b_F^2},\] an affine linear function. Since $\alpha$ is monotonic by \hyperref[lem:monotonicity]{Lemma~\ref*{lem:monotonicity}(3)}, necessarily $\alpha(w(t))$ is increasing. 

(3 $\Rightarrow$ 4) The assumption that $\alpha(w(t))$ is affine linear tells us that the quadratic term in  \eqref{eqn:skewness_formula2} vanishes. Given this, it is clear we only need to prove that $\Delta = 0$. Again, we consider $g(t) := \beta(w(t)\mid \nu_F)$. From \eqref{eqn:relative_skewness_formula2}, we have $g(t) = \alpha(w(t))/h(t)$, where $h(t)$ is the nonnegative affine linear function \[h(t) = \left[\Delta + \frac{1}{b_Eb_F}\right]t + \alpha(\nu_E) + \Delta.\] The second derivative of $g$ is easily computed to be \[
g''(t)  = \frac{-2\Delta\alpha(\nu_F)\left[\Delta + \frac{1}{b_Eb_F}\right]}{h(t)^3} = \frac{-2\Delta\alpha(\nu_F)\left[\alpha(\nu_F) - \alpha(\nu_E) - \Delta\right]}{h(t)^3},
\]
the equalities coming from the vanishing of the quadratic term of \eqref{eqn:skewness_formula2}. Given the fact that $\Delta\leq 0$ and, by our assumption that $\alpha$ is increasing, $\alpha(\nu_F) > \alpha(\nu_E)$, we conclude that $g''(t) \geq 0$, with equality only if $\Delta = 0$. It therefore suffices to show that $g$ cannot be a strictly convex function. Let $\pi'$ be the good resolution of $(X,x_0)$ obtained by blowing up the point $p$, and let $G$ be the new exceptional prime thus obtained. Clearly for any $\mf{m}$-primary ideal $\mf{a}$ of $R$, we have $\divi_G(\mf{a}) \geq \divi_E(\mf{a}) + \divi_F(\mf{a})$. The normalized divisorial valuation $\nu_G = (b_E + b_F)^{-1}\divi_G$ associated to $G$ is exactly $w(t_0)$  with $t_0 = b_F/(b_E + b_F)$, so we can rewrite this inequality as \[
w(t_0)(\mf{a})\geq (1-t_0)\nu_E(\mf{a}) + t_0\nu_F(\mf{a}).\] Applying \refprop{famous_equality}, this proves that $g(t_0)\geq (1-t_0)g(0) + t_0g(1)$. Thus $g$ is not strictly convex, completing the proof. 

(4 $\Rightarrow$ 1) We wish to prove that $\beta(w(t)\mid w(s)) = 1$ whenever $s\geq t$. Since statement 4 is assumed to hold, the function $\beta(w(t)\mid w(s))$ takes the simple form \begin{equation}\label{simple_rs}
\beta(w(t)\mid w(s)) = \frac{-2\Delta t^2 + 2\Delta t + \frac{t}{b_Eb_F} + \alpha(\nu_E)}{-2\Delta st + \Delta s + \Delta t + \frac{t}{b_Eb_F} + \alpha(\nu_E)}\end{equation} whenever $s\geq t$. The fact that this quantity must be $\geq 1$ implies that necessarily \[
-2\Delta t^2 + 2\Delta t \geq -2\Delta st + \Delta s + \Delta t,\] or equivalently that $\Delta (2t - 1)(s-t)\geq 0$ whenever $1\geq s \geq t\geq 0$. Since $\Delta \leq 0$, this is only possible if $\Delta = 0$. But then from \eqref{simple_rs} we get $\beta(w(t)\mid w(s)) = 1$ whenever $s\geq t$, completing the proof. 
\end{proof}

\begin{rmk}
One cannot remove the assumption in statement 2 of \refprop{affine_linear} that $t_1 > 0$. One example of this is the cusp singularity $(X,0)$ with $X = \{x^2 + y^3 + z^9 - xyz = 0\}\subset \C^3$, see \S\ref{ssec:example_cusp_322}.
It has a unique minimal good resolution $\pi$, with three rational exceptional primes $E_1$, $E_2$, and $E_3$ that pairwise intersect in $X_\pi$.
Their self-intersections are $E_1^2=E_2^2=-2$, and $E_3^2 = -3$, and the generic multiplicities of all three equal $1$.
The resolution $\pi$ is not a log resolution of the maximal ideal $\mf{m}_X$, but there is a unique base point, which is a \emph{free} point in $E_3$ (see below).
It follows that $\check{E}_3^2 = \check{E}_3\cdot \check{F}_1 = -1$, and thus $\beta(\nu_{E_3} \mid \nu_{E_1})  = 1$, saying precisely that $\nu_{E_3}\leq \nu_{E_1}$. On the other hand, $\alpha(\nu_{E_1}) - \alpha(\nu_{E_3}) = 2/3\neq 1=d(\nu_{E_3}, \nu_{E_1})$, so the equivalent conditions of \refprop{affine_linear} are not satisfied. 
\end{rmk}

Recall that for each log resolution $\pi$ of $\mf{m}$, the map $\emb_\pi$ gives a homeomorphism of the dual graph $\Gamma_\pi\subset \Ediv(\pi)_\R$ onto a subset $\mc{S}_\pi\subset\mc{V}_X$. Moreover, if $b_E^{-1}\check{E}$ and $b_F^{-1}\check{F}$ are adjacent vertices in $\Gamma_\pi$, then the map $w\colon [0,1]\to \mc{V}_X$ given by $w(t) = \emb_\pi((1-t)b_E^{-1}\check{E} + tb_F^{-1}\check{F})$ is precisely monomial parameterization at $p$. Thus $w$ gives a continuous parameterization of the edge in $\mc{S}_\pi$ between the adjacent vertices $\nu_E$ and $\nu_F$. 

Suppose now that $\pi$ is a log resolution of $\mf{m}$, and $\pi'$ is the good resolution of $(X,x_0)$ obtained from $\pi$ by blowing up a single point $p\in X_\pi$. If $p$ is a \emph{satellite point} of $\pi$ (that is, if $p$ is the intersection of two distinct exceptional primes of $\pi$), then $\mc{S}_{\pi'} = \mc{S}_\pi$. If $p$ is a \emph{free point} of $\pi$ (that is, if $p$ lies on a unique exceptional prime $E$ of $\pi$), then $\mc{S}_{\pi}\subsetneq \mc{S}_{\pi'}$; more precisely, $\mc{S}_{\pi'}$ is obtained from $\mc{S}_\pi$ by adjoining the edge from $\nu_E\in \mc{S}_{\pi}$ to a new vertex $\nu_F$, with $F$ the exceptional divisor of $\eta_{\pi\pi'}$. We claim that the parameterization $w$ of this new edge satisfies the equivalent conditions of \refprop{affine_linear}. Indeed, it is easy to see that $b_F = b_E$ and that in $\Ediv(\pi')_\R$ one has $\check{F} = \check{E} - F$. It follows that \[
\alpha(\nu_F) = -\frac{\check{F}^2}{b_F^2} = -\frac{(\check{E} - F)^2}{b_E^2} = -\frac{\check{E}^2}{b_E^2}  - \frac{1}{b_E^2} = \alpha(\nu_E) + \frac{1}{b_E^2}, \] which is exactly condition 4 of \refprop{affine_linear}. Since any good resolution $\pi'$ of $(X,x_0)$ dominating $\pi$ is obtained from $\pi$ by a composition of point blowups, the above analysis applied inductively gives that $\mc{S}_{\pi'}$ is obtained from $\mc{S}_\pi$ by adjoining a collection of ordered trees. More precisely, the following holds. 

\begin{prop}\label{prop:baby_tree} Let $\pi_0$ be a fixed log resolution of $\mf{m}$, and suppose $\hat{\nu}\in \mc{S}_{\pi_0}$ is a chosen point. Let $\pi$ be any good resolution of $(X,x_0)$ dominating $\pi_0$, and let $\mc{S}_{\pi, \hat{\nu}}\subseteq \mc{S}_\pi$ denote the set of valuations $\nu$ whose retraction $r_{\pi_0}(\nu)$ onto $\mc{S}_{\pi_0}$ is the point $\hat{\nu}$. Then $\mc{S}_{\pi, \hat{\nu}}$ is a rooted tree with respect to the partial order $\geq$, having root $\hat{\nu}$. Moreover, the skewness function $\alpha$ is a parameterization of $\mc{S}_{\pi, \hat{\nu}}$ inducing the metric $d$. 
\end{prop}

Here the terms \emph{rooted tree} and \emph{parameterization} are in the sense of \cite[\S2]{jonsson:berkovich}. Explicitly, a rooted tree is a partially ordered set $(\mc{T},\leq )$             in  which \begin{enumerate}[itemsep=-1ex]
\item there exists a unique minimal element $\hat{x}$, called the \emph{root},
\item any two elements $x, y\in \mc{T}$ have an infimum $x\wedge y\in \mc{T}$,
\item the set $\{x\in \mc{T} : x\leq y\}$ is order isomorphic to $[0,1]$ when $y\neq \hat{x}$, and 
\item any nonempty totally ordered subset of $\mc{T}$ has a least upper bound  in $\mc{T}$. 
\end{enumerate} A \emph{parameterization} of a rooted tree $(\mc{T}, \leq)$ is a function $\sigma\colon \mc{T}\to \R\cup\{+\infty\}$ which is monotonically increasing and which maps intervals $[x,y] := \{z\in \mc{T} : x\leq z\leq y\}$ order isomorphically onto their images $[\sigma(x), \sigma(y)]$. Any such parameterization gives rise to a metric $d_\sigma$ on the set of all $x\in \mc{T}$ with $\sigma(x) < +\infty$, namely the metric defined by $d_\sigma(x,y) = \sigma(x) + \sigma(y) - 2\sigma(x\wedge y)$.

With the same setup as \refprop{baby_tree}, let $\mc{V}_{X,\hat{\nu}}$ denote the collection of all semivaluations $\nu\in\mc{V}_X$ whose retraction $r_{\pi_0}(\nu)$  onto $\mc{S}_{\pi_0}$ is the point $\hat{\nu}$. More or less immediately from \refthm{structure} we see that $\mc{V}_{X,\hat{\nu}}\cong \varprojlim\mc{S}_{\pi, \hat{\nu}}$. It is not hard to see that the rooted tree structure of the the $\mc{S}_{\pi, \hat{\nu}}$ induces in the limit a rooted tree structure on $\mc{V}_{X,\hat{\nu}}$.

\begin{prop}\label{prop:big_tree} The set $\mc{V}_{X,\hat{\nu}}$ is a rooted tree with respect to the partial order $\geq$, having root $\hat{\nu}$. The skewness function $\alpha$ is a parameterization of $\mc{V}_{X,\hat{\nu}}$ whose induced metric $d_\alpha$ agrees with $d$ on the set $\mc{V}_{X,\hat{\nu}}^\qm$ of quasimonomial valuations. 
\end{prop}

\refprop{big_tree} generalizes the fact that $\mc{V}_X$ is a rooted tree in the smooth setting when $X = \C^2$, see \cite[\S7.6]{jonsson:berkovich} or \cite[\S3.2]{favre-jonsson:valtree}. Finally, we note that \refprop{big_tree} gives a natural way of extending the metric $d$ to the set of all finite skewness valuations of $\mc{V}_X$. Explicitly, if $\nu$ and $\mu$ are finite skewness valuations that lie in $\mc{V}_{X,\hat{\nu}}$ for some $\hat{\nu}\in \mc{S}_{\pi_0}$, then $d(\nu, \mu) := d_\alpha(\nu, \mu)$, whereas if $\nu$ and $\mu$ lie in $\mc{V}_{X,\hat{\nu}}$ and $\mc{V}_{X,\hat{\mu}}$ for distinct $\hat{\nu}$ and $\hat{\mu}$, then $d(\nu, \mu) := d_\alpha(\nu, \hat{\nu}) + d(\hat{\nu}, \hat{\mu}) + d_\alpha(\hat{\mu}, \mu)$. Henceforth, we will extend $d$ in this way, and call the resulting topology the \emph{strong topology} of finite skewness valuations.

\subsection{The angular metric}\label{angular_metric}

The metric $d$ constructed above has proved useful a number of times in previous works for studying the geometry of valuation spaces and developing potential theory on them, see \cite{favre-jonsson:valtree, favre-jonsson:eigenval, favre-jonsson:valmultideals, favre-jonsson:valanalplanarplurisubharfunct, favre-jonsson:dynamicalcompactifications, boucksom-favre-jonsson:izumi, baker-nicaise:weight-functions}. For us, however, another metric, which we will call the \emph{angular metric}, shall prove to be of greater use in our dynamical setting. 

\begin{defi} The \emph{angular distance} $\rho(\nu, \mu)$ between two semivaluations $\nu, \mu\in \mc{V}_X$ is given by \begin{equation}\rho(\nu, \mu) := \log [\beta(\nu\mid \mu)\beta(\mu\mid\nu)].\end{equation} This quantity may of course be $+\infty$. 
\end{defi}

The term angular distance alludes to the fact that  $\beta(\nu\mid\mu)\beta(\mu\mid \nu)$ is the limit over good resolutions $\pi$ of $\beta_\pi(\nu\mid \mu)\beta_\pi(\mu\mid \nu)$, which in turn is a measure of the angle between the divisors $Z_\pi(\nu)$ and $Z_\pi(\mu)$ as vectors in $\Ediv(\pi)_\R$, see \eqref{skewnesses}. 

\begin{prop} The angular distance $\rho$ gives a metric on the set of valuations of finite skewness in $\mc{V}_X$.
\end{prop}
\begin{proof} If $\nu, \mu\in\mc{V}_X$ are both valuations of finite skewness, then $\rho(\nu, \mu)$ is finite by \refcor{skewness_comparison}. Clearly $\rho$ is symmetric. It is positive definite by \hyperref[lem:monotonicity]{Lemma~\ref*{lem:monotonicity}(1)}, so we need only to prove the triangle inequality. If $\gamma\in \mc{V}_X$ is another  valuation of finite skewness, then by \refprop{famous_equality} \begin{align*}
\beta(\nu\mid \mu)\beta(\mu\mid \nu) & = \sup_\mf{a}\left(\frac{\nu(\mf{a})\gamma(\mf{a})}{\mu(\mf{a})\gamma(\mf{a})}\right)\times\sup_\mf{a}\left(\frac{\mu(\mf{a})\gamma(\mf{a})}{\nu(\mf{a})\gamma(\mf{a})}\right) \\
& \leq \sup_\mf{a} \left(\frac{\nu(\mf{a})}{\gamma(\mf{a})}\right)\times \sup_\mf{a} \left(\frac{\gamma(\mf{a})}{\mu(\mf{a})}\right)\times \sup_\mf{a}\left(\frac{\mu(\mf{a})}{\gamma(\mf{a})}\right)\times \sup_\mf{a}\left(\frac{\gamma(\mf{a})}{\nu(\mf{a})}\right)\\
& = \beta(\nu\mid\gamma)\beta(\gamma\mid \mu)\beta(\mu\mid\gamma)\beta(\gamma\mid\nu). 
\end{align*} Upon taking logarithms, this gives $\rho(\nu, \mu)\leq \rho(\nu, \gamma) + \rho(\gamma, \mu)$, as desired. 
\end{proof}

\begin{rmk}
In a recent work \cite{garciabarroso-gonzalezperez-popescupampu:ultrametricspacesarborescentsing}, the authors introduce a distance, called \emph{determinant distance}, on the set $\mc{S}_\pi^*$ of divisorial valuations realized in a given good resolution $\pi \colon X_\pi \to (X,x_0)$ of an \emph{arborescent singularity}. An arborescent singularity is a normal surface singularity whose dual graph of any good resolution is a tree.

Up to a rescaling factor, their determinant distance coincide with the restriction of the angular distance on $S_\pi^*$.
With this point of view, the angular distance can be seen as an inverse limit of the determinant distances along all good resolutions $\pi$ of $(X,x_0)$.
\end{rmk}

We spend the remainder of the section comparing the two metrics $d$ and $\rho$. We saw in \S\ref{ssec:partial_order} that the metric $d$ is the induced by the skewness parameterization on subsets of the form $\mc{V}_{X,\hat{\nu}}$. As we will now see, the angular metric is also induced by a parameterization on these sets, namely the parameterization by $\log \alpha$. 

\begin{prop}\label{prop:log_alpha} Let $\pi_0$ be a fixed log resolution of $\mf{m}$, and suppose $\hat{\nu}\in \mc{S}_{\pi_0}$ is a chosen point. Then $\log \alpha$ is a parameterization on the rooted tree $\mc{V}_{X,\hat{\nu}}$ that induces the angular metric on the set of finite skewness valuations. 
\end{prop}

\begin{proof} That $\log \alpha$ is a parameterization of $\mc{V}_{X, \hat{\nu}}$ is immediate from the fact that $\alpha$ is a parameterization of $\mc{V}_{X,\hat{\nu}}$. To see it induces the angular metric, we need only show that if $\nu, \mu\in \mc{V}_{X,\nu}$ are such that $\mu < \nu$, then $\rho(\mu, \nu) = \log\alpha(\nu) - \log\alpha(\mu) = \log[\alpha(\nu)/\alpha(\mu)]$. Indeed, if $\pi$ is a good resolution dominating $\pi_0$ in which the centers of $\nu$ and $\mu$ are distinct, we have by \reflem{monotonicity} that  \[
\rho(\nu, \mu) =\log[\beta(\mu\mid\nu)\beta(\nu\mid\mu)]= \log\beta(\nu\mid\mu) = \log\frac{\alpha(\nu)}{-Z_\pi(\nu)\cdot Z_\pi(\mu)} = \log\frac{\alpha(\nu)}{\alpha(\mu)},\] completing the proof.
\end{proof}

It follows essentially immediately that when restricted to sets of the form $\mc{V}_{X,\hat{\nu}}$ the metrics $d$ and $\rho$ are equivalent in that they both induce the strong topology on finite skewness valuations. To compare the metrics on all of $\mc{V}_X$, we need the following lemma. 

\begin{lem}\label{lem:additiverho}
Suppose that $\mu, \nu\in \mc{V}_X$ are finite skewness valuations that lie in $\mc{V}_{X,\hat{\mu}}$ and $\mc{V}_{X,\hat{\nu}}$, respectively, where $\hat{\mu}\neq \hat{\nu}$. Then $\rho(\nu, \mu) = \rho(\nu, \hat{\nu}) + \rho(\hat{\nu}, \hat{\mu}) + \rho(\hat{\mu}, \mu)$.
\end{lem}
\begin{proof} After possibly replacing $\pi_0$ with a good resolution obtained from $\pi_0$ by a sequence of satellite point blowups (which do not change $\mc{S}_{\pi_0}$), we may assume without loss of generality that the centers of $\hat{\nu}$ and $\hat{\mu}$ are distinct in $X_{\pi_0}$. By definition, $Z_{\pi_0}(\nu) = Z_{\pi_0}(\hat{\nu})$ and similarly $Z_{\pi_0}(\mu) = Z_{\pi_0}(\hat{\mu})$, so applying \reflem{monotonicity} again, we see \[
\beta(\nu\mid\mu)\beta(\mu\mid \nu) = \frac{\alpha(\nu)\alpha(\mu)}{(Z_{\pi_0}(\hat{\nu})\cdot Z_{\pi_0}(\hat{\mu}))^2} = \frac{\alpha(\nu)\alpha(\mu)}{\alpha(\hat{\nu})\alpha(\hat{\mu})}\cdot\frac{\alpha(\hat{\nu})\alpha(\hat{\mu})}{(Z_{\pi_0}(\hat{\nu})\cdot Z_{\pi_0}(\hat{\mu}))^2} = \frac{\alpha(\nu)\alpha(\mu)}{\alpha(\hat{\nu})\alpha(\hat{\mu})}\beta(\hat{\nu}\mid\hat{\mu})\beta(\hat{\mu}\mid \hat{\nu}).\] We saw in \refprop{log_alpha} that $\rho(\nu, \hat{\nu}) = \log \big[\alpha(\nu)/\alpha(\hat{\nu})\big]$, and similarly for $\mu$, so upon taking logarithms we get exactly the desired equality. 
\end{proof}

We saw in \S\ref{ssec:partial_order} that the analogous formula holds for $d$. Given that $d$ and $\rho$ are known to be equivalent on $\mc{V}_{X,\hat{\nu}}$ and $\mc{V}_{X,\hat{\mu}}$, it will therefore follow that $d$ and $\rho$ are equivalent on all of $\mc{V}_X$ if we can show that $d$ and $\rho$ are equivalent on $\mc{S}_{\pi_0}$. The metric $d$ induces the Euclidean topology on $\mc{S}_{\pi_0}\cong \Gamma_{\pi_0}$ by definition. The fact that $\rho$ also induces the Euclidean topology on $\mc{S}_{\pi_0}$ follows immediately from \eqref{eqn:relative_skewness_formula}. We conclude that $d$ and $\rho$ both induce the strong topology on the set of finite skewness valuations of $\mc{V}_X$.

\section{Log discrepancy, essential skeleta, and special singularities}\label{sec:log_disc}

In the proofs of our main theorems, the most important singularities one needs to consider turn out to be the \emph{log canonical} singularities. In this section we review the classification of such singularities and discuss the related notion of log discrepancy for valuations. As usual, $(X,x_0)$ denotes a normal surface singularity and $(R,\mf{m})$ denotes the completed local ring $\hat{\mc{O}}_{X,x_0}$. We also fix a nontrivial holomorphic $2$-form $\omega$ on $(X,x_0)$; the vanishing of $\omega$ defines a Weil divisor on $X$ we denote by $\Div(\omega)$.

Given a good resolution $\pi\colon X_\pi\to (X, x_0)$ of $(X,x_0)$, the \emph{relative canonical divisor} of $\pi$ is the divisor $K_\pi\in \Ediv(\pi)_\Q$ defined by the equality $\Div(\pi^*\omega) = \pi^*\Div(\omega) + K_\pi$, where here $\pi^*\Div(\omega)$ refers to the Mumford pull-back of $\Div(\omega)$, see  \cite[p.\ 195]{matsuki:mori-program}. The relative canonical divisor $K_\pi$ does not depend on the choice of $\omega$. Indeed, the Mumford pull-back $\pi^*\Div(\omega)$ is numerically trivial on $\Ediv(\pi)$ by definition, so the adjunction formula for $X_\pi$ gives
\begin{equation}\label{eqn:adjunction_can_div}
K_\pi\cdot E = 2g(E) - 2 - E^2
\end{equation}
for every exceptional prime $E$ of $\pi$ independently of $\omega$, where here $g(E)$ denotes the genus of $E$. Since the intersection form on $\Ediv(\pi)_\R$ is definite, this uniquely determines $K_\pi$ independently of $\omega$. 

If $E$ is an exceptional prime of a good resolution $\pi$, the \emph{log discrepancy} of $E$ is  $A(E) := 1 + \ord_E(K_\pi)$. More generally, one can define a log discrepancy function $A\colon \hat{\mc{V}}_X^*\to \R\cup\{+\infty\}$ on valuations which is uniquely characterized by the following properties: \begin{enumerate}[itemsep=-1ex]
\item If $E_1$ and $E_2$ are exceptional primes of a good resolution $\pi$ intersecting in a point $p$, and if $\nu_{r,s}$ is the monomial valuation at $p$ with weights $r$ and $s$, then $A(\nu_{r,s}) = rA(E) + sA(F)$. 
\item $A$ is lower semicontinuous on $\hat{\mc{V}}_X^*$. 
\end{enumerate} Taking  $r = 1$ and $s = 0$ in the first statement, we derive that $A(\divi_E) = A(E)$ for any exceptional prime $E$ of a good resolution. Note also that $A$ is homogeneous in the sense that $A(\lambda \nu) = \lambda A(\nu)$, and thus if $\nu_E$ is the normalized divisorial valuation $\nu_E = b_E^{-1}\divi_E$ we have $A(\nu_E) = b_E^{-1}A(E)$.
For the details of the construction of $A$ as well as some other properties, see \cite{favre:holoselfmapssingratsurf}. Here we will call $A(\nu)$ the \emph{log discrepancy} of $\nu$, in agreement with \cite{boucksom-defernex-favre-urbinati, jonsson:berkovich}, but it should be noted that $A(\nu)$ has also been called the \emph{thinness} of $\nu$ \cite{favre-jonsson:valtree, favre-jonsson:eigenval, favre:holoselfmapssingratsurf, boucksom-favre-jonsson:valuationsplurisubharsing} and the \emph{weight} of $\nu$  \cite{mustata-nicaise:weight-functions, nicaise-xu:essential-skeleton, baker-nicaise:weight-functions}.

Suppose that $\pi$ is a log resolution of $\mf{m}$, that $E$ is an exceptional prime of $\pi$, and that $\pi'$ is the good resolution obtained from $\pi$ by blowing up a free point $p\in E$. If $F$ is the exceptional locus of $\eta_{\pi\pi'}$, then it is easy to verify that $A(F) = A(E) + 1$ and $b_F = b_E$. Therefore, if $\nu_E$ and $\nu_F$ are the normalized divisorial valuations associated to $E$ and $F$, we derive that $A(\nu_F) = A(\nu_E) + b_E^{-1}$. We showed in \S\ref{ssec:partial_order} that $\alpha(\nu_F) = \alpha(\nu_E) + b_E^{-2}$, so we see $A(\nu_F) - A(\nu_E) \geq \alpha(\nu_F) - \alpha(\nu_E)$. Applying this analysis inductively on point blowups, we obtain the following analogue of \refprop{big_tree} for log discrepancy. 

\begin{prop} Let $\pi_0$ be a log resolution of $\mf{m}$ and let $\hat{\nu}$ be a chosen point of $\mc{S}_{\pi_0}$. Then the function $A$ gives a parameterization of $\mc{V}_{X,\hat{\nu}}$. Moreover, if $d_A$ denotes the metric induced by this parameterization, then $d_A(\nu, \mu) \geq d(\nu, \mu)$ for all $\nu, \mu\in \mc{V}_{X,\hat{\nu}}$. In particular, if $A(\nu) < +\infty$, then $\alpha(\nu) < +\infty$ as well. 
\end{prop}

Note, it is possible for $\alpha(\nu) < +\infty$ while $A(\nu) = +\infty$, so the metrics $d$ and $d_A$ are not equivalent at the ends of the tree $\mc{V}_{X,\hat{\nu}}$. On the other hand, for any good resolution $\pi\geq\pi_0$, the parameterizations $\alpha$ and $A$ both induce the Euclidean topology on the finite tree $\mc{S}_{\pi, \hat{\nu}} := \mc{S}_\pi\cap \mc{V}_{X,\hat{\nu}}$, and thus $\alpha$ and $A$ both induce the strong topology on quasimonomial valuations $\mc{V}_{X,\hat{\nu}}^\qm \cong \varinjlim \mc{S}_{\pi, \hat{\nu}}$. 

\subsection{Log canonical and log terminal singularities}\label{ssec:lc_and_lt}

Because $A$ is lower semicontinuous and $\mc{V}_X$ is compact,  the log discrepancy function $A$ necessarily takes a minimum value on $\mc{V}_X$, which is called the \emph{log canonical threshold} of $\mf{m}$. If this minimum is nonnegative, then $(X,x_0)$ is said to be \emph{log canonical} (lc); if it is strictly positive, then $(X, x_0)$ is \emph{log terminal} (lt). Log  canonical and log terminal singularities  feature prominently in  the minimal model program, and thus have been studied intensively. They are relatively mild singularities, and in fact in our two-dimensional setting they are completely classified. The details of this classification, due originally to Kawamata \cite[\S10]{kawamata:lc_classification}, are discussed in depth in \cite[Ch.\ 4]{matsuki:mori-program}, see also \cite[\S3]{kollar:flips} or \cite[Ch.\ 3]{kollar:MMP-singularities} for an alternative approach due to Alexeev. We provide now a rough summary of this classification sufficient for our purposes.

\begin{defi} We say $(X,x_0)$ is a \emph{finite quotient} of another irreducible normal surface germ $(Y,y_0)$ if there is a finite group $G$ acting on $(Y,y_0)$ by analytic automorphisms, the action being free on $Y \setminus \{y_0\}$, such that $(X,x_0)\cong (Y,y_0)/G$. 
\end{defi}

If $(X,x_0)$ is a finite quotient of a normal surface germ $(Y,y_0)$, then $(X,x_0)$ is lc (resp.\ lt) if and only if $(Y,y_0)$ is lc (resp.\ lt). This can be seen, for instance, by applying the Jacobian formula (see \S\ref{ssec:Jacobian_formula}) to the quotient map $(Y,y_0)\to (X,x_0)$. 

\begin{thm}\label{thm:lc_classification} Every log canonical normal surface germ $(X, x_0)$ is a finite quotient $(Y,y_0)/G$, where $(Y,y_0)$ is either nonsingular, a cusp singularity, or a simple elliptic singularity. Moreover, $(X,x_0)$ is log terminal if and only if $(Y, y_0)$ is nonsingular. 
\end{thm}

Cusp and simple elliptic singularities are defined in terms of the geometry of the exceptional loci of their minimal good resolutions $\pi_0$. We say $(X,x_0)$ is a \emph{cusp singularity} if the exceptional primes  of $\pi_0$ are $n \geq 2$ rational curves $E_1,\ldots, E_n$ which intersect in a \emph{cycle}, that is, for which up to re-indexing one has $E_i\cdot E_j = 1$ if and only if $i - j \equiv \pm 1\pmod{n}$. We say $(X,x_0)$ is a \emph{simple elliptic singularity} if the exceptional locus of $\pi_0$ consists of a single genus one curve $E$. In particular, the dual graph $\Gamma_{\pi_0}$ of $\pi_0$ is homeomorphic to a circle for cusp singularities and homeomorphic to a point for simple elliptic singularities. 

\refthm{lc_classification} gives that the class of lt surface singularities $(X,x_0)$ coincides with the class of \emph{quotient surface singularities}, that is, finite quotients $(\C^2,0)/G$. The group $G$ can always be taken to be a subgroup of $GL_2(\C)$ acting linearly on $(\C^2,0)$, so one can classify quotient singularities by understanding the finite subgroups of $GL_2(\C)$, see \cite{brieskorn:rational-singularities}. Of particular interest to us are the \emph{cyclic quotient singularities}, where $G$ is a cyclic group; these are sometimes called \emph{Hirzebruch-Jung quotient singularities}. It can be shown that $(X,x_0)$ is a cyclic quotient singularity if and only if the exceptional primes of the (unique) minimal good resolution $\pi_0$ of $(X,x_0)$ are $n\geq 1$ rational curves $E_1,\ldots, E_n$ intersecting in a \emph{chain}, that is, for which up to re-indexing one has $E_i\cdot E_j = 1$ if and only if $|i-j| = 1$. In particular, the dual graph $\Gamma_{\pi_0}$ for a cyclic quotient singularity 
is homeomorphic to either a point or an interval. If $(X,x_0)$ is a non-cyclic quotient singularity, then its (unique) minimal good resolution $\pi_0$ again only has rational exceptional primes, but now $\Gamma_{\pi_0}$ is homeomorphic to a tree with one degree-three \emph{fork}. 

The remaining lc singularities $(X,x_0)$ are quotients $(Y,y_0)/G$ of cusps or simple elliptic singularities. Suppose that $(X, x_0)$ is a quotient of a cusp. It may be that $(X,x_0)$ is again a cusp, but if it is not, then necessarily $G\cong \Z/2\Z$, the (unique) minimal good resolution $\pi_0$ of $(X,x_0)$ has only rational exceptional primes, and $\Gamma_{\pi_0}$ is homeomorphic to a tree with two forks, both of which have degree three.
These singularities are usually referred to as \emph{quotient-cusp} singularities.

If $(X,x_0)$ is a non-trivial quotient of a simple elliptic singularity, then necessarily $G\cong \Z/n\Z$ with $n\in \{2,3,4,6\}$. In all of these cases, the (unique) minimal good resolution $\pi_0$ of $(X,x_0)$ has only rational exceptional primes and $\Gamma_{\pi_0}$ is homeomorphic to a tree with one fork. If $n = 2$, this fork has degree four, while if $n \in \{3,4,6\}$ this fork has degree three. 

We have now determined the possible homeomorphism types of all the dual graphs $\Gamma_{\pi_0}$ of minimal good resolutions $\pi_0$ of lc singularities $(X,x_0)$. If $\pi$ is a log resolution of $\mf{m}$ dominating $\pi_0$, then $\emb_\pi$ maps $\Gamma_{\pi_0}$ homeomorphically onto the skeleton $\skel{X}=\mc{S}_{\pi_0} \subset \mc{V}_X$. The log discrepancy $A$ will be strictly positive off of $\skel{X}$, but if $(X,x_0)$ is not lt then $A$ must vanish somewhere on $\skel{X}$. In \reffig{dual_graphs} we have drawn all possible homeomorphism types of $\skel{X}$ for lc singularities $(X,x_0)$, and marked exactly where the log discrepancy function vanishes. Notice that $\skel{X}$ is not contractible only when $(X,x_0)$ is a cusp singularity.

\begin{figure}[ht]
\centering
\def\svgwidth{\columnwidth}
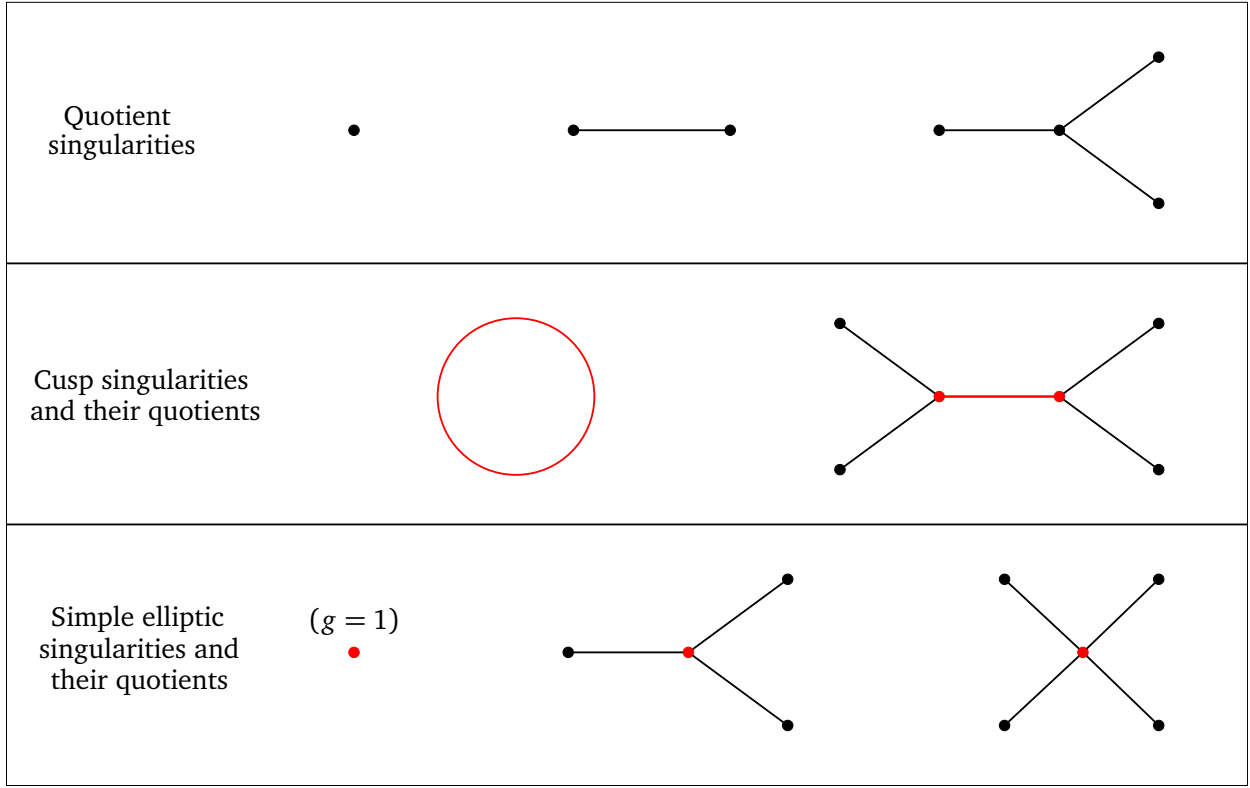
\caption{The possible skeleta $\skel{X} \subset \mc{V}_X$ of lc singularities $(X,x_0)$. Here dots are drawn only at forks and endpoints of $\skel{X}$, and a point $\skel{X}$ is red if and only if $A(\nu) = 0$.}\label{fig:dual_graphs}
\end{figure}

\subsection{The essential skeleton}\label{ssec:essential_skeleton}

Now suppose that $(X,x_0)$ is any irreducible normal surface germ, and let $\pi_0$ be a minimal good resolution of $(X,x_0)$. Within its dual graph $\Gamma_{\pi_0}$, let $\Gamma_X\subseteq \Gamma_{\pi_0}$ denote the smallest connected subgraph that contains all cycles, all forks, and all vertices corresponding to exceptional primes $E$ of genus $g(E)>0$. If $\pi$ is any log resolution of $\mf{m}$ which dominates $\pi_0$, let $\eskel{X}\subset \mc{V}_X$ denote the homeomorphic image of $\Gamma_X$ under the embedding $\emb_\pi$; this set is independent of both the choice of $\pi_0$ and the choice of $\pi$. 

\begin{defi} The subgraph $\eskel{X}\subset \mc{V}_X$ is called the \emph{essential skeleton} of $(X, x_0)$. 
\end{defi}

\begin{figure}[ht]
	\def\svgwidth{\columnwidth}
	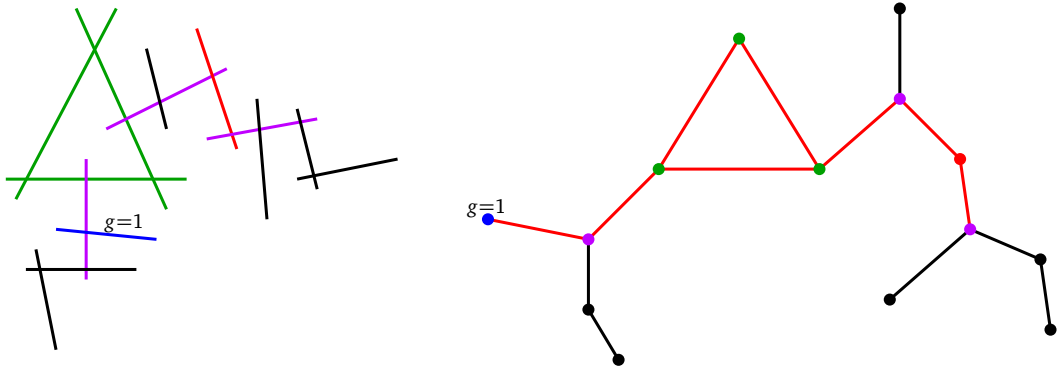
\caption{On the left is the exceptional locus of some minimal good resolution $\pi_0$. On the right is the corresponding dual graph $\Gamma_{\pi_0}$. The edges marked in red are the edges corresponding to the essential skeleton $\eskel{X}$.} \label{fig:essentialskeleton}
\end{figure}

An illustration of the essential skeleton is given in \reffig{essentialskeleton}. We make a couple of easy observations. First, the essential skeleton $\eskel{X}$ may be empty, but this happens if and only if $(X,x_0)$ is nonsingular or a cyclic quotient singularity. Second, the connected components of $\skel{X} \setminus \eskel{X}$ are always intervals, and the vertices appearing in any such interval correspond to chains of rational curves in $X_{\pi_0}$. In fact, one should think of $\eskel{X}$ as the (embedded) dual graph of the space $\hat{X}$ obtained from $X_{\pi_0}$ by contracting the chains of rational curves corresponding to the components of $\skel{X} \setminus \eskel{X}$. This space $\hat{X}$ is a modification of $(X,x_0)$ which has only cyclic quotient singularities (one for each of the contracted chains), and, assuming $\eskel{X}$ is nonempty, each of these singular points will lie on exactly one exceptional prime of $\hat{X}$. 

\begin{rmk}\label{rmk:dlt} It would be more natural to modify the definition of the essential skeleton $\eskel{X}$ and the space $\hat{X}$ in the following manner. We let $\pi\colon \hat{X}\to (X,x_0)$ be the minimal modification of $(X,x_0)$ for which the pair $(\hat{X}, B)$ is divisorial log terminal (dlt), where here $B$ is the reduced divisor supported on the exceptional locus of $\pi$. We then set $\eskel{X}$ to be the dual graph of $\pi$ embedded into $\mc{V}_X$ in the usual manner.  This agrees with the previous definition of $\hat{X}$ and $\eskel{X}$ in all but one case, namely when $(X,x_0)$ is a non-cyclic quotient singularity; in this case, by the previous definition $\eskel{X}$ would consist of a single point but with this new definition $\hat{X} = X$ and $\eskel{X}$ is empty. While this new definition is more natural, we will stick to the previous definition simply to avoid  a discussion of log discrepancies for pairs and dlt models, of which we will have no need.
\end{rmk}

The main result of the section is the following.

\begin{thm}\label{thm:nonpositive-discrepancies} If $(X,x_0)$ is not log terminal, then $A(\nu)\leq 0$ for all $\nu\in\eskel{X}$. 
\end{thm}

This theorem was proved by Veys in \cite{veys:stringy}. Note, if we used the definition of $\eskel{X}$ in \refrmk{dlt}, then the theorem holds without the assumption that $(X,x_0)$ is not lt, because in this case $\eskel{X}$ is empty for lt singularities. We give a brief outline of the proof below. The key tool is the adjunction formula for the space $\hat{X}$. Since $\hat{X}$ has singular points, the usual adjunction formula is invalid, and one has to modify it by adding certain correction terms for each of the singular points $p_1,\ldots, p_r$ of $\hat{X}$. Recall that each $p_i$ is a cyclic quotient singularity, say $(\hat{X}, p_i)\cong (\C^2,0)/G_i$ where $G_i$ is a cyclic group. The adjunction formula for $\hat{X}$ is the following. For each exceptional prime $E$ of $\hat{X}$, \begin{equation}\label{modified_adjunction}
K_{\hat{X}}\cdot E = 2g(E) - 2 - E^2 + \sum_{p_i\in E}\left(1 - \frac{1}{|G_i|}\right) \deg_{E}[p_i],\end{equation} see \cite[Thm.\ 3.36]{kollar:MMP-singularities} for a proof.

\begin{proof}[Proof sketch for \refthm{nonpositive-discrepancies}] It suffices to check that $A(E)\leq 0$ for every exceptional prime $E$ of $\hat{X}$. Let $E_1,\ldots, E_n$ be the exceptional primes of $\hat{X}$, and let $a$ be the column vector $a = (A(E_1), \ldots, A(E_n))^t$. If $M$ is the intersection matrix $M = (E_i\cdot E_j)_{ij}$, then the adjunction formula \eqref{modified_adjunction} says exactly that $Ma = b$, where $b = (b_1,\ldots, b_n)^t$ is the column vector with entries\[
b_j = 2g(E_j) - 2 + d_j + \sum_{p_i\in E_j}\left(1 - \frac{1}{|G_i|}\right)\deg_{E_j}[p_i],\] with $d_j$ being the degree of the vertex in $\eskel{X}$ corresponding to $E_j$. Note, each term in the expression for $b_j$ is nonnegative except for the term $-2$; thus if either $g(E_j) \geq 1$ or $d_j\geq 2$, then necessarily $b_j\geq 0$. Assume then that $g(E_j) = 0$ and $d_j\in \{0, 1\}$. The case when $d_j = 0$ is special, since this says that $E_j$ is the only exceptional prime of $\hat{X}$; in this case the theorem is proved if $A(E_j)\leq 0$, and if $A(E_j)>0$ then in fact $A(E)>0$ for each of the contracted primes $E$ of $\pi_0$, so that $(X,x_0)$ is log terminal. We may assume therefore that $d_j = 1$. Then by the definition of $\eskel{X}$, there must be at least two singular points $p_i$ lying within $E_j$. Since $2g(E_j) - 2 + d_j = -1$ and for each of the (at least two) singular points $p_i\in E_j$ the term $(1 - |G_i|^{-1})\deg_{E_j}[p_i]$ is at least $1/2$, we conclude that again $b_j \geq 0$. To 
sum 
up, this proves that the column vector $b$ has nonnegative entries in all cases except for if $(X,x_0)$ is  log terminal, when it necessarily consists of a single negative entry. Note, the inverse  $M^{-1}$ has strictly negative entries by \reflem{linear-algebra}, so unless $(X,x_0)$ is lt, $a = M^{-1}b$ must have nonpositive entries, completing the proof.
\end{proof}

It is worth pointing out that the converse is false: it is not generally true that if $A(\nu)\leq 0$ then $\nu\in \eskel{X}$. For example, let  $(X,x_0)$ be a cone singularity over a hyperbolic Riemann surface $E$. This means that $(X,x_0)$ is obtained from the total space $X(L)$ of a line bundle $L$ of degree $\leq -1$ over $E$ by contracting the zero section of $L$. The minimal resolution of $(X,x_0)$ is then $X(L)$, and the essential skeleton $\eskel{X}$ consists only of the divisorial point $\nu_E$ corresponding to the zero section of $L$. As the log discrepancy $A(\nu_E)$ is strictly negative and $A$ is continuous on quasimonomial valuations, there is a whole neighborhood of $\nu_E$ in the strong topology where $A$ is strictly negative. 

\section{Dynamics on valuation spaces}\label{sec:dynamics_valspaces}

\subsection{Induced maps on valuation spaces}

Suppose now that $(X,x_0)$ and $(Y,y_0)$ are two normal surface singularities and $f\colon (X,x_0)\to (Y,y_0)$ is a dominant holomorphic map between them. In this section, we will define an associated continuous map $f_\bullet \colon \mc{V}_X\to \mc{V}_Y$, and discuss how this map $f_\bullet$ interacts with the various constructions from \S\ref{sec:valuation_spaces} and \S\ref{sec:log_disc}.
Note, in the dynamical setting when $(Y,y_0) = (X,x_0)$, we obtain a dynamical system $f_\bullet\colon \mc{V}_X\to \mc{V}_X$, which is the focus of our \refthm{valdynamics}.
In this section we will state a precise version of this theorem, while the next two sections are devoted to its proof.
Let $(R_X, \mf{m}_X)$ and $(R_Y, \mf{m}_Y)$ denote the completed local rings $\hat{\mc{O}}_{X,x_0}$ and $\hat{\mc{O}}_{Y,y_0}$, respectively, and let $f^*\colon R_Y\to R_X$ be the induced local ring homomorphism. 

\begin{defi}
The holomorphic map $f$ is said to be \emph{finite} if it is finite-to-one in a neighborhood of $x_0$. 
\end{defi}

There are many other equivalent characterizations of finiteness. Algebraically, $f$ is finite if and only if $f^*\mf{m}_Y$ is an $\mf{m}_X$-primary ideal. Geometrically, $f$ is finite if and only if it has no \emph{contracted curves}, that is, no formal curve germs $(C, x_0)\subset (X,x_0)$ for which $f(C) = \{y_0\}$.
Note, any contracted curve $(C,x_0)$ would necessarily have to be a holomorphic curve instead of simply a formal curve, since it would be critical for $f$, see \refrmk{contractedcurves}.
Therefore $f$ is finite if and only if it contracts no holomorphic curve $C$ through $\{x_0\}$ to a point. Observe that there are at most finitely many contracted curve semivaluations, since $f^*\mf{m}_Y$ is finitely generated. 
In the following, we denote by $\CC{f}$ the set of irreducible curves $C$ with $f(C)=y_0$, and by $\VC{f}$ the associated set of contracted curve valuations $\nu_C$.

%

Associated to $f$ is a map $f_*\colon \hat{\mc{V}}_X\to \hat{\mc{V}}_Y$ of valuation spaces, defined by $f_*\nu := \nu\circ f^*$. This map $f_*$ is continuous with respect to the weak topologies on $\hat{\mc{V}}_X$ and $\hat{\mc{V}}_Y$ and commutes with scalar multiplication of semivaluations. If $\nu\in \hat{\mc{V}}_X^*$ is a finite semivaluation then $f_*\nu$ will again be finite, except in exactly one situation: if $\nu$ is a curve semivaluation associated to a curve germ $(C, x_0)\subset (X, x_0)$ which is contracted by $f$, then $f_*\nu$ will not be finite. In particular, $f$ is finite if and only if $f_*$ maps finite semivaluations to finite semivaluations. 

\begin{defi}
Let $\nu\in {\mc{V}}_X$ be a semivaluation. The \emph{attraction rate} $c(f,\nu)$ of $f$ along $\nu$ is the quantity $c(f, \nu) := (f_*\nu)(\mf{m}_Y) = \nu(f^*\mf{m}_Y)$. 
\end{defi}

By definition $c(f, \nu)< +\infty$ if and only if $f_*\nu$ is finite, that is, if and only if $\nu$ is not a contracted curve semivaluation. Assuming $c(f,\nu)<+\infty$, the semivaluation $f_\bullet \nu := c(f,\nu)^{-1}f_*\nu\in \mc{V}_Y$ is normalized. In this way we obtain a map $f_\bullet\colon \mc{V}_X\dashrightarrow\mc{V}_Y$, defined and continuous away from the finite set of contracted curve semivaluations.  As the next proposition shows, $f_\bullet$ can be extended naturally to all of $\mc{V}_X$.  

\begin{prop}\label{prop:extensionfbullet}
There is a unique continuous extension $f_\bullet\colon \mc{V}_X\to \mc{V}_Y$ of $f_\bullet$ to all of $\mc{V}_X$. 
\end{prop}
\begin{proof} Suppose that $(C, x_0)\subset (X, x_0)$ is an irreducible curve germ which is contracted by $f$, and let $\nu_C$ be the associated normalized curve semivaluation. Let $\mf{p}\subset  R_X$ be the prime ideal of $C$. Since $(X, x_0)$ is normal, the local ring $\mc{O}_{X, C}$ is a DVR, and thus we get a valuation $\ord_C$ on $R_X$ associated to $C$, measuring the order of vanishing of functions along $C$. While $\ord_C$ is not centered at $x_0$ and thus does not belong to $\hat{\mc{V}}_X$, its image $f_*\ord_C := \ord_C\circ f^*$ is centered at $y_0$ because $C$ is contracted by $f$. The rational rank and transcendence degree of $f_*\ord_C$ are both $1$, so $f_*\ord_C$ is a divisorial valuation $\lambda \divi_E$ associated to some exceptional prime $E$ of some good resolution of $(Y, y_0)$. We define $f_\bullet \nu_C$ to be the normalized divisorial valuation $\nu_E = b_E^{-1}\divi_E$. To check that $f_\bullet$ defined in this way is continuous at $\nu_C$, take any sequence $\{\nu_i\}$ of 
semivaluations 
$\nu_i\in \mc{V}_X\setminus\{\nu_C\}$ converging to $\nu_C$. Then, since $\nu_i(\mf{p})\to \nu_C(\mf{p}) = +\infty$, we see that the sequence $\nu_i/\nu_i(\mf{p})$ converges to $\ord_C$. Because $f_*$ is continuous, it follows that
\[
\frac{c(f, \nu_i)}{\nu_i(\mf{p})}f_\bullet \nu_i = f_*\left(\frac{\nu_i}{\nu_i(\mf{p})}\right) \to f_*\ord_C = \lambda b_E\nu_E.
\]
Since evaluating semivaluations on ideals is weakly continuous, evaluating the above limit on $\mf{m}_Y$ yields that $c(f, \nu_i)/\nu_i(\mf{p})\to \lambda b_E>0$. Because scalar multiplication is continuous, we conclude that
\[
f_\bullet\nu_i = \frac{\nu_i(\mf{p})}{c(f, \nu_i)}\times \frac{c(f, \nu_i)}{\nu_i(\mf{p})}f_\bullet\nu_i\to \frac{1}{\lambda b_E}\times \lambda b_E\nu_E = \nu_E,\]
completing the proof. 
\end{proof}

In the next proposition, we collect some easy facts about the map $f_\bullet$ and the attraction rates $c(f,\nu)$. Each of them is proved in \cite{favre:holoselfmapssingratsurf} or in \cite{favre-jonsson:eigenval} in the case when $(X,x_0) = (Y,y_0) = (\C^2,0)$, but these proofs are very general and apply equally in our setting; the first three are essentially translations of standard results in valuation theory about how valuations on a field $K$ extend to valuations on a finite extension $L$ of $K$, and the last two are just a matter of unraveling definitions. 

\begin{prop}\label{prop:propertiesfbullet}
The following hold for any dominant holomorphic germ $f\colon (X,x_0)\to (Y,y_0)$.
\begin{enumerate}[itemsep=-1ex]
\item If $\nu\in \mc{V}_X$ is not a contracted curve semivaluation, then $f_\bullet\nu$ is of the same type (divisorial, irrational, infinitely singular, curve) as $\nu$. If $\nu$ is a contracted curve semivaluation, then $f_\bullet\nu$ is divisorial.
\item Every $\nu\in \mc{V}_Y$ has at most $N$ preimages under $f_\bullet$, where $N=[\Frac(R_X) : f^*\Frac(R_Y)]$ is the degree of the field extension $\Frac(R_X)/f^*\Frac(R_Y)$.
\item If $f$ is finite, then $f_\bullet$ is surjective.
\item The inequality $c(f, \nu)\geq 1$ holds for all $\nu\in \mc{V}_X$.
\item Attraction rates are multiplicative in the sense that if $g\colon (Y,y_0)\to (Z,z_0)$ is another dominant holomorphic germ between normal surface singularities, then  $c(g\circ f, \nu) = c(g, f_\bullet\nu)c(f,\nu)$. 
\end{enumerate}
\end{prop}

As a special case of the first statement, a divisorial valuation $\nu_E\in \mc{V}_X$ necessarily maps to a divisorial valuation $\nu_{E'}\in \mc{V}_Y$, while a curve semivaluation $\nu_C \in \mc{V}_X$ maps to a curve semivaluation $\nu_{C'} \in \mc{V}_Y$, unless $f(C)=y_0$, and in this case it is mapped to a divisorial valuation $\nu_G \in \mc{V}_Y$.

In the following, we will often consider modifications (or good resolutions) realizing a finite set of divisorial valuations, and/or giving embedded resolutions of some curves, or with other features regarding the dynamics. We introduce here some notation to avoid long statements afterwards.

\begin{defi}
Let $f\colon (X,x_0) \to (Y,y_0)$ be a dominant germ, and let $V\subset \mc{V}_X$ be a finite set of divisorial or curve (semi-)valuations.
We say that two modifications $\pi \colon X_{\pi}\to (X,x_0)$ and $\varpi \colon Y_{\varpi}\to (Y,y_0)$ give a \emph{resolution of $f$ with respect to $V$} if the following conditions are satisfied :
\begin{itemize}
\item For any divisorial valuation $\nu_E \in V$, $\nu_E$ is realized by $\pi$, and $f_\bullet \nu_E$ is realized by $\varpi$.
\item For any curve valuation $\nu_C \in V$ non contracted by $f$, $\pi$ is an embedded resolution of $(C,x_0) \subset (X,x_0)$ and $\varpi$ is an embedded resolution of $(f(C),y_0) \subset (Y,y_0)$.
\item For any contracted curve valuation $\nu_C \in V$, $\pi$ is an embedded resolution of $(C,x_0) \subset (X,x_0)$ and $f_\bullet \nu_C$ is realized by $\varpi$.
\item The lift $\tilde{f}=\varpi^{-1}\circ f \circ \pi \colon X_\pi \to Y_\varpi$ is regular.
\end{itemize}
\end{defi}
Notice that such modifications always exists, and can be taken dominating any two given modifications.
In fact, since $V$ is finite, we can easily find good resolutions $\pi$ and $\varpi$ satisfying the first three conditions. By taking an opportune good resolution $\pi'$ dominating $\pi$, we can assure that the lift $\tilde{f}$ is also regular.

We now come back to the geometric interpretation of the action of $f_\bullet$ on divisorial and curve valuations.
Consider two good resolutions $\pi \colon X_\pi \to (X,x_0)$ and $\varpi \colon Y_\varpi \to (Y,y_0)$ giving a resolution of $f$ with respect to the empty set (meaning that we just ask the lift $\tilde{f}=\varpi^{-1}\circ f \circ \pi$ to be regular).
Let $D$ be a prime divisor in $X_\pi$ (it could be an exceptional prime $E$, or the strict transform $C_\pi$ of an irreducible curve $(C,x_0)$), and $D'$ be a prime divisor in $Y_\varpi$.
We denote by $k_D\geq 0$ the coefficient of $D$ in the divisor $\tilde{f}^*D'$.
Notice that $k_D \geq 1$ if and only if $\tilde{f}(D)\subset D'$.
The integer $k_D$ can also be characterized in the language of valuation theory as the ramification index $e(\divi_D/f_*\divi_D)$, i.e., the index of the value group of $f_*\divi_D$ within the value group of $\divi_D$.
Note that $k_D$ does not depend on the choice of the good resolutions $\pi$ and $\varpi$, since it is a quantity associated to the valuation associated to $D$.
If we have to keep track of the prime divisor $D'$, we will write $\kk{D}{D'}$ or $\kkk{D}{D'}{f}$ instead of $k_D$.

Analogously, if $f(D)=D'$, we denote by $e_{D} = \ee{D}{D'}=\eee{D}{D'}{f}$ the topological degree of the map $\tilde{f}|_{D} \colon D\to D'$.

\begin{rmk}
Notice that the values $\kk{D}{D'}$ and $\ee{D}{D'}$ do not depend on the good resolutions $\pi$ and $\varpi$ chosen (as far as they realize the valuations $\nu_E$ and $\nu_{E'}$).
In fact they are quantities associated to the divisorial or curve semivaluations associated more than to the primes $D$ and $D'$ themselves.
\end{rmk}

Notice that the values $k_D$ and $e_D$ do not depend on the choice of $\pi$ and $\varpi$ satisfying the above conditions.

We can now state the geometrical interpretation of the action of $f_*$ and of the attraction rate $c(f,\nu)$.
Some of these results have been stated and proved in \cite{favre-jonsson:eigenval} in the smooth setting, but the proofs easily generalize to the singular setting. We start with divisorial valuations.

\begin{prop}[{\cite[Prop.\ 2.5]{favre-jonsson:eigenval}}]\label{prop:divisorial_image}
Let $f \colon (X,x_0) \to (Y,y_0)$ be a dominant germ, $\nu_E \in \mc{V}_X$ a divisorial valuation, and $\nu_{E'}=f_\bullet \nu_E$.
Then for any modifications $\pi\colon X_\pi \to (X,x_0)$ and $\varpi\colon Y_\varpi \to (Y,y_0)$ realizing $\nu_E$ and $\nu_{E'}$ respectively, the lift $\wt{f}=\varpi^{-1} \circ f \circ \pi$ satisfies $\wt{f}|_E(E)=E'$. Moreover
\begin{equation}\label{eqn:divisorial_attraction_rate}
c(f, \nu_E) = \frac{b_{E'}}{b_E}\kk{E}{E'},
\end{equation}
\end{prop}
This geometrical interpretation of the image of divisorial valuations allows to describe tangent maps on valuation spaces. In fact, in the situation of \refprop{divisorial_image}, we may define the map $df_\bullet \colon T_{\nu_E} \mc{V}_X \to T_{\nu_{E'}} \mc{V}_Y$ by sending the tangent vector $\vect{v_p} \in T_{\nu_E} \mc{V}_X$ corresponding to a point $p \in E$, to the tangent vector $df_\bullet \vect{v_p} \in T_{\nu_{E'}} \mc{V}_Y$ corresponding to the point $q=\wt{f}|_E(p)$.
Notice that if $\pi$ and $\varpi$ give a resolution of $f$ with respect to $\{\nu_E\}$, again by \refprop{divisorial_image} we have $f_\bullet(U_\pi(p)) \subseteq U_{\pi'}(q)$. This corresponds to the description of tangent maps on real trees described in \cite{favre-jonsson:eigenval}.

We now focus on the geometric interpretation of the image of a curve semivaluation $\nu_C$. We first assume that $C$ is not contracted by $f$.

\begin{prop}[{\cite[Prop.\ 2.6]{favre-jonsson:eigenval}}]\label{prop:curve_image}
Let $f \colon (X,x_0) \to (Y,y_0)$ be a dominant germ, $\nu_C \in \mc{V}_X$ a non-contracted curve valuation, and $\nu_{C'}=f_\bullet \nu_C$.
Then $f(C)=C'$, and moreover
\begin{equation}\label{eqn:curve_attraction_rate}
c(f, \nu_C) = \frac{m(C')}{m(C)}\ee{C}{C'}.
\end{equation}
\end{prop}

For contracted curve valuations, we need some notations before proceeding.

\begin{defi}
Let $f \colon (X,x_0) \to (Y,y_0)$ be a dominant germ, and $\nu_C \in \mc{V}_X$ be a normalized curve semivaluation.
We set
$$
c_\alpha(f, \nu_C)=\lim_{\nu \to \nu_C} \frac{c(f,\nu)}{\alpha(\nu)}, \qquad
c_A(f, \nu_C)=\lim_{\nu \to \nu_C} \frac{c(f,\nu)}{A(\nu)}.
$$
\end{defi}

Notice that for any curve semivaluation $\nu_C$, there exists a subtree $\mc{V}_{X,\hat{\nu}}$ which contains $\nu_C$.
In particular the interval $I=[\hat{\nu},\nu_C]$ is totally ordered, and the limit can be taken over $I$.
Moreover, the skewness $\alpha$ and the log discrepancy $A$ give a parameterization of such interval $I$.
So $c_\alpha(f,\nu_C)$ and $c_A(f,\nu_C)$ can be seen as the derivatives of $c(f,\nu)$ at $\nu_C$ with respect to such parameterizations.

Finally, recall that for curve semivaluations, $\alpha(\nu_C)=A(\nu_C)=+\infty$. So $c_\alpha(f,\nu_C)=c_A(f,\nu_C)=0$ unless $\nu_C$ is a contracted curve valuation. In the latter case, $c(f,\nu_C)$ is affine with respect to both the $\alpha$ and $A$ parameterizations. Then the derivative of $c(f,\nu)$ is just the coefficient of the linear part.

\begin{prop}\label{prop:contractedcurve_image}
Let $f\colon (X,x_0) \to (Y,y_0)$ be a non-finite germ between normal surface singularities. Let $\nu_C \in \VC{f}$ be a contracted curve valuation, and set $\nu_G = f_\bullet \nu_C \in \mc{V}_Y$.
Then
$$
c_\alpha(f,\nu_C)=m(C) b_{G} \kk{C}{G}, \qquad
c_A(f,\nu_C)=b_{G}\kk{C}{G}, \qquad
$$
\end{prop}
\begin{proof}
Let $\pi\colon X_\pi \to (X, x_0)$ and $\varpi\colon Y_\varpi \to (Y,y_0)$ give a resolution of $f$ with respect to $\{\nu_C\}$.
Up to taking higher good resolutions, we may also assume that $\pi$ is a log resolution of $\mf{m}_X$.
The strict transform $C_\pi$ of $C$ intersects $\pi^{-1}(x_0)$ in a unique exceptional prime $H$, transversely at a point $p$. Notice that $b_H=m(C)=:m$.
Take local coordinates $(x,y)$ at $p$ so that $H=\{x=0\}$ and $C=\{y=0\}$.
For any $t \geq 0$, let $\mu_t$ be the monomial valuation at $p$ of weights $\frac{1}{m}$ and $t$ with respect to the coordinates $(x,y)$, and set 
$\nu_t=\pi_* \mu_t \in \mc{V}_X$.
Notice that $\nu_t$ is normalized thanks to the choice of the weight on $x$.
By direct computation, we get that $A(\nu_t)=A(\nu_0)+t$, while $\alpha(\nu_t)=\alpha(\nu_0)+\frac{t}{m}$.
In particular $c_\alpha(f,\nu_C)=m c_A(f,\nu_C)$.

By \refprop{extensionfbullet}, $\tilde{f}_*\ord_{C_\pi} = k\ord_G$, where $k=\kk{C_\pi}{G}$.
By \refprop{divisorial_image} we get $\tilde{f}(C_\pi)=G$.
Set $q=\tilde{f}(p)$, and pick local coordinates $(z,w)$ at $q$, so that $G=\{w=0\}$.

We consider the sets of valuations $\mc{V}_p^{x^m}$ at $p$ normalized with respect to the value at $x^m$, and analogously $\mc{V}_q^w$.
Notice that with our choices, $\mu_t \in \mc{V}_p^{x^m}$ is normalized, and we have
$$
c(f,\nu_t)=\mu_t(\tilde{f}^*w) \cdot \mu'_t(\varpi^*\mf{m}_Y),
$$
where $\mu'_t$ is the valuation in $\mc{V}_q^w$ proportional to $\tilde{f}_*\mu_t$.
Moreover, since $f \circ \pi = \varpi \circ \tilde{f}$, we infer that $\varpi_*\mu'_t$ is proportional to $f_\bullet \nu_t$.
In particular, for $t$ big enough, $\mu'_t(\varpi^*\mf{m}_Y) = \ord_G(\varpi^*\mf{m}_Y) = b_G$.
We conclude by noticing that $\mu_t(\tilde{f}^*w) = kt$ for $t$ big enough. 
\end{proof}
\begin{rmk}
In the proof of \refprop{contractedcurve_image}, we may also consider the derivative of $c(f,\nu)$ with respect to the parameterization given by the skewness $\alpha_p$ computed at $p$ of the corresponding normalized valuation $m\nu_t$, instead of the skewness $\alpha$ computed at $x_0$.
In this case we have $\alpha_p(m\mu_t)=mt$ (for $t \geq 1/m$).
In particular, the derivative $c^\alpha(f,\nu_C)$ of $c(f,\nu)$ with respect to $\alpha_p$ satisfies
$$
c^\alpha(f,\nu_C)=\frac{b_{G}}{m(C)}\kk{C}{G}.
$$
\end{rmk}


We end this section by recalling a criterion that allows to recover all the indeterminacy points the lift $\tilde{f}$ of a dominant germ $f \colon (X,x_0) \to (Y,y_0)$ with respect to given modifications $\pi\colon X_\pi\to (X,x_0)$ and $\varpi\colon Y_\varpi\to (Y,y_0)$.
Recall that in general $\tilde{f} = \varpi^{-1}\circ f\circ \pi\colon X_\pi\dashrightarrow Y_\varpi$ is a meromorphic map.
If $p\in \pi^{-1}(x_0)$ is a closed point, we denote by $U_X(p)\subset \mc{V}_X$ the subset of semivaluations $\nu$ whose center in $X_\pi$ is $p$. This is a weak open set for all closed points $p$. Similarly, if $q\in \varpi^{-1}(y_0)$ is a closed point, then $U_Y(q)$ will denote those semivaluations $\nu\in \mc{V}_Y$ with center $q$ in $Y_\varpi$. Certainly if $\tilde{f}$ is holomorphic at a closed point $p\in \pi^{-1}(x_0)$, and if $q = \tilde{f}(p)$, then $f_\bullet(U_X(p))\subseteq U_Y(q)$; this is simply a matter of unraveling definitions. More importantly for us, the converse is also true, see \cite[Prop.\ 3.2]{favre-jonsson:eigenval} for a proof.

\begin{prop}\label{prop:detecting_holomorphicity}
The lift $\tilde{f}$ is holomorphic at a closed point $p\in \pi^{-1}(x_0)$ and has $\tilde{f}(p) = q$ if and only if $f_\bullet(U_X(p))\subseteq U_Y(q)$. 
\end{prop}

\subsection{Action on dual divisors}\label{ssec:action_dual_divisors}

Let $f\colon (X,x_0) \to (Y,y_0)$ be a dominant map between two normal surface singularities.
Pick good resolutions $\pi\colon X_{\pi}\to (X,x_0)$ and $\varpi\colon Y_{\varpi}\to (Y,y_0)$ such that the lift $\tilde{f} := \varpi^{-1}\circ f\circ \pi$ is holomorphic.
If $f$ is \emph{finite}, then $\tilde{f}$ induces $\Z$-linear push-forward and pull-back operations $\tilde{f}_*\colon \Ediv(\pi)\to \Ediv(\varpi)$ and $\tilde{f}^*\colon \Ediv(\varpi)\to \Ediv(\pi)$ in the usual way; note, the finiteness of $f$ is needed in order for $\tilde{f}^*$ to map into $\Ediv(\pi)$. These operations satisfy the projection formula
\begin{equation}\label{eqn:projformula}
\tilde{f}_*D_1\cdot D_2 = D_1\cdot \tilde{f}^*D_2
\end{equation}
for all divisors $D_1\in \Ediv(\pi)$ and $D_2\in \Ediv(\varpi)$.
In this case, one can compute the push-forward and the pull-back of dual divisors, and more generally, divisors associated to valuations (see \cite[Lemma 1.10]{favre:holoselfmapssingratsurf}).

In fact, if $E$ is an exceptional prime of $\pi$ which is not contracted by $\tilde{f}$, then $E' = \tilde{f}(E)$ is an exceptional prime of $\varpi$ and a straightforward computation shows $\tilde{f}_*\check{E} = \kk{E}{E'}\check{E'}$.

Similarly, if $E'$ is an exceptional prime of $\varpi$ and $E_1,\ldots, E_n$ are all of the exceptional primes of $\pi$ for which $\wt{f}(E_i) = E'$, then $\wt{f}^*\check{E'} = \sum_i \ee{E_i}{E'} \check{E}_i$.


Finally, let us note that even if the lift $\tilde{f} = \varpi^{-1}\circ f\circ \pi$ is not holomorphic, we still get well-defined $\Z$-linear push-forward and pull-back operations $f_*\colon \Ediv(\pi)\to \Ediv(\varpi)$ and $f^*\colon \Ediv(\varpi)\to \Ediv(\pi)$, simply by passing to a good resolution $\pi'\geq \pi$ for which the lift $g := \varpi^{-1}\circ f\circ \pi'$ is holomorphic and defining $f_* :=g_*\circ  \eta_{\pi\pi'}^*$ and $f^* = (\eta_{\pi\pi'})_*\circ g^*$. These operations $f_*$ and $f^*$ again satisfy the projection formula.

We want to generalize these formulas to non-finite maps.
In this case, the pull-back $\tilde{f}^*$ acting on $\Ediv(\pi)$ does not give an element of $\Ediv(\pi)$, but in general a divisor supported in $\pi^{-1}(f^{-1}(y_0))$.
To deal with this situation, we need a few more notations.
For any curve $(C,x_0) \subset (X,x_0)$, and any modification $\pi \colon X_\pi \to (X,x_0)$, we denote by $C_\pi$ the strict transform of $C$ by $\pi$.

Let now $\pi \colon X_\pi \to (X,x_0)$ be a good resolution. We denote by $\Weil(\pi)$ the set of Weil divisors of $X_\pi$, i.e., a formal sum of the form
$$
D=\sum_{E \in \varGamma_\pi^*} a_E E + \sum_C a_C C_\pi, 
$$
where $E$ varies in the set $\varGamma_\pi^*$ of all exceptional primes of $\pi$, $C$ among all irreducible curves $(C,x_0) \subset (X,x_0)$, and $a_E, a_C \in \nZ$, all zero but for a finite number of indices.
We may define analogously the set of $\nQ$-divisors $\Weil(\pi)_\nQ$ and of $\nR$-divisors $\Weil(\pi)_\nR$.
Notice that the intersection product of two elements in $\Weil(\pi)$ is not well defined in general, since the self-intersection of non-exceptional divisors is not.
For any divisor $D \in \Weil(\pi)_\nR$ as above, we denote by $\Exc_\pi(D) = \sum_E a_E E$ its exceptional part, and by $\Rest_\pi(D) = D - \Exc_\pi(D)$ the non-exceptional part.

To any irreducible curve $(C,x_0) \subset (X,x_0)$, we associate a divisor $\check{C} \in \Weil(\pi)_\nQ$ as the only divisor so that $\check{C}+C_\pi \in \Ediv(\pi)_\nQ$ and $\check{C} \cdot E = 0$ for all $E \in \varGamma_\pi^*$.
We denote $\check{C}=\hat{Z}_\pi(\inte_C)$ if we need to remember the good resolution $\pi$ we are working with.

Notice that if $\pi \colon X_\pi \to (X,x_0)$ is an embedded resolution of $(C,x_0) \subset (X,x_0)$, then $C_\pi$ intersects a unique exceptional prime $H$ transversely, and $C+\check{C}=\check{H}=Z_\pi(\inte_C)$.

As before, a (not necessarily finite) map $f\colon (X,x_0) \to (Y,y_0)$ induces $\nZ$-linear push-forward and pull-back operators $\tilde{f}_* \colon \Weil(\pi) \to \Weil(\varpi)$ and $\tilde{f}^* \colon \Weil(\varpi) \to \Weil(\pi)$, and analogously for $\nQ$- and $\nR$-divisors.
Push-forward and pull-back operators satisfy the projection formula $\tilde{f}_*D_1\cdot D_2 = D_1\cdot \tilde{f}^*D_2$ for all divisors $D_1\in \Ediv(\pi)$ and $D_2\in \Ediv(\varpi)$.
Note that the right-hand side of the projection formula \eqref{eqn:projformula} is well defined only as far as the supports of $D_1$ and $\tilde{f}^*D_2$ have no common non-exceptional components.
When this happens, the projection formula still holds.

In this more general setting, we obtain the following formulae for pushforward and pullback of dual divisors.

\begin{prop}\label{prop:pushforward}
Let $f\colon (X,x_0) \to (Y,y_0)$ be a dominant germ, and $\nu_E \in \mc{V}_X$ be a divisorial valuation.
Let $\pi\colon X_\pi \to (X, x_0)$ and $\varpi\colon Y_\varpi \to (Y,y_0)$ give a resolution of $f$ with respect to $\VC{f} \cup \{\nu_E\}$.
Then we have
\begin{equation}\label{eqn:pushforward}
\tilde{f}_* \check{E} = \kk{E}{E'} \check{E'} - \sum_{C \in \CC{f}} \big(-C_\pi \cdot \check{E}\big)\kk{C}{G_C} \check{G}_C,
\end{equation}
where $\tilde{f}=\varpi^{-1}\circ f \circ \pi$ is the lift of $f$, and $E'$ and $G_C$ satisfy $\nu_{E'} = f_\bullet \nu_E$ and $f_\bullet \nu_C = \nu_{G_C}$ for all $C \in \CC{f}$.
\end{prop}
\begin{proof}
It suffices to prove that the intersections of both sides of \eqref{eqn:pushforward} with any exceptional prime $D'$ in $Y_\varpi$ coincide. Notice that for any $D' \in \varGamma_\varpi^*$, we have
$$
\tilde{f}^*D'=\sum_{D \in \varGamma_\pi^*} \kk{D}{D'}D + \sum_{C \in \CC{f}} \one_{D'}^{G_C}\kk{C}{G_C}C_\pi,
$$
where $\one$ denotes the Kronecker's delta function.
The result easily follows.
\end{proof}

\begin{prop}\label{prop:pullback}
Let $f\colon (X,x_0) \to (Y,y_0)$ be a dominant germ, and $\nu_{E'} \in \mc{V}_Y$ be a divisorial valuation.
Let $\pi\colon X_\pi \to (X, x_0)$ and $\varpi\colon Y_\varpi \to (Y,y_0)$ give a resolution of $f$ with respect to $\VC{f} \cup f_\bullet^{-1}(\{\nu_{E'}\})$.
Then we have
\begin{equation}\label{eqn:pullback}
\tilde{f}^*\check{E'} = \sum_{i=1}^r \ee{E_i}{E'} \check{E_i} + \sum_{C \in \CC{f}} \big(-\check{G_C}\cdot \check{E'}\big) \kk{C}{G_C} \check{C},
\end{equation}
where $\tilde{f}=\varpi^{-1}\circ f \circ \pi$ is the lift of $f$, $\nu_{E_1}, \ldots, \nu_{E_r}$ are the divisorial valuations in $f_\bullet^{-1}(\{\nu_{E'}\})$, and $G_C$ satisfies $f_\bullet \nu_C = \nu_{G_C}$ for all $C \in \CC{f}$.
\end{prop}
\begin{proof}
First, let us compute the non-exceptional part of $\tilde{f}^* \check{E'}$.
For any $C \in \CC{f}$, the $C_\pi$-coefficient of $\tilde{f}^* \check{E'}$ is given by $\kk{C}{G_C}$ times the $G_C$-coefficient of $\check{E'}$.
This last coefficient is given by $\check{G_C} \cdot \check{E'}$ (see \refrmk{checkcoordinates}).
Recalling that the $C_\pi$-coefficient of $\check{C}$ is $-1$, it follows that
$$
L := \tilde{f}^*\check{E'} + \sum_{C \in \CC{f}} (\check{G_C}\cdot \check{E'}) \kk{C}{G_C} \check{C} \in \Ediv(\pi).
$$
Now we may compute the intersection of $L$ with any exceptional prime $D$ of $X_\pi$. Since $\check{C} \cdot D = 0$ for all $D \in \Ediv(X)$, using the projection formula we get
$$
D \cdot L = D \cdot \tilde{f}^* \check{E}' = \tilde{f}_*(D) \cdot \check{E'} = \ee{D}{\tilde{f}(D)} \tilde{f}(D) \cdot \check{E'} = \one_D^{E_i} \ee{E_i}{E'},
$$ 
and we are done.
\end{proof}

Notice that in the expressions \eqref{eqn:pushforward} and \eqref{eqn:pullback}, the coefficients $-C_\pi \cdot \check{E}$ and $-\check{G_C}\cdot \check{E'}$ are always positive.

\subsection{Action on b-divisors}\label{ssec:action_bdivisors}

This section is devoted to generalize equations \eqref{eqn:pushforward} and \eqref{eqn:pullback} to the setting of b-divisors.


Let $f\colon (X,x_0) \to (Y,y_0)$ be a dominant germ between normal surface singularities.
As for divisors, we will have to consider a more general class of b-divisors, whose support is not over the singular point.

\begin{defi}
Let $(X,x_0)$ be a normal surface singularity. A (not necessarily exceptional) b-divisor is a collection $Z=(Z_\pi)_\pi$ of divisors $Z_\pi \in \Weil(\pi)_\nR$ for all modification $\pi$, satisfying $Z_\pi = (\eta_{\pi\pi'})_*Z_{\pi'}$ whenever $\pi'$ is a modification dominating $\pi$. 

We say that $Z$ is \emph{Cartier} if there exists a modification $\pi$ so that $Z_\pi' = \eta_{\pi\pi'}^* Z_\pi$ for all modifications $\pi'$ dominating $\pi$. 

We denote by $\Weil(X)$ the set of all b-divisors over $X$ and by $\Weil_C(X)$ the subset of Cartier b-divisors.

A b-divisor $Z$ is \emph{nef} if $Z_\pi \cdot D \geq 0$ for any effective divisor $D \in \Ediv(\pi)_\nR$.
It is \emph{exceptional} if $Z_\pi \in \Ediv(\pi)_\nR$ for any modification $\pi$. 
\end{defi}

\begin{rmk}
Given a b-divisor $Z \in \Weil(X)$, we may define its \emph{exceptional part} $\Exc(Z)=(\Exc_\pi(Z_\pi))_\pi \in \Ediv(X)$.
Notice that $\Exc(Z)$ could be non-Cartier even though $Z$ is. A nice example of this phenomenon is given by b-divisors associated to curve semivaluations.
In fact, to a curve semivaluation $\inte_C$, we may associate a Cartier $b$-divisor $\hat{Z}(\inte_C)$, determined by any embedded resolution of $(C,x_0) \subset (X,x_0)$.
In this case, $\Exc(\hat{Z}(\inte_C)) = Z(\inte_C)$, which is not Cartier.

In general, since modifications are isomorphisms outside the exceptional divisor, for any b-divisor $X \in \Weil(X)$ we have
$$
Z-\Exc(Z)=\sum_{C} -a_C Z(C),
$$
where $C$ varies among irreducible curves in $(X,x_0)$, $a_C$ vanishes for all but a finite number of $C$, and $Z(C)=(C_\pi)_\pi$ denotes the Weil b-divisor whose incarnation in $\pi$ is the strict transform $C_\pi$ of $C$. Notice that $Z(C)$ is nef.
We set $\displaystyle \on{Supp}(Z)=\{x_0\} \cup \bigcup_{a_C \neq 0} \on{Supp}(C)$ the \emph{support} of $Z$.
It follows that $\Weil(X)=\Rest(X)\oplus \Ediv(X)$, and $\Weil_C(X)=\Rest(X)\oplus \Ediv_C(X)$, where $\Rest(X)=\bigoplus_{C} \check{C}$, with $C$ varying among irreducible curves in $(X,x_0)$, is an infinite dimensional $\nR$-vector space.
In particular, any Weil b-divisor $Z \in \Weil(X)$ can be uniquely written as
$$
Z = \sum_{C} a_C \check{C} + W, \qquad W \in \Ediv(X).
$$

Notice that $W \neq \Exc(Z)$. We will use the notation $W=\pr_\Exc(Z)$, and call it the \emph{projection} of $Z$ on $\Ediv(X)$.
Analogously, we set $\pr_\Rest(Z)=Z-\pr_\Exc(Z)$.

\end{rmk}

\begin{rmk}\label{rmk:defZhata}
The construction of the non-exceptional b-divisor associated to a curve semivaluation can be generalized to the case of non $\mf{m}_X$-primary ideals $\mf{a}$.
In fact, given such an ideal $\mf{a}$, we may consider any log resolution $\pi \colon X_\pi \to (X,x_0)$ of $\mf{a}$. In this case we have $\pi^* \mf{a} = \mc{O}_{X_\pi}(-D_\pi)$ for some divisor $D_\pi \in \Weil(\pi)$.
The projection of $D_\pi$ to $(X,x_0)$ gives a germ of curve $C$, which induces a decomposition of $D_\pi$ in its exceptional and non-exceptional parts $D_\pi=E_\pi + C_\pi$, where $C_\pi$ stands for the strict transform of $C$.
We set $\hat{Z}_\pi(\mf{a}):=-D_\pi$, and we denote by $\hat{Z}(\mf{a})$ the associated b-divisor.
As for curve semivaluations, in this case we have $\Exc(\hat{Z}(\mf{a}))=Z(\mf{a})$, where $Z(\mf{a})$ is the b-divisor defined in \refrmk{defZanonprimary}
\end{rmk}

We are ready to define pushforward and pullback of b-divisors (see \cite{boucksom-favre-jonsson:degreegrowthmeromorphicsurfmaps}).

We first define the pushforward $f_*:\Weil(X) \to \Weil(Y)$ as follows.
Let $Z \in \Weil(X)$ be a b-divisor, and fix a modification $\varpi\colon Y_\varpi \to (Y,y_0)$.
Pick a modification $\pi \colon X_\pi \to (X,x_0)$ so that the lift $\tilde{f}=\varpi^{-1} \circ f \circ \pi \colon X_\pi \to Y_\varpi$ of $f$ is holomorphic, and set $(f_*Z)_\varpi := \tilde{f}_*Z_\pi$.
This defines a continuous map $f_*\colon \Weil(X) \to \Weil(Y)$.
Moreover, if $Z \in \Weil_C(X)$ is Cartier, then $f_*(Z)$ is Cartier, determined by any $\varpi\colon Y_\varpi \to (Y,y_0)$ so that $\pi\colon X_\pi \to (X,x_0)$ is a determination of $Z$ and $\tilde{f}=\varpi^{-1}\circ f \circ \pi \colon X_\pi \dashrightarrow Y_\varpi$ does not contract any exceptional prime of $X_\pi$ or (the strict transform of) any curve in the support of $Z$ (see \cite[Corollary 2.6]{boucksom-favre-jonsson:degreegrowthmeromorphicsurfmaps}).

We now define the pullback $f^*:\Weil(Y) \to \Weil(X)$.
Let $Z' \in \Weil(Y)$ be a b-divisor, and fix a modification $\pi \colon X_\pi \to (X,x_0)$.
Pick a modification $\varpi \colon Y_\varpi \to (Y,y_0)$ so that the lift $\tilde{f}=\varpi^{-1} \circ f \circ \pi \colon X_\pi \to Y_\varpi$ of $f$ does not contract any (exceptional or not) irreducible component of $(f\circ \pi)^{-1}(\on{Supp}(Z'))$.
Then we set $(f^*Z')_\pi := \tilde{f}^*Z'_\varpi$.
This defines a continuous map $f^*\colon \Weil(Y) \to \Weil(X)$ (see \cite[Corollary 2.5]{boucksom-favre-jonsson:degreegrowthmeromorphicsurfmaps}), which preserves Cartier b-divisors. In fact if $Z'$ is determined by $\varpi$, then $Z$ is determined by a modification $\pi$ so that $\tilde{f}$ is regular.

\begin{thm}\label{thm:operations_bdivisors}
Let $f\colon (X,x_0) \to (Y,y_0)$ be a dominant germ, and $\nu \in \mc{V}_X$ and $\nu' \in \mc{V}_Y$ be normalized valuations.
Then we have
\begin{align}
f_*Z(\nu)&=c(f,\nu)Z(f_\bullet \nu) - \sum_{\nu_C \in \VC{f}} \big(-Z(\nu_C) \cdot Z(\nu)\big) c_\alpha(f,\nu_C)Z(f_\bullet \nu_C), \label{eqn:pushforward_bdivisors}\\
f^*Z(\nu')&=\sum_{f_\bullet\nu=\nu'}\frac{m(f,\nu)}{c(f,\nu)}Z(\nu) + \sum_{\nu_C \in \VC{f}} \big(-Z(f_\bullet \nu_C) \cdot Z(\nu')\big) c_\alpha(f,\nu_C)\hat{Z}(\nu_C),\label{eqn:pullback_bdivisors}
\end{align}
where $m(f,\nu) \in \nN^*$, and $\sum_{f_\bullet\nu=\nu'} m(f,\nu)$ is bounded by the local topological degree of $f$.
\end{thm}
\begin{proof}
By continuity we may assume that $\nu=\nu_E$ and $\nu'=\nu_{E'}$ are divisorial.
To obtain \eqref{eqn:pushforward_bdivisors}, we divide both sides of \eqref{eqn:pushforward} by $b_E$, and express the equation with respect to b-divisors associated to normalized valuations.
We conclude by applying \refprop{divisorial_image}, \refprop{contractedcurve_image}, and the fact that
$C_\pi \cdot Z_\pi(\nu_E)=Z_\pi(\inte_C) \cdot Z_\pi(\nu_E) = m(C) Z_\pi(\nu_C) \cdot Z_\pi(\nu_E)$, since $Z_\pi(\nu_E)$ is an exceptional divisor.

We proceed analogously for the pullback, dividing both sides of \eqref{eqn:pullback} by $b_{E'}$.
First, we analyze the projection of $f^*Z(\nu_{E'})$ to $\Rest(X)$. We get in this way the second sum of \eqref{eqn:pullback_bdivisors}. Notice that it varies continuously on $\nu'$.
It remains to study the projection of $f^*Z(\nu_{E'})$ to $\Ediv(X)$, which is given by
$$
\sum_{i=1}^r e_{E_i}\frac{b_{E_i}}{b_{E'}}Z(\nu_{E_i}) = \sum_{i=1}^r \frac{e_{E_i} k_{E_i}}{c(f,\nu_{E_i})}Z(\nu_{E_i})
$$
by \eqref{eqn:divisorial_attraction_rate}, where $e_{E_i}=\ee{E_i}{E'}$ and $k_{E_i}=\kk{E_i}{E'}$.
We set $m(f,\nu_E)=e_Ek_E \in \nN^*$, and get the same expression as in \eqref{eqn:pullback_bdivisors}, with possibly one exception.
In fact, if $\nu_{E'} = f_\bullet \nu_C$ then the two expressions differ for a term $\frac{m(f,\nu_C)}{c(f,\nu_C)}Z(\nu_C)$.
Assume we can prove that $m(f,\nu)$ is uniformly bounded. Then since $C$ is a contracted curve, $c(f,\nu_C)=+\infty$ and $\frac{m(f,\nu_C)}{c(f,\nu_C)}$ vanishes, giving \eqref{eqn:pullback_bdivisors}.

Set now $M(\nu'):=\sum_{f_\bullet \nu = \nu'} m(f,\nu)$.
Since the pull back of b-divisors is continuous, the map
\begin{equation}\label{eqn:pullback_bdivisors_exc}
\nu \mapsto 
\sum_{f_\bullet \nu = \nu'} \frac{m(f,\nu)}{c(f,\nu)}Z(\nu)
\end{equation}
is continuous. By intersecting with $Z(\mf{m}_X)$, we deduce that $\nu' \mapsto M(\nu')$ is continuous as far as $c(f,\nu) < +\infty$, i.e., on $\mc{V}_Y \setminus f_\bullet (\VC{f})$.
We want to prove that $M(\nu')\leq e(f)$ the topological degree of $f$, concluding the proof.
Pick good resolutions $\pi \colon X_\pi \to (X,x_0)$ and $\varpi\colon Y_\varpi \to (Y,y_0)$ that give a resolution of $f$ with respect to $f_\bullet^{-1}(\nu')$.
We want to compute the number of preimages for a generic point $q \in Y_\varpi$ close to $E'$.
The map $\tilde{f}|_E : E \to E'$ has topological degree $e_{E}=\ee{E}{E'}$, hence any point $q$ in $E'$ has $e_{E}$ preimages in $E$ (counted with multiplicities).
For any such preimage $p$, pick local coordinates $(x,y)$ at $p$ and $(z,w)$ at $q$ so that $E=\{x=0\}$ and $E'=\{z=0\}$.
Up to taking higher models, we may assume that $\wt{f}$ is monomial. In these coordinates it is of the form
$$
\tilde{f}(x,y)=\big(x^ky^b \eps_1(x,y), y^d \eps_2(x,y)\big),
$$
for suitable $\eps_i$ non-vanishing at $0$, $b \in \nN$ and $d \in \nN^*$. Notice that $d$ is the multiplicity of $p$ as a solution of $\tilde{f}|_{E}(p)=q$.
It follows that the number of preimages near $p$ of a generic point near $q$ is $kd$.
If we sum over all preimages $p$ of $q$, we get $ke$ preimages of a (generic) point near $q$.
Summing up for all preimages $\nu_E$ of $\nu_{E'}$, we obtain the value $M(\nu')$, which is hence bounded by the topological degree.
\end{proof}

\begin{rmk}
When $f \colon (X,x_0) \to (Y,y_0)$ is \emph{finite}, $\nu' \mapsto M(\nu')$ is a constant map in the whole space $\mc{V}_Y$, which coincides with the topological degree of $f$.
From the proof of \refthm{operations_bdivisors}, we also deduce that for (possibly) \emph{non-finite} maps, the topological degree can still be computed as $e(f)=\max\{M(\nu')\ |\ \nu' \in \mc{V}_Y\}$.
One can also study how $m(f,\nu)$ varies when $\nu$ varies. Its properties are related to the behavior of tangent maps $df_\bullet$ at $\nu$ (see \hyperref[rmk:multiplicity_finite]{Remarks~\ref*{rmk:multiplicity_finite}} and \ref{rmk:multiplicity_nonfinite}).
\end{rmk}

\begin{defi}
Let $f\colon (X,x_0) \to (Y,y_0)$ be a dominant germ, and $\nu \in \mc{V}_X$ be a normalized valuation.
The value $m(f,\nu)$ of \refthm{operations_bdivisors} is called \emph{multiplicity of $f$ at $\nu$}.
\end{defi}

\begin{rmk}\label{rmk:pullback_bdivisors_rest}
Analogous formulae can be obtained for pushforwards and pullbacks of any b-divisor. In particular, in the setting of \refthm{operations_bdivisors}, for any exceptional nef b-divisor $Z \in \Ediv(Y)$ we get
\begin{equation}
\Rest \big(f^*Z\big)=\sum_{\nu_C \in \VC{f}} \big(-Z(f_\bullet \nu_C) \cdot Z\big) c_\alpha(f,\nu_C) \frac{-Z(C)}{m(C)}.\label{eqn:pullback_bdivisors_rest}
\end{equation}
\end{rmk}

\subsection{Angular distance is non-increasing}

We come now to a remarkable property of the maps $f_\bullet\colon \mc{V}_X\to \mc{V}_Y$, a property that makes a classification result like \refthm{valdynamics} possible: $f_\bullet$ is distance non-increasing with respect to the angular metrics on $\mc{V}_X$ and $\mc{V}_Y$.

\begin{thm}\label{thm:non-expansion} For each $\nu_1, \nu_2\in \mc{V}_X$ of finite skewness one has $\rho(f_\bullet\nu_1, f_\bullet\nu_2) \leq \rho(\nu_1, \nu_2)$. 
\end{thm}
\begin{proof} Using \refprop{famous_equality}, we have
\[
\rho(f_\bullet\nu_1,f_\bullet\nu_2) = \log\left(\sup_\mf{a} \frac{f_\bullet\nu_1(\mf{a})}{f_\bullet\nu_2(\mf{a})}\times \sup_\mf{a} \frac{f_\bullet\nu_2(\mf{a})}{f_\bullet\nu_1(\mf{a})}\right) = \log\left(\sup_\mf{a}\frac{\nu_1(f^*\mf{a})}{\nu_2(f^*\mf{a})}\times \sup_\mf{a}\frac{\nu_2(f^*\mf{a})}{\nu_1(f^*\mf{a})}\right),\]
the suprema taken over all $\mf{m}_Y$-primary ideals $\mf{a}\subset R_Y$.
The ideals $f^*\mf{a}$ will not be $\mf{m}_X$-primary if $f$ is not finite, but in any case for large enough $n\in \N$ the ideals $\mf{a}_n := f^*\mf{a} + \mf{m}_X^n$ are $\mf{m}_X$-primary and satisfy $\nu_i(\mf{a}_n) =\nu_i(\mf{a})$ for $i=  1,2$, and thus we can conclude that
\[
\rho(f_\bullet\nu_1, f_\bullet\nu_2) \leq \log\left(\sup_{\mf{b}}\frac{\nu_1(\mf{b})}{\nu_2(\mf{b})}\times \sup_\mf{b}\frac{\nu_2(\mf{b})}{\nu_1(\mf{b})}\right) = \rho(\nu_1, \nu_2),
\]
the suprema taken over all $\mf{m}_X$-primary ideals $\mf{b}\subset R_X$.
This completes the proof.
\end{proof}

We emphasize that this holds for \emph{every} dominant holomorphic map $f\colon (X,x_0)\to (Y,y_0)$. In particular, if we are in the dynamical setting where $(Y,y_0) = (X,x_0)$, then $\rho(f^n_\bullet \nu_1, f^n_\bullet\nu_2)\leq \rho(\nu_1, \nu_2)$ for all $n\geq 1$. This says that the family of maps $\{f^n_\bullet\}_{n\geq 1}$ is $\rho$-equicontinuous in a very strong sense. 

\begin{rmk} \refthm{non-expansion} is analogous to the well-known result in the theory of Riemann surfaces that if $f\colon \Sigma_1\to \Sigma_2$ is a holomorphic map between hyperbolic Riemann surfaces, then $f$ does not increase distances with respect to the hyperbolic metrics on $\Sigma_1$ and $\Sigma_2$. Using this fact and Montel's theorem, Fatou gave a classification of all the dynamics possible on a hyperbolic Riemann surface, see \cite{milnor:dyn1cplxvar, carleson-gamelin:cplxdyn}. With this in mind, our \refthm{valdynamics} can be viewed as an analogue of Fatou's classification theorem. 
\end{rmk}

In the following, we need a stronger characterization of when the map $f_\bullet$ is strictly decreasing the angular distance.

Given two distinct valuations $\nu, \mu \in \mc{V}_X$, we denote by $U_\nu(\mu)$ the connected component of $\mc{V}_X \setminus \{\nu\}$ containing $\mu$.

\begin{thm}\label{thm:strong_contraction}
Let $f \colon (X,x_0) \to (Y,y_0)$ be a dominant non-invertible germ between two normal surface singularities. For any $\nu \neq \mu \in \mc{V}_X$ of finite skewness, we have $\rho(f_\bullet \nu, f_\bullet \mu) \leq \rho(\nu,\mu)$.
The equality holds if and only if $f$ is finite, and moreover
\begin{itemize}
\item $\nu$ disconnects $\mu$ and any preimage of $f_\bullet \nu$, and symmetrically 
\item $\mu$ disconnects $\nu$ and any preimage of $f_\bullet \mu$.
\end{itemize}
\end{thm}
\begin{proof}
We want to prove the same estimate as in \refthm{non-expansion} in terms of b-divisors. This will allow to study the equality cases by \refprop{positivity_bdivisors}.
In particular, we will show that $\beta(f_\bullet \nu | f_\bullet \mu) \leq \frac{c(f,\mu)}{c(f,\nu)}\beta(\nu|\mu)$, which gives the desired inequality by symmetry.
This relations is equivalent to proving
\begin{equation}\label{eqn:angulardistance_bdivisor}
\big(Z(f_\bullet \nu) \cdot c(f,\nu) Z(f_\bullet \nu)\big) \big(Z(\nu) \cdot Z(\mu)\big) \leq 
\big(Z(f_\bullet \nu) \cdot c(f,\mu) Z(f_\bullet \mu)\big) \big(Z(\nu) \cdot Z(\nu)\big). 
\end{equation}
By \eqref{eqn:pushforward_bdivisors}, we get
\begin{align*}
c(f,\nu) Z(f_\bullet \nu) &= f_*Z(\nu) + \sum_{\nu_C \in \VC{f}}(-Z(\nu_C) \cdot Z(\nu))c_\alpha(f,\nu_C)Z(f_\bullet \nu_C),\\
c(f,\mu) Z(f_\bullet \mu) &= f_*Z(\mu) + \sum_{\nu_C \in \VC{f}}(-Z(\nu_C) \cdot Z(\mu))c_\alpha(f,\nu_C)Z(f_\bullet \nu_C).
\end{align*}
Notice that all coefficients of $Z(f_\bullet \nu_C)$ in the two expressions are positive.
By plugging these expressions in \eqref{eqn:angulardistance_bdivisor}, we get to the equivalent system of inequalities:
\begin{align}
\big(Z(f_\bullet \nu) \cdot f_*Z(\nu)\big) \big(Z(\mu) \cdot Z(\nu)\big) \leq & 
\big(Z(f_\bullet \nu) \cdot f_*Z(\mu)\big) \big(Z(\nu) \cdot Z(\nu)\big), \label{eqn:strongcontraction_bdivisors1}\\
\big(Z(\nu_C) \cdot Z(\nu)\big) \big(Z(\mu) \cdot Z(\nu)\big) \leq & 
\big(Z(\nu_C) \cdot Z(\mu)\big) \big(Z(\nu) \cdot Z(\nu)\big),\label{eqn:strongcontraction_bdivisors2}
\end{align}
where the equality holds for \eqref{eqn:angulardistance_bdivisor} if and only if the equality holds for all such equations.
Recall that by \eqref{eqn:pullback_bdivisors}, we have
$$
f^*Z(f_\bullet \nu) = \sum_{i=1}^r a_i Z(\nu_i) + \sum_{\nu_C \in \VC{f}} a_C Z(f_\bullet \nu_C)
$$
for suitable $a_i, a_C > 0$, where $\{\nu_1, \ldots, \nu_r\}$ are the preimages of $f_\bullet \nu$.
By using the projection formula and plugging in this pullback formula, \eqref{eqn:strongcontraction_bdivisors1} is equivalent to the system of inequalities
\begin{align}
\big(Z(\nu_i) \cdot Z(\nu)\big) \big(Z(\mu) \cdot Z(\nu)\big) \leq & 
\big(Z(\nu_i) \cdot Z(\mu)\big) \big(Z(\nu) \cdot Z(\nu)\big), \label{eqn:strongcontraction_bdivisors3}\\
\big(Z(f_\bullet \nu_C) \cdot Z(\nu)\big) \big(Z(\mu) \cdot Z(\nu)\big) \leq & 
\big(Z(f_\bullet \nu_C) \cdot Z(\mu)\big) \big(Z(\nu) \cdot Z(\nu)\big). \label{eqn:strongcontraction_bdivisors4}
\end{align}

All these inequalities \eqref{eqn:strongcontraction_bdivisors2}, \eqref{eqn:strongcontraction_bdivisors3} and \eqref{eqn:strongcontraction_bdivisors4} hold by \refprop{positivity_bdivisors}, and we deduce by symmetry that $\rho(f_\bullet\nu,f_\bullet\mu) \leq \rho(\nu,\mu)$.

Notice that if $f$ is finite, the set $\VC{f}$ is empty, and the inequalities \eqref{eqn:strongcontraction_bdivisors2} and \eqref{eqn:strongcontraction_bdivisors4} do not occur.
By \refprop{positivity_bdivisors}, the equality holds in \eqref{eqn:strongcontraction_bdivisors3} if and only if $\nu$ disconnects $\mu$ and $\nu_i$.
By symmetry, we get the statement for finite maps.

Assume now that $f$ is not finite.
Again by \refprop{positivity_bdivisors}, the equality in \eqref{eqn:strongcontraction_bdivisors2} holds if and only if
$\nu$ disconnects $\nu_C$ and $\mu$.
We apply the same argument with the role of $\nu$ and $\mu$ interchanged, and for the angular distance to be preserved, we need that $\nu$ disconnects $\nu_C$ and $\mu$, and $\mu$ disconnects $\nu_C$ and $\nu$.
It is easy to check that this cannot happen, and the angular distance strictly decreases for non-finite maps.
\end{proof}

We state here more explicitly what we proved for non-finite germs.

\begin{cor}\label{cor:strong_contraction_non-finite}
Let $f \colon (X,x_0) \to (Y,y_0)$ be a dominant \emph{non-finite} germ. Then $\rho(f_\bullet \nu, f_\bullet \mu) < \rho(\nu,\mu)$ for all valuations $\nu \neq \mu \in \mc{V}_X^\alpha$ of finite skewness.
\end{cor}

\subsection{The Jacobian formula}\label{ssec:Jacobian_formula}

Another important tool in understanding the maps $f_\bullet$ on valuation spaces is the \emph{Jacobian formula}, which says precisely how the log discrepancy function $A$ behaves with respect to this action.
This proof has been suggested by Charles Favre; a version for finite maps can be found in \cite[Prop.\ 1.9]{favre:holoselfmapssingratsurf}.

\begin{prop}[The Jacobian Formula]\label{prop:jacobian_formula}
Let $(X,x_0)$ and $(Y,y_0)$ be two normal surface singularities, and let $f\colon(X,x_0) \to (Y,y_0)$ be a dominant holomorphic map.
Then there exists a Weil divisor $R_f$ on $X$ so that
\begin{equation}\label{eqn:jacobian_formula}
A(f_*\nu)= A(\nu) + \nu(R_f)
\end{equation}
for any $\nu \in \hat{\mc{V}}_X^*$.
\end{prop}
We call $R_f$ the \emph{Jacobian divisor} of $f$.
\begin{proof}
By density, we can reduce ourselves to the case of divisorial valuations.
Moreover, since \eqref{eqn:jacobian_formula} is linear homogeneous, we may assume that $\nu = \divi_E$ for some exceptional prime $E$.
Let $\pi \colon X_\pi \to (X,x_0)$ and $\varpi \colon Y_\varpi \to (Y,y_0)$ give a resolution of $f$ with respect to $\{\nu_E\}$. 
By homogeneity of the log discrepancy, equation \eqref{eqn:jacobian_formula} in this setting becomes
\begin{equation}\label{eqn:jacobian_formula_ord}
k_E(1+\ord_{E'}(K_\varpi))= 1+\ord_E(K_\pi) + \ord_E(\pi^*R_f),
\end{equation}
where $K_\pi$ and $K_\varpi$ denote the relative canonical divisors of $\pi$ and $\varpi$ respectively.

Pick $\omega_Y = \omega$ a non-trivial $2$-form on $(Y,y_0)$, and set $\omega_X = f^*\omega$.
We may define canonical divisors as $K_X=\on{Div}(\omega_X)$, $K_{X_\pi} = \on{Div}(\pi^*\omega_X)$, and analogously for $K_Y$, $K_{Y_\varpi}$, so that, by definition,
\[
K_\pi= K_{X_\pi} - \pi^* K_X \quad \text{ and } \quad K_\varpi= K_{Y_\varpi} - \varpi^* K_{Y_\varpi}.
\]
Since $\wt{f}:X_\pi \to Y_\varpi$ is holomorphic, there exists an effective divisor $J_{\wt{f}}$ of $X_\pi$ so that
\[
K_{X_\pi}-\wt{f}^*K_{Y_\varpi} = J_{\wt{f}}.
\]
We plug in all expressions together, and we get
\begin{align*}
A(\nu)=1+\ord_E(K_\pi) & = 1+\ord_E(K_{X_\pi})-\ord_E (\pi^* K_X)\\
& = 1+\ord_E(\wt{f}^*K_{Y_\varpi}) + \ord_E(J_{\wt{f}})-\ord_E (\pi^* K_X)\\
& = 1+k_E\ord_{E'}(K_{Y_\varpi}) + k_E-1 -\ord_E (\pi^* K_X)\\
& = k_E\big[1+\ord_{E'}(K_\varpi) + \ord_F(\varpi^* K_Y)\big] -\ord_E (\pi^* K_X)\\
& = k_E A(\divi_{E'}) + k_E \ord_{E'}(\varpi^* K_Y) -\ord_E (\pi^* K_X)\\
& = A(f_* \nu) + \ord_E(\wt{f}^*\varpi^* K_Y) -\ord_E (\pi^* K_X).
\end{align*}
Set $R_{\wt{f}}= \pi^* K_X - \wt{f}^*\varpi^* K_Y$, so that the previous equations translate to
$$
A(f_* \nu) = A(\nu) + \ord_E(R_{\wt{f}}).
$$
Notice that
$$
R_{\wt{f}} = \pi^* K_X - \wt{f}^*\varpi^* K_Y 
= \pi^* \on{Div}(f^*\omega) - \wt{f}^*\varpi^* \on{Div}(\omega)  
= \pi^* \big[\on{Div}(f^*\omega) - f^* \on{Div}(\omega)\big] = \pi^* R_f,
$$
where we set $R_f = \on{Div}(f^*\omega) - f^* \on{Div}(\omega)$.
Notice that $R_f$ does not depend on the choice of $\pi$ and $\varpi$, since if $\pi' \geq \pi$ and $\varpi'\geq\varpi$ are two higher good resolutions so that $\wt{f}'=(\varpi')^{-1} \circ f \circ \pi'$ is regular, then $R_{\wt{f}'} = \eta_{\pi,\pi'}^* R_{\wt{f}}$, and their projections on $X$ coincide.
We then conclude
$$
A(f_* \nu) = A(\nu) + \ord_E(\pi^*R_f) = A(\nu) + \nu(R_f).
$$

\end{proof}

\begin{rmk}\label{rmk:jacobian_charp}
Here we use in an essential way the hypothesis on the zero characteristic of the field. In fact the Jacobian formula is not valid, not even in the smooth case, for dominant maps between surfaces defined over positive characteristic fields (see \refex{Frobenius4}).
\end{rmk}

\begin{rmk}\label{rmk:jacobian_formula_finite}
For normalized semivaluations, the Jacobian formula \eqref{eqn:jacobian_formula} can be rewritten as
\begin{equation}\label{eqn:jacobian_formula_bullet}
c(f, \nu) A(f_\bullet \nu) = A(\nu) + \nu(R_f).
\end{equation} 
Notice that for any $x \in X$ so that $f(x) \neq y_0$, then $(X,x)$ and $(Y,f(x))$ are both smooth, and $R_f$ is given, locally at $x$, by the Weil divisor defined by the Jacobian determinant of $f\colon(X,x) \to (Y,f(x))$.
In particular, the divisor $R_f$ is effective whenever $f$ is finite, but it may have negative components at the contracted curves, as the following example shows.
\end{rmk}

\begin{ex}
Let $(X,x_0)$ be any non log canonical singularity. Let $\pi\colon X_{\pi} \to (X,x_0)$ be any good desingularization of $(X,x_0)$, and let $E$ be an exceptional prime for $\pi$ so that $A(E) < 0$. By continuity of log discrepancy on dual graphs, for any $p \in E$ there exists a valuation $\nu \in \mc{V}_{X_{\pi}}$ centered at $p$ such that $A(\pi_*\nu)<0$.
The Jacobian formula \eqref{eqn:jacobian_formula} applied to $\pi\colon(X_{\pi},p) \to (X,x_0)$ at $\nu$ gives
$$
\nu(R_\pi) = A(\pi_* \nu)- A(\nu) < 0,
$$
where the last inequality follows from $A(\nu)> 0$ since $X_{\pi_0}$ is smooth.
In particular, $R_\pi$ is not effective.
\end{ex}

Non-effective Jacobian divisors are quite common, even for dynamical systems. Here we give some construction of selfmaps with some prescribed behavior.

\begin{lem}\label{lem:nutoord0}
Let $\mc{V}=\mc{V}_{\nC^2}$ denote the valuative tree. Then for every divisorial valuation $\nu \in \mc{V}$, there exists a dominant superattracting germ $f:(\nC^2,0) \to (\nC^2,0)$ so that $f_\bullet \nu = \ord_0$. The map $f$ may be chose both finite or non-finite.
\end{lem}
\begin{proof}
First, assume that $m_0(\nu)=1$, where $m_0$ denotes the multiplicity at $0$ of $\nu$, see \cite[p. 63]{favre-jonsson:valtree}.
We may find coordinates $(x,y)$ at $0 \in \nC^2$ so that $\nu$ belongs to the segment $[\ord_0,\nu_x)$, where $\nu_x$ is the curve semivaluation associated to $\{x=0\}$.
Then there exists a monomial map $f_1(x,y)=(x^a y^b,x^c y^d)$, with $a,b,c,d \in \nN$ and $ad-bc \neq 0$ (we may assume $ad-bc= \pm 1$), so that $(f_1)_\bullet \nu = \ord_0$.
We may pick the coefficients so that $f_1$ is not finite.

Assume now that $m_0(\nu) \geq 2$. Up to a linear change of coordinates, we may assume that $\nu$ belongs to the connected component $U$ of $\mc{V} \setminus \{\ord_0\}$ containing $\nu_x$.
Let $C=\{\phi(x,y)=0\}$ be an irreducible curve so that the associated curve semivaluation $\nu_C$ satisfies $\nu < \nu_C$.
Then the map $f_2(x,y)=(\phi(x,y),y)$ is order preserving on $U$, and sends $\nu_C$ to $\nu_x$. Hence $(f_2)_\bullet \nu$ is a divisorial valuation (of multiplicity $1$) which belongs to $[\ord_0,\nu_x)$.
The map $f=f_1 \circ f_2$ satisfies the desired properties.

If $f$ is not finite, a small perturbation with an element of $\mf{m}^N$ for $N$ big enough gives a finite map which preserves the property $f_\bullet \nu = \ord_0$.
\end{proof}

In the next proposition, we call $\emph{essential}$ divisorial valuation any divisorial valuation $\nu_E \in \mc{V}_X$ associated to an exceptional prime $E$ in the minimal (possibly non good) resolution of $(X,x_0)$.

\begin{prop}\label{prop:nutomu}
Let $(X,x_0)$ be a normal surface singularity, $\nu \in \mc{V}_X$ be any \emph{divisorial} valuation, and $\mu \in \mc{V}_X$ a \emph{non-essential} divisorial valuation.
Then there exists a (non-finite) superattracting germ $f\colon(X,x_0) \to (X,x_0)$ so that $f_\bullet \nu = \mu$.
\end{prop}
\begin{proof}
The hypothesis on $\mu$ assures that there exists a (not necessarily good) resolution $\pi:X_{\pi} \to (X,x_0)$ and a point $p \in \pi^{-1}(x_0)$ so that $\mu = \pi_\bullet \ord_p$, where $\ord_p$ is the vanishing order at $p$ (also called the multiplicity valuation at $p$).
Embed the singularity $(X,x_0)$ in an affine space $(\nC^n,0)$, and take a (generic) projection to $(\nC^2,0)$. Call $g\colon(X,x_0) \to (\nC^2,0)$ the composition.
By \reflem{nutoord0} applied to the divisorial valuation $g_\bullet\nu$, there exists a superattracting germ $h \colon (\nC^2,0) \to (\nC^2,0)$ so that $h_\bullet g_\bullet \nu = \ord_0$.
Let $\sigma \colon (\nC^2,0) \to (X_\pi,p)$ be any local isomorphism sending $0$ to $p$, so that $\sigma_\bullet \ord_0 = \ord_p$. 
Then the map $f=\pi \circ \sigma \circ h \circ g \colon (X,x_0) \to (X,x_0)$ satisfies
$$
f_\bullet \nu = \pi_\bullet \sigma_\bullet h_\bullet g_\bullet \nu = \pi_\bullet \sigma_\bullet \ord_0 = \pi_\bullet \ord_p = \mu.
$$
\end{proof}
Notice that the condition of $\mu$ being non-essential is sufficient, but not necessary.
On the one hand, it is not always possible to send $\nu$ to an essential divisorial valuation $\mu$.
An easy example is given by a divisorial valuation $\mu=\nu_{E'}$ with $E'$ a non-rational exceptional prime. If $f_\bullet$ sends a divisorial valuation $\nu_E$ to $\mu$, then by \refprop{divisorial_image} $f$ induces a non-constant map $E \to E'$, and $g(E) \geq g(E')$, where $g$ denotes the genus.

On the other hand, essential valuations can be eigenvaluations, even for non-finite maps. See \S \ref{ssec:example_quotient1}, \S\ref{ssec:example_elliptic}, \S\ref{ssec:example_quasihom} for specific examples.

We conclude with some considerations on the effectiveness of the Jacobian divisor for non-finite selfmaps.

\begin{rmk}
Let $(X,x_0)$ be either non lc singularity, or (the finite quotient of) a cusp.
Let $\nu \in \mc{V}_X$ be any divisorial valuation with positive log discrepancy $A(\nu) > 0$, and $\mu \in \mc{V}_X$ any non-essential divisorial valuation with non-positive log discrepancy $A(\mu) \leq 0$.
In the first case, this is possible because the log discrepancy $A$ is continuous on dual graphs, and by definition there exists an essential valuation with negative log discrepancy.
In the second case, the subset of $\mc{V}_X$ where the log discrepancy is zero is not discrete, and we can find a non-essential divisorial valuation with zero log discrepancy.
By \refprop{nutomu}, there exists a superattracting germ $f\colon(X,x_0) \to (X,x_0)$ for which $f_\bullet \nu = \mu$.
Then by the Jacobian formula \eqref{eqn:jacobian_formula_bullet}, we get
$$
\nu(R_f) = c(f, \nu) A(f_\bullet \nu) - A(\nu) = c(f, \nu) A(\mu)-A(\nu) < 0,
$$
and $R_f$ is not effective.
\end{rmk}

With similar arguments, one can prove that Jacobian divisors are always effective as far as the target space is \emph{canonical}.
We recall that a singularity $(Y,y_0)$ is \emph{canonical} (resp., \emph{terminal}) if for any exceptional prime $E$ in any good resolution, we have
$\ord_E(K_\varpi) = A(\divi_E)-1 \geq 0$ (resp., $> 0$), where $\varpi\colon Y_\varpi \to (Y,y_0)$ is any good resolution.

\begin{prop}
Let $f \colon (X,x_0) \to (Y,y_0)$ be a dominant germ between normal surface singularities.
If $(Y,y_0)$ is \emph{canonical}, then the jacobian divisor $R_f$ is effective.
\end{prop}
\begin{proof}
Let $C$ be an irreducible contracted critical curve, and $\nu_C$ the associated semivaluation. Let $\alpha_C C$ be the component of $C$ in $R_f$. We want to show that $\alpha_C \geq 0$.
We proceed as in the proof of \refprop{contractedcurve_image}.
Let $\pi\colon X_\pi \to (X, x_0)$ and $\varpi\colon Y_\varpi \to (Y,y_0)$ give a resolution of $f$ with respect to $\{\nu_C\}$.
The strict transform $C_\pi$ of $C$ intersects $\pi^{-1}(x_0)$ in a unique exceptional prime $H$, transversely at a point $p$.
Up to taking higher models, we may also assume that $\pi$ is a log resolution of $\mf{m}_X$, and that the strict transform of any other irreducible component of $f^{-1}(y_0)$ doesn't pass through $p$.
Set $b_H=m(C)=:m$.
Take local coordinates $(x,y)$ at $p$ so that $H=\{x=0\}$ and $C_\pi=\{y=0\}$.
For any $t \geq 0$, let $\mu_t$ be the monomial valuation at $p$ of weights $\frac{1}{m}$ and $t$ with respect to the coordinates $(x,y)$, and set 
$\nu_t=\pi_* \mu_t \in \mc{V}_X$.
By direct computation, we have $A(\nu_t)=A(\nu_0)+t$, while $\nu_t(C)= t + \text{const}$.
Taking the limit for $t \to +\infty$ of the Jacobian formula \eqref{eqn:jacobian_formula_bullet} divided by $t$ on both sides, we get
$$
\alpha_C=\lim_{t \to +\infty} \frac{\nu_t(R_f)}{t} = \lim_{t \to +\infty} \frac{c(f,\nu_t) A(f_\bullet \nu_t) - A(\nu_t)}{t} =c_A(f,\nu_C) A(f_\bullet \nu_C) - 1
$$
Set $f_\bullet \nu_C=\nu_G$.
By \refprop{contractedcurve_image}, we get $c_A(f,\nu_C) = \kk{C}{G}b_G$, and 
$$
\alpha_C=\kk{C}{G} b_G A(\nu_G)-1 = \kk{C}{G} A(\divi_G) -1 \geq A(\divi_G) -1 \geq 0.
$$
\end{proof}

\begin{rmk}\label{rmk:contractedcurves}
Notice that if $(Y,y_0)$ is \emph{terminal}, then every irreducible component of $f^{-1}(y_0)$ has a non-trivial contribution to $R_f$.
This is not always the case for non-terminal singularities, see \S\ref{ssec:example_elliptic} for an example.

In particular, not all curves contracted to $y_0$ belong to the support of $R_f$. Nevertheless, following the notations used in the proof of \refprop{jacobian_formula}, their strict transforms belong to the critical set of the lift $\tilde{f}$.
In particular, contracted curves are necessarily holomorphic.
\end{rmk}

\subsection{Critical skeleton}

Here we study how the maps $\nu \mapsto c(f, \nu)$ and $\nu \mapsto \nu(R_f)$ behave.
It turns out that in fact, they are locally constant outside a finite graph.
In the smooth setting, such maps belong to the class of \emph{tree potentials}, see \cite[Section 1.8]{favre-jonsson:eigenval}.
To describe the analogous on valuative spaces associated to normal surface singularities $(X,x_0)$, we need to introduce a few objects.
We already introduced in \S\ref{ssec:dualgraphs} the skeleton $\skel{X}$ of $X$, which is given by the embedded dual graph of any minimal good resolution of $X$, and the essential skeleton $\eskel{X}$ in \S\ref{ssec:essential_skeleton}.
In the following we will need to consider other subgraphs of $\mc{V}_X$ depending on $\mf{m}_X$-primary ideals and finite sets of valuations. 

Recall that for any $\mf{m}_X$-primary ideal $\mf{a}$, the associated b-divisor $Z(\mf{a})$ is Cartier (and nef).
In particular, there exists a finite set $\rees(\mf{a})$ of divisorial valuations so that $Z(\mf{a})=\sum_{\nu} a_\nu Z(\nu)$, with $a_\nu > 0$ for all $\nu \in \rees(\mf{a})$. The set $\rees(\mf{a})$ is called the set of \emph{Rees valuations} of $\mf{a}$.

\begin{defi}
Let $(X,x_0)$ be a normal surface singularity, and $V \subset \mc{V}_X$ be a finite set of semivaluations.
Denote by $r_X \colon \mc{V}_X \to \mc{S}_X$ the retraction map.
We call the set $\displaystyle \mc{S}_X(V):=\mc{S}_X \cup \bigcup_{\nu \in V} [r_X \nu, \nu]$ the \emph{skeleton generated by} $V$.
If moreover $\mf{a}$ is a $\mf{m}_X$-primary ideal, we refer to $\mc{S}_X(V \cup \rees(\mf{a}))$ as the \emph{skeleton generated by} $V$ and \emph{adapted to} $\mf{a}$. We denote it by $\mc{S}_\mf{a}(V)$.
\end{defi}

Recall that given two valuations $\nu, \mu \in \mc{V}_X$, we may compute the intersection of the associated b-divisors $Z(\mu) \cdot Z(\nu) \in [-\infty,0)$.
We now fix $\mu$, and study how this intersection varies when $\nu$ varies in the valuative space $\mc{V}_X$.

\begin{prop}\label{prop:intersection_locallyconst}
Let $(X,x_0)$ be a normal surface singularity, and $\mu \in \mc{V}_X$ be any normalized semivaluation.
Then the map $g\colon \nu \mapsto Z(\nu) \cdot Z(\mu)$ is locally constant on $\mc{V}_X \setminus \mc{S}_{\mf{m}_X}(\mu)$.
\end{prop}

\begin{proof}
Denote by $r \colon \mc{V}_X \to \mc{S}_{\mf{m}_X}(\mu)$ the natural retraction.
To avoid trivial cases, we may assume that $\nu \not \in \mc{S}_{\mf{m}_X}(\mu)$. In particular
we want to prove that
$
Z(\nu) \cdot Z(\mu)=Z(r\nu) \cdot Z(\mu)
$.

Notice that either $\nu=r\nu$, or $\nu$ and $\mu$ belong to different components of $\mc{V}_X \setminus \{r\nu\}$.
Then by \refprop{positivity_bdivisors}, we have
$$
(Z(r\nu) \cdot Z(\nu)) (Z(r\nu) \cdot Z(\mu)) = (Z(\nu) \cdot Z(\mu)) (Z(r\nu) \cdot Z(r\nu)).
$$
By \refprop{baby_tree} we have $r\nu \leq \nu$, and by \reflem{monotonicity} we have $\beta(r\nu|\nu)=1$, which written in terms of b-divisors gives
$$
Z(r\nu)\cdot Z(r\nu)=Z(r\nu) \cdot Z(\nu).
$$
Putting these last two equations together, we deduce the desired equality.
\end{proof}

\refprop{intersection_locallyconst} allows to study the properties of other functionals over $\mc{V}_X$. In particular, we will study the maps $\nu \mapsto \nu(\phi)$ for any $\phi \in R_X$, $\nu \mapsto \nu(R_f)$ for the Jacobian divisor $R_f$, and $\nu \mapsto c(f,\nu)$.

\smallskip
	
Let $\pi \colon X_\pi \to (X,x_0)$ be a good resolution, and $E \in \varGamma_\pi$ be an exceptional prime of $\pi$.
Let $D \in \Weil(\pi)$ be the divisor associated to $\phi \circ \pi$, so that $\divi_E(\phi)=\ord_E(\phi \circ \pi)$ is given by the coefficient of $E$ in $D$.
Since the non-exceptional part of $D$ coincides with the strict transform $C_\pi$ of the curve $C=\{\phi=0\}$, and $D \cdot E=0$ for all exceptional primes, we infer that
$$
D= C_\pi - \sum_{k} d_k Z_\pi(\inte_{C_k}),
$$
where we decomposed the curve $C$ in its irreducible components $C_k$ of multiplicity $d_k$.

To simplify notations, we set $Z_\pi(\inte_C)=Z_\pi(\inte_\phi)=\sum_{k} d_k Z_\pi(\inte_{C_k})$, and denote by $Z(\inte_C)=Z(\inte_\phi)$ its associated b-divisor.
Notice that in general the map $\inte_C$ of local intersection with $C=\{\phi=0\}$ fails to be a valuation (unless $C$ is irreducible).

The coefficient of $E$ in $D$ is then given by $-\check{E} \cdot Z(\inte_{C})$.
Since divisorial valuations are dense in $\mc{V}_X$, by continuity we deduce the following result.

\begin{prop}
Let $(X,x_0)$ be a normal surface singularity, $\phi \in R_X$ and $C=\{\phi=0\}$. Then 
\begin{equation}\label{eqn:valuationasintersection}
\nu(\phi)= -Z(\nu) \cdot Z(\inte_{C})
\end{equation}
for all $\nu \in \hat{\mc{V}}_X$.

In particular, $\nu \mapsto \nu(\phi)$ is locally constant on $\mc{V} \setminus \mc{S}_{\mf{m}_X}(\{\nu_{C_k}\})$, where $C_k$ varies among the irreducible components of $C$.
\end{prop}

Similarly, for Jacobian divisors we get
\begin{cor}\label{cor:jacobiandivisor_locallyconst}
Let $f\colon(X,x_0) \to (Y,y_0)$ be a dominant map between two normal surface singularities, and denote by $R_f$ the Jacobian divisor of $f$.
Then the map $\nu \mapsto \nu(R_f)$ is locally constant outside the skeleton $\mc{S}_{R_f}$ adapted to $\mf{m}_X$ and generated by the curve valuations associated to irreducible components in the support of $R_f$.
\end{cor}

We now focus our attention on the map $\nu \mapsto c(f, \nu)$.
First, we give an interpretation of the attraction rate $c(f,\nu)$ in the setting of b-divisors.

\begin{prop}\label{prop:attractionrate_intersection}
Let $f\colon (X,x_0) \to (Y,y_0)$ be a dominant germ between two normal surface singularities, and $\nu \in \mc{V}_X$.
Then
\begin{equation}\label{eqn:attractionrate_intersection}
c(f,\nu) = - Z(\nu) \cdot \Exc \big(f^* Z(\mf{m}_Y)\big).
\end{equation}
\end{prop}
\begin{proof}
Applying \eqref{eqn:pushforward_bdivisors} and intersecting with $-Z(\mf{m}_Y)$, we have
$$
-f_*Z(\nu) \cdot Z(\mf{m}_Y) = c(f,\nu) - \sum_{\nu_C \in \VC{f}}\big(-Z(\nu_C)\cdot Z(\nu)\big)c_\alpha(f,\nu_C),
$$
where we used the fact that $-Z(\nu') \cdot Z(\mf{m}_Y)=1$ for all normalized semivaluations $\nu' \in \mc{V}_Y$.
By the projection formula, $-f_*Z(\nu)\cdot Z(\mf{m}_y) = -Z(\nu) \cdot f^*Z(\mf{m}_Y)$.
By \eqref{eqn:pullback_bdivisors_rest} we have
\begin{align*}
-Z(\nu) \cdot \Rest\big(f^*Z(\mf{m}_Y)\big) 
&=-\sum_{\nu_C \in \VC{f}} \big(-Z(f_\bullet \nu_C) \cdot Z(\mf{m}_Y)\big) c_\alpha(f,\nu_C) \frac{-Z(C)}{m(C)}\cdot Z(\nu)\\
&=-\sum_{\nu_C \in \VC{f}} c_\alpha(f,\nu_C) \big(-Z(\nu_C)\cdot Z(\nu)\big),
\end{align*}
where we used the fact that $Z(C) \cdot Z = Z(\inte_C) \cdot Z$ for any exceptional b-divisor $Z \in \Ediv(X)$.
Equation \eqref{eqn:attractionrate_intersection} directly follows, recalling that $f^*Z(\mf{m}_Y) =\Exc\big(f^*Z(\mf{m}_Y)\big) +\Rest\big(f^*Z(\mf{m}_Y)\big)$. 
\end{proof}

\begin{cor}\label{cor:attractionrate_locallyconst}
Let $f\colon(X,x_0) \to (Y,y_0)$ be a dominant map between two normal surface singularities.
Set $V_f=\VC{f} \cup f_\bullet^{-1}(\rees_Y)$, where $\VC{f}$ is the set of contracted curve valuations, and $\rees_Y=\rees(\mf{m}_Y)$ is the set of Rees valuations of the maximal ideal $\mf{m}_Y$.
Then the function $\nu \mapsto c(f, \nu)$ is locally constant on $\mc{V}_X \setminus \mc{S}_{c(f)}$, where $\mc{S}_{c(f)}=\mc{S}_{\mf{m}_X}(V_f)$ is the skeleton generated by $V_f$ and adapted to $\mf{m}_X$.
\end{cor}
\begin{proof}
By \refprop{attractionrate_intersection}, we may compute $c(f,\nu)$ as minus the intersection of $Z(\nu)$ with the exceptional part of $f^*(\mf{m}_Y)$.
Taking the exceptional part of \eqref{eqn:pullback_bdivisors}, we have
$$
\Exc\big(f^*Z(\mf{m}_Y)\big)=\sum_{\nu \in V_f} a_\nu Z(\nu),
$$
for suitable $a_\nu > 0$. We conclude by \refprop{intersection_locallyconst}.
\end{proof}

Notice that when $(Y,y_0) = (\nC^2,0)$ is smooth, then $Z(\mf{m}) = Z(\ord_0)$, where $\ord_0$ is the multiplicity valuation at $0 \in \nC^2$.
Then $\rees_{\nC^2}=\{\ord_0\}$, and \refcor{attractionrate_locallyconst} coincide with the characterization given by \cite[Proposition 3.4]{favre-jonsson:eigenval}.

The skeleta $\mc{S}_{c(f)}$ and $\mc{S}_{R_f}$ will play an important role on the study of the dynamical features of $f_\bullet$.
We will denote by $\mc{S}_f = \mc{S}_{c(f)} \cup \mc{S}_{R_f}$ their union, and refer to it as the \emph{critical skeleton} of $f$. It is generated by a finite number of divisorial or curve semivaluations, and the functions $\nu \mapsto c(f, \nu)$ and $\nu \mapsto \nu(R_f)$ are both locally constant on $\mc{V}_X \setminus \mc{S}_f$.

\subsection{Classification of valuative dynamics}

Finally, we come now to a precise statement of \refthm{valdynamics}, our classification of dynamics on the valuation spaces $\mc{V}_X$. The setup for the theorem the following: $(X,x_0)$ is an irreducible normal surface germ, $f\colon (X,x_0)\to (X,x_0)$ is a dominant  holomorphic map, and $R_f$ is the Weil divisor on $(X,x_0)$ defined by the vanishing of the Jacobian determinant of $f$, as discussed in \S\ref{ssec:Jacobian_formula}. In addition, we make the crucial assumption that $f$ is \emph{non-invertible}; note, everything discussed in this section up to now is valid without this assumption, but we will use it heavily in our proof of the classification.
In fact for \emph{invertible} germs, the action $f_\bullet$ is always an isometry with respect to the angular distance $\rho$. Moreover, $A(f_\bullet \nu)=A(\nu)$ and $c(f, \nu) = 1$ for all valuations $\nu \in \mc{V}_X$.

In the statement, we will use the following terminology. A set $S\subseteq\mc{V}_X$ is said to be \emph{totally invariant} if $f_\bullet^{-1}(S) = S$. This is stronger than usual notion of invariance $f_\bullet(S)\subseteq S$. If $f_\bullet$ is surjective and $S$ is totally invariant, then $f_\bullet(S) = S$. 

\begin{thm}\label{thm:classification}
For the dynamical system $f_\bullet\colon \mc{V}_X\to \mc{V}_X$, exactly one of the following statements holds.
\begin{enumerate}
\item The map $f$ is not finite. In this case, there is a semivaluation $\nu_\star\in \mc{V}_X$ for which $f_\bullet^n\nu\to \nu_\star$ weakly as $n\to \infty$ for all $\nu\in \mc{V}_X$ of finite skewness.
Obviously any such $\nu_\star$ is unique and fixed by $f_\bullet$.
If $\nu_\star$ is of finite skewness, then in fact $f^n_\bullet\nu\to \nu_\star$ in the strong topology for each $\nu\in \mc{V}_X$ of finite skewness.
\item The map $f$ is finite and $R_f\neq 0$. In this case, $(X,x_0)$ must be a quotient singularity (or non-singular). There is a subset $I\subset \mc{V}_X$ which is homeomorphic to a closed interval or a point satisfying the following properties: (a) $I$ is totally invariant, (b) $I$ is fixed pointwise by $f_\bullet^2$, and (c) for every $\nu\in \mc{V}_X$ of finite skewness there is a $\nu_\star\in I$ such that $f^{2n}_\bullet\nu\to \nu_\star$ weakly as $n\to \infty$. If this $\nu_\star$ itself has finite skewness, then in fact $f^{2n}_\bullet\nu\to \nu_\star$ in the strong topology.
\item The map $f$ is finite and $R_f = 0$. In this case $(X,x_0)$ is a lc but not lt singularity. The set $S := A^{-1}(0)\subset \mc{V}_X$ is totally invariant for $f_\bullet$, and $f_\bullet|_S$ is an isometry for the angular distance $\rho$. The set $S$ is either a point (for simple elliptic singularities or their quotients), a segment (quotient-cusp singularities) or a circle (cusp singularities). Denote by $r_S\colon \mc{V}_X \to S$ the natural retraction. Then for every $\nu\in \mc{V}_X$ of finite skewness we have $\rho(f^n_\bullet \nu, f^n_\bullet r_S\nu) \to 0$ as $n\to \infty$. 
\end{enumerate}
\end{thm}

An important special case of \refthm{classification} was proved in the papers \cite{gignac-ruggiero:attractionrates, ruggiero:rigidification}, namely the case when $(X,x_0) = (\C^2,0)$ is non-singular. More precisely, \cite{gignac-ruggiero:attractionrates} proved the theorem for dominant holomorphic maps $f\colon (\C^2,0)\to (\C^2,0)$ which are \emph{superattracting}, that is, for which the differential $df_0$ is nilpotent. If $f$ is non-invertible but not superattracting, then $df_0$ has exactly one non-zero eigenvalue; in this case $f$ is said to be \emph{semi-superattracting}, and the theorem was proved in \cite{ruggiero:rigidification}. Note, the only two statements that apply in the non-singular setting are the first two. 

A strong dichotomy appears in the statement of \refthm{classification}.
When the map $f\colon(X,x_0)\to (X,x_0)$ is \emph{non-finite}, we will use the contraction properties of $f_\bullet$ with respect to the angular distance given by \refcor{strong_contraction_non-finite}.
When the map $f$ is \emph{finite}, we will apply the following results, see Wahl \cite{wahl:charnumlinksurfsing}, \cite{favre:holoselfmapssingratsurf} for proofs.

\begin{thm}[{\cite{wahl:charnumlinksurfsing}, \cite[Theorem B]{favre:holoselfmapssingratsurf}}]\label{thm:wahl}
Let $f\colon(X,x_0) \to (X,x_0)$ be a non-invertible germ. 
If $f$ is finite, then $(X,x_0)$ is log canonical. Moreover, $(X,x_0)$ is log terminal if and only if $R_f\neq 0$. 
\end{thm}

\begin{rmk}
\refthm{wahl} holds only over fields of characteristic zero.
Favre's strategy to prove such a result is to use the Jacobian formula \eqref{eqn:jacobian_formula} (which does not hold in positive characteristic), and use the fact that the map $f_\bullet$ is surjective for finite germs by \refprop{propertiesfbullet}.
Notice that in positive characteristic, any singularity (defined over $\nF_q$) admits non-invertible finite germs (see \refex{Frobenius5}).
\end{rmk}

The next two sections will be devoted to the proof of \refthm{classification} in the non-finite and in the finite cases.

\section{Dynamics of non-finite germs}\label{sec:dynamics_nonfinite}

This section is devoted to the proof of \refthm{classification} in the non-finite case.

\subsection{Construction of an eigenvaluation}

The strong contraction property stated in \refthm{strong_contraction} will allow us to construct a unique (attracting) fixed valuation $\nu_\star \in \mc{V}_X$ for $f_\bullet$. An analogous construction in the smooth case can be found in \cite[Theorems~4.2 and 4.5]{favre-jonsson:eigenval}.

We first need a few definitions.

\begin{defi}\label{def:locallyattracting}
Let $(X,x_0)$ be a normal surface singularity, and $f \colon (X,x_0) \to (X,x_0)$ be a non-invertible dominant germ.
A normalized semivaluation $\nu_\star \in \mc{V}_X$ is called \emph{locally attracting} if it is fixed by $f_\bullet$ and there exists a weak open set $U \subset \mc{V}_X$ containing $\nu_\star$ such that $f_\bullet^n \nu \to \nu_\star$ in the weak topology, \begin{change}for any $\nu \in U$ of finite skewness.\end{change}
\end{defi}

\begin{change}
\begin{rmk}
We will be interested in locally attracting semivaluations mainly when they are non-quasimonomial, which correspond to \emph{ends} in the language of $\nR$-trees.
Similar definitions of local attracting behaviors were previously introduced in \cite[Section 4.1]{favre-jonsson:eigenval} for ends of the valuative tree.
Transferring their definition to our setting, a non-quasimonomial semivaluation $\nu_\star \in \mc{V}_X \setminus \mc{V}_X^{\qm}$ is a \emph{strongly attracting end} if there exists a $f_\bullet$-invariant segment $I=[\nu_0, \nu_\star]$ with $\nu_0 < \nu_\star$ so that $f_\bullet \nu > \nu$ for all $\nu \in [\nu_0,\nu_\star)$.
We will see that a strongly attracting end is locally attracting in the sense of \refdef{locallyattracting}.
In fact, in the case of a strongly attracting end $\nu_\star$, one gets the existence of a weak open neighborhood $U$ of $\nu_\star$ so that $f_\bullet^n\nu \to \nu_\star$ for all $\nu \in U$ not necessarily of finite skewness.

The behavior of semivaluations with infinite skewness (in particular, of curve semivaluations) may be quite chaotic, see e.g. \cite{gignac:conjarnold}.
As an easy example, consider the map $f\colon (\nC^2,0) \to (\nC^2,0)$ given in coordinates by $f(x,y)=(x^2,xy)$.
In this case $f_\bullet$ admits a unique quasimonomial fixed point, the multiplicity $\ord_0$.
On the one hand, the orbit of every valuation of finite skewness is attracted by $\ord_0$, and $\ord_0$ is locally attracting.
On the other hand, the curve semivaluation $\nu_{C_\theta}$ associated to the curve $C_\theta=\{y=\theta x\}$ is fixed by $f_\bullet$ for any $\theta \in \nC$. It follows that there are no weak open neighborhoods $U$ of $\ord_0$ where $f_\bullet^n \nu \to \nu_\star$ weakly for all $\nu \in U$. 
 
\end{rmk}	
\end{change}

\begin{change}
\begin{defi}
Let $(X,x_0)$ be a normal surface singularity, and $f \colon (X,x_0) \to (X,x_0)$ be a non-invertible dominant germ.
We say that a normalized semivaluation $\nu \in \mc{V}_X$ fixed by $f_\bullet$ is an \emph{eigenvaluation} for $f$ if it is either a locally attracting end, or if it belongs to the (weak) closure of the set $\on{Fix}(f_\bullet) \cap \mc{V}_X^{\qm}$ of fixed quasimonomial valuations.
\end{defi}
\end{change}

\begin{rmk}
Notice that, by definition, all quasimonomial valuations fixed by $f_\bullet$ are eigenvaluations; nevertheless we could have fixed semivaluations that are not eigenvaluations.
Consider for example the map $f\colon(\nC^2,0)\to(\nC^2,0)$ given by $(x,y)=(x,y^d)$, where $d \geq 2$.
Denote by $\nu_{x,t}$ the monomial valuation of weights $(t,1)$ with respect to the coordinates $(x,y)$.
Then $f_\bullet \nu_{x,t} = \nu_{x,t/d}$ for any $t \geq d$. In particular, the curve semivaluation $\nu_x=\nu_{x,+\infty}$ is fixed, but it is not an eigenvaluation, since it is not locally attracting.

Quasimonomial valuations are not necessarily locally attracting, since they could belong to a segment or circle of fixed eigenvaluations, as indicated by \refthm{classification}.
Consider for example the map $f\colon(\nC^2,0)\to(\nC^2,0)$ given by $(x,y)=(x^d,y^d)$ for any $d \geq 2$. In this case the eigenvaluations are all elements of the segment $[\nu_x, \nu_y]$, since valuations in $I=(\nu_x, \nu_y)$ are fixed quasimonomial valuations, and $\nu_x$ and $\nu_y$ are fixed and belong to the weak closure of $I$.
\end{rmk}

\begin{change}
\begin{rmk}
From the study of the dynamics of $f_\bullet$ that will lead to the proof of \refthm{classification} in this and the next sections, we may deduce several properties of the set of eigenvaluations of a given map $f$.
In particular, we will show that if $f$ admits a locally attracting eigenvaluation, then this is the unique eigenvaluation for $f$.

We can also deduce an alternative description of the set of eigenvaluations of a given map $f$.
It is the weak closure of the set of semivaluations $\nu_\star \in \mc{V}_X$ admitting a weak open set $U \subseteq \mc{V}_X$ so that $\nu_\star \in \overline{U}$ and for any semivaluation $\nu \in U$ of finite skewness we have $f_\bullet^n \nu \to \nu_\star$ weakly as $n \to +\infty$. 

Notice that when $\nu_\star \in U$, this is exactly the definition of locally attracting.
\end{rmk}	
\end{change}

We will need the following (easy) fixed point theorem, see for example \cite{edelstein:fixedpoint}.
\begin{lem}\label{lem:fixedpoint}
Let $(S,\rho)$ be a compact metric space, and $F\colon S \to S$ be a continuous selfmap satisfying
$$
\rho(F(v),F(w))<\rho(v,w) \quad \text{ for all } v \neq w.
$$
Then there exists a unique fixed point $v_\star \in S$, and $F^n(v) \to v_\star$ for all $v \in S$.
\end{lem}

We deduce directly the following fixed point theorem for the maps induced by $f_\bullet$ on skeleta.

\begin{cor}\label{cor:fixedpoint_skeleta}
Let $(X,x_0)$ be a normal surface singularity, and $f \colon (X,x_0) \to (X,x_0)$ be a non-finite germ.
For every good resolution $\pi \colon X_\pi \to (X,x_0)$, denote by $F_\pi = r_\pi \circ f_\bullet \colon \mc{S}_\pi \to \mc{S}_\pi$ be the induced map on the skeleton $\mc{S}_\pi$ associated to $\pi$.
Then $F_\pi$ as a unique fixed point $\nu_\pi$, and for all $\nu \in \mc{S}_\pi$ we have $F_\pi^n \nu \to \nu_\pi$.
\end{cor}
\begin{proof}
Let $\nu$ and $\mu$ be two different valuations in $\mc{S}_\pi$.
Then we have
$$
\rho(F_\pi(\nu),F_\pi(\mu)) = \rho(r_\pi f_\bullet \nu, r_\pi f_\bullet \mu) \leq \rho(f_\bullet \nu, f_\bullet \mu) < \rho(\nu,\mu),
$$
where the first inequality follows from \reflem{additiverho}, while the second inequality is given by \refthm{strong_contraction}. We conclude by applying \reflem{fixedpoint} to $F_\pi$.
\end{proof}

We can now proceed to construct eigenvaluation for non-finite germs.

\begin{thm}\label{thm:unique_eigenval}
Let $(X,x_0)$ be a normal surface singularity, and let $f \colon (X,x_0) \to (X,x_0)$ be a non-finite germ. Then there exists a \emph{unique} eigenvaluation for $f$.
\end{thm}
\begin{proof}
For any good resolution $\pi\colon X_\pi \to (X,x_0)$, we consider the map $F_\pi = r_\pi \circ f_\bullet \colon \mc{S}_\pi \to \mc{S}_\pi$ induced on the skeleton associated to $\pi$.
By \refcor{fixedpoint_skeleta}, any $F_\pi$ admits a unique fixed point $\nu_\pi \in \mc{S}_\pi$. The unicity guarantees that $\nu_{\pi}= r_\pi \nu_{\pi'}$ whenever $\pi'$ dominates $\pi$.

Assume there is a $\pi$ so that $f_\bullet \nu_\pi = \nu_\pi$. In this case $\nu_\star = \nu_\pi$ is a quasimonomial eigenvaluation.
Assume by contradiction that $\mu_\star$ is another eigenvaluation for $f$.
If $\mu_\star$ is quasimonomial, then there exists $\pi'$ dominating $\pi$ so that $\nu_\star$ and $\mu_\star$ belong to $\mc{S}_{\pi'}$.
In this case $F_{\pi'}$ would have two fixed points $\nu_\star$ and $\mu_\star$, a contradiction.
Hence there are no other quasimonomial eigenvaluations, and the only other possibility for $\mu_\star$ is to be a locally attracting valuation.
But in this case, there exists a good resolution $\pi'$ dominating $\pi$ so that $\mu = r_{\pi'}\mu_\star \neq \nu_\star$ is fixed by $F_{\pi'}$, a contradiction.

Assume now that $f_\bullet \nu_\pi \neq \nu_\pi$ for any good resolution $\pi$.
In particular $\nu_\pi$ is divisorial for all $\pi$, since for irrational valuations we have $r_\pi^{-1}(\nu_\pi) = \{\nu_\pi\}$.

Fix any good resolution $\pi_0$, and set $\nu_0 = \nu_{\pi_0}$.
Recall that by \refprop{big_tree}, the set $\mc{V}_{\pi_0, \nu_0}$ of valuations in $\mc{V}_X$ whose retraction to $\mc{S}_{\pi_0}$ is $\nu_0$ is a tree rooted at $\nu_0$.
By our assumption, $f_\bullet \nu_0$ is a divisorial valuation which does not belong to $\mc{S}_{\pi_0}$.
Since $\nu_0=F_{\pi_0}(\nu_0)$, we have $r_{\pi_0} f_\bullet \nu_0 = \nu_0$, hence $f_\bullet \nu_0 \in \mc{V}_{\pi_0,\nu_0}$.
Let $\pi_1$ be any good resolution so that $f_\bullet \nu_0 \in \mc{S}_{\pi_1}$, set $\nu_1=\nu_{\pi_1}$.
Notice that $\nu_1 \in \mc{S}_{\pi_1} \setminus \mc{S}_{\pi_0} \subset \mc{V}_{\pi_0,\nu_0}$, and in particular $\nu_1 > \nu_0$.

Proceeding by induction, we construct a strictly increasing sequence of divisorial valuations $(\nu_n)_{n \in \nN}$ such that $F_{\pi_n}(\nu_n) = \nu_n$ and $r_{\pi_n}(\nu_m) = \nu_n$ for any $m \geq n$.
By construction, this sequence weakly converges to a non-quasimonomial semivaluation $\nu_\star \in \mc{V}_X$.
We get
$$
f_\bullet \nu_\star 
= \lim_{n\to +\infty} f_\bullet \nu_n
\geq \lim_{n \to +\infty} \nu_n = \nu_\star,
$$
and $\nu_\star=f_\bullet \nu_\star$.

We now prove that $\nu_\star$ is locally attracting, and hence an eigenvaluation.
We recall that $f_\bullet$ is order-preserving where $c(f, - )$ is locally constant.
By \refcor{attractionrate_locallyconst}, this is the case outside the skeleton $\mc{S}_{c(f)}$ which contains $\mc{S}_X$, and it is generated by all critical curve valuations and a finite number of divisorial valuations. 
Since $\nu_\star$ is either a infinitely-singular valuation or a curve semivaluation non-contracted by $f$, there exists a weakly-open neighborhood $U$ of $\nu_\star$ where $c(f,-)$ is constant.
and where consequently $f_\bullet$ is order-preserving.
We may assume that $U=\{\nu \in \mc{S}_{\pi_0, \nu_0}\ |\ \nu > \nu_N\}$ for some $N$ big enough.
Again by \refprop{big_tree}, $\ol{U}$ is a tree rooted at $\nu_N$, $f_\bullet \nu_N \geq \nu_N$ and $f_\bullet U \subset U$.
We deduce that $f_\bullet|_U$ defines a regular tree map in the sense of \cite[\S 4]{favre-jonsson:eigenval}, and $\nu_\star$ is \begin{change}a strongly attracting end\end{change} by \cite[Theorem 4.5]{favre-jonsson:eigenval}.
\begin{change}
Up to increasing $N$, we may assume that $I=[\nu_N, \nu_\star]$ is $f_\bullet$-invariant and $f_\bullet^n \nu \to \nu_\star$ for all $\nu \in I$. Since $f_\bullet$ is order-preserving on $U$, we deduce that $f_\bullet^n \nu \to \nu_\star$ for all $\nu \in U$, and $\nu_\star$ is locally attracting.
\end{change}
Finally, arguing as in the previous case, we deduce that $\nu_\star$ is the unique eigenvaluation in this case as well.
\end{proof}

\subsection{Weak convergence}

Recall that $\mc{V}_X^\alpha := \{\nu \in \mc{V}_X\ |\ \alpha(\nu) < +\infty\}$ denotes the set of normalized valuations with finite skewness.

\begin{lem}\label{lem:closed_basin}
Let $(X,x_0)$ be a normal surface singularity, and $f\colon (X,x_0) \to (X,x_0)$ a dominant germ.
Let $S$ be any subset of $\mc{V}_X^\alpha$.
Then the set
$$
B_S^\alpha = \{\nu \in \mc{V}_X^\alpha\ |\ f_\bullet^n\nu \text{ strongly converges to } S\}
$$
is closed in the strong topology.
\end{lem}
\begin{proof}
Let $(\nu_k)_k$ be a sequence of valuations in $B_S^\alpha$ strongly converging to some valuation $\nu_\infty \in \mc{V}_X^\alpha$.
For any $\eps > 0$, let $k$ be big enough so that $\rho(\nu_\infty, \nu_k) < \eps/2$.
Since $f_\bullet^n \nu_k$ converges to $S$, there exists $N$ such that for all $n \geq N$, $\rho(f_\bullet^n \nu_k, S) < \eps/2$.
Then for all $n \geq N$, we have
$$
\rho(f_\bullet^n \nu_\infty, S) \leq \rho(f_\bullet^n \nu_\infty, f_\bullet^n \nu_k) + \rho(f_\bullet^n \nu_k, S) \leq \rho(\nu_\infty, \nu_k) + \rho(f_\bullet^n \nu_k, S) < \frac{\eps}{2} + \frac{\eps}{2} = \eps.
$$
\end{proof}

\begin{thm}\label{thm:weak_convergence}
Let $(X,x_0)$ be a normal surface singularity, and $f\colon (X,x_0) \to (X,x_0)$ a non-finite germ. Let $\nu_\star$ be the unique eigenvaluation of $f$ given by \refthm{unique_eigenval}.
Assume that $\nu_\star$ is not quasimonomial.
Then for any normalized valuation $\nu \in \mc{V}_X^\alpha$ of finite skewness, we have $f_\bullet^n \nu \to \nu_\star$ in the weak topology.
\end{thm}
\begin{proof}
We set $B^\alpha = \{\nu \in \mc{V}_X^\alpha\ |\ f_\bullet^n\nu \text{ weakly converges to } \nu_\star\}$
the weak basin of attraction to $\nu_\star$.
The strategy of the proof is to prove that $B^\alpha$ is both open and closed in the strong topology. Since $\mc{V}_X^\alpha$ is connected, this gives the statement.

Since $\nu_\star$ is not quasimonomial, it is a weakly attracting end,
and there exists a weakly open neighborhood $U$ of $\nu_\star$ so that $f_\bullet^n \nu \to \nu_\star$ weakly for all $\nu\in U$.
Note that we may pick $U$ of the form $\{\nu \in \mc{V}_X\ |\ \nu > \mu\}$ for a suitable divisorial valuation $\mu$.

The basin of attraction $B^\alpha$ is clearly a non-empty open set (in the weak and strong topologies), since it coincides with
$$
B^\alpha = \mc{V}_X^{\alpha} \cap \bigcup_{n \in \nN} f_\bullet^{-n}(U).
$$
We now show that $B^\alpha$ is also closed in the strong topology.

Let $(\nu_k)_k$ be a sequence of valuations belonging to $B^\alpha$ and converging to some valuation $\nu_\infty \in \mc{V}_X^\alpha$.

Set $\eps = \rho(\mu, f_\bullet \mu)>0$, so that $U$ contains a $\eps$-neighborhood of $f_\bullet U \cap \mc{V}_X^\alpha$.
Pick $k$ big enough so that $\rho(\nu_k, \nu_\infty)< \eps$.
Since $f_\bullet^n \nu_k \to \nu_\star$ there exists $N \gg 1$ so that $f_\bullet^n \nu_k \in f_\bullet (U)$ for all $n \geq N$.
By \refthm{non-expansion}, for all $n \in \nN$ we have 
$$
\rho(f_\bullet^n \nu_k, f_\bullet^n \nu_\infty) \leq \rho (\nu_k, \nu_\infty) < \eps.
$$
In particular, $f_\bullet^n \nu_\infty$ belong to $U$ for all $n \geq N$, and $\nu_\infty \in B^\alpha$.
\end{proof}

In the next section, we will deal with the case when the eigenvaluation $\nu_\star$ is quasimonomial.
In this case, and more in general when $\nu_\star$ has finite skewness, we will show (see \refthm{strong_convergence}) that the orbit of all valuations with finite skewness strongly converge to $\nu_\star$.

\begin{rmk}
For all known applications, the weak convergence of orbits of valuations of finite skewness to $\nu_\star$ would be sufficient.

With this in mind, one could try to prove an analogous statement to the one of  \refthm{weak_convergence} when the eigenvaluation $\nu_\star$ is quasimonomial.

Since the basin of attraction to $\nu_\star$ is closed by \reflem{closed_basin}, we only need to show that it is also open in $\mc{V}_X^\alpha$.
In some cases it is easy to show that, given a quasimonomial valuation $\nu \in \mc{V}_X$, the sequence $\nu_n=f_\bullet^n \nu$ weakly converges to $\nu_\star$.
\begin{itemize}
\item The first trivial case is when there exists $N$ so that $\nu_N=\nu_\star$, since in this case $\nu_n=\nu_\star$ for all $n \geq N$. In the next cases, we will assume that $\nu_n \neq \nu_\star$ for all $n$.
\item If there exists a good resolution $\pi \colon X_\pi \to (X,x_0)$ so that $\nu_n \in \mc{S}_\pi$ for all $n$, then $\nu_n = f_\bullet^n(\nu)=F_\pi^n(\nu)$ for all $n \in \nN$, where $F_\pi\colon \mc{S}_\pi \to \mc{S}_\pi$ is the map induced by $f$ on the skeleton $\mc{S}_\pi$. By 
\refcor{fixedpoint_skeleta}, $\nu_n \to \nu_\star$ strongly, and hence weakly.
\item \begin{change} Suppose that for any connected component $U$ of $\mc{V}_X \setminus \{\nu_\star\}$, the set $I_U=\{n \in \nN\ |\ \nu_n \in U\}$ is finite.
Then again $\nu_n \to \nu_\star$ weakly.
Recall that a weak open set is the union of connected components of the complement of a finite number of points in $\mc{V}_X$.
In particular, any weak open connected neiborhood $V$ of $\nu_\star$ contains all but finitely many connected components $U_1, \ldots, U_r$ of $\mc{V}_X \setminus \{\nu_\star\}$.
Fixed any such $V$, take $N=N(V)$ big enough so that $I_{U_j} \subseteq \{0, \ldots, N\}$ for all $j=1, \ldots, r$. Then $\nu_n$ belongs to $V$ for all $n > N$, and $\nu_n \to \nu_\star$ weakly.\end{change}
\end{itemize}

\begin{change}
The situations described above are quite special, and do not cover all possible behaviors of orbits of a valuation $\nu$. 
To deal with the general situation, even for weak convergence, we need to measure the speed of contraction towards the eigenvaluation through log discrepancies, using the Jacobian Formula \eqref{eqn:jacobian_formula_bullet}.\end{change}
\end{rmk}


\subsection{Semi-superattracting germs}

To prove the strong convergence of $f_\bullet$ to eigenvaluations, for non-finite germs as well as for finite germs, we need to establish the \emph{superattracting} behavior of $f$, i.e., that the orbits of points converge to $x_0$ faster than exponentially.
In fact, when $(X,x_0)\cong(\nC^2,0)$ is a smooth point, the study of maps $f\colon(\nC^2,0)\to (\nC^2,0)$ differs according to the properties of the differential $df_0$ at $0$.
If $f$ is non-invertible, there are essentially two cases: either $df_0$ is nilpotent (\emph{superattracting} case), or it has a non-zero eigenvalue (\emph{semi-superattracting case}).

Although this definition does not carry over the singular setting, its algebraic interpretation does.
\begin{defi}
Let $(X,x_0)$ be a normal surface singularity, and $f\colon(X,x_0) \to (X,x_0)$ be a holomorphic germ.
We say that $f$ is \emph{superattracting} if $(f^n)^*\mf{m} \subseteq \mf{m}^2$ for $n \in \nN^*$ big enough.
\end{defi}
Superattractivity can be rephrased in terms of the behavior of the maps $\nu \mapsto c(f^n,\nu)$ on $\mc{V}_X$.

%

\begin{prop}
A dominant selfmap $f \colon (X,x_0) \to (X,x_0)$ on a normal surface singularity is superattracting if and only if there exists $C > 1$ and $n$ big enough so that $c(f^n,\nu) \geq C$ for all normalized semivaluations $\nu \in \mc{V}_X$.
\end{prop}
\begin{proof}
The direct implication is straigthforward, since if $(f^n)^*\mf{m} \subseteq \mf{m}^2$, then for any normalized semivaluation $\nu \in \mc{V}_X$ we have
$$
c(f^n,\nu)=\nu((f^n)^*\mf{m}) \geq \nu(\mf{m}^2) = 2\nu(\mf{m})=2.
$$

Conversely, assume there exists $C > 1$ so that $c(f, \nu) \geq C$ for all $\nu \in \mc{V}_X$ (achieved from the hypothesis up to replacing $f$ by an iterate).
Notice that
$$
c(f^n,\nu)=\prod_{j=0}^{n-1} c(f,f_\bullet^j \nu) \geq C^n.
$$
Since $C>1$, for any $k \geq 2$ there exists $n=n(k)$ so that $c(f^n,\nu) \geq k$ for all $\nu \in \mc{V}_X$.
This implies that $(f^n)^*\mf{m}\subseteq \ol{(f^n)^*\mf{m}} \subseteq \ol{\mf{m}^k}$, where, given an ideal $\mf{a}$, we denote by $\ol{\mf{a}}$ is integral closure (see, e.g., \cite[Section 6.8]{huneke-swanson:integralclosure}).
By the Brian\c{c}on-Skoda theorem (see \cite[Theorem 4.13]{huneke:uniformbounds}), we have that $\ol{\mf{m}^k}\subseteq \mf{m}^2$ for $k$ big enough, and we are done.
\end{proof}

In the case of non-finite germs, being superattracting is the same as having $c(f, \nu_\star)> 1$, where $\nu_\star$ is the unique eigenvaluation of $f$ given by \refthm{unique_eigenval}.

In the smooth setting, the valuative dynamics of semi-superattracting germs $f\colon (\nC^2,0) \to (\nC^2,0)$ is similar to the one of superattracting germs, and in fact simpler.
Since by definition the eigenvalues of $df_0$ are $0$ and $\lambda \neq 0$, we have a ``superstable'' manifold $D$ (associated to the eigenvalue $0$) and a ``relatively unstable'' manifold $C$ (associated to the eigenvalue $\lambda \neq 0$). In this case $\nu_C$ is the unique eigenvaluation of $f$, and all valuations (but at most $\nu_D$) weakly-converge to $\nu_C$ (see \cite{ruggiero:rigidification}).  

In the singular setting, a similar phenomenon happens.

\begin{prop}\label{prop:semisuper_nonfinite}
Let $(X,x_0)$ be a normal surface singularity, and $f\colon (X,x_0) \to (X,x_0)$ a non-finite germ. Let $\nu_\star \in \mc{V}_X$ be the unique eigenvaluation of $f$ given by \refthm{unique_eigenval}.
If $c(f,\nu_\star)=1$, then $\alpha(\nu_\star)=+\infty$.
\end{prop}
\begin{proof}
Since $c(f,\nu_\star)=1$ and $f_\bullet \nu_\star=\nu_\star$, \eqref{eqn:pushforward_bdivisors} in this case gives
$$
f_*Z(\nu_\star)=Z(\nu_\star) - \sum_{\nu_C \in \VC{f}} a_C Z(f_\bullet \nu_C)
$$
for suitable $a_C > 0$.
By intersecting with $-Z(\nu_\star)$, we infer that
\begin{equation}\label{eqn:semisuper_push}
-Z(\nu_\star) \cdot f_* Z(\nu_\star)=\alpha(\nu_\star) - a
\end{equation}
for a suitable $a > 0$.
Analogously, from \eqref{eqn:pullback_bdivisors} we get
$$
f^*Z(\nu_\star) = m(f,\nu_\star) Z(\nu_\star) + \sum_{\nu_\star \neq \mu \in f_\bullet^{-1}\{\nu_\star\}} b_\mu Z(\mu) + \sum_{\nu_C \in \VC{f}} b_C Z(\nu_C)
$$
for suitable $b_\mu, b_C > 0$, and by intersecting with $-Z(\nu_\star)$ we get
\begin{equation}\label{eqn:semisuper_pull}
-Z(\nu_\star) \cdot f^*Z(\nu_\star) = m(f, \nu_\star) \alpha(\nu_\star) + b
\end{equation}
for a suitable $b > 0$ (notice that $Z(\nu_C) \cdot Z(\nu_\star) < +\infty$, since $\nu_\star$ cannot be a contracted curve semivaluation).
By the projection formula, \eqref{eqn:semisuper_push} and \eqref{eqn:semisuper_pull} coincide, and we have
$$
\alpha(\nu_\star) = m(f,\nu_\star) \alpha(\nu_\star) + a + b,
$$
which implies $\alpha(\nu_\star)=+\infty$.
\end{proof}

We will deal later with the situation of semi-superattracting finite germs (\refthm{superattracting_finite}).

\subsection{Strong convergence}

We can now state and prove the strong convergence of valuations towards the eigenvaluation for non-finite germs.

\begin{thm}\label{thm:strong_convergence}
Let $(X,x_0)$ be a normal surface singularity, and $f\colon (X,x_0) \to (X,x_0)$ a non-finite germ. Let $\nu_\star$ be the unique eigenvaluation of $f$ given by \refthm{unique_eigenval}, and assume that $\nu_\star \in \mc{V}_X^\alpha$ has finite skewness.
Then for any normalized valuation $\nu \in \mc{V}_X^\alpha$, we have $f_\bullet^n \nu \to \nu_\star$ in the \emph{strong} topology.
\end{thm}
\begin{proof}
Following the proof of \refthm{weak_convergence}, we only need to prove that there exists a strongly open neighborhood of $\nu_\star$ in $\mc{V}_X^\alpha$ where the orbits strongly converge to $\nu_\star$.
Notice that by \refprop{semisuper_nonfinite}, the attraction rate $c=c(f,\nu_\star)$ is strictly bigger than $1$.
We split the proof according to the type of eigenvaluation $\nu_\star$.
\begin{enumerate}[label=\textit{Case} \arabic*.,leftmargin=0pt, itemindent=40pt]
\item The eigenvaluation $\nu_\star$ is infinitely singular and has finite skewness.
Assume first that $\nu_\star$ has finite thinness as well.
As in the proof of \refthm{unique_eigenval}, we may find a connected weakly open set $U \subset \mc{V}_X$ where $c(f,\nu)=c$ and $\nu(R_f)=\eps$ are constant, and $f_\bullet$ is order preserving.
We deduce by the Jacobian Formula \eqref{eqn:jacobian_formula_bullet} that
$$
A(f_\bullet \nu)-A(\nu_\star) = \frac{A(\nu)+\eps}{c}-\frac{A(\nu_\star)+\eps}{c} = \frac{A(\nu) - A(\nu_\star)}{c}.
$$
By iterating this estimate, we get that the thinness distance between $f_\bullet^n\nu$ and $\nu_\star$ decreases to $0$ when $n$ goes to $+\infty$, and we have strong convergence.

Assume now that $\nu_\star$ has infinite thinness (but finite skewness). We can take the open neighborhood of $\nu_\star$ as above of the form $U=\{\nu \in \mc{V}\ |\ \nu > \nu_0\}$ for a suitable $\nu_0$. Set $I=(\nu_0,\nu_\star]$, and denote by $r_I\colon U \to I$ the natural retraction.
Notice that $I$ is $f_\bullet$-invariant, and for any $\mu \in I$ the orbit $f_\bullet^n \mu \to \nu_\star$ in the strong topology.
We conclude by remarking
$$
A(f_\bullet^n \nu)-A(f_\bullet^n r_I \nu) = \frac{A(\nu) - A(r_I \nu)}{c^n}\to 0.
$$
\item The eigenvaluation $\nu_\star$ is irrational.
Let $I=[\nu_0,\nu_1]$ be an interval with divisorial endpoints, containing $\nu_\star$ in its interior.
Up to shrinking $I$, we may assume that $f_\bullet$ is injective on $I$, and $\nu \mapsto c(f,\nu)$ and $\nu \mapsto \nu(R_f)$ are locally constant on $U(I) \setminus I$ (notice that such maps need not to be constant on $I$ itself).
Since $f_\bullet I$ is an interval containing $\nu_\star$, we have that $J=I \cap f_\bullet I$ is an interval containing $\nu_\star$ in its interior, satisfying $f_\bullet(J) \subseteq I$.
In particular, $f_\bullet$ coincides with the map $F_\pi\colon \mc{S}_\pi \to \mc{S}_\pi$ for a suitable good resolution $\pi \colon X_\pi \to (X,x_0)$, as defined in the proof of \refthm{unique_eigenval}. By \refcor{fixedpoint_skeleta}, $f_\bullet^n \mu \to \nu_\star$ for all $\mu \in J$.

We now want to show strong convergence of $f_\bullet^n \nu \to \nu_\star$ for all $\nu \in U$, for a suitable strong open neighborhood of $\nu_\star$.
Consider the set $K=f_\bullet^{-1}(J)$. It is made by a finite number of connected components, and denote by $K_0$ the one containing $\nu_\star$.
Then $K_0$ is a finite union of segments with divisorial endpoints.

Up to shrinking $J$, we may assume that $K_0 \cap U(J) \subset J$. Set $U=U(J)$.
By construction, $r_\pi f_\bullet^n \nu = f_\bullet^n r_\pi \nu$ for all $\nu \in U$.
Since $f_\bullet$ is order preserving on $U \setminus J$, $f_\bullet$ is increasing on the interval $[r_\pi \nu,\nu]$.
Moreover, $c(f,\nu)=c(f,r_\pi \nu)$ and $\nu(R_f)=r_\pi\nu(R_f)$ for all $\nu \in U$.
Pick any $\nu \in U$ of finite thinness, and set $\mu=r_\pi \nu$.
By the Jacobian formula \eqref{eqn:jacobian_formula_bullet}, we infer that
$$
A(f_\bullet^n \nu,f_\bullet^n \mu) = \frac{A(\nu)-A(\mu)}{c(f^n, \mu)} \to 0 
$$
when $n \to \infty$.
While $A$ could be constant on $J$ (if $J$ belongs to the essential skeleton $\eskel{X}$ of $\mc{V}_X$), $A$ is not locally constant on $U \setminus J$, and we deduce the strong convergence of $f_\bullet^n \nu \to \nu_\star$ for all $\nu \in U$ of finite thinness. Since $U \cap \{A(\nu) < +\infty\}$ is a strong neighborhood of $\nu_\star$, we are done.

\item The eigenvaluation $\nu_\star = \nu_{E_\star}$ is divisorial.
To construct a strongly open neighborhood of $\nu_\star$ that belongs to the strong basin of attraction to $\nu_\star$, we follow the same strategy used in the smooth case, see \cite[Lemma 3.9]{gignac-ruggiero:attractionrates}.
Let $\pi \colon X_{\pi} \to (X,x_0)$ be any good resolution that realizes $E_\star$. Recall that for any $p$, the open set $U_{\pi}(p)$ represents a tangent vector $\vect{v_p} \in T_{\nu_\star} \mc{V}_X$.
For any such tangent vector, we want to find a positive number $\rho(\vect{v_p})$ so that:
\begin{itemize}
\item for any $\nu \in U_{\pi}(p)$ so that $\rho(\nu,\nu_\star) < \rho(\vect{v_p})$, then $f_\bullet^n \nu \to \nu_\star$ strongly;
\item $\inf_{p \in E_\star} \rho(\vect{v_p}) > 0$.
\end{itemize}
Let $\mc{S}_f$ denote the critical skeleton, so that $\nu \mapsto c(f,\nu)$ and $\nu \mapsto \nu(R_f)$ are locally constant on $\mc{V}_X \setminus \mc{S}_f$.
Let $\mc{T}$ be the (finite) set of tangent directions $\vect{v_p}$ so that $U_{\pi_0}(p) \cap \mc{S}_f \neq \emptyset$.
Consider the sets $\mc{T}_1$ of tangent vectors $\vect{v}$ so that $df_\bullet^n \vect{v} \not \in \mc{T}$ for all $n \in \nN$, and $\mc{T}_2$ of tangent vectors in $\mc{T}$ which are $df_\bullet$-periodic.
Any tangent vector $\vect{v}$ is eventually mapped by $df_\bullet$ to either $\mc{T}_1$ or $\mc{T}_2$.

First, assume that $\vect{v_p} \in \mc{T}_1$, and pick any quasimonomial valuation $\nu \in U_\pi(p)$. Notice that $c(f, f_\bullet^n \nu)= c(f, \nu_\star)=:c$ and $f_\bullet^n\nu(R_f)=\nu_\star(R_f)=:\eps$ for all $n \in \nN$.
As in \textit{Case $1$} we get
$$
A(f_\bullet^n \nu) - A(\nu_\star) = c^{-n}(A(\nu) - A(\nu_\star)) \to 0
$$
as $n \to +\infty$.
Since the thinness function $A$ is not constant on any segment in $\mc{V} \setminus \mc{S}_\pi$, this implies the strong convergence.
We set in this case $\rho(\vect{v_p})=+\infty$.

Assume now that $\vect{v_p} \in \mc{T}_2$, and pick again any quasimonomial valuation $\nu'' \in U_\pi(p)$. Consider the segment $[\nu_\star, \nu'']$.
Since $\vect{v_p}$ is periodic, there exists $m \in \nN^*$ and another segment $J=[\nu_\star, \nu']$ with $\nu' \in (\nu_\star, \nu'']$, so that $f_\bullet^m (J) \subseteq I$.
Let $\pi'\colon X_{\pi'} \to (X,x_0)$ be a good resolution dominating $\pi$ and so that $J \subset \mc{S}_{\pi'}$.
By construction, $f_\bullet^m$ coincides with $F_{\pi'}^m$ on $J$, where $F_{\pi'}^m= r_{\pi'}\circ f_\bullet^m \colon \mc{S}_{\pi'} \to \mc{S}_{\pi'}$.
By \refcor{fixedpoint_skeleta}, $f_\bullet^{nm}\mu \to \nu_\star$ strongly for any $\mu \in J$.
Repeating the same argument for $df_\bullet^k \vect{v_p}$ and $f_\bullet^k \nu''$ for all $k=1, \ldots, m-1$, we deduce that $f_\bullet^n \mu \to \nu_\star$.
Let $U'$ denote the connected component of $\mc{V}_X \setminus \{\nu_\star, \nu'\}$ intersecting $J$.
Up to shrinking $J$, we may assume that $(U' \setminus J) \cap \mc{S}_f = \emptyset$.
By the same argument used in the irrational case, we can deduce that $f_\bullet^n \nu \to \nu_\star$ strongly for any $\nu \in U'$.
In this case we set $\rho(\vect{v_p})= \rho(\nu', \nu_\star) > 0$.

Finally, assume that $\vect{v_p} \not \in \mc{T}_1 \cup \mc{T}_2$.
We define $\rho(\vect{v_p})$ recursively as follows.
Suppose that $\vect{v_p}$ is such that $\rho(df_\bullet \vect{v_p})$ has been defined.
If $\vect{v_p}$ is one of the (finite number of) tangent vectors in $\mc{T} \setminus \mc{T}_2$, then we just set $\rho(\vect{v_p})$ to be a value $\rho>0$ sufficiently small to have that for any $\nu \in U_\pi(p)$ so that $\rho(\nu, \nu_\star) < \rho$, then $\rho(f_\bullet \nu, \nu_\star) < \rho(df_\bullet \vect{v_p})$.
If $\vect{v_p} \not \in \mc{T}$, then $f_\bullet$ is order preserving on $U_\pi(p)$. It follows that $f_\bullet(U_\pi(p)) \subset \mc{V}_X \setminus \{\nu_\star\}$, and by continuity, it belongs to the connected component associated to $df_\bullet \vect{v_p}$.
We may define $\rho(\vect{v_p})=\rho(df_\bullet \vect{v_p})$.
In fact, if $\nu \in U_\pi(p)$ and $\rho(\nu,\nu_\star) \leq \rho$, by \refthm{non-expansion} we have $\rho(f_\bullet\nu, \nu_\star)$.

We have defined $\rho(\vect{v})$ for any $\vect{v} \in T_{\nu_\star} \mc{V}_X$. Clearly by construction, we have that $\rho_\star := \displaystyle \inf_{\vect{v}} \rho(\vect{v}) > 0$.
In particular, the basin of attraction to $\nu_\star$ with respect to the strong topology contains the strong open neighborhood $\{\nu \in \mc{V}_X\ |\ \rho(\nu,\nu_\star) < \rho_\star\}$.
This completes the proof.
\end{enumerate}
\end{proof}

\begin{rmk}\label{rmk:weak_convergence_charp}
Notice that the weak convergence towards the eigenvaluation established by \refthm{weak_convergence} holds also over fields of positive characteristic.
This will be enough for proving the existence of dynamically stable models for non-finite germs having a non-quasimonomial eigenvaluation even in positive characteristic.
\end{rmk}

\section{Dynamics of non-invertible finite germs}\label{sec:dynamics_finite}

This section is devoted to the proof of \refthm{classification} for finite germs.
By \refthm{wahl}, $(X,x_0)$ is log canonical (lc).
We will analyze separately the cases when $(X,x_0)$ is log terminal (lt, which are the quotient singularities), when $(X,x_0)$ is not.
In the latter case, we will show that $f$ is necessarily superattracting (\refthm{superattracting_finite}).

\subsection{Quotient singularities}\label{ssec:dynamics_quotientsing}

Suppose first that $(X,x_0)$ is log terminal, i.e., a quotient singularity.
It is isomorphic to the quotient $(\C^2,0)/G$ of $\nC^2$ by a finite subgroup $G$ of $GL_2(\nC)$, acting freely on $\C^2\setminus\{0\}$.
We denote by $\pr\colon \C^2\to X$ the natural projection. Notice that when restricted to $\nC^2 \setminus \{0\}$, the map $\pr$ gives the universal covering of $X \setminus \{x_0\}$.

Let now $f \colon (X,x_0) \to (X,x_0)$ be a non-invertible germ.
While in the non-finite case, $f$ cannot in general be lifted to a holomorphic germ $g\colon (\nC^2,0) \to (\nC^2,0)$ (see \S\ref{ssec:example_quotient1}), it is always the case when $f$ is finite.
\begin{prop}
Let $f \colon (X,x_0) \to (X,x_0)$ be a \emph{finite} germ on a quotient singularity $(X,x_0)=(\nC^2,0)/G$. Let $\pr \colon (\nC^2,0) \to (X,x_0)$ be the natural projection.
Then there exists a finite germ $g \colon (\nC^2,0) \to (\nC^2,0)$ so that $f \circ \pr = \pr \circ g$.
\end{prop}
\begin{proof}
Since $f$ is finite, $f^{-1}(x_0)$ is discrete, and there exists a neighborhood $U \subset X$ of $x_0$ such that $f\colon U \to f(U)$ satisfies $f^{-1}(x_0)=\{x_0\}$.
The map $f \circ \pr$, which sends $pr^{-1}(U)\setminus \{0\}$ to $f(U) \setminus \{x_0\}$, lifts to a holomorphic map $g \colon \pr^{-1}(U)\setminus \{0\} \to \pr^{-1}(f(U)) \setminus \{0\}$.
By Hartog's theorem, $g$ extends through $0$ to a map defined on $pr^{-1}(U)$, which clearly satisfies the wanted properties.
\end{proof}

The commutative diagram $f \circ \pr = \pr \circ g$ induces a commutative diagram on the action on valuative spaces, $f_\bullet \circ \pr_\bullet = \pr_\bullet \circ g_\bullet$.
In particular $\pr_\bullet$ acts as a semi-conjugation between $f_\bullet$ and $g_\bullet$.

By the result for non-invertible maps on smooth surfaces \cite[Theorem 3.1]{gignac-ruggiero:attractionrates}, we know that $g_\bullet$ satisfies case $2$  of \refthm{classification}.
Assume we are in the first case and let $\mu_\star$ be the unique eigenvaluation for $g_\bullet$.
Then it is easy to check that $\nu_\star=\pr_\bullet \mu_\star$ is the unique eigenvaluation for $f_\bullet$, and the orbit of any quasimonomial valuation converges to $\nu_\star$, strongly as far as $\mu_\star$ (and hence $\nu_\star$) has finite skewness.
Assume we are in the second case, and let $J$ be the segment of eigenvaluations for $g_\bullet^2$.
Then again $I=\pr_\bullet J$ is a segment of eigenvaluations for $f_\bullet^2$, and $f_\bullet^{2n} \nu \to I$ strongly for any quasimonomial valuation $\nu$.
This concludes the proof in this case.

\subsection{Non-lt singularities}

Recall that by \refthm{wahl}, for any non-invertible finite germ $f\colon(X,x_0) \to (X,x_0)$ on a lc non lt singularity, the Jacobian divisor $R_f$ is trivial.

\begin{lem}\label{lem:thinness_isometry}
Let $(X,x_0)$ be a lc non lt singularity, and $f \colon (X,x_0) \to (X,x_0)$ be a non-invertible finite germ.
Then the set $S=\{\nu \in \mc{V}_X\ |\ A(\nu)=0\}$ is totally invariant, and $f_\bullet|_S : S \to S$ is an isometry with respect to the angular distance $\rho$.
\end{lem}
\begin{proof}
The Jacobian formula \eqref{eqn:jacobian_formula_bullet} states in this case that
$
A(f_\bullet \nu) = \frac{A(\nu)}{c(f,\nu)}
$.
In particular, the set $S$ is totally invariant by $f_\bullet$.
By \S\ref{ssec:lc_and_lt}, the set $S$ is either a point, a segment or a circle.
By \refthm{non-expansion}, the map $f_\bullet$ is non-expanding for the angular distance.
By \refprop{propertiesfbullet}, $f_\bullet|_S$ is surjective.
If follows that $f_\bullet|_S$ is an isometry.
\end{proof}

As for non-finite germs, to prove strong convergence to $S$ we need to show that $f$ is necessarily superattracting.
Notice that quotient singularities admit finite non-invertible germs which are not superattracting. Consider for example the cyclic group $G$ generated by $\Phi(x,y)=(\zeta x, \zeta^q y)$, with $\zeta$ a primitive $p$-root of unity, and $q$ coprime with $p$.
Then the map $g(x,y)=(\lambda x, y^d)$ commutes with $\Phi$ as far as $d \equiv 1$ modulo $p$, and defines a semi-superattracting germ on the cyclic quotient singularity $(\nC^2,0)/G$.
This cannot happen for other lc singularities.

\begin{thm}\label{thm:superattracting_finite}
Let $f \colon (X,x_0) \to (X,x_0)$ be a non-invertible finite germ on a normal surface singularity. Suppose that $(X,x_0)$ is lc not lt. Then $f$ is superattracting.
\end{thm}
\begin{proof}
Set $S=\{\nu \in \mc{V}_X\ |\ A(\nu)=0\}$.
By \reflem{thinness_isometry}, $f_\bullet|_S$ is an isometry with respect to $\rho$.
By \S\ref{ssec:lc_and_lt}, $S$ is either a point, a segment, or a circle.
Since $f_\bullet|_S$ is an isometry, it follows that it is either of finite order, or $S$ is a circle and $f_\bullet$ acts as an irrational rotation.

Let $\nu \in S$ be any valuation. Since $S$ are totally invariant and $f|_S$ is a bijection, we get $f_\bullet \nu \in S$ and $f_\bullet^{-1}(f_\bullet \nu)=\{\nu\}$. By \eqref{eqn:pushforward_bdivisors} and \eqref{eqn:pullback_bdivisors} we have
\begin{equation}\label{eqn:c_and_e}
c(f,\nu) Z(f_\bullet \nu) \cdot Z(f_\bullet\nu)
= Z(f_\bullet \nu) \cdot f_*Z(\nu)
=f^*Z(f_\bullet \nu) \cdot Z(\nu)
=e(f)Z(\nu) \cdot Z(\nu),
\end{equation}
where $e(f)$ denotes the topological degree of $f$.

Assume that $f_\bullet|_S$ has finite order.
Up to taking a suitable iterate, we may assume that $f_\bullet$ acts as the identity on $S$.
Suppose by contradiction that there exists a point $\nu \in S$ satisfying $c(f,\nu)=1$.
Then \eqref{eqn:c_and_e} applied to $\nu$ says that $e(f)=c(f,\nu)=1$, which is a contradiction, since finite maps with topological degree $1$ are invertible.

Assume now that $f_\bullet|_S$ is an irrational rotation on the circle $S$.
Suppose by contradiction that there exists a $\nu \in S$ so that $c(f,f_\bullet^n \nu)=1$ for all $n$.
We may find a subsequence $(f_\bullet^{n_k} \nu)_k$ of the orbit of $\nu$ which converges to $\nu$.
Then \eqref{eqn:c_and_e} tells us that
$$
e(f)^{n_k}Z(\nu) \cdot Z(\nu)=
Z(f_\bullet^{n_k} \nu) \cdot Z(f_\bullet^{n_k}\nu)
\to Z(\nu) \cdot Z(\nu),
$$
and again $e(f)=1$, a contradiction.

By continuity of $\nu \mapsto c(f^n,\nu)$ (with respect to the weak topology) and compacity of $S$, we infer that there exists $C>1$ and $N \gg 1$ so that $c(f^N,\nu) \geq C$ for all $\nu \in S_\eps$, where $S_\eps =\{\nu \in \mc{V}_X\ |\ A(\nu) < \eps\}$.
In particular, $c(f^n, \nu) \to +\infty$ for all $\nu \in S_\eps$.
By the Jacobian formula \eqref{eqn:jacobian_formula_bullet}, for any $s \in S_\eps$ we have that
$$
A(f_\bullet^n \nu) = \frac{A(\nu)}{c(f,\nu)} \to 0,
$$
and $S_\eps$ is contained in the basin of attraction $B_S^\alpha$ to $S$ (of valuations with finite skewness).
Hence $B_S^{\alpha}$ is open in $\mc{V}_X^\alpha$. But $B_S^\alpha$ is also closed in $\mc{V}_X^\alpha$ by \reflem{closed_basin}. By connectedness, $B_S^\alpha=\mc{V}_X^\alpha$.

In particular $c(f^n,\nu) \to \infty$ for any valuation of finite skewness, and $f$ is superattracting.
\end{proof}

As a consequence of the proof of \refthm{superattracting_finite}, we get strong convergence towards $S$.

\begin{cor}\label{cor:thinness_convergence}
Let $(X,x_0)$ be a lc non lt singularity, and $f \colon (X,x_0) \to (X,x_0)$ be a non-invertible finite germ.
Set $S=\{\nu \in \mc{V}_X\ |\ A(\nu)=0\}$. Then for any valuation $\nu \in \mc{V}_X$ of finite skewness, we have $f_\bullet^n \nu \to S$ with respect to the strong topology.
\end{cor}

We have seen in \S\ref{ssec:lc_and_lt} that $S$ can be either a point, a segment or a circle, according to whether $(X,x_0)$ is (a quotient of) a simple elliptic singularity, a quotient-cusp singularity, or a cusp singularity.

In the first case, \refcor{thinness_convergence} gives directly case $3$ of the statement of \refthm{classification}.
In the second and third cases, we conclude by the following proposition.

\begin{prop}
Let $f\colon(X,x_0) \to (X,x_0)$ be a non-invertible finite germ. Let $S \subset \mc{V}_X$ be a (closed, connected) set of normalized semivaluations, totally invariant for $f_\bullet$.
Denote by $r_S : \mc{V}_X \to S$ the retraction to $S$.
Then for any valuation $\nu \in \mc{V}_X$ such that $f_\bullet^n \nu \to S$, we have $\rho(f_\bullet^n,f_\bullet^n r_S \nu) \to 0$. 
\end{prop}
\begin{proof}
We claim that $r_S f_\bullet^n \nu = f_\bullet^n r_S\nu$ for any $n \in \nN^*$.
This implies the statement.

If $\nu \in S$, by invariance of $S$ we also have $f_\bullet^n \nu \in S$, and there is nothing to prove. We assume that $\nu \not \in S$, and set $\nu_\star:=r_S \nu \neq \nu$.
Suppose by contradiction that there exists $n \in \nN^*$ so that $r_S f_\bullet^n \nu \neq f_\bullet^n \nu_\star$.
Take a minimal $n$ satisfying this condition. Up to replacing $\nu$ by $f_\bullet^{n-1} \nu$, we may assume that $n=1$.

Consider the segment $[\nu_\star, \nu]$.
Notice that by hypothesis, $f_\bullet \nu$ and $f_\bullet \nu_\star$ belong to different connected components of $\mc{V}_X \setminus S$.
Since $f_\bullet$ is continuous, there exists $\mu \in (\nu_\star, \nu)$ so that $f_\bullet \mu \in S$. This contradicts the total invariance of $S$, since $\mu \not \in S$.
\end{proof}

\subsection{Irrational rotations on cusp singularities}

We conclude by showing that irrational rotations occur on any cusp singularity.
The examples given are a generalization of \cite[\S 2.5]{favre:holoselfmapssingratsurf}.

\begin{prop}\label{prop:irrational_rotations}
For any cusp singularity $(X,x_0)$, there exists a finite germ $f\colon (X,x_0) \to (X,x_0)$ so that the action of $f_\bullet$ on the set $S$ of normalized valuations of zero log discrepancy is conjugated to an irrational rotation.
\end{prop}
\begin{proof}
Any cusp singularity $(X,x_0)$ is analytically isomorphic to one constructed as in \S\ref{ssec:arith_cusp}.
It depends on a square-free positive integer $d \geq 2$, a totally positive element $\omega \in \nQ(\sqrt{d})$, generating a rank $2$ free $\nZ$-module $N_\omega=\nZ \oplus \omega \nZ$, and a totally positive unit $\eps=a+b\sqrt{d}$ so that $\eps N_\omega = N_\omega$ (in particular, $\eps^n \in N_\omega$ for all $n \in \nZ$). Notice that $a$ and $b$ are not necessarily integers depending on $d$, but $2a$ and $2b$ are always positive integers.

We consider now the finite endomorphism $f_\alpha$ associated to a totally positive element $\alpha \in N_\omega \cap \eps^{-1} N_\omega$ (see \S\ref{ssec:arith_finitegerm}).
The action of $(f_\alpha)_\bullet$ on the cycle $\mc{S}_X \cong \nR/\nZ$ is conjugated to the translation of
$$
\beta=\frac{\log(\alpha/\alpha')}{2\log \eps}
$$
modulo $\nZ$.

To prove the statement, it suffice to find a $\alpha \in N_\omega \cap \eps^{-1} N_\omega$ which is totally positive and so that $\beta \in \nR\setminus \nQ$.
Notice that $\beta \in \nQ$ if and only if some power of $\alpha/\alpha'$ is an integer multiple of a power of $\eps^2$.
Write $\eps=a+b\sqrt{d}$, and set consider $\alpha = p+b\sqrt{d}$, where $p \in \nN$ is to determine.
Notice that $\alpha = p-a+\eps \in N_\omega$ and $\eps \alpha = \eps^2 + (p-a)\eps \in N_\omega$. Moreover, $\alpha$ is totally positive whenever $p > a$, since in this case it is given by a positive linear combination of totally positive numbers.
By direct computation,
$$
\frac{\alpha}{\alpha'}=\frac{p+b\sqrt{d}}{p-b\sqrt{d}}=\frac{p^2+db^2 + 2bp\sqrt{d}}{p^2-db^2}.
$$
Notice that for $p$ big enough, $2\abs{p^2-db^2} > \abs{2bp}$. It follows that for such $p$ the number $\frac{\alpha}{\alpha'}$ is not integral in $\nQ(\sqrt{d})$.
The same holds for any power of $\alpha/\alpha'$.
Since $\eps$ is integral in $\nQ(\sqrt{d})$, so is any positive multiple of any power of $\eps$.
It follows that $\beta$ is irrational, and we are done.
\end{proof}

\section{Algebraic stability}\label{sec:algebraic_stability}

This section is devoted to proving \refthm{geometricstability}.
First, we introduce a definition to simplify the statement.
\begin{defi}
Let $f\colon (X,x_0) \to (X,x_0)$ be a dominant non-invertible germ.
We say that a birational model $\pi\colon X_\pi \to (X,x_0)$ is a \emph{geometrically stable} model for $f$ if the lift $f_\pi \colon X_\pi \dashrightarrow X_\pi$ has the following property. For every exceptional prime $E$ of $X_\pi$, $f_\pi^n (E)$ is an indeterminacy point of $f_\pi$ for at most finitely many $n$.
\end{defi}

Geometrical stability should be thought as the right concept of algebraic stability in the local setting.
In fact, from geometrical stability one can recover easily algebraic stability.

\begin{lem}\label{lem:pullback_is_functorial}
Let $f\colon(X,x_0) \to (Y,y_0)$ and $g\colon(Y,y_0) \to (Z,z_0)$ be two dominant germs between normal surface singularities. Let $\pi\colon X_\pi \to (X,x_0)$, $\varpi\colon Y_\varpi \to (Y,y_0)$ and $\eta \colon Z_\eta \to (Z,z_0)$ be good resolutions. Set $\wt{f}=\varpi^{-1}\circ f \circ \pi \colon X_\pi \dashrightarrow Y_\varpi$ and $\wt{g}=\eta^{-1}\circ g \circ \varpi \colon Y_\varpi \dashrightarrow Z_\eta$.
If for any exceptional prime $E \in \varGamma_\pi^*$ we have that $\wt{f}(E)$ is not an indeterminacy point of $\wt{g}$, then $(g \circ f)^* = f^* \circ g^*$.
\end{lem}
The proof of \reflem{pullback_is_functorial} is a straightforward adaptation of the arguments used in \cite[Lemma 5.1]{gignac-ruggiero:attractionrates}.
As an immediate consequence, we get: 

\begin{prop}\label{prop:geometrical_and_algebraic_stability}
Let $f\colon (X,x_0) \to (X,x_0)$ be a dominant non-invertible germ. Let $\pi\colon X_\pi \to (X,x_0)$ be a proper birational map, and denote by $f_\pi\colon X_\pi \dashrightarrow X_\pi$ the lift of $f$ to $X_\pi$.
If $\pi$ is a \emph{geometrically stable} model for $f$, then there exists $N \in \nN$ so that for any $n \geq N$, we have
$$
(f^n_\pi)^* = (f^N_\pi)^* (f_\pi^*)^{n-N}\qquad \text{ and } \qquad \Exc_\pi \circ (f^n_\pi)^* = \Exc_\pi \circ (f^N_\pi)^* (\Exc_\pi \circ f_\pi^*)^{n-N}. 
$$
\end{prop}

\refthm{geometricstability} can be reformulated as the existence of geometrically stable models, dominating any given model.

\begin{thm}\label{thm:dimensional_stability}
Let $(X,x_0)$ be an irreducible germ of a normal complex surface at a point $x_0\in X$, and let $f\colon (X,x_0)\to (X,x_0)$ be a dominant non-invertible holomorphic map.
Assume that $f$ is not a finite germ at a cusp singularity inducing an irrational rotation.
Then for any modification $\pi\colon X_{\pi}\to (X,x_0)$, one can find another modification $\pi'\colon X_{\pi'}\to (X,x_0)$ dominating $\pi$ for which $f$ is geometrically stable.
In general $X_{\pi'}$ may have cyclic quotient singularities. It can be taken smooth up to replacing $f$ by an iterate.
\end{thm}

Notice that the singular surfaces $X_{\pi'}$ will only have (cyclic) quotient singularities. Such surfaces can be also described with the formalism of \emph{orbifolds} (also called $V$-manifolds).
Hence \refthm{dimensional_stability} could be also restated as a result on the existence of geometrically stable orbifolds.

Finally, in view of \refprop{geometrical_and_algebraic_stability}, we clearly have that \refthm{geometricstability} implies \refthm{algebraicstability}.

\subsection{Existence of geometrically stable models}

This section is dedicated to proving \refthm{dimensional_stability}.

We split the proof of \refthm{dimensional_stability} according to the dynamics of $f_\bullet$ (see \refthm{classification}):
\begin{enumerate}
\item $f_\bullet$ admits a unique eigenvaluation $\nu_\star$, which is not quasimonomial; every valuation $\nu \in \mc{V}_X^\alpha$ weakly converges to $\nu_\star$.
\item $f_\bullet$ admits a unique eigenvaluation $\nu_\star$, which is irrational; every valuation $\nu \in \mc{V}_X^\alpha$ strongly converges to $\nu_\star$.
\item\label{item:condition3} There exists $S \subseteq \mc{V}_X$ which is either a divisorial point, a segment with divisorial or curve endpoints, or a circle, so that $f_\bullet^k|_S=\on{id}_S$, and every valuation $\nu \in \mc{V}_X^\alpha$ converges strongly to $S$.
\end{enumerate}

\begin{enumerate}[label=\textit{Case} \arabic*.,leftmargin=0pt, itemindent=40pt]
\item $f_\bullet$ admits a unique eigenvaluation $\nu_\star$, which is not quasimonomial.

Let $U$ be a weak open neighborhood of $\nu_\star$ which avoids the critical skeleton $\mc{S}_f$ where either $\nu \mapsto c(f,\nu)$ or $\nu \mapsto \nu(R_f)$ are not locally constant.
As we have seen in the proof of \refthm{strong_convergence}, this implies that $f_\bullet(U)\subrelcpct U$ and $f_\bullet^n \nu \to \nu_\star$ for any $\nu \in U$.
In fact, we may take $U$ of the form $U=U(\nu_0)=\{\nu \in \mc{V}_X\ |\ \nu > \nu_0\}$, where $\nu_0$ is any valuation such that $\nu_0 < \nu_\star$ and sufficiently close to $\nu_\star$.

For any modification $\pi \colon X_{\pi} \to (X,x_0)$, consider another modification $\pi' \colon X_{\pi'} \to (X,x_0)$ dominating $\pi$ and so that there exists an exceptional prime $E$ in $X_{\pi'}$ so that $\nu_E \geq \nu_0$.

We may assume (by taking the exceptional prime $E$ closer to $\nu_\star$) that the center of $\nu_\star$ in $X_{\pi'}$ is a free point $p$ in $E$.
Then $U(\nu_E)=U_{\pi'}(p)$ is a neighborhood of $\nu_\star$ which is $f_\bullet$-invariant.

By \refprop{detecting_holomorphicity}, the lift $f_{\pi'} \colon X_{\pi'} \to X_{\pi'}$ defines a holomorphic fixed point at $p$.
By \refthm{weak_convergence}, for any exceptional prime $D$ of $X_{\pi'}$, there exists $N$ so that $f_\bullet^n(\nu_D) \in U$ for all $n \geq N$.
Hence $f_{\pi'}^n(D) = p$ for all $n \in \nN$, and the model $\pi'$ is geometrically stable.

Notice that in this case $X_{\pi'}$ is smooth.

\item $f_\bullet$ admits a unique eigenvaluation $\nu_\star$, which is irrational.

As seen in the proof of \refthm{strong_convergence} for non-finite germs, or in \S\ref{ssec:dynamics_quotientsing} for finite germs, we may find a closed interval $J$ containing $\nu_\star$ such that $f_\bullet J \subset J$, and $f_\bullet U(J) \subrelcpct U(J)$.
We may assume that $J$ has divisorial endpoints $\nu_E$ and $\nu_F$.
Let $\pi''\colon X_{\pi''} \to (X,x_0)$ be a good resolution dominating $\pi$ so that $\nu_\star \in \mc{S}_{\pi''}$, and $\nu_E$ and $\nu_F$ are realized in $X_{\pi''}$. Notice that this implies $J \subset \mc{S}_{\pi''}$.
Let $\eta_{\tilde{\pi}\pi''}\colon X_{\pi''} \to X_{\tilde{\pi}}$ be the contraction of the (possibly empty) set of exceptional primes $D$ in $X_{\pi''}$ for which $\nu_D$ belongs to the open segment $\mathring{J}=]\nu_E, \nu_F[$, and denote by $\wt{\pi}\colon X_{\wt{\pi}} \to (X,x_0)$ the induced modification.
By our choice of $J$, the modification $\wt{\pi}$ still dominates $\pi$.
The image through $\eta_{\tilde{\pi}\pi''}$ of the contracted divisor is a (possibly singular) point $\wt{p} \in \wt{\pi}^{-1}(x_0)$, and $U(J)=U_{\wt{\pi}}(\wt{p})$. Up to taking a smaller $J$, we may assume that the contracted divisor is a chain of rational curves, so that $\wt{p}$ is either a smooth point or a cyclic quotient singularity.

By \refprop{detecting_holomorphicity}, the lift $f_{\wt{\pi}} \colon X_{\wt{\pi}} \dashrightarrow X_{\wt{\pi}}$ has a holomorphic fixed point at $\wt{p}$.
By \refthm{weak_convergence}, for any exceptional prime $D$, there exists $N$ so that $f_\bullet^n(\nu_D) \in U(J)$ for all $n \geq N$.
Hence $f_{\wt{\pi}}^n(D) = \wt{p}$ for all $n \in \nN$, and the model $\wt{\pi}$ is geometrically stable.

\item The set $S$ of eigenvaluations for the iterates of $f$ contains at least a divisorial valuation.

We split this case according to the geometry of $S$ (which is either a point, a segment, or a circle).
Before proceeding, we need to introduce some definitions and small lemmas.

Suppose we have a good resolution $\pi_0 \colon X_{\pi_0} \to (X,x_0)$ so that the set $\mc{S}_{\pi_0}^*$ of divisorial valuations realized by $\pi_0$ contains at least the orbit of a divisorial valuation $\nu_{E_0} \in S$. Notice that by construction, the whole orbit belong to $S$.
Set $\nu_{E_i}=f_\bullet^i \nu_{E_0}$ for all $i=1, \ldots, k$, where $E_k=E_0$ (we allow also other exceptional primes to coincide, so $k$ needs not to be the minimal period of $\nu_{E_0}$).

We recall (see \S\ref{ssec:tangent_vectors}) that to each point $p \in E_i$ correspond a tangent vector $\vect{v_p}^{i}$ at $\nu_{E_i}$.
Moreover, any such tangent vector is associated to a connected component of $\mc{V}_X \setminus \{\nu_{E_i}\}$, that we denote by $U_i(p)$.
Notice that different tangent vectors may be associated to the same connected component, notably when $S$ is a circle.  

\vspace{3mm}

\begin{defi}
We say that a point $p$ is \emph{parallel to $S$} if $U_i(p) \cap S \neq \emptyset$.
\end{defi}

\vspace{3mm}

\begin{rmk}
Notice that if $p \in E_i \cap E_j$, then the segment $]\nu_{E_i},\nu_{E_j}[_p$ belong to $S \cap U_i(p) \cap U_j(p)$.
We deduce two properties.
Firstly, being parallel depend only on $p$ and not on the exceptional prime $E_i$ that contains it, so the definition is well posed.
Secondly, points that are non-parallel to $S$ belong to a unique exceptional prime among the $\{E_i, i=0, \ldots, k-1\}$.

In fact, if we denote by $S_{\pi_0}^* = \{\nu_E \in S\ |\ E \in \Gamma_{\pi_0}^*\}$ the set of divisorial valuations on $S$ realized by $\pi_0$, and by $D(S_{\pi_0})$ the union of exceptional primes $E \subset X_{\pi_0}$ so that $\nu_E \in S_{\pi_0}^*$, then the points parallel to $S$ are exactly the singular points of $D(S_{\pi_0})$.
In particular, if $S$ consists of a single divisorial valuation $\nu_{E_\star}$, all points of $E_\star$ are not parallel to $S$.
\end{rmk}

Let now $p^0$ be a point non parallel to $S$. It belongs to a unique exceptional prime $E^0 \subset \pi_0^{-1}(x_0)$.
Let $\pi_1\colon X_{\pi_1} \to (X,x_0)$ be a modification dominating $\pi_0$, so that $\eta = \eta_{\pi_0\pi_1} \colon X_{\pi_1} \to X_{\pi_0}$ is an isomorphism over $X_{\pi_0} \setminus \{p^0\}$.
The strict transform of $E^0$ by $\eta$ intersects $\eta^{-1}(p^0)$ in a unique point, which we denote by $p^1$. In fact $p^0 \in X_{\pi_0}$ and $p^1 \in X_{\pi_1}$ correspond to the same tangent vector $\vect{v_p}^i$ at $\nu_{E_i}$. To ease notations, we write $p^0=p^1=p$, having in mind this interpretation through tangent vectors.

Moreover, for any point $p \in E$, we denote by $f_\pi(p)$ the image of $p$ though $f_\pi|_E \colon E \to E$, which is well defined even if $p \in \on{Ind}(f_\pi)$.

We now need a lemma, identical in spirit to the analogous \cite[Lemma 4.6]{favre-jonsson:dynamicalcompactifications} and \cite[Lemma 5.4]{gignac-ruggiero:attractionrates}. Here the proof is simpler with respect to the cited papers, since we allow the creation of cyclic quotient singularities.
As usual, for any modification $\pi \colon X_\pi \to (X,x_0)$, we denote by $f_\pi \colon X_\pi \dashrightarrow X_\pi$ the lift of $f$ to $X_\pi$.

\begin{lem}\label{lem:noperiodic_cyclicquot}
Let $f\colon(X,x_0) \to (X,x_0)$ be a dominant germ satisfying condition \ref{item:condition3}.
Consider a periodic divisorial valuation $\nu_{E_0} \in S$, and let $\pi_0\colon X_{\pi_0} \to (X,x_0)$ be any good resolution realizing $\nu_{E_i}=f_\bullet^i\nu_{E_0}$ for all $i\in \nN$.
Let $p_0 \in E_0$ be any periodic point for $f_{\pi_0}$ satisfying the condition $p_i=f_{\pi_0}^i(p_0)$ is not parallel to $S$ for any $i \in \nN$.
Then there exists a modification $\pi_1\colon X_{\pi_1} \to (X,x_0)$ dominating $\pi_0$, obtained as modifications over the orbit of $p_0$, so that $f_{\pi_1}$ is regular along the orbit of $\tilde{p}_0$.
Moreover, $\pi_1$ can be chosen so that $X_{\pi_1}$ has at most cyclic quotient singularities.
\end{lem}
\begin{proof}
Set $r \in \nN^*$ the period of $p_0$. Notice that since $p_i$ is not parallel to $S$ for all $i$, it belongs to a unique component in the orbit of $E_0$, which is $E_i$ by construction.
Let $\nu_0$ be any divisorial valuation on the tangent direction associated to $p_0$, and set $\nu_i=f_\bullet^i\nu_0$ for $i=1, \ldots, r$.
By taking $\nu_0$ sufficiently close to $\nu_{E_0}$, we may assume that $\nu_i$ is on the tangent direction associated to $p_i$ at $\nu_{E_i}$ for all $i=0,\ldots, r$.
Denote by $J_i$ the segment $[\nu_{E_i},\nu_i]_{p_i}$ for all $i=0 \ldots, r$.
Again by taking $\nu_0$ sufficiently close to $\nu_\star$, we may assume that $U(J_i) \cap \mc{S}_f \subseteq J_i$. It follows that $f_\bullet(U(J_i)) \subseteq U(J_{i+1})$ for all $i=0, \ldots, r-1$.
Moreover, we may also assume that $J_i$ does not contain in its relative interior any divisorial valuation associated to non-rational exceptional primes.
Consider a good resolution $\pi_2$ dominating $\pi_0$ and realizing $\nu_0, \ldots, \nu_{r-1}$.
Now let $\eta_{\tilde{\pi}\pi_2}\colon X_{\pi_2} \to X_{\tilde{\pi}}$ be the contraction of all prime divisors $D$ whose associated valuation $\nu_D$ belongs to $]\nu_{E_i}, \nu_i[_{p_i}$ for $i =0, \ldots, r-1$.

The induced modification $\tilde{\pi}\colon X_{\tilde{\pi}} \to (X,x_0)$ has (at most) $r$ cyclic quotient singularities $\tilde{p}_0, \ldots, \tilde{p}_{r-1}$, corresponding to the contraction of the chains of rational curves belonging to $J_0, \ldots, J_{r-1}$.
By construction, $U(J_i)=U_{\tilde{\pi}}(\tilde{p}_i)$, and by \refprop{detecting_holomorphicity}, the action of the lift $f_{\tilde{\pi}} \colon X_{\tilde{\pi}} \dashrightarrow X_{\tilde{\pi}}$ is regular at the points $\tilde{p}_i$, $i=0, \ldots, r-1$.
\end{proof}

We now come back to the proof of \refthm{dimensional_stability}.

\begin{enumerate}[label=\textit{Case} \textit{\arabic{enumi}\alph*}.,leftmargin=0pt, itemindent=48pt]
\item $S=\{\nu_{\star}\}$ is a unique divisorial eigenvaluation.

Let $\pi\colon X_\pi \to (X,x_0)$ be any modification, and consider a good resolution $\pi_0 \colon X_{\pi_0} \to (X,x_0)$ dominating $\pi$ and so that $\nu_\star = \nu_{E_\star}$ is realized by $\pi_0$.
Let $U$ be any weak open neighborhood of $\nu_{\star}$ that does not contain any other divisorial valuation realized by $\pi_0$.
By construction, for any exceptional prime $D$ in $X_{\pi_0}$, there exists $N$ so that $f_\bullet^n \nu_D \in U$ for all $n \geq N$.

If there exists $M$ so that $f_\bullet^M \nu_D = \nu_\star$, then we get that $f_\pi(D) = E_\star$ for all $n \geq M$, and $D$ satisfies the geometrical stability condition.
Assume this is not the case. Then for any $n \geq N$, $f_\bullet^n \nu_D$ belongs to the connected component of $\mc{V}_X \setminus \Gamma_{\pi_1}^*$ associated to some tangent direction $\vect{v_n}$ corresponding to a point $p_n \in E_\star$.
For $n$ big enough, $\vect{v_{n+1}}= df_\bullet \vect{v_n}$.
If the orbit $\vect{v_n}$ is infinite, it will avoid indeterminacy points after a finite number of iterates, and the geometrical stability condition is satisfied.
Assume the orbit is finite, and up to replacing $\nu_D$ by $f_\bullet^n \nu_D$ for $n$ big enough, we may assume that $p_n$ is periodic of some period $r$.
By applying \reflem{noperiodic_cyclicquot} to the periodic point $p_n$, we may find another modification $\pi_1 \colon X_{\pi_1} \to (X,x_0)$ dominating $\pi_0$ so that the orbit of $\tilde{p}_n$ does not meet $\on{Ind}(f_{\pi_1})$.
In particular $D$ satisfies the geometrical stability condition.

Notice that since $E_\star$ is compact, the number of indeterminacy points of $f_{\pi_0}$ in $E_\star$ is finite.
Moreover, the modification $\pi_1$ is an isomorphism outside of the orbit of $p_n$. In particular, the number of indeterminacy points of $f_{\pi_1}$ on $E_\star$ is strictly smaller than the one for $f_{\pi_0}$.
By applying this argument recursively on the number of periodic orbits in $E_\star$ meeting indeterminacy points, we find a model $\pi'\colon X_{\pi'} \to (X,x_0)$ where $f_{\pi'}$ acts regularly along all periodic orbits in $E_\star$.
The model $\pi'$ is geometrically stable for $f$.

\item $S$ is either a (non-trivial) segment or a circle of eigenvaluations for $f$.
Let $\pi\colon X_\pi \to (X,x_0)$ be any modification, and let $S_\pi^*$ be the set of all divisorial valuations in $S$ realized by $\pi$.
Let $\pi_0 \colon X_{\pi_0} \to (X,x_0)$ be any good resolution dominating $\pi$ and realizing all divisorial valuations in $S_\pi^*$, plus possibly the divisorial endpoints of $S$ if it is a segment. Up to taking an even higher model, we may assume that $S_{\pi_0}^*$ is not empty.

For any divisorial valuation $\nu_E \in S_{\pi_0}^*$, the exceptional prime $E$ is invariant by $f_{\pi_0}$. By applying recursively \reflem{noperiodic_cyclicquot}, we find a modification $\pi_1 \colon X_{\pi_1} \to (X,x_0)$ dominating $\pi_0$, with at most cyclic quotient singularities, so that $f_{\pi_1}$ is regular along all periodic orbits which are not parallel to $S$.
Now, let $\nu_E, \nu_F \in S_{\pi_1}^*$ be any two divisorial valuations realized by $\pi_1$ for which $E \cap F = \{p\}$ in $X_{\pi_1}$. By construction, $p$ is a smooth point, and by total invariance of $S$, we have that $f_\bullet(U_{\pi_1}(p)) \subseteq U_{\pi_1}(p)$.
In particular, $f_{\pi_1}$ is regular at $p$ for all $p \in \mc{V}_X$.
If $S$ has a curve endpoint, say $\nu_C$, then an analogous property holds for the point of intersection between the strict transform of $C$ and the exceptional divisor of $\pi_1$.
We deduce that, if $E$ is the union of all exceptional primes associated to valuations in $S_{\pi_1}^*$, then $f_{\pi_1}$ is regular along every periodic orbit in $E$.
Since $f_\bullet^n \nu \to S$ for any $\nu \in \mc{V}_X^\alpha$, we have that for any exceptional prime $D \in X_{\pi_1}$, $f_{\pi_1}^n(D)$ is either a component $E$ so that $\nu_E \in S$, or $f_{\pi_1}^n(D)=p_n$ is a point in $E$, for $n$ big enough.
Notice that if there is $N$ so that $f_{\pi_1}^N(D)=E$, then $f_{\pi_1}^n(D)=E$ for all $n \geq N$, and the geometrical stability condition holds.
If this is not the case, then either $p_n$ has infinite orbit, and it will avoid the indeterminacy points of $f_{\pi_1}$ for $n$ big enough, or $p_n$ has finite orbit, and again it will avoid indeterminacy points as far as it enters the periodic cycle in $E$.
Hence $\pi_1$ is geometrically stable.

\item $S$ is a (non-trivial) segment or a circle of eigenvaluations for $f^k$, $k \geq 2$, but not for $f$.

Let $\pi_0 \colon X_{\pi_0} \to (X,x_0)$ be any good resolution dominating $\pi$ and realizing all divisorial valuations in the set $V$ given by $\displaystyle \bigcup_{i=0}^{k-1} f_\bullet^i S_\pi^*$, plus possibly the divisorial endpoints of $S$ when it is a segment.
Again by \reflem{noperiodic_cyclicquot}, we may find a modification $\pi_1 \colon X_{\pi_1} \to (X,x_0)$ dominating $\pi_0$ and so that $f_{\pi_1}$ is regular along any periodic orbit of $E=\bigcup_{\nu_D \in V} D$ not parallel to $S$.
In this case, there may be points $p \in E$, parallel to $S$, where $f_{\pi_1}$ is not regular.
To fix this, we consider the contraction $\eta_{\tilde{\pi}\pi_1} \colon X_{\pi_1} \to X_{\tilde{\pi}}$ of every exceptional prime $D \in S_{\pi_1}^* \setminus V$. These primes are organized in chains of rational curves between consecutive divisorial elements of $V$. In particular, the induced modification $\tilde{\pi} \colon X_{\tilde{\pi}} \to (X,x_0)$ has at most cyclic quotient singularities (the ones given by $\pi_1$, plus the ones obtained by the contraction $\eta_{\tilde{\pi}\pi_1}$). Notice that by construction, $\tilde{\pi}$ dominates $\pi_0$, and hence $\pi$.
By invariance of $V$ under the action of $f_\bullet$, and since $f_\bullet|_S$ is totally invariant, we deduce that the action of $f_{\tilde{\pi}}$ is regular along every periodic orbit of $E$.
As before, we deduce that $\tilde{\pi}$ is a geometrically stable model for $f$.
\end{enumerate}

\end{enumerate}

\subsection{Smoothness of geometrically stable models}

In the previous section, we constructed geometrically stable models which have cyclic quotient singularities. While this result is quite natural, and sufficient for all known applications, one could be interested in knowing when we can construct \emph{smooth} geometrically stable models.
Here we collect a few techniques and remarks to get such smooth models.

First, notice that the proof of \refthm{dimensional_stability} produces already a smooth geometrically stable model when $f$ falls into \textit{Case 1}, i.e., it admits a unique eigenvaluation, which is non-quasimonomial.
For the other cases, we first need some preliminary lemmas.

Consider a dominant germ $f\colon (X,x_0) \to (Y,y_0)$, and its induced map $f_\bullet\colon \mc{V}_X\to \mc{V}_Y$ on valuative spaces.
If we are interested in the local behavior of this map on monomial valuations at some infinitely-near point $p$, it is natural to consider the action with respect to the monomial weights.

\begin{lem}\label{lem:action_monomialweights}
Let $f\colon (X,x_0) \to (Y,y_0)$ be a dominant germ between normal surface singularities, and let $\nu \in \mc{V}_X$ be any quasimonomial valuation.
Then there exists good resolutions $\pi\colon X_\pi \to (X,x_0)$ and $\varpi\colon Y_\varpi \to (Y,y_0)$ and infinitely near points $p \in \pi^{-1}(x_0)$, $q \in \varpi^{-1}(y_0)$, so that:
\begin{itemize}
\item $\nu$ is a monomial valuation at $p$.
\item For all $r,s \geq 0$ not both zero, the monomial valuation $\nu_{r,s}$ at $p$ is sent to a monomial valuation $\nu_{r',s'}=f_*\nu_{r,s}$ at $q$.
\item There exists an invertible matrix $\left(\begin{smallmatrix}a&b\\c&d\end{smallmatrix}\right)$ with non-negative integer entries, so that
$
(r',s')=(ar+bs,cr+ds).
$
\end{itemize}
\end{lem}
\begin{proof}
Take any good resolution $\pi$ so that $\nu$ is monomial at a certain point in $\pi^{-1}(x_0)$.
By \cite[Theorem 3.2]{cutkosky:monomialization3foldstosurf}, there exists good resolutions $\pi$ and $\varpi$ and infinitely-near points $p$ and $q$ as in the statement, so that the map $\tilde{f}=\varpi^{-1}\circ f \circ \pi \colon X_\pi \to Y_\varpi$ is regular, sends $p$ to $q$, and can be written as
$$
\tilde{f}(x,y)=\big(x^ay^b u(x,y), x^c y^d v(x,y)\big),
$$
where $(x,y)$ are local coordinates at $p$ adapted to $\pi^{-1}(x_0)$, the coordinates in the target space are also centered at $q$ and adapted to $\varpi^{-1}(y_0)$, and $u$ and $v$ are suitable holomorphic functions not vanishing at $p$.
The statement easily follows from a direct computation.
\end{proof}

For selfmaps, we need to express this behavior on monomial weights with the respect to the same monomial weights at the source and at the target, or equivalently, we need to impose $p=q$.

\begin{lem}\label{lem:selfmap_monomialweights}
Let $f\colon (X,x_0) \to (X,x_0)$ be a dominant germ, and let $\nu_\star \in \mc{V}_X$ be any quasimonomial eigenvaluation.
Assume there are intervals $I \subset \mc{V}_X$ containing $\nu_\star$, as small as wanted, that are $f_\bullet$-invariant.
Let $\pi\colon X_\pi \to (X,x_0)$ be any good resolution so that all valuations in $I$ are monomial at a suitable point $p \in \pi^{-1}(x_0)$.
Denote by $(r_\star,s_\star)$ the weights of $\nu_\star$ as a monomial valuation at $p$. Then the following properties hold.
\begin{itemize}
\item For all $r,s \geq 0$ so that $s/r$ is sufficiently close to $s_\star/r_\star$, the monomial valuation $\nu_{r,s}$ at $p$ is sent to a monomial valuation $\nu_{r',s'}=f_*\nu_{r,s}$ at $p$.
\item There exists an invertible matrix $\left(\begin{smallmatrix}a&b\\c&d\end{smallmatrix}\right)$ with non-negative integer entries, so that
$
(r',s')=(ar+bs,cr+ds)
$
for all $(r,s)$ with $s/r$ sufficiently close to $s_\star/r_\star$.
\end{itemize}
\end{lem}
\begin{proof}
Up to taking an higher good resolution, we may assume that $\{p\} = E \cap F$ for two exceptional primes $E$, $F$ in $X_{\pi_1}$. 
The weights $(r',s')$ may be computed using b-divisors, as $r'=Z(f_*\nu) \cdot Z(E)$ and $s'=Z(f_*\nu) \cdot Z(F)$, where $Z(E)$ denotes the Cartier b-divisor determined by $E$ in the good resolution $\pi_1$. 

We make the computation for $r'$, the one for $s'$ being completely analogous.
By \eqref{eqn:pushforward_bdivisors}, we have
$$
r'=Z(f_*\nu) \cdot Z(E) = f_*Z(\nu)\cdot Z(E) + \sum_{\nu_C \in \VC{f}} (-Z(\nu_C) \cdot Z(\nu)) c_\alpha(f, \nu_C)Z(f_\bullet \nu_C) \cdot Z(E).
$$
Up to shrinking $I$ and taking a higher good resolution $\pi$, we may assume that for all contracted curve valuations $\nu_C \in \VC{f}$, $f_\bullet \nu_C$ does not belong to $[\nu_{E},\nu_{F}]_p$.
This implies that $Z(f_\bullet \nu_C)\cdot Z(E)=0$ for all $\nu_C \in \VC{f}$.
Hence
$$
r'=f_*Z(\nu)\cdot Z(E) = Z(\nu)\cdot f^*{Z(E)} = \sum_{D' \in \Gamma_{\pi'}^*} \kk{D'}{E} Z_{\pi'}(\nu) \cdot D' + \sum_{C \in \CC{f}} \kk{C}{E} Z_{\pi'}(\nu) \cdot C,
$$
where $\pi' \colon X_{\pi_2} \to (X,x_0)$ is a good resolution dominating $\pi$ and so that $\pi'$ and $\pi$ give a resolution of $f$ with respect to $\VC{f}$.
Notice that by our assumption on $f_\bullet \nu_C$ not belonging to $[\nu_{E}, \nu_{F}]_p$, we infer that $\kk{C}{E}=0$ for all $C \in \CC{f}$.
We hence have
$$
r'
= \sum_{D' \in \Gamma_{\pi'}^*} \kk{D'}{E} \eta^*Z_{\pi}(\nu) \cdot D' 
= \sum_{D' \in \Gamma_{\pi'}^*} \kk{D'}{E_1} Z_{\pi}(\nu) \cdot \eta_*D'
= \sum_{D \in \Gamma_{\pi}^*} \kk{\eta^\star D}{E} Z_{\pi}(\nu) \cdot D,
$$
where $\eta=\eta_{\pi\pi'}$ and $\eta^\star D$ denotes the strict transform of $D$ through $\eta$.
It follows that $r'=ar+bs+h$, where $a=\kk{\eta^\star E_1}{E_1}$ and $b=\kk{\eta^\star F_1}{E_1}$ are non-negative integers, and $h$ is some constant.
We proceed analogously $F$, obtaining $s'=cr+ds+k$ with analogous properties for $c,d,k$.
We conclude by noticing that $h=k=0$ when $s/r \in I$, and the matrix $M$ is invertible, since we know that for such $(r,s)$ we have $f_\bullet \nu_{r,s}=\nu_{r',s'}$, and the map $f_\bullet$ cannot be locally constant since $f$ is dominant, see \refprop{propertiesfbullet}(2).
\end{proof}

Assume now we are in \textit{Case 2} of the proof of \refthm{dimensional_stability}, and $f$ admits a unique eigenvaluation $\nu_\star$, which is irrational.
To get a smooth geometrically stable model, we need to find an interval $J'$ so that $f_\bullet J' \subseteq J'$ of the form $J'=[\nu_{E'}, \nu_{F'}]$, where $E'$ and $F'$ are two exceptional primes which intersect at a point in a suitable good resolution.

Consider $J$ as in the proof of \refthm{dimensional_stability}. By taking $J$ small enough, we may assume that $J \subset U=U_{\pi_0}(p_0)$ for a suitable (smooth) point $p_0$ in the exceptional divisor of a good resolution $\pi_0$ dominating $\pi$.
Notice that the closure of $U$ is isomorphic to the space of normalized valuations $\mc{V}_{p_0}$ at $(X_{\pi_0},p_0) \cong (\nC^2,0)$.

We my find a good resolution $\pi_1 \colon X_{\pi_1} \to (X,x_0)$ dominating $\pi_0$ so that there exists two exceptional primes $E_1$, $F_1$ in $X_{\pi_1}$ intersecting in a smooth point $p_1$ and so that $J \subseteq [\nu_{E_1}, \nu_{F_1}]_{p_1}$.
Pick coordinates $(x,y)$ at $p_1$ so that $E_1=\{x=0\}$ and $F_1 = \{y=0\}$.
Denote by $\nu_{r,s}$ the monomial valuation at $p_1$ of weights $(r,s)$ with respect to the coordinates $(x,y)$.
Notice that $\nu_{r,s}$ is normalized in $\hat{\mc{V}}_{p_0}$ the set of centered valuations at $p_0$ as long as $b^0_{E_1}r+b^0_{F_1} s = 1$, where $b^0$ denotes the general multiplicity in $\hat{\mc{V}}_{p_0}$.
Moreover, $J$ is given by such normalized valuations so that $s/r \in I$ for a suitable interval $I \subset [0,\infty]$.
Since $f_\bullet J \subset J$, then for any $(r,s)$ with $s/r \in I$, we have $f_*\nu_{r,s} = \nu_{r',s'}$ for suitable $(r',s')$ with $s'/r' \in I$.
Notice that the irrational eigenvaluation $\nu_\star \in J$ also satisfies $\nu_\star=\nu_{r_\star, s_\star}$ and $f_*\nu_\star = \nu_{r'_\star, s'_\star}$ is so that $s'_\star/r'_\star = s_\star / r_\star \in I \setminus \nQ$.

By \reflem{selfmap_monomialweights}, for all $(r,s)$ so that $s/r \in I$ is close enough to $s_\star/r_\star$, we have $(r',s')=(ar+bs, cr+ds)=:M(r,s)$, with $M=\left(\begin{smallmatrix} a&b\\c&d \end{smallmatrix}\right)$ an invertible matrix with non-negative integer coefficients.

Now we can argue as in \cite[p.330]{favre-jonsson:eigenval}.
By \cite[Lemma 5.7]{favre-jonsson:eigenval}, there exists arbitrarily large integers $r_0,s_0,r_1,s_1$ such that $\frac{s_0}{r_0} < \frac{s_\star}{r_\star} < \frac{s_1}{r_1}$, $s_1r_0-s_0r_1=1$, and $M$ maps the interval $\left]\frac{s_0}{r_0},\frac{s_1}{r_1}\right[$ into itself.
This says exactly that there exists a toric modification (as high as desired) $\eta_2$ over $p_1$, and exceptional primes $E_2, F_2$ which intersect at a smooth point $p_2$, so that the segment $]\nu_{E_2}, \nu_{F_2}[_{p_2}$ is invariant by the action of $f_\bullet$. The toric modification is given in local charts at $p_2$ and $p_1$ by the monomial map associated to the matrix $\left(\begin{smallmatrix} s_1&s_0\\r_1&r_0 \end{smallmatrix}\right)$ for suitable coordinates adapted to $E_2, F_2$ and $E_1, F_1$.

\medskip

Assume we are in \textit{Case 3c}, i.e., $f$ admits a segment or a circle of eigenvaluations for $f^k$, $k \geq 2$, but not for $f$.
In this case smooth geometrically stable models don't exist, not even when $(X,x_0)$ is smooth (see \cite{gignac-ruggiero:attractionrates}).
But the $k$-th iterate of $f$ falls into \textit{Case 3b}, and we may try to construct a smooth geometrically stable model for $f^k$.

Assume hence we are in the remaining \textit{Cases 3a} or \textit{3b}.
In these cases, the only cyclic quotient singularities which appear in the construction given by the proof of \refthm{dimensional_stability} are given by applying \reflem{noperiodic_cyclicquot}.
In the smooth case, one can prove a similar result, without creating cyclic quotient singularities. This proves the existence of smooth geometrically stable models in these cases.
The key point of the argument is a statement on the speed of convergence of $f_\bullet^n \nu$ to a (divisorial) eigenvaluation $\nu_\star$, when $\nu$ belongs to a small segment $I$ with endpoints $\nu_\star=\nu_{E_\star}$ and some divisorial valuation $\nu_F$.
In fact, take any sufficiently small segment $I=[\nu_{E_\star}, \nu_F]$ so that $E_\star$ and $F$ intersect at a smooth point $p$ of some good resolution $\pi$, so that we may parameterize $I$ with respect to the monomial parameterization $w(t)$, $t \in [0,1]$.
Assume $f_\bullet I \subset I$; then by \reflem{selfmap_monomialweights}, $f_\bullet(w(t))=w(h(t))$, where $h$ is a suitable M\"obius map.

By \refprop{skewness_parameter}, $t \mapsto \alpha(w(t))$ is a polynomial mapping of degree $\leq 2$.
If its degree is exactly $1$, then $\alpha(f_\bullet(w(t)))=H(\alpha (w(t))$, with $H$ again a M\"obius map (with non-negative integer coefficients).

The easy \cite[Lemma 5.5]{favre-jonsson:eigenval} states that if $H$ is such a map and $t_\star$ is a \emph{strictly positive} fixed point for $H$, then either $\abs{H'(t_\star)} < 1$, or $H$ has finite order. This is exactly the technical statement we need to ensure the smooth analogous of \reflem{noperiodic_cyclicquot}.

As we have seen above, this same property holds when $\nu_\star$ is the unique irrational eigenvaluation for $f$.
By \refprop{big_tree}, the desired property on $\alpha$ holds as far as there exists a log resolution $\pi$ of $\mf{m}_X$ so that the tangent direction at $\nu_\star$ corresponding to $\nu_F$ is not parallel to $\mc{S}_\pi$.
Notice that if we work with respect to the monomial parameterization, we always get a M\"obius map $h$ which has the desired properties, but in this case the weight corresponding to the fixed point $\nu_{E_\star}$ is zero, and \cite[Lemma 5-5]{favre-jonsson:eigenval} does not apply.

To deal with tangent directions which lie in $\mc{S}_\pi$, we would need to prove a strong contraction property with respect to another weight function different from $\alpha$.
Although we believe that this strong contraction property still holds in general, we are unable to prove it and leave it as an open question.

\section{Attraction rates}\label{sec:attraction_rates}

Let $f \colon (X,x_0) \to (X,x_0)$ be a dominant germ on a normal surface singularity $(X,x_0)$.
Here we study for any (normalized) valuation $\nu \in \mc{V}_X$ (of finite skewness), the sequence $c(f^n, \nu)$ of attraction rates, and its asymptotic behavior the \emph{first dynamical degree} (see below).

\subsection{First dynamical degree}

Before defining the first dynamical degree, we need to describe a few properties of the sequence of attraction rates.

\begin{prop}
Let $f\colon (X,x_0)\to (X,x_0)$ be a dominant germ on a normal surface singularity.
Let $\nu \in \mc{V}_X^\alpha$ be any normalized valuation of finite skewness.
Then $c_\infty(f,\nu)=\lim_{n \to \infty} \sqrt[n]{c(f^n, \nu)}$ exists, belongs to $[1, +\infty)$, and it does not depend on $\nu$.
\end{prop}
\begin{proof}
We first assume that the limit exists for a suitable normalized valuation $\nu \in \mc{V}_X^\alpha$ of finite skewness.
Take any other valuation $\mu \in \mc{V}_X^\alpha$. By \refcor{skewness_comparison}, $\beta(\nu|\mu)$ and $\beta(\mu|\nu)$ are finite.
Assume for the moment that $f$ is finite. Then
$$
c(f^n,\nu)=\nu((f^n)^*\mf{m}) \leq \beta(\nu|\mu) \mu((f^n)^*\mf{m}),
$$
and we deduce $c_\infty(f,\nu) \leq \liminf_n \sqrt[n]{c(f^n,\mu)}$ by taking the inferior limit for $n \to \infty$ of the $n$-th root of this inequality.
In the general case, we may consider $\mf{a}_k = (f^n)^*\mf{m} + \mf{m}^k$. Then $\nu(\mf{a}_k)$ is an increasing sequence converging to $\nu((f^n)^*\mf{m})$, the same for $\mu$, and we deduce the same estimate obtained previously for the finite case.
By repeating the argument by interchanging the roles of $\nu$ and $\mu$, and taking the superior limit, we deduce $c_\infty(f,\nu) \geq \limsup_n \sqrt[n]{c(f^n,\mu)}$.
Hence the limit $c_\infty(f,\mu)$ exists and coincide with $c_\infty(f,\nu)$.

We now prove that there exists $\nu \in \mc{V}_X^\alpha$ so that $\sqrt[n]{c(f^n,\mu)}$ converges.
By taking the logarithm, we want to study limit of the sequence
$$
\frac{1}{n} \log c(f^n,\nu) = \frac{1}{n} \sum_{j=0}^{n-1} \log c(f,f_\bullet^j \nu).
$$
If the sequence $\log c(f,f_\bullet^j \nu)$ converges to some value $c \in \nR$, then its associated sequence of Cesaro means $\frac{1}{n} \log c(f^n,\nu)$ does as well.
If $f_\bullet^j \nu$ converges (weakly) to an eigenvaluation $\nu_\star$, then by continuity $c(f_\bullet^j \nu)$ converges to $c(f, \nu_\star)$, and we are done.
It may be that $f$ does not admit eigenvaluations, exactly when $(X,x_0)$ is a cusp singularity, and $f_\bullet$ acts as a rotation on the circle $\mc{S}_X$.
If the rotation is rational, then there exists $k \in \nN^*$ so that $f_\bullet^k$ is the identity on $\mc{S}_X$, and $c(f,f_\bullet^{hk+i}\nu)$ converges to some constant $c$ for all $i \in \nN$. It follows that the sequence $\log c(f^j, \nu)$ converges to $c$.
If the rotation is irrational, assume that $\nu \in \mc{S}_X$. Then by the ergodic theorem,
$$
\lim_{n \to +\infty} \frac{1}{n} \sum_{j=0}^{n-1} \log c(f,f_\bullet^j \nu) = \int_{\mc{S}_X} \log c(f,\nu) d\rho(\nu),
$$
where $\rho(\nu)$ is the ergodic measure on the circle $\mc{S}_X$.
\end{proof}

\begin{defi}
Let $f\colon (X,x_0)\to (X,x_0)$ be a dominant germ on a normal surface singularity.
The \emph{first dynamical degree} of $f$ is $c_\infty(f)=c_\infty(f,\nu)$ for any valuation $\nu \in \mc{V}_X^\alpha$ of finite skewness.
\end{defi}

The dynamical degree describes the speed of convergence of the orbits of $X$ towards $x_0$. It is a fundamental invariant of (bimeromorphic) conjugacy for global dynamics as well as in our setting, see e.g. \cite{diller-favre:dynbimeromapsurf}.
We now show that the first dynamical degree is always a quadratic integer in our setting, generalizing \cite[Theorem A]{favre-jonsson:eigenval} to normal surface singularities.
Our contribution mainly lies in the study of the case of rotations on cusps singularities.

\begin{thm}
Let $f\colon (X,x_0)\to (X,x_0)$ be a dominant germ on a normal surface singularity.
Then the first dynamical degree $c_\infty(f)$ is a quadratic integer.
\end{thm}
\begin{proof}
First, notice that if $\nu_\star$ is an eigenvaluation, then $c_\infty(f)=c(f,\nu_\star)$.

When $f$ admits a divisorial eigenvaluation $\nu_\star = \nu_{E_\star}$, by \refprop{divisorial_image} we get $c(f,\nu_{E_\star})=\kk{E_\star}{E_\star} \in \nN^*$.
The situation is analogous when $f$ admits a curve eigenvaluation $\nu_\star = \nu_{C_\star}$, since by \refprop{curve_image} $c(f, \nu_\star)=\ee{C}{C}$ is an integer.

Assume now that $f$ admits an irrational eigenvaluation.
Then there exists a good resolution $\pi \colon X_\pi \to (X,x_0)$ and a point $p$ in the intersection of two exceptional primes of $X_\pi$ so that the lift of $f$ to $X_\pi$ is a holomorphic monomial map, associated to a suitable $2 \times 2$ matrix $M$. It follows that $c(f,\nu_\star)$ is a quadratic integer, since it satisfies $c(f,\nu_\star)^2 - \on{tr}{M} c(f, \nu_\star) + \det M = 0$.
We could also argue by applying \reflem{selfmap_monomialweights} and noticing that $c(f, w(t))$ varies affinely with respect to the monomial weight.

The case of an infinitely singular eigenvaluation can be treated similarly, see \cite[Theorem 5.1]{favre-jonsson:eigenval}, to show that $c(f, \nu_\star)$ is an integer.

The remaining case, when $f$ does not admit eigenvaluations, consists of finite maps on cusp singularities, which induce a rotation on the circle $\mc{S}_X$.

Let $(X,x_0)$ be a cusp singularity, that we may assume is given by the toric construction of \S\ref{ssec:arith_cusp}.
By \refprop{arith_action_on_cycle}, the action of $f_\bullet$ on $\hat{\mc{S}}_X^*$ corresponds to the action of a linear map $g_\alpha$ on $\mc{C}=(\nR_+^*)^2$, where $\alpha \in K=\nQ(\sqrt{d})$ is such that $Q(\alpha)=\alpha\alpha' \in \nN^*$.

We claim that $c_\infty(f)=\sqrt{Q(\alpha)}=:q(\alpha)$, which is in particular a quadratic integer.
In fact, given a ray generated by $v \in \mc{C}_0$, its image $Lv$ by the linear action $L$ induced by $g_\alpha$ satisfies $q(Lv)=q(\alpha)q(v)$.
Denote by $b(v)$ its norm, i.e., the value the valuation associated to $v$ takes on the maximal ideal. The attraction rate is then given by $c(f,v)=q(\alpha)b(v)/b(Lv)$. By iterating, $c(f^n,v)=q(\alpha)^n b(v) / b(L^n v)$.
Now, by compacity, $b(v)$ is bounded. It follows that $\lim_{n \to +\infty} \sqrt[n]{c(f^n,v)}=q(\alpha)$.
\end{proof}

\subsection{Recursion relations for the sequence of attraction rates}

In the previous section we have showed that $c_\infty(f)$ is a quadratic integer.
We now focus our attention on the sequence $c(f^n, \nu)$ for a given normalized valuation $\nu \in \mc{V}_X^\alpha$ of finite skewness.
In the smooth setting, this sequence always satisfies eventually a linear recursion relation with integer coefficients, see \cite[Theorem A]{gignac-ruggiero:attractionrates}. Here we show that the same property holds in the singular case, with a very specific exception, that will be dealt with in the next section.
In fact, here we assume we are not in the following situation: $(X,x_0)$ is a cusp singularity, $f$ is a finite non-invertible germ on $(X,x_0)$, which induces an irrational rotation on the skeleton $\mc{S}_X$.
Our assumption is equivalent to the fact that $f^k$ admits an eigenvaluation for some $k \in \nN^*$, and by \refthm{dimensional_stability}, it implies the existence of geometrically stable models.
We will see in the next section that the existence of such models is actually equivalent to admitting an eigenvaluation for an iterate of $f$.

We first prove that the existence of a geometrically stable model, together with the interpretation of the attraction rate $c(f,\nu)$ in terms of intersections of suitable b-divisors established by \refprop{attractionrate_intersection}, implies the existence of the desired recurrence relation for the sequence of attraction rates.
This result is a straightforward generalization of \cite[Corollary 5.5]{gignac-ruggiero:attractionrates}.

\begin{cor}
Let $f\colon (X,x_0) \to (X,x_0)$ be a dominant non-invertible germ at a normal surface singularity. Assume we are in the hypotheses of \refthm{geometricstability}.
Then for any quasimonomial valuation $\nu \in \mc{V}_X$, the sequence $(c(f^n, \nu))_n$ eventually satisfies a linear recursion relation with integer coefficients.
\end{cor}
\begin{proof}
By regularity of $f_\bullet$, we may assume that $\nu=\nu_E$ is divisorial.
Let $\pi \colon X_{\pi} \to (X,x_0)$ be a good resolution high enough to have $E$ as an exceptional prime.
By \refthm{geometricstability}, up to taking a higher model, we may assume that $\pi$ is geometrically stable for $f$. Denote by $f_\pi = \pi^{-1} \circ f \circ \pi \colon X_\pi \dashrightarrow X_\pi$ the lift of $f$ to $X_\pi$.
By \refprop{geometrical_and_algebraic_stability}, there exists $N \in \nN$ so that for any $n \geq N$, we have
$$
\Exc_\pi \circ (f^n_\pi)^* = \Exc_\pi \circ (f^N_\pi)^* (\Exc_\pi \circ f_\pi^*)^{n-N}.
$$
By \refprop{attractionrate_intersection}, $c(f,\nu)=-Z_\pi(\nu) \cdot \Exc_\pi \circ f_\pi^*Z_\pi(\mf{m}_X)$.
Since the map $\Exc_\pi \circ f_\pi^*\colon\Ediv(\pi) \to \Ediv(\pi)$ is $\nZ$-linear, by the Cayley-Hamilton theorem we have that $\Exc_\pi \circ f_\pi^*$ satisfies a monic polynomial $t^m + a_1 t^{m-1} + \cdots + a_m$ with integer coefficients.
If $n \geq N + m$, this gives
\begin{align*}
c(f^n,\nu)&=-Z_\pi(\nu) \cdot \Exc_\pi \circ (f_\pi^n)^*Z_\pi(\mf{m}_X)\\
&=-Z_\pi(\nu) \cdot \Exc_\pi \circ (f_\pi^N)^* (\Exc_\pi \circ (f_\pi)^*)^{n-N} Z_\pi(\mf{m}_X) \\
&=-\sum_{j=1}^m - a_j Z_\pi(\nu) \cdot \Exc_\pi \circ (f_\pi^N)^* (\Exc_\pi \circ (f_\pi)^*)^{n-N-j} Z_\pi(\mf{m}_X) \\
&=-\sum_{j=1}^m - a_j c(f^{n-j},\nu).
\end{align*}
This is the desired linear recursion relation eventually satisfied by the sequence of attraction rates.
\end{proof}

To prove \refthm{recursion}, we need to prove this recurrence relation has order at most $2$, at least for a suitable iterate of $f$.
The argument used in \cite[Theorem 6.1]{gignac-ruggiero:attractionrates} equally works in our setting. Here we outline the proof, details are left to the reader.
Consider a normalized valuation $\nu \in \mc{S}_X^\alpha$ of finite skewness.
Since we are working in the case of existence of eigenvaluations for $f^k$, by \refthm{classification} there exists $\nu_\star$ an eigenvaluation for $f^k$ so that $f_\bullet^{nk}\nu\to \nu_\star$ when $n \to \infty$. The convergence is strong as far as $\alpha(\nu_\star)<+\infty$.
Up to replacing $f$ by $f^k$, we may always assume $k=1$, and $\nu_\star$ is an eigenvaluation.

\begin{enumerate}[label=\textit{Case} \arabic*.,leftmargin=0pt, itemindent=40pt]
\item The eigenvaluation $\nu_\star$ is not quasimonomial.

In this case, we have seen that $c(f,\nu)=c_\infty(f)$ is a constant integer for all $\nu$ in a suitable (weak open) neighborhood $U$ of $\nu_\star$.
By \refthm{classification} $f_\bullet^n \nu \to \nu_\star$ (with respect to the weak topology) for all $\nu \in \mc{V}_X^\alpha$.
We deduce that $c(f^{n+1},\nu)=c(f,f_\bullet^n \nu) c(f^n,\nu)= c_\infty(f) c(f^n,\nu)$ for all $n\geq N$ so that $f_\bullet^N\nu \in U$.

\item The eigenvaluation $\nu_\star$ is irrational.

In this case, we have seen that there exists a log resolution $\pi\colon X_\pi \to (X,x_0)$ of the maximal ideal, and a point $p \in \pi^{-1}(x_0)$ in the intersection of two exceptional primes $E$ and $F$, so that $c(f,\nu_{r,s})$ is linear with respect to the weights $(r,s)$, where $\nu_{r,s}$ is the monomial valuation at $p$ of weights $(r,s)$, with respect to coordinates adapted to $E \cup F$.
Up to shrinking the segment $I=[\nu_E, \nu_F]_p$, by \reflem{selfmap_monomialweights}, we may assume that $f_\bullet$ acts a linear map $M$ with non-negative coefficients on the weights $(r,s)$.
In particular, let $n$ is big enough so that $f_\bullet^n \nu$ belongs to $U(I)$. Up to shrinking $I$ if necessary, we may also assume that $c(f,\nu)$ is locally constant on $U(I) \setminus I$, so that $c(f, f_\bullet^n \nu)= c(f,r_i f_\bullet^n\nu)$.
In this case (see also \cite[Lemma 6.2]{gignac-ruggiero:attractionrates}), we get the recursion relation of order $2$
$$
c(f^{n+2},\nu)=\on{tr}(M) c(f^{n+1},\nu) -\det(M)c(f^n, \nu).
$$
\item The eigenvaluation $\nu_\star=\nu_{E_\star}$ is divisorial.

If $f_\bullet^n \nu = \nu_\star$ for some $n$, $c(f^n, \nu)$ clearly eventually satisfies a linear integral recursion relation of order $1$.
Assume this is not the case. We can then associate to $f_\bullet^n\nu$ the tangent vector $\vect{v}_n$ at $\nu_\star$ towards $f_\bullet^n\nu$. Moreover, we have that $\vect{v_{n+1}}=d(f_\bullet)_{\nu_\star} \vect{v_n}$ for $n$ big enough.
Let $\mc{S}_f$ be the critical skeleton of $f$, and denote by $S$ the set of tangent vectors at $\nu_\star$ parallel to $S$.
If $\vect{v_n} \not \in S$ for all $n$ big enough, then $c(f, f_\bullet^n \nu)$ is constant for such n, and $c(f^n, \nu)$ satisfies a recursion relation of order $1$.
If $\vect{v_n}$ belongs to $S$ for infinitely many $n$, then the orbit $\vect{v_n}$ is preperiodic, and can be considered fixed if we replace $\nu$ by $f_\bullet^N \nu$ and $f$ by a suitable iterate.
In this case we may conclude as before, and find a recursion relation of order $2$.
\end{enumerate}

\subsection{Finite germs on cusp singularities}\label{ssec:finitecusp}

This section is devoted to study the sequence of attraction rates for finite germs on cusp singularities.
The only situation that is not covered by \refthm{recursion} is when $f_\bullet$ acts on the circle $S=\{\nu \in \mc{V}_X\ |\ A(\nu)=0\}$ as an irrational rotation.
We shall show that in this case, no linear recursion relations are satisfied.

\begin{prop}\label{prop:irr_rot_c_not_constant}
Let $f\colon (X,x_0)\to (X,x_0)$ be a superattracting germ on a cusp which induces an irrational rotation on $S$. Then $c(f,-)$ cannot be constant on $S$.
\end{prop}
\begin{proof}
Suppose for contradiction that $c(f,-)\equiv M$ on $S$.
That is to say, for any $\nu$ on $S$, we have
$$	
M = c(f,\nu) = -\lim_\pi Z_{\pi}(\nu)\cdot f^*Z_\pi(\mf{m}).
$$
Let $\pi_1 \colon X_{\pi_1}\to (X,0)$ be a log resolution of $\mf{m}$.
Then we can write
\[
Z_{\pi_1}(\mf{m}) = \sum_{i\in I} \lambda_i \check{E_i} + \sum_{j \in J} \mu_j\check{F}_j,
\]
where $\lambda_i, \mu_j>0$ for each $i$ and $j$,  $\nu_{E_i} \in S$ and $\nu_{F_j} \not \in S$.

Take any exceptional prime $E \subset X_{\pi_1}$ so that $\nu_E \in S$.
By \refprop{pullback}, it is of the form $f^*\check{E} = \ee{G}{E}\check{G}$ for some $\nu_G \in S$.
Since $f_\bullet$ is an irrational rotation on $S$, up to taking a suitable iterate, we can assume that $G$ does not appear as an exceptional prime in $X_{\pi_1}$.
In particular, the center of $\nu_G$ in $X_{\pi_1}$ must be a closed point $p$ lying in the intersection of two exceptional primes $D_1$ and $D_2$ which belong to the cycle. 

Blow up $p$ in $X_{\pi_1}$, thus obtaining a modification $\pi_2\colon X_{\pi_2}\to (X,x_0)$ and a new exceptional prime $H$.
Let $\pi_3 \colon X_{\pi_3} \to X_{\pi_2}$ be a modification dominating $\pi_2$, large enough to contain all the primes appearing in $f^*Z_{\pi_1}(\mf{m})$. Set $\wt{\eta}= \pi_2^{-1} \circ \pi_3$. Then we get
\[
M = c(f,\nu_H) = -Z_{\pi_3}(\nu_H)\cdot f^*Z_{\pi_1}(\mf{m}) = -\frac{1}{b_H}\wt{\eta}^*\check{H}\cdot f^*Z_{\pi_1}(\mf{m}) = -\frac{1}{b_H}\check{H}\cdot \wt{\eta}_*f^*Z_{\pi_1}(\mf{m}).
\]
Now we use the fact that $\check{H} = \check{D}_1 + \check{D}_2 - H$ in $X_{\pi_2}$ to get
\begin{align*}
M & = -\frac{1}{b_H}\check{D}_1\cdot \eta_*f^*Z_{\pi_1}(\mf{m}) - \frac{1}{b_H}\check{D}_2\cdot \eta_*f^*Z_{\pi_1}(\mf{m}) + \frac{1}{b_H}H\cdot \eta_*f^*Z_{\pi_1}(\mf{m})\\
& = -\frac{1}{b_H}\eta^*\check{D}_1\cdot f^*Z_{\pi_1}(\mf{m}) - \frac{1}{b_H}\eta^*\check{D}_2\cdot f^*Z_{\pi_1}(\mf{m}) + \frac{1}{b_H}H\cdot \eta_*f^*Z_{\pi_1}(\mf{m})\\
& = \frac{1}{b_H}b_{D_1}c(f,\nu_{D_1}) + \frac{1}{b_H}b_{D_2}c(f,\nu_{D_2}) + \frac{1}{b_H}H\cdot \eta_*f^*Z_{\pi_1}(\mf{m})\\
& = \frac{b_{D_1} + b_{D_2}}{b_H}M + \frac{1}{b_H}H\cdot \eta_*f^*Z_{\pi_1}(\mf{m})\\
& = M + \frac{1}{b_H}H\cdot \eta_*f^*Z_{\pi_1}(\mf{m})
\end{align*}
We conclude that $H\cdot \eta_*f^*Z_{\pi_1}(\mf{m}) = 0$.
We claim that this gives a contradiction, since by construction, this intersection number should be positive.

Assume first that $I$ is not empty, and take $E=E_1$. Then
\begin{align*}
H\cdot \eta_*f^*Z_{\pi_1}(\mf{m}) & = H\cdot \eta_*f^*\check{E}_1 + \sum_{i>1}H\cdot \eta_*f^*\check{E}_i + \sum_j H\cdot \eta_*f^*\check{F}_j,
\end{align*}
and each of the terms in the two summations is nonnegative.
The first term can be computed by
\[
H\cdot \eta_*f^*\check{E}_1 = \ee{G}{E_1}H\cdot \eta_*\check{G},\] and, since the valuation $\nu_G$ is centered in $H$ (by our choice of $H$), it follows that the intersection $H\cdot \eta_*\check{G}$ is positive.

Assume now that $I$ is empty. Set $E$ to be the exceptional prime in the cycle corresponding to the retraction to the cycle of the divisorial valuation $\nu_{F_1}$.
In this case
\begin{align*}
H\cdot \eta_*f^*Z_{\pi_1}(\mf{m}) & = H\cdot \eta_*f^*\check{F}_1 + \sum_{j>1} H\cdot \eta_*f^*\check{F}_j.
\end{align*}
Again each of the terms in the summation is nonnegative, and the first one is positive, since the center of any preimage of $\nu_{F_1}$ in $X_{\pi_2}$ is $H$.
This completes the proof.
\end{proof}

Using an argument due to \cite{hasselblatt-propp:degreegrowthmonmaps}, we show that for irrational rotations on cusps, $c(f^n, \nu)$ does not satisfy any linear recursion relation. As a corollary, such maps do not admit geometrically stable models.

\begin{prop}
Let $f\colon (X,x_0)\to (X,x_0)$ be a superattracting germ on a cusp which induces an irrational rotation on the cycle.
Let $\nu \in \mc{V}_X^\alpha$ be any valuation of finite skewness.
Then the sequence $c(f^n,\nu)$ does not satisfy any linear recursion relation.
\end{prop}
\begin{proof}
We may assume that $(X,x_0)$ is constructed as in \S\ref{ssec:arith_cusp}. The construction depends on a lattice $N_\omega$, and a totally positive unit $\eps$ leaving $N_\omega$ invariant, which induces a linear map $g_\eps$ leaving the open cone $\mc{C}=(\nR_+^*)^2$ invariant.
By \refrmk{arith_valuations}, valuations in $\hat{\mc{S}}_X^*$ are in $1$-to-$1$ correspondence with $\mc{C}_\eps = \mc{C}/\langle g_\eps \rangle$.
Moreover, the function $\on{ev}_\mf{m}\colon \hat{\mc{S}}_X^* \to \nR_+^*$ given by $\on{ev}_\mf{m}(\nu)= \nu(\mf{m})$ lifts to a continuous, $g_\eps$-equivariant map on $\mc{C}$, which is piecewise linear and in fact made up of linear maps on sectors.
By \refprop{arith_action_on_cycle}, the action of $f_*$ on $\hat{\mc{S}}_X^*$ lifts to a linear map $\Lambda = g_\alpha \colon \mc{C}\to \mc{C}$.

Suppose at first that $\nu$ belongs to the circle $\mc{S}_X$.
We have seen that $\nu$ corresponds to a point $[p] \in \mc{C}_\eps$. Take any representative $p \in \mc{C}$. Then $c(f^n, \nu) = (f^n_*\nu)(\mf{m}) = \on{ev}_\mf{m}(\Lambda^n p)$.
Define $c_n := c(f^n, \nu)$.
Now, let $L$ be one of the linear maps making up $\on{ev}_\mf{m}$, and define $d_n = L(\Lambda^n p)$.
Note that $d_n$ satisfies a linear recursion relation. 
Suppose that $c_n$ also eventually satisfies a recursion relation.
Then so would the sequence $c_n - d_n$.
By Skolem-Mahler-Lech's theorem, we conclude that $\{n : c_n = d_n\}$ consists of a finite union of arithmetic progressions.
Since $f$ gives an irrational rotation, the only way this can happen is if $\on{ev}_\mf{m}\equiv L$, which implies $c(f,\cdot)$ is constant on the cycle, a contradiction by \refprop{irr_rot_c_not_constant}.

Let now take any valuation $\nu \in \mc{V}_X^\alpha$.
Notice that $c(f,\cdot)$ is locally constant outside the (finite) skeleton $\mc{S}_{c(f)}$.
In particular, there is only a finite number of tangent vectors at divisorial valuations in $\mc{S}_X$ along which $c(f,\cdot)$ is not locally constant. Since $f_\bullet$ acts as an irrational translation on $\mc{S}_X$ and $r_X f_\bullet^n \nu = f_\bullet^n r_X \nu$, we avoid these directions after a finite number of iterates.
We deduce that $c(f^n, \nu)$ does not eventually satisfy any linear recursion relation for any valuation $\nu$ of finite skewness.
\end{proof}


\section{Examples and remarks}\label{sec:examples}

\subsection{A finite map at a smooth point}\label{ssec:example_smooth_finite}
Consider a smooth point $(X,x_0)=(\nC^2,0)$. In this case the space of normalized valuations is called the \emph{valuative tree} $\mc{V}$.
It is a complete real tree, rooted at the multiplicity valuation $\ord_0$, which is the divisorial valuation associated to the exceptional prime $E_0=\pi_0^{-1}(0)$, where $\pi_0 \colon X_{\pi_0} \to (\nC^2,0)$ is the blow-up of the origin.
Notice that $\pi_0$ is also a log resolution of the maximal ideal, and $Z_{\pi_0}(\mf{m}) = Z_{\pi_0}(\ord_0) = \check{E}_0 = -E_0$.
For any $k \in \nN$, let $f_k \colon (\nC^2,0)\to (\nC^2,0)$ be the dominant germ
$$
f_k(x,y)=(x^k(y+x^3),x^2y).
$$
Notice that $f_k$ is non-invertible, and finite if and only if $k=0$.
For any irreducible $\phi \in R=\hat{\mc{O}}_{\nC^2,0}$, we denote by $\nu_\phi$ the associated curve semivaluation. For any $t \in [1,\infty]$, we denote by $\nu_{\phi,t}$ the unique normalized valuation in $[\ord_0,\nu_\phi]$ with skewness $t$.
Consider first the case $k=0$ and write $f=f_0$.
By direct computation, we notice that $f_\bullet^{-1}(\ord_0)=\{\nu_{y+x^3,5}\}$.
By \refcor{attractionrate_locallyconst}, the attraction rate $c(f,\nu)$ is locally constant outside $\mc{T}_{c(f)} = [\ord_0,\nu_{y+x^3,5}]$.
Moreover,
$$
f_* \nu_{x,t} = \nu_{y,1+2t},
\quad
f_* \nu_{y,t} =
\begin{cases}
t \nu_{y,1+\frac{2}{t}} & \text{if } t \in [1,3],\\
3 \nu_{y,\frac{2+t}{3}} & \text{if } t \in [3,+\infty],
\end{cases}
\quad
f_* \nu_{y+x^3,t} =
\begin{cases}
t \nu_{y,\frac{5}{t}} & \text{if } t \in [3,5],\\
5 \nu_{x,\frac{t}{5}} & \text{if } t \in [5,+\infty].
\end{cases}
$$
In particular, $f$ admits a unique divisorial eigenvaluation $\nu_\star=\nu_{y,2}$, and $c_\infty(f)=c(f,\nu_\star)=2$.
It is easy to check (see the similar computation on \S\ref{ssec:example_smooth_nonfinite}) that the sequence of attraction rates $c_n=c(f^n,\ord_0)$ satisfies the recursion relation
$$
c_0=1, c_1 = 1, c_{n+2}=c_{n+1} + 2 c_n, \forall n \in \nN.
$$
Notice that the sequence $c'_n=c(f^n, \nu_{y+x^3,5})$ satisfies $c'_{n+1}=c'_1 c_n$, and hence satisfies the same recursion relation for $n \geq 1$, while $c'_0=1$, $c'_1=5$, $c'_2 = 5$, so in particular the recursion relation does not hold for $n=0$.

\begin{rmk}\label{rmk:multiplicity_finite}
We can also compute the multiplicity of $f$ at any valuation on $[\nu_x, \nu_y] \cup [\nu_{y,3}, \nu_{y+x^3}]$.
We get
$$
m(f,\nu)=
\begin{cases}
5 & \text{if } \nu \in [\nu_{y,3}, \nu_{y+x^3}],\\
3 & \text{if } \nu \in ]\nu_{y,3}, \nu_{y}],\\
2 & \text{if } \nu \in ]\nu_{y,3}, \nu_{x}].\\
\end{cases}
$$
By \refthm{operations_bdivisors}, the map $M(f, \nu')=\displaystyle\sum_{f_\bullet \nu = \nu'}m(f,\nu)$ is the constant $5=e(f)$ the topological degree of $f$.

The values of $m(f,\nu)$ are quite difficult to control in general. In fact, at any branch point (i.e., divisorial valuation) $\nu_E$, the value of $\mc(f,\nu)$ on branches at $\nu_E$ depend on the values of $k_E$ and $e_E$.
For example, one can check that if $\nu_E$ is the divisorial valuation associated to $\nu_{y+x^3,5}$, we get $k_E=5$ and $e_E=1$.
In particular, if we denote by $E_0$ the exceptional prime obtained by blowing up the origin, the tangent map $df_\bullet \colon T_{\nu_E} \mc{V} \to T_{\nu_{E_0}} \mc{V}$ has degree $1$, and $m(f,\nu)=5$ for valuations $\nu$ along any branches, close enough to $\nu_E$.
\end{rmk}

We summarize in \reffig{valAdyn} the dynamics of $f_\bullet$. We denote by black dots curve semivaluations, and by white dots divisorial valuations.
\begin{figure}[ht]
\centering
\def\svgwidth{0.90\columnwidth}
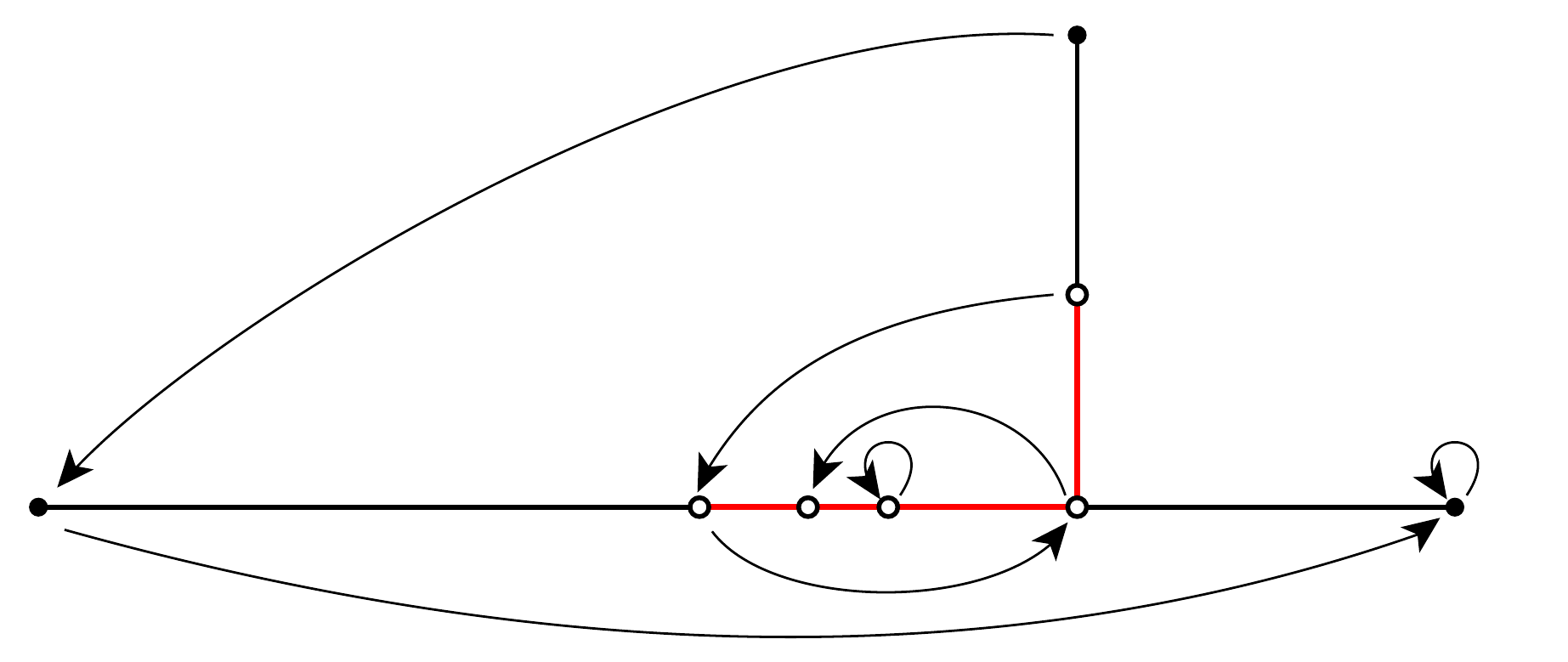
\caption{Action of $f_\bullet$ (\S\ref{ssec:example_smooth_finite}).}\label{fig:valAdyn}
\end{figure}

The model $\pi_0$ is not algebraically stable for $f$. In fact, denoting by $p_0 \in E_0$ the point associated to the tangent vector at $\nu_{E_0}$ towards $\nu_y$, we have $\on{Ind}(f_{\pi_0}) = \{p_0\}$, and $f_{\pi_0}(E_0) = p_0$.
Consider $\pi_1 \colon X_{\pi_1} \to (\nC^2,0)$, obtained from $\pi_0$ by blowing up $p_0$. Denote by $E_1$ the new exceptional prime. We have $\nu_{E_1}=\nu_\star$, hence $f_{\pi_1}$ acts on $E_1$ as the rational selfmap $h \colon \zeta \mapsto \zeta^{-1}$, where $\zeta$ corresponds to the tangent vector at $\nu_\star$ towards $\nu_{y-\zeta x^2}$. We notice that the point $p_1$ corresponding to $\zeta=0$ is of indeterminacy for $f_{\pi_1}$, and it is periodic for $h$.
We blow up $p_1$, obtaining another model $\pi_2 \colon X_{\pi_2} \to (\nC^2,0)$. Although $f_{\pi_2}$ has still indeterminacy points (corresponding to the tangent vectors at $\nu_{y,3}$ towards $\nu_{y+x^3}$ and $\nu_y$), it acts regularly on $E_1$, and the model is algebraically stable.
\reffig{dynAalg} resumes the situation, where black dots are regular points, and white dots are indeterminacy points.

\begin{figure}[ht]
\centering
\def\svgwidth{0.20\columnwidth}
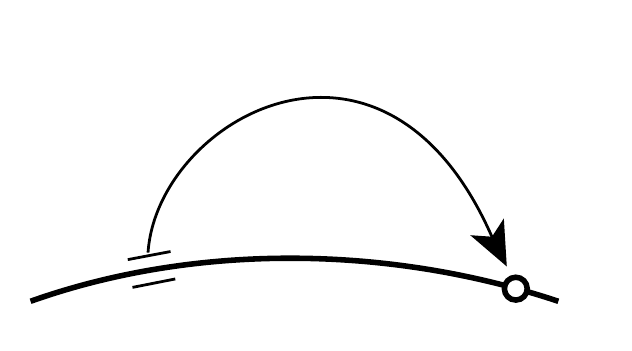
\hspace{5mm}
\def\svgwidth{0.35\columnwidth}
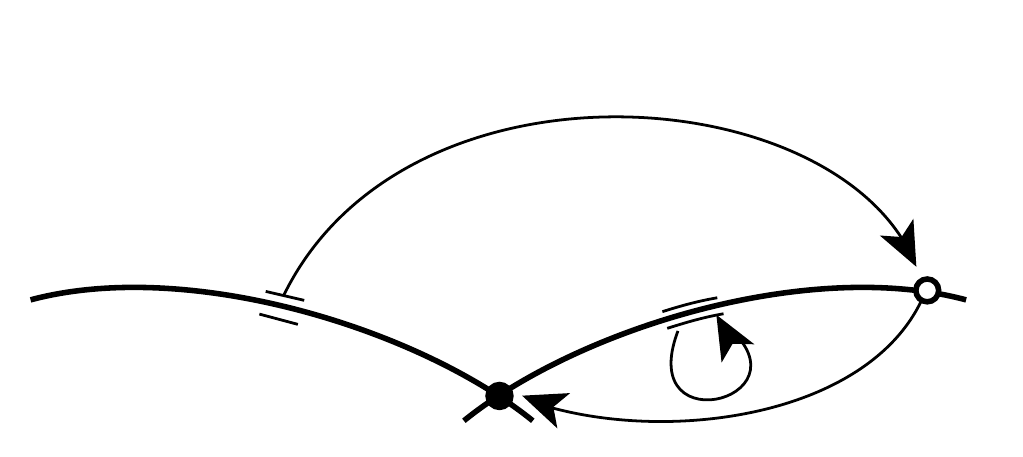
\hspace{5mm}
\def\svgwidth{0.35\columnwidth}
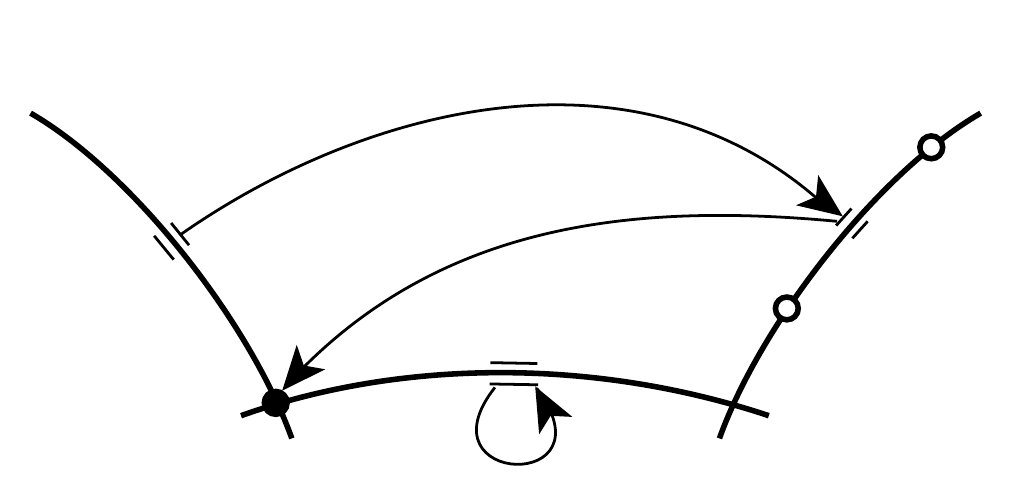
\caption{Construction of an algebraically stable model (\S\ref{ssec:example_smooth_finite}).}\label{fig:dynAalg}
\end{figure}

\subsection{A non-finite map at a smooth point}\label{ssec:example_smooth_nonfinite}

We now study the case $k=1$ of the previous section.
In this case, $\VC{f}=\{\nu_x\}$, and $f_\bullet^{-1}(\ord_0)=\{\nu_{y+x^3,4}\}$, so $\mc{T}_{c(f)} = [\nu_x,\nu_{y+x^3,4}]$.
Moreover,
$$
f_\bullet \nu_{x,t} = \nu_{y,\frac{2t+1}{t+1}},
\quad
f_\bullet \nu_{y,t} =
\begin{cases}
\nu_{y,\frac{t+2}{t+1}} & \text{if } t \in [1,3],\\
\nu_{y,\frac{t+2}{4}} & \text{if } t \in [3,+\infty],
\end{cases}
\quad
f_\bullet \nu_{y+x^3,t} =
\begin{cases}
\nu_{y,\frac{5}{t+1}} & \text{if } t \in [3,4],\\
\nu_{x,\frac{t+1}{5}} & \text{if } t \in [4,+\infty].
\end{cases}
$$
$$
c(f,\nu_{x,t}) = (t+1),
\quad
c(f,\nu_{y,t}) =
\begin{cases}
(t+1) & \text{if } t \in [1,3],\\
4 & \text{if } t \in [3,+\infty],
\end{cases}
\quad
c(f,\nu_{y+x^3,t}) =
\begin{cases}
(t+1) & \text{if } t \in [3,4],\\
5 & \text{if } t \in [4,+\infty].
\end{cases}
$$
In this case the unique eigenvaluation given by \refthm{unique_eigenval} is the irrational valuation $\nu_\star = \nu_{y,\sqrt{2}}$, and $c_\infty(f)=c(f,\nu_\star) = 1+\sqrt{2}$.
Notice also that $f_\bullet \nu_x = \nu_{y,2}$.
Consider the sequence of attraction rates $c_n=c(f^n,\ord_0)$. Set $t_n = \alpha(f_\bullet^n \ord_0)$. Then for any $n \in \nN$ we get
\begin{align*}
c_{n+2}&=c(f^{n+2},\ord_0)=c(f^n, \ord_0) c(f,f_\bullet^n \ord_0) c(f,f_\bullet^{n+1} \ord_0) \\
&= c_n (t_n+1)(t_{n+1}+1) = c_n (t_n+1)\left(\frac{t_n+2}{t_n+1}+1\right) = c_n (2t_n + 3) = 2c_n(t_n+1) + c_n\\
&= 2c_{n+1} + c_n.
\end{align*}

\begin{rmk}\label{rmk:multiplicity_nonfinite}
As in the previous example, we can compute the multiplicity of $f$ at any valuation on $[\nu_x, \nu_y] \cup [\nu_{y,3}, \nu_{y+x^3}]$, and we get
$$
m(f,\nu)=
\begin{cases}
5 & \text{if } \nu \in [\nu_{y,3}, \nu_{y+x^3}],\\
4 & \text{if } \nu \in ]\nu_{y,3}, \nu_{y}],\\
1 & \text{if } \nu \in ]\nu_{y,3}, \nu_{x}].
\end{cases}
$$
In this case, the map $M(f, \nu')=\displaystyle\sum_{f_\bullet \nu = \nu'}m(f,\nu)$ is locally constant on $\mc{V} \setminus \{\nu_{y,2}\}$.
In fact, it is equal to $5=e(f)$ in the connected component containing $\nu_{y,t}$ with $t<2$ (and on $\nu_{y,2}$ itself), while it is equal to $4$ on the other connected components.

If we focus on $\nu_E=\nu_{y,4}$, by direct computation we get $k_E=2$ and $e_E=2$.
One can compute in this case that on all branches at $\nu_E$ (but the ones associated to the directions towards $\nu_x$ and $\nu_y$), the multiplicity $m(f,\nu)=2$ for $\nu$ close enough to $\nu_E$.
\end{rmk}

We summarize in \reffig{valBdyn} the dynamics of $f_\bullet$. Here the square dot denotes irrational valuations.

\begin{figure}[ht]
\centering
\def\svgwidth{0.90\columnwidth}
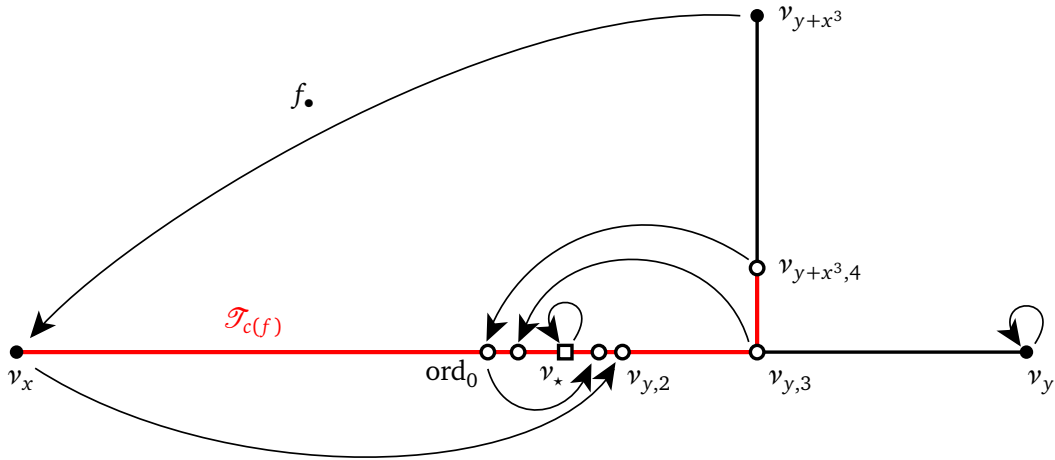
\caption{Action of $f_\bullet$ (\S\ref{ssec:example_smooth_nonfinite}).}
\label{fig:valBdyn}
\end{figure}

As in the previous example, the model $\pi_0$ is not geometrically stable for $f$, since again $f_{\pi_0}(E_0) = p_0 \in \on{Ind}(f_{\pi_0})$.
Notice that we also have $f_{\pi_0}(C_x)=p_0$, where $C_x$ is the strict transform of the curve $\{x=0\}$ to $X_{\pi_0}$.
In this case, it suffices to blow-up $p_0$, to get a geometrically stable model $\pi_1$. 
We get a new exceptional prime $E_1$, intersecting the strict transform of $E_0$ at a point $p_1$. We easily notice that $p_1$ is a fixed point for $f_{\pi_1}$, and the orbit of any exceptional prime through $f_{\pi_1}$ will eventually go to $p_1$.
Notice also that since $f_\bullet \nu_{x,t} \to \nu_{y,2}$ when $t \to +\infty$, from the tangent direction at $\nu_{y_2}$ associated to $p_1$, then $f_{\pi_1} (C_x,q) \to (E_1,p_1)$. 
See \reffig{dynBalg} for a description of the actions of $f_{\pi_0}$ and $f_{\pi_1}$.
\begin{figure}[h]
\centering
\def\svgwidth{0.35\columnwidth}
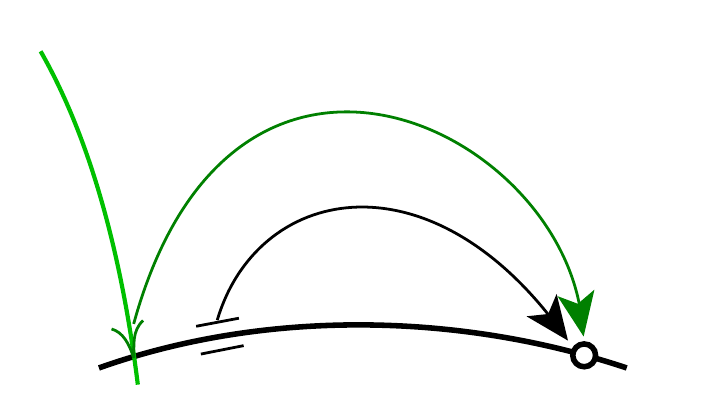
\hspace{5mm}
\def\svgwidth{0.55\columnwidth}
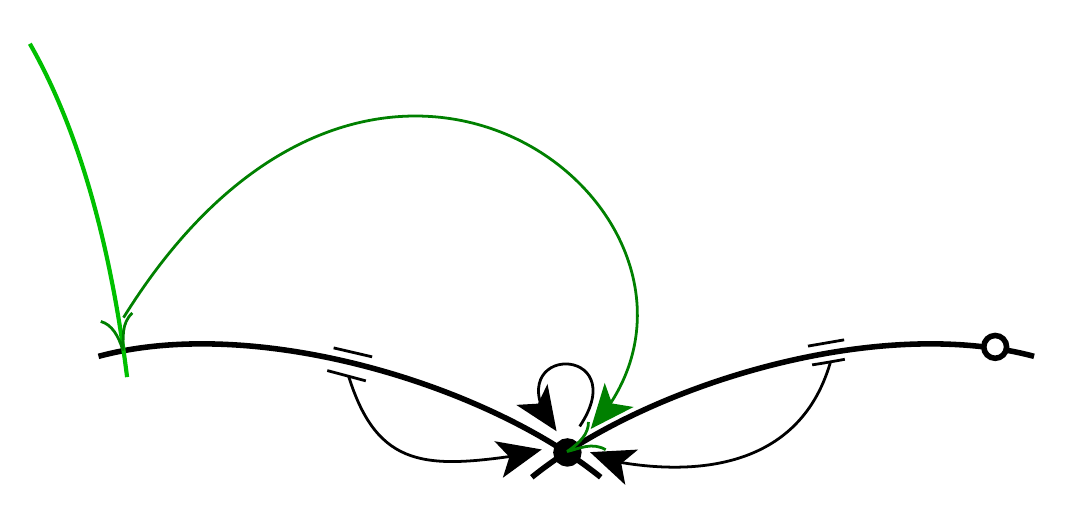
\caption{Construction of a geometrically stable model (\S\ref{ssec:example_smooth_nonfinite}).}
\label{fig:dynBalg}
\end{figure}

We refer to \cite{favre-jonsson:eigenval,gignac-ruggiero:attractionrates} for further examples on the smooth case.

\subsection{A quotient singularity}\label{ssec:example_quotient1}

Let $(X,x_0)$ be the singularity defined by $X=\{(x,y,z) \in \nC^3\ |\ xy=z^2\}$ at the origin $x_0=(0,0,0)$.
This is the simplest cyclic quotient singularity, usually denoted by $A_2$.
Its minimal good resolution, which is also a log resolution of the maximal ideal, has a unique rational exceptional prime $E_0$, and self-intersection $E_0^2=-2$.
This singularity can be obtained by quotient of $(\nC^2,0)$ by the action of $(u,v) \mapsto (-u,-v)$. Denote by $pr\colon(\nC^2,0) \to (X,x_0)$ the natural projection, given by $pr(u,v)=(u^2,v^2,uv)$.

Consider the map $f(x,y,z)=(x^2z,y^2z,xyz)$.
The maps $f$ leaves $X$ invariant, so it defines a germ at $(X,x_0)$.
Moreover, the curves $C_x = \{x=z=0\}$ and $C_y = \{y=z=0\}$ are contracted by $f$, which is hence non-finite.

Notice that $f \circ pr(u,v)=(u^5 v,u v^5, u^3v^3)$.
In particular there does not exist a holomorphic map $g\colon(\nC^2,0) \to (\nC^2,0)$ so that $f \circ pr=pr \circ g$.

An embedded resolution of $(X,x_0)$ is obtained by a single blow-up $\pi_0$ of the origin.
Pick coordinates so that $\pi_0(x_0,y_0,z_0)=(x_0,x_0y_0,x_0z_0)$. Then $X_{\pi_0}=\{y_0=z_0^2\}$, and $(x_0,z_0)$ are local parameters of $X_{\pi_0}$.
By direct computation, we find that $E_0=\{x_0=y_0-z_0^2=0\}$ is invariant by the action of $f_{\pi_0}$, and hence $\nu_{E_0}$ is the unique eigenvaluation given by \refthm{unique_eigenval}.
The map $f_{\pi_0}$ restricted to $E_0$ corresponds to the map $z_0 \mapsto z_0^2$. In particular, $\ee{E_0}{E_0} = 2$, while we can compute $\kk{E_0}{E_0}=c(f,\nu_{E_0})=c_\infty(f)=3$.
For any $t \in [0,+\infty]$ and $z_0 \in \nC$, denote by $\mu_{z_0,t}$ the monomial valuation at the point $(0,z_0^2,z_0) \in E_0$ with weights $1$ and $t$ with respect to the parameters $(x_0,z_0)$.
Notice that the valuation $\nu_{z_0,t} = (\pi_0)_* \mu_{z_0,t}$ is normalized.
By interchanging the coordinates $x$ and $y$, we recover the last point in $E_0$, associated to the parameter $z_0=\infty \in \nP^1$.
By direct computation, we get
$$
f_* \nu_{0,t} = (3+t)\nu_{0,\frac{2t}{3+t}},
\quad
f_* \nu_{z_0,t} = 3 \nu_{z_0^2,\frac{t}{3}},
\quad
f_* \nu_{\infty,t} = (3+t)\nu_{\infty,\frac{2t}{3+t}}.
$$
In particular we get $f_\bullet \nu_{0,\infty}=\nu_{0,2}$, and analogously $f_\bullet \nu_{\infty,\infty}=\nu_{\infty,2}$.
The Jacobian divisor can be computed directly, and we obtain $R_f = 2C_x+2C_y$.

\subsection{A simple elliptic singularity}\label{ssec:example_elliptic}

Let $(X,x_0)$ be the singularity defined by $X=\{(x,y,z) \in \nC^3\ |\ x^2+y^3+z^6=0\}$ at $x_0=(0,0,0)$.
One can show that $(X,x_0)$ is a simple elliptic singularity: its minimal good resolution $\pi_0 \colon X_{\pi_0} \to (X,x_0)$ has a unique exceptional prime $E_0$ of genus $1$, and self-intersection $E_0^2=-1$.

One can check that $\pi_0$ is not a log resolution of the maximal ideal $\mf{m}$. To obtain it, we need to blow up a point $p_0 \in E_0$, obtaining a new exceptional prime $E_1$ of genus $0$, that intersect transversely the strict transform of $E_0$.
In this model, the self intersection of $E_1$ is $-1$, and the self intersection of the strict transform of $E_0$ is $-2$.
One can also check that $b_{E_0}=1$ and $b_{E_1}=2$. It follows that $A(\nu_{E_0})=0$, and $A(\nu_{E_1})=\frac{1}{2}$.




Consider the map $f\colon(X,x_0) \to (X,x_0)$ given by
$$
f(x,y,z)=\left(xy^3z^3,y^3z^2,yz^2\right).
$$
The map $f$ defines a non-finite germ at $x_0$.
In fact, $f$ contracts three irreducible curves: $C_y^+ = \{y=x-\ui z^3=0\}$, $C_y^- = \{y=x+\ui z^3=0\}$, and $C_z = \{z=x^2+y^3=0\}$.
One can check by direct computation that the unique eigenvaluation given by \refthm{unique_eigenval} is the divisorial valuation associated to $E_0$.
The action on the tangent space attached to $\nu_{E_0}$ is the identity.
Since the critical points of $f$ in $\nC^3$ are contained in $\{yz=0\}$, the jacobian divisor takes the form $R_f=a_y^+ C_y^+ + a_y^- C_y^- + a_z C_z$.
A direct computation shows that $a_y^+ = a_y^- = a_z = 0$. In particular $R_f$ is trivial even though $f$ is non-finite. 

\subsection{Quasihomogeneous singularities}\label{ssec:example_quasihom}

The previous example can be easily generalized to \emph{quasihomogeneous} singularities (also called \emph{weighted homogeneous} singularities).
Up to isomorphisms, they are constructed as follows.
Let $\omega=(\omega_1, \ldots, \omega_n)$ be a vector of positive integers (we may assume without common factors).
Consider a finite family of polynomials $P_j \in \nC[x_1, \ldots, x_n]$ which are homogeneous with respect to the weight $\omega$, i.e., there exists $d_j \in \nN^*$ so that $P_j(\lambda^{\omega_1}x_1, \ldots, \lambda^{\omega_n} x_n) = \lambda^{d_j} P_j(x_1, \ldots, x_n)$ for all $\lambda \in \nC$.
The common zero locus of this family of polynomials defines a quasihomogeneous singularity at $x_0=(0, \ldots, 0)$.
Quasihomogeneous singularities can be also described as finite quotients of \emph{cone singularities}.
They are obtained by contracting the zero section of a negative degree line bundle $L \to E$ over a compact curve $E$.
We refer to \cite{wagreich:structurequasihomogensing} for further details on quasihomogeneous singularities.

For example, take any $p,q \geq 2$ coprime, we consider the surface $X=\{(x,y,z) \in \nC^3\ |\ x^p+y^q+z^{pq}=0\}$ at its singular point $x_0=(0,0,0)$.
This is a quasihomogeneous singularity with respect to the weight $(q,p,1)$.
By direct computation, $(X,x_0)$ is a cone singularity given by a line bundle of degree $-1$ over a curve $E_0$ of genus $(p-1)(q-1)/2$.
The minimal good resolution of $(X,x_0)$ is not a log resolution of the maximal ideal. The latter is obtained by blowing-up $p-1$ times the intersection between the exceptional divisor and the strict transform of the curve $\{z=0\}$.

Come back to the general case of a quasihomogeneous singularity $(X,0) \subset (\nC^n,0)$ with respect to some weight $\omega=(\omega_1, \ldots, \omega_n)$ on the coordinates $x=(x_1, \ldots, x_n)$.
Let $\phi \in \mc{O}_{X,0}$ be any holomorphic function at $0 \in X$, and $\sigma \colon (X,0) \to (X,0)$ any automorphism (see \cite{muller:liegroupsanalyticalgebras,favre-ruggiero:normsurfsingcontrauto} for constructions of automorphisms over quasihomogeneous singularities).
We may consider the germ $f=\phi^\omega \sigma$, defined in coordinates by
$$
f(x)=\big(\phi^{\omega_1} \sigma_1(x), \ldots, \phi^{\omega_n} \sigma_n(x)\big),
$$
where $(\sigma_1, \ldots, \sigma_n)$ are the coordinates of $\sigma$. Then $f$ defines an endomorphism of $(X,0)$, which is non-invertible as long as $\phi$ is not a unit.
Notice also that the curve $\{\phi=0\}$ is contracted to $0$, so $f$ is non-finite (when non-invertible).

In the cone case described above, the map $f$ leaves invariant the curve $E_0$ in the minimal good resolution, hence $\nu_{E_0}$ is the unique eigenvaluation given by \refthm{unique_eigenval}.

\subsection{A non-finite map on a cusp singularity}\label{ssec:example_cusp_322}

Consider the cusp singularity defined by $X=\{(x,y,z) \in \nC^3\ |\ x^2+y^3+z^9-xyz=0\}$ at $x_0=(0,0,0)$.
The exceptional divisor of its minimal good resolution $\pi_0\colon X_{\pi_0} \to (X,x_0)$ is a cycle of three rational curves $E_1, E_2, E_3$ of self-intersection $-2$, $-2$ and $-3$ respectively.
Moreover $b_{E_j}=1$ for $j=1,2,3$.
By direct computation, one can check that
\begin{align*}
Z_{\pi_0}(\nu_{E_1}) = -\left(\frac{5}{3}E_1+\frac{4}{3}E_2+E_3\right),\quad 
Z_{\pi_0}(\nu_{E_2}) = -\left(\frac{4}{3}E_1+\frac{5}{3}E_2+E_3\right),\quad
Z_{\pi_0}(\nu_{E_3}) = -(E_1+E_2+E_3).
\end{align*}
In particular by \refprop{order_effective}, $\nu_{E_3} \leq \nu_{E_i}$ for $i=1,2$, while $E_1$ and $E_2$ (and any two normalized monomial valuation at $E_1 \cap E_2$) are not comparable.

The minimal log resolution $\pi_1 \colon X_\pi \to (X,x_0)$ of $\mf{m}$ is obtained from $\pi_0$ by blowing up a free point in $E_3$. We denote by $E'_1, E'_2, E'_3$ the strict transforms of $E_1, E_2, E_3$, and by $E'_4$ the new exceptional prime.
One can check that $b_{E'_4}=2$, and $Z_{\pi_1}(\mf{m}_X)=\check{E}'_4=-(E'_1+E'_2+E'_3+2E'_4)$.

We now construct a non-finite germ, similar to the construction used in \refprop{nutomu}.

Set $p_0=E'_1 \cap E'_2$, $p_1=E'_1 \cap E'_3$, $p_2= E'_2 \cap E'_3$, and $p_3=E'_3 \cap E'_4$. For any $j=0, \ldots, 3$, we let $a_j < b_j$ so that $p_j=E'_{a_j} \cap E'_{b_j}$.
Computing directly $\pi_1$ as an embedded resolution of $(X,x_0) \subset (\nC^3,0)$, we may describe $\pi$ in coordinates at the points $p_0$, $p_1$, $p_2$, $p_3$.
For example, we may chose local coordinate $(x_0,y_0,z_0)$ at $p_0$, so that
$$
X_{\pi_1}=\{z_0(1+y_0^3)=x_0(y_0-x_0)\}, \quad \pi_1(x_0,y_0,z_0)=(x_0 z_0^4, y_0 z_0^3, z_0).
$$
$$
E_1 = \{z_0=x_0=0\}, \qquad E_2=\{z_0=y_0-x_0=0\},
$$
We now consider the map $f\colon (X,x_0) \to (X,x_0)$ obtained as the composition $f=\pi_1 \circ \sigma \circ g$, where $g\colon(X,x_0) \to (\nC^2,0)$ is the projection to the first two coordinates, $\sigma\colon (\nC^2,0) \to (X_{\pi_1},p_0)$ is the automorphism given by $\sigma(x,y)=(x,y,x(y-x)(1+y^3)^{-1})$ with respect to the coordinates $(x_0,y_0,z_0)$ at $p_0$.
In particular we can write $f$ in coordinates as
$$
f(x,y,z)=\left(\frac{x^5(y-x)^4}{(1+y^3)^4},\frac{x^3y(y-x)^3}{(1+y^3)^3},\frac{x(y-x)}{1+y^3}\right).
$$
Notice that $f$ is non-finite, since it contracts to $0$ the curve $X \cap \{x(y-x)=0\}$.

Denote by $\nu^j_{r,s} \in \hat{\mc{V}}_X$ be the monomial valuation at $p_j$ of weights $(r,s)$ with respect to coordinates adapted to $\pi_1^{-1}(x_0)$, so that $\mu_{1,0}^j = \ord_{E'_{a_j}}$ and $\mu_{0,1}^j = \ord_{E'_{b_j}}$.
Notice that $\nu^j_{r,s}$ is normalized as far as $r + s = 1$ when $j=0,1,2$, while $\nu^3_{r.s}$ is normalized as far as $r+2s=1$.
We first study the dynamics of $f_*$ on the circle $\hat{\mc{S}}_X$.
By direct computation, we get
$$
f_* \nu^0_{r,s} = \begin{cases}
                 \nu^0_{5r+4s, 4r+3s} & \text{if } r \leq s,\\
                 \nu^0_{5r+4s, 3r+4s} & \text{if } r \geq s,
                \end{cases}
\quad
f_* \nu^1_{r,s} = \nu^0_{5r+3s, 3r+2s}, 
\quad
f_* \nu^2_{r,s} = \nu^0_{4r+3s, 3r+2s}. 
$$
With respect to the monomial parameterization $w^j \colon [0,1] \to [\nu_{E_{a_j}}, \nu_{E_{b_j}}]_{p_j}$ given by $w(t)=\nu^j_{1-t,t}$, we get
$$
f_\bullet w^0(t)= \begin{cases}
                 w^0\left(\frac{4-t}{9-2t}\right) & \text{if } t \geq \frac{1}{2},\\
                 w^0\left(\frac{3+t}{8}\right) & \text{if } t \leq \frac{1}{2}.\\
                \end{cases}
\quad
f_\bullet w^1(t)= w^0\left(\frac{3-t}{8-3t}\right),
\quad
f_\bullet w^2(t)= w^0\left(\frac{3-t}{7-2t}\right).
$$

$$
c(f,w^0(t))= \begin{cases}
                 9-2t & \text{if } t \geq \frac{1}{2},\\
                 8 & \text{if } t \leq \frac{1}{2}.\\
                \end{cases}
\quad
c(f,w^1(t))= 8-3t,
\quad
c(f,w^2(t))= 7-2t.
$$

In particular, $f_\bullet \nu_{E_1} = w^0(3/8)$, $f_\bullet \nu_{E_2} = w(3/7)$, $f_\bullet \nu_{E_3}=w^0(2/5)$.
Notice also that $f_\bullet w^0(3/7)=w^0(3/7)=:\nu_\star$ is the unique eigenvaluation given by \refthm{unique_eigenval}, and $c_\infty(f)=c(f, \nu_\star)=8$.
Notice also that the M\"obius maps appearing in the expressions of $f_\bullet w^j(t)$ are all monotone, and $f_\bullet w^0(1/2)=w^0(7/16)$.
Similar computations show that the segment $[\nu_{E_3},\nu_{E'_4}]_{p_3}$ is sent to a segment $[w^0(2/5),\nu_{E''_5}]$, where $E''_5$ is the exceptional prime obtained by blowing up a suitable free point in the exceptional prime associated to $w^0(2/5)$. Finally, $c(f,w^3(t))=5-t$.
Notice that $c(f,\nu)$ is locally constant on a weak open neighborhood of $\nu_\star$. It follows that for any valuation $\nu_0 \in \mc{V}_X^\alpha$, the sequence $c_n=c(f^n,\nu_0)$ eventually satisfy the recursion relation $c_{n+1}=8c_n$. 

\begin{figure}[ht]
\centering
\def\svgwidth{0.80\columnwidth}
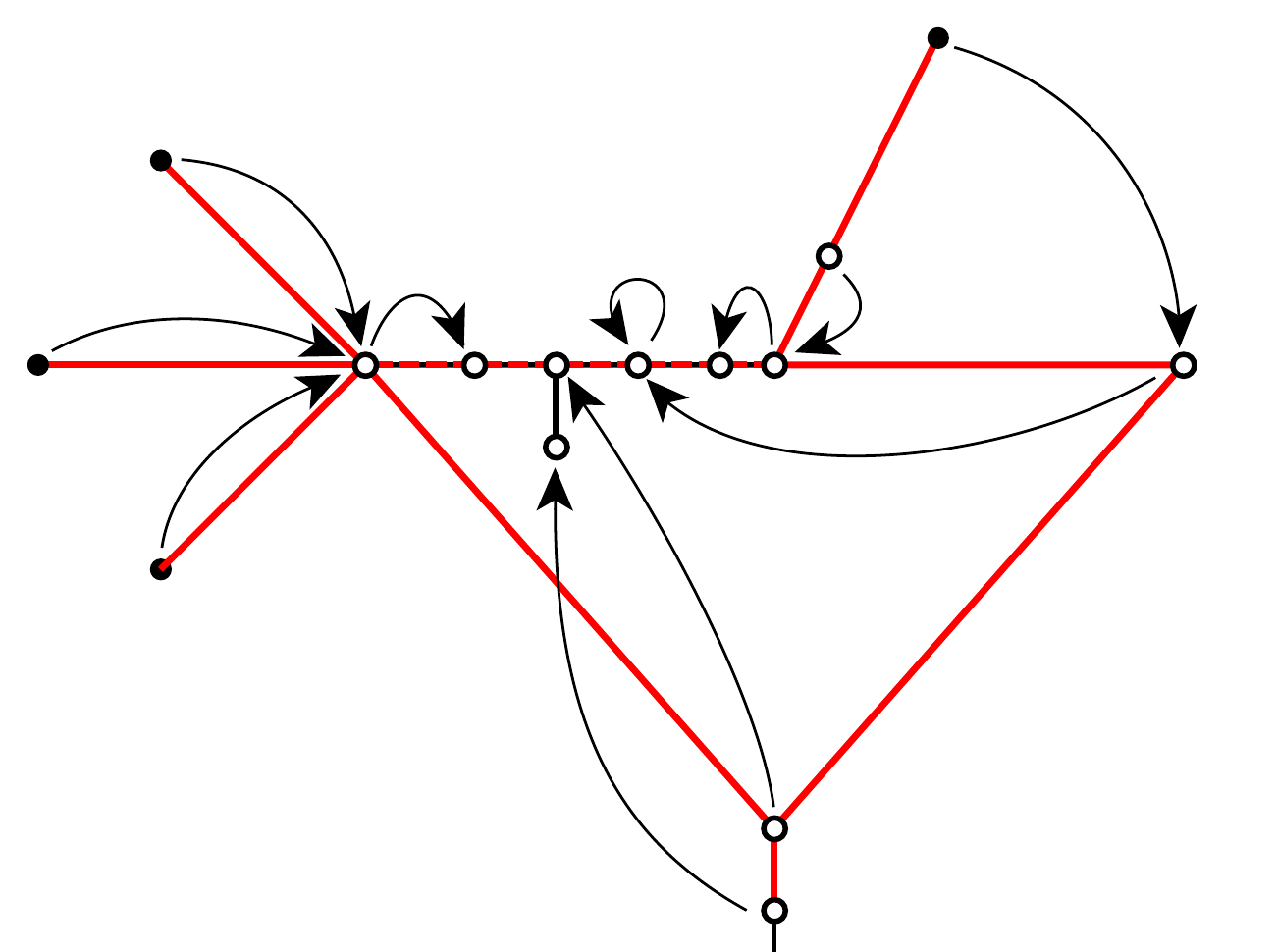
\caption{Action of $f_\bullet$ (\S\ref{ssec:example_cusp_322}).}
\label{fig:valFdyn}
\end{figure}

We now study the behavior of $f_\bullet$ towards contracted curve semivaluations in $\VC{f}$. It consists of $4$ curve semivaluations, associated to the irreducible curves $C_x^\eta=\{x=y-\eta z^3=0\}$ for any $\eta$ satisfying $\eta^3=-1$, and $C'=\{y-x=0\} \cap X$.
Notice that with respect to the coordinates $(x_0,y_0,z_0)$ at $p_0$, $C_x^\eta$ lifts to the curve $\tilde{C}_x^\eta=\{x_0=y_0-\eta = 0\}$. It intersects the exceptional prime $E_1$ transversely at the point $p_\eta=(0,\eta,0)$.
At $p_\eta \in X_{\pi_1}$ local parameters are given by $(y_0-\eta, z_0)$. The monomial valuation $\nu^\eta_t$ of weights $t$ and $1$ with respect to $y_0-\eta$ and $z_0$ is normalized for all $t \in [0,+\infty]$, and by direct computation we get
$$
f_* \nu^\eta_{t} = \nu^0_{t+5, 3} = (8+t) w^0\left(\frac{3}{8+t}\right).
$$
In particular $f_\bullet \nu_{C_x^\eta} = \nu_{E_1}$.

We turn our attention to the curve $C'$, which lifts to the curve $\{y_0=x_0z_0\} \cap X_{\pi_1}$. In particular, it intersects $\pi_1^{-1}(x_0)$ in $p_0$, transversely to $E'_1$ and $E'_2$. Denote by $\mu'_t$ the monomial valuation at $p_0$ of weights $1$ and $t$ with respect to the coordinates $x_0$ and $y_0-x_0z_0$. It is easy to check that
$$
\mu'_t=(1+t) w^0\left(\frac{t}{1+t}\right) \text{ for } t \leq 1, \qquad \mu'_t(\mf{m}) = 2 \text{ for }t \geq 1.
$$
Let $\nu'_t$ be the normalized valuation proportional to $\mu'_t$. By direct computations, we get
$$
f_* \mu'_t
=
\begin{cases}
\nu^0_{5+4t, 3+4t} & \text{if } t \leq 1,\\
\nu^0_{9, 6+t} & \text{if } t \geq 1,
\end{cases}
\quad \text{and we deduce }\quad
f_* \nu'_t
=
\begin{cases}
8w^0\left(\frac{3+4t}{8(1+t)}\right) & \text{if } t \leq 1,\\
\frac{15+t}{2} w^0\left(\frac{6+t}{15+t}\right) & \text{if } t \geq 1.
\end{cases}
$$
Notice in particular that $f_\bullet\nu'_3 = w^0(1/2)$.
\reffig{valFdyn} describes the action of $f_\bullet$ on $\mc{V}_X$, while \reffig{dynFalg} describes the action induced by $f$ on the minimal resolution $X_{\pi_0}$, the minimal log resolution $X_{\pi_1}$ of $\mf{m}$, and the minimal log resolution $X_{\pi_2}$ of $\mf{m}$ which also is an embedded resolution of $f^{-1}(x_0)$. Notice that all three models are geometrically stable in this case.
\begin{figure}[ht]
\centering
\begin{minipage}{0.3 \columnwidth}
\def\svgwidth{\columnwidth}
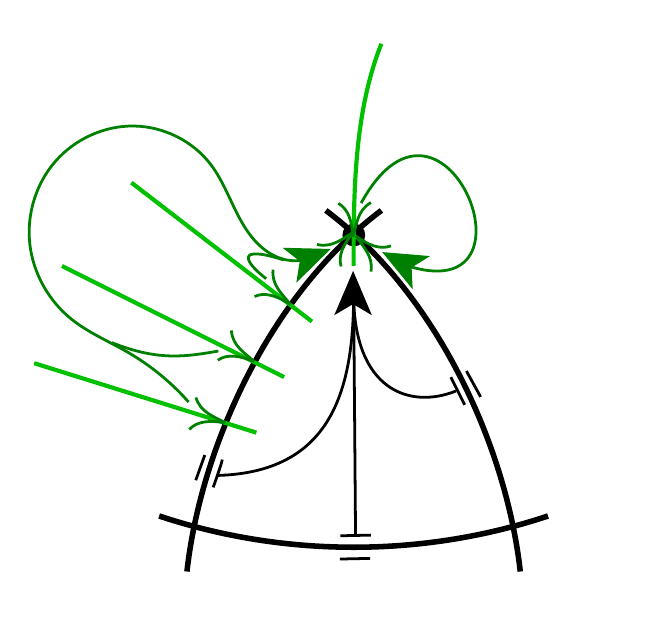
\end{minipage}
\hspace{3mm}
\begin{minipage}{0.3 \columnwidth}
\def\svgwidth{\columnwidth}
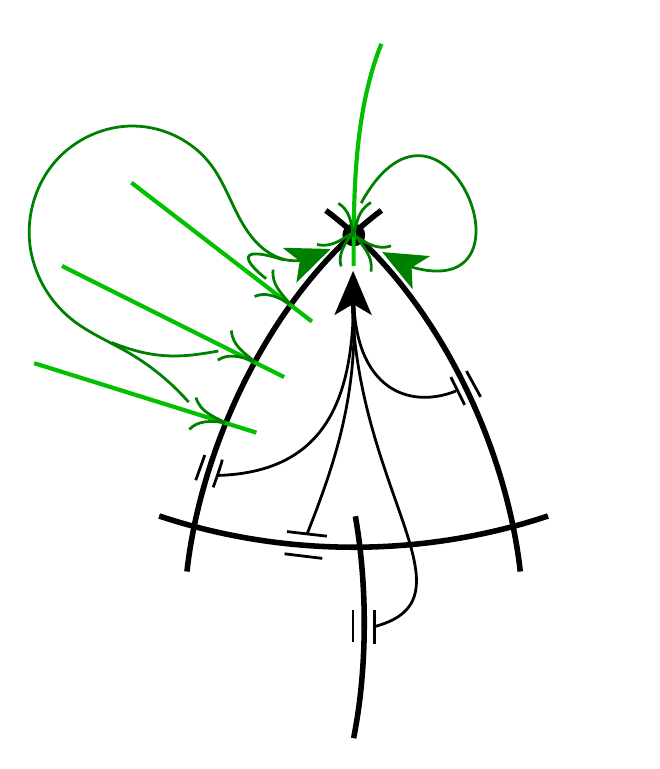
\end{minipage}
\hspace{3mm}
\begin{minipage}{0.34 \columnwidth}
\def\svgwidth{\columnwidth}
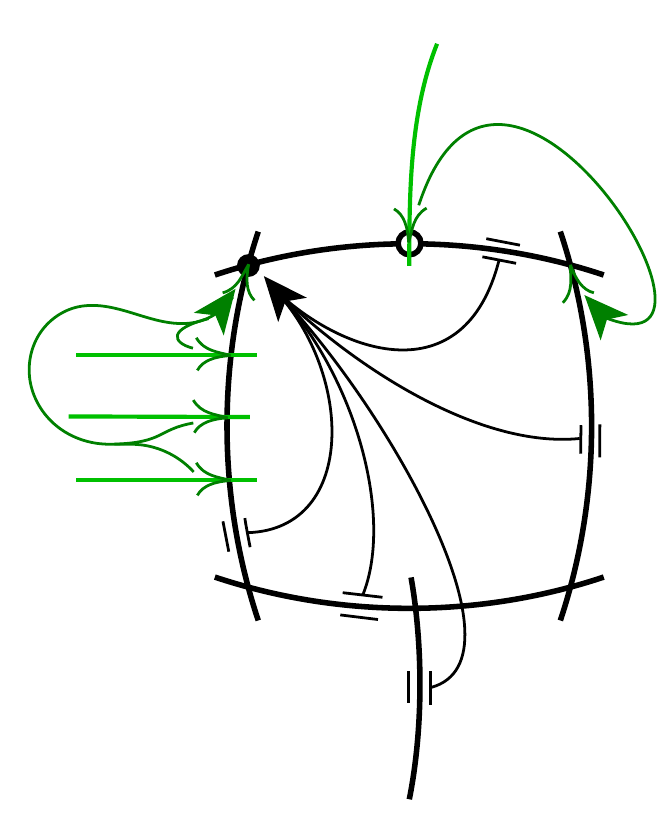
\end{minipage}
\caption{Examples of actions induced by $f$ on modified spaces (\S\ref{ssec:example_cusp_322}).}\label{fig:dynFalg}
\end{figure}

\subsection{A finite map on a cusp singularity}\label{ssec:example_cusp_42}

This example was proposed in \cite{favre:holoselfmapssingratsurf}.
Let $(X,x_0)$ be the cusp singularity whose minimal desingularization is a cycle of two rational curves, of self-intersections $-4$ and $-2$.
It can be obtained as the hypersurface singularity $X=\{(x,y,z) \in \nC^3\ |\ x^2+y^4+z^6=xyz\}$.
Here we consider the toric construction of this cusp, as described in \S\ref{ssec:arith_cusp}.
Set $\omega = [[\ol{4,2}]]$ as in \eqref{eqn:continued_fraction}. By direct computation, $\omega = 2+\sqrt{2}$.
We follow the notations of \S\ref{ssec:arith_cusp}, and set $N_\omega = \nZ \oplus \omega \nZ$.
The ring $\mf{o}^>$ of positive units is generated by $1+\sqrt{2}$, while the ring $\mf{o}^+$ of totally positive units is generated by its square $3+2\sqrt{2}$.
Notice that the lattice $N_\omega$ is preserved, and $\eps_\omega=3+2\sqrt{2}$ generates $\mf{o}_\omega^+$.
The cusp $(X,x_0)$ is then constructed by choosing $\eps=\eps_\omega$ in the construction given in \S\ref{ssec:arith_cusp}.
Recall that by \refrmk{arith_valuations}, the vectors $(1,0)$ and $(0,1)$ in the basis $\{1,\omega\}$ correspond to the exceptional primes $E_0$ and $E_1$ in the cycle, whose selfintersections are $-4$ and $-2$ respectively.
Notice that a basis of a lattice $N_\omega$ is given by $\{1,\sqrt{2}\}$. With respect to this basis, $E_0$ and $E_1$ correspond to the vectors $e_0=(1,0)$ and $e_1=(2,1)$.
The exceptional primes $E_n, n \in \nZ$ that appear in the universal covering $X_\infty$ of $(X,x_0)$ are associated to $e_n$, constructed inductively so that $e_{n+2}=Ae_n$ for all $n$, where $A$ is the matrix $A=\begin{pmatrix}3&4\\2&3\end{pmatrix}$ representing the multiplication by $\eps$.

Consider now the matrix $B=\begin{pmatrix}3&2\\1&3\end{pmatrix}$, representing the multiplication by $\alpha = 3+\sqrt{2}$.
The action of $B$ induces a finite endomorphism $f\colon (X,x_0) \to (X,x_0)$ on the cusp, of topological degree $e(f)=7=\det B$ and first dynamical degree $c_\infty(f)=\sqrt{7}$.
Using the correspondence described in \refrmk{arith_valuations}, we can study the action induced by $f$ on the cycle.

Any given vector $v \in \mc{C}$ belongs to at least one (and at most two) face of the fan defining $X_\infty$, generated by some $(e_n, e_{n+1})$.
So we can write $v=r e_n + s e_{n+1}$ for suitable $r,s \geq 0$ (not both $0$).
We may do the same for $w=Bv$, which will belong to some other face generated by $(e_m, e_{m+1})$. We set $w=r'e_m + r'e_{m+1}$.

We denote by $\nu = \nu_{r,s}$ the monomial valuation at $E_n \cap E_{n+1}$ corresponding to $v=re_n+se_{n+1}$ (modulo the action of $g_\eps$).
The normalization with respect to the maximal ideal corresponds here to $\nu(\mf{m})=r+s$. 
We denote by $\nu_t$ the normalized valuation corresponding to $(1-t)e_0 + t e_1$, and by $\mu_t$ the monomial valuation corresponding to $(1-t)e_1 + t e_2$.
Notice that $\nu_1=\mu_0$ and $\nu_0 = \mu_1$.
In fact, $t \mapsto \nu_t$ and $t \mapsto \mu_t$ are the monomial parameterizations of $[\nu_{E_0}, \nu_{E_1}]_{p_0}$ and $[\nu_{E_1}, \nu_{E_0}]_{p_1}$ respectively, where $p_0$ and $p_1$ are the two points of intersections between the two exceptional primes $E_0$ and $E_1$ in the minimal resolution of $(X,x_0)$.
By direct computation, we get:
\begin{align*}
&c(f, \nu_t)= t+2, 
&f_\bullet \nu_t =
\begin{cases}     
\nu_{\frac{1+4t}{2+t}} &\text{if } 0 \leq t \leq \frac{1}{3},\\
\mu_{\frac{3t-1}{2+t}} &\text{if } \frac{1}{3} \leq t \leq 1,
\end{cases}
\\
&c(f, \mu_t)=
\begin{cases}     
t+3 &\text{if } 0 \leq t \leq \frac{1}{2},\\
5-3t &\text{if } \frac{1}{2} \leq t \leq 1,
\end{cases}
&f_\bullet \mu_t =
\begin{cases}     
\mu_{\frac{2+3t}{3+t}} &\text{if } 0 \leq t \leq \frac{1}{2},\\
\nu_{\frac{2t-1}{5-3t}} &\text{if } \frac{1}{2} \leq t \leq 1.
\end{cases}
\end{align*}

The parameterization $\gamma$ of rays of $\mc{C}$ induces a parameterization $\gamma_X\colon \nR/\nZ \to \mc{S}_X$ (see \S\ref{ssec:arith_finitegerm}).
Explicitly, $\nu_t$ correspond to the vector $(1-t)e_0+ te_1$, which correspond to the real $a(t)=(1-t)+t(2+\sqrt{2})$. Notice that $Q(a(t))=a(t)a'(t)=(1+t)^2-2t^2$, where $a'$ denotes the extension of the conjugation in $K=\nQ(\sqrt{2})$ to $\nR$. Set $q=\sqrt{Q}$. Then $\nu_t=\gamma_X(u(t))$, where
$$
u(t)=\frac{\log(a(t))-\log(q(a(t)))}{2 \log \eps}.
$$
Analogous computations can be done for $\mu_t$.
With respect to this parameterization, $f_\bullet$ acts as an irrational rotation, by the value $\beta \in \nR/\nZ$ given by
$$
\beta=\frac{\log(\alpha)-\log(q(\alpha))}{2\log \eps} = \frac{\log(3+\sqrt{2})-\log(\sqrt{7})}{2\log (3+2\sqrt{2})}.
$$

\subsection{Different normalizations}\label{ssec:different_normalizations}

In the whole paper, to study the dynamics induced by a holomorphic germ $f\colon (X,x_0) \to (X,x_0)$ on the set of (finite) semivaluations $\hat{\mc{V}}_X^*$, we fixed a normalization, by considering the value they take on the maximal ideal $\mf{m}$.
While being a quite natural choice, this is only one of the possible normalizations.

The normalization with respect to the maximal ideal can be reformulated as follows.
For any valuation $\nu \in \hat{\mc{V}}_X^*$, we consider its ``norm''
\begin{equation}\label{eqn:norm_valuation}
\norm{\nu}_Z = - Z(\nu) \cdot Z, 
\end{equation}
where $Z=Z(\mf{m})$ is the b-divisor associated to $\mf{m}$.
In this case, the intersection \eqref{eqn:norm_valuation} is always well defined and finite, since $Z$ is a Cartier b-divisor, and this intersection coincides with $-Z_{\pi}(\nu) \cdot Z_\pi(\mf{m}_X)$, where $\pi$ is a log-resolution of $\mf{m}$.
But \eqref{eqn:norm_valuation} makes sense for any \emph{nef} b-divisor $Z$, for which we get $-Z(\nu)\cdot Z \in (0,+\infty]$ for all $\nu \in \hat{\mc{V}}_X^*$.
As meaningful examples, we can take $Z=Z(\mf{a})$ for any ($\mf{m}$-primary) ideal $\mf{a}$, or $Z=Z(\mu)$ for any semivaluation $\mu \in \hat{\mc{V}}_X^*$.

Notice that if $Z$ is not Cartier, the value $-Z(\nu) \cdot Z$ could be $+\infty$.
For example, consider the valuative space $\hat{\mc{V}}^*$ associated to a smooth point $(\nC^2,0)$, and its normalization with respect to the value at some coordinate $x$. In this case $\inte_x(x)=+\infty$, and the normalization is not well defined here. Nevertheless the normalization on quasimonomial valuations can be extended by continuity, and in this case it would give the monomial valuation of weights $(1,0)$ with respect to the coordinates $(x,y)$.
See \cite{favre-jonsson:valtree} for an extensive literature on different normalizations on the valuative tree.

In fact, one should think of $\mc{V}_X$ as the quotient of $\hat{\mc{V}}_X^*$ by the natural multiplicative action of $\nR_+^*=(0,+\infty)$.
The normalized space $\mc{V}_X[Z] = \{\nu \in \hat{\mc{V}}_X^*\ |\ \norm{\nu}_Z=1\}$ would then correspond to a section of the natural projection $\hat{\mc{V}}_X^*\to \mc{V}_X$.

All results obtained in this paper can be reformulated more generally for any choice of normalizations as above.
As we have seen, the skewness $\alpha$, the thinness $A$, and the attraction rate $c(f,-)$ depend on the normalizations chosen. Nevertheless, one can easily pass from one normalization to the other by homogeneity.
Notice finally that the angular distance does not depend on the normalization, and hence it defines a distance on the intrinsic space $\mc{V}_X$ more than on the normalized space $\mc{V}_X[\mf{m}]$.

\subsection{Automorphisms}

The valuative analysis described for non-invertible germs can be worked out for invertible germs as well.
For an automorphism $f\colon (X,x_0) \to (X,x_0)$, one can easily check that $f_\bullet = f_*$ acts as an isomorphism with respect to the angular distance, and preserves log discrepancies.

In particular, for smooth points $(X,x_0)=(\nC^2,0)$ this implies that $\ord_0$ is a totally invariant valuation.
In general the sets $\on{Fix}(f_*)$ and $\on{Per}(f_*)$ of fixed and periodic valuations is quite huge, and contains enough information to allow to distinguish different conjugacy classes.
As an example, for any $\lambda \neq 0$ and $u \in \nN^*$, consider
\begin{equation}\label{eqn:ex_resonant}
f(x,y)=(\lambda x, \lambda^u y + \eps x^u).
\end{equation}
When $\abs{\lambda} < 1$, this is the Poincar\'e-Dulac normal form of contracting invertible automorphisms which are resonant, see \cite{sternberg:localcontractions,rosay-rudin:holomorphicmaps,berteloot:methodeschangementechelles}.
Denote by $\nu_{y,t}$ the monomial valuation of weights $1$ and $t$ with respect to the coordinates $(x,y)$.
It is easy to check that $f_*\nu_{y,u}=\nu_{y,u}$, and the action of $f$ on the tangent space is given in suitable coordinates by $\zeta \mapsto \zeta +\eps \lambda^{-u}$. By further computations, one can notice that $\on{Fix}(f_*)=[\nu_x,\nu_{y,u}]$ as far as $\eps \neq 0$, while $\on{Fix}(f_*)$ contains all segments $[\ord_0, \nu_{y+\theta x^u}]$ for all $\theta \in \nC$ when $\eps = 0$.

For singular spaces, studying finite order automorphisms $g\colon (X,x_0) \to (X,x_0)$ (for which $\on{Fix}(g)=\{x_0\}$) and their action on the valuative space allows to relate the valuative spaces of a singularity and of its quotient by the action of $g$.
We have seen examples of this phenomenon in \S\ref{ssec:dynamics_quotientsing}.
%
%

\subsection{Positive characteristic}\label{ssec:positive_char}

All along the paper, we considered maps and singularities over the complex numbers. Here we comment on the more general situation where we replace $\nC$ by any field $\nK$.
First, remark that we can always work on the algebraic closure of a given field. We will hence assume that $\nK$ is algebraically closed.
As commented in the introduction, all of the results in this paper hold over any field $\nK$ \emph{of zero characteristic}, as we never use the archimedean structure of the complex numbers.

Consider now an algebraically closed field $\nK$ of characteristic $p > 0$.
The contents of \S\ref{sec:valuation_spaces} still hold in this setting, as well the contents of \S\ref{sec:dynamics_valspaces}, but for the Jacobian formula of \S\ref{ssec:Jacobian_formula}.
In particular, \refthm{weak_convergence}, and all its corollaries on the existence of geometrically stable models in \S\ref{sec:algebraic_stability} and recursion relations for the sequence of attraction rates in \S\ref{sec:attraction_rates}, still hold for \emph{non-finite} germs over fields $\nK$ of positive characteristic having a non-quasimonomial eigenvaluation.

While we believe that for non-finite germs \refthm{strong_convergence} should also hold in characteristic $p$, 
for finite germs new phenomena arise, due to the presence of \emph{wild ramifications}.
\begin{ex}\label{ex:Frobenius1}
Consider the map $F\colon (\nA^2_\nK,0)\to(\nA^2_\nK,0)$ given by
$$
F(x,y)=(x^p,y^p).
$$
Notice that the jacobian determinant $J_F$ of $F$ vanishes identically, even though the map $F$ is dominant.
It is easy to show that the set of eigenvaluations for $F$ is given by the set of valuations whose center in $X_\pi$ belongs to $X_\pi(\nF_p)$ for any modification $\pi \colon X_\pi \to (\nC^2,0)$.
Similarly, eigenvaluations for $F^k$ corresponds to valuations whose center belongs to $X_\pi(\nF_{p^k})$.
This is in sharp contrast with the situation described by \refthm{classification}.
This phenomenon is given by the fact that for any divisorial eigenvaluation $\nu_E$, the map $F$ acts on $E$ as the Frobenius map $\zeta \mapsto \zeta^p$.
\end{ex}

Notice that in the non-finite case, Frobenius maps may still appear, but not for the self-action at an exceptional prime associated to a (totally invariant) divisorial eigenvaluation, since it would contradict the uniqueness of the eigenvaluation established by \refthm{unique_eigenval}.
Ramification indices may increase with respect to the characteristic zero case, as the following example shows.


\begin{ex}\label{ex:Frobenius2}
Consider the map $f\colon (\nA^2_\nK,0)\to(\nA^2_\nK,0)$ given by
$$
f(x,y)=(x^p, x^{p-1}y).
$$
As in the previous example, the jacobian determinant $J_f$ vanishes identically.
In this case however, we still have a unique eigenvaluation $\ord_0$, and any other valuation $\nu \in \mc{V}^\alpha$ converges to $\ord_0$.
Denote by $\nu_{y,t}$ the monomial valuation at $0$ of weights $1$ and $t$ with respect to coordinates $(x,y)$.
We have $f_*\nu_{y,2}=p \nu_{y,1+1/p}$. If we denote by $E$ and $F$ the exceptional primes associated to $\nu_{y,2}$ and $\nu_{y,1+1/p}$ respectively, than $f$ induces the Frobenius map $\zeta \mapsto \zeta^p$.
For any $t \geq 2$ and $\zeta \in \nK^*$, denote by $\nu_{\zeta, t}$ the valuation in $[\ord_0, \nu_{y-\zeta x^2}]$ whose skewness is $t$.
Then $f_* \nu_{\zeta,t}$ is the valuations in $[\ord_0, \nu_{y^p - \zeta^p x^{p+1}}]$ whose skewness is $h(t)=(p-1 + t)/p$.
If $f$ is seen as an action on $(\nC^2,0)$, the same holds, but with $h_\nC(t)=((p+2)(p-1)+t)/p^2$ replacing $h(t)$.
\end{ex}

\refex{Frobenius1} can be easily generalized. For example, for any germ $\sigma\colon(\nA^2_\nK,0)\to (\nA^2_\nK,0)$ and any $k \in \nN^*$, we may consider maps of the form $f=\sigma \circ F^k$.
Notice that, unlike the case of dimension one, we cannot split any germ $f$ in such a way so that the Jacobian of $\sigma$ is not vanishing (see \refex{Frobenius2}).

\begin{ex}
Let $u \in \nN^*$, $\eps \in \nK$, and consider the map $f\colon (\nA^2_\nK,0)\to(\nA^2_\nK,0)$ given by
$$
f(x,y)=(x^p, y^p + \eps x^{pu}).
$$
This map is obtained as the composition $\sigma \circ F$, where $\sigma$ is of the form \eqref{eqn:ex_resonant}.
In this case, one can easily notice that $[\nu_x, \nu_{y,u}]$ is fixed by $f_\bullet$.
The action of $f$ on the exceptional prime associated to $\nu_{y,u}$ is given by $h \colon \zeta \mapsto \zeta^p + \eps$. Its properties depend on whether $\eps$ belongs to $\nF_p$ (or its algebraic closure $\ol{\nF_p}$) or not.
Notice that $\displaystyle h^n(\zeta)=\zeta^{p^n}+ \sum_{k=0}^{n-1} \eps^{p^k}$.
If $\eps=0$, we have again the Frobenius map described in \refex{Frobenius1}.
If $\eps \in \nF_p^*$, then $h^p$ is the $p$-th iterate of the Frobenius map.
Analogously, if $\eps \in \ol{\nF_p}$, then some iterate of $h$ and the Frobenius map coincide.
If $\eps \not \in \ol{\nF_p}$, then we have no periodic directions associated to points over $\ol{\nF_p}$.

In particular, the study of the structure of $\on{Fix}(f_\bullet)$ becomes quite intricate, and depends on the properties of the field $\nK$.
\end{ex}

Notice that even though the Jacobian formula \eqref{eqn:jacobian_formula} doesn't hold in positive characteristic, it remains as the inequality
$$
A(f_*\nu)\leq A(\nu) + \nu(R_f),
$$
even when the Jacobian determinant doesn't vanish identically.

\begin{ex}\label{ex:Frobenius4}
Let $u \in \nN^*$, $\eps \in \nK$, and consider the map $f\colon (\nA^2_\nK,0)\to(\nA^2_\nK,0)$ given by
$$
f(x,y)=(x^p(1+x),y^p(1+y)).
$$
In this case, the jacobian determinant $J_f=x^py^p$ does not vanish, but it has an higher multiplicity than expected.
In particular,
$$
2p=A(p\ord_0) = A(f_* \ord_0) < A(\ord_0)+\ord_0(J_f)= 2+2p.
$$
\end{ex}

The proof of \refthm{wahl} relies on the Jacobian formula, and hence it does not remain valid in positive characteristic.
In fact, any singularity defined over $\nF_{q}$ admits non-invertible finite self-maps.

\begin{ex}\label{ex:Frobenius5}
Let $(X,x_0)$ be a singularity over a field $\nK$ of characteristic $p > 0$.
Suppose that $(X,x_0)$ is embedded in an affine space $\nA^m_\nK$, and obtained as the vanishing locus of some polynomials (or more generally analytic functions) with coefficients in $\nF_{p^s}$ for some $s \in \nN^*$.
Consider the map $f\colon \nA^m_\nK \to \nA^m_\nK$ given by $(x_1, \ldots, x_m)=(x_1^p, \ldots, x_m^p)$.
Then $f^s$ induces a finite non-invertible self-map on $(X,x_0)$.
\end{ex}

The conclusions of \refthm{wahl} remain valid as far as the Jacobian formula holds, while they are clearly false when the Jacobian determinant identically vanishes. As showed by \refex{Frobenius4}, some mixed behavior may occur. We believe it could be interesting to investigate which classes of singularities admit non-invertible finite self-maps whose Jacobian determinant doesn't identically vanish. We refer to this case as maps admitting a Jacobian divisor, which can be constructed by considering the Jacobian determinant on $X \setminus \{x_0\}$, which is defined since $f$ is finite.

\begin{quest}
Let $(X,x_0)$ be a normal surface singularity over an algebraically closed field $\nK$ of positive characteristic. Assume there exists a \emph{finite} germ $f\colon (X,x_0) \to (X,x_0)$ which is non-invertible, and admits a Jacobian divisor. Does this give obstructions for the geometry of $(X,x_0)$ ?
\end{quest}

\appendix
\section{Cusp singularities}\label{sec:cusps}

\subsection{Arithmetic construction of cusp singularities}\label{ssec:arith_cusp}
We describe here the arithmetic construction of cusps singularities, see e.g. \cite[Section 4.1]{oda:convexbodies}.
Let $r \in \nN^*$ and $k_0, \ldots, k_{r-1}$ be a finite sequence of integer numbers $\geq 2$, not all equal to $2$.
Set $k_{nr+j} := k_j$ for all $n \in \nZ$, so that $(k_j)_{j \in \nZ}$ defines a periodic sequence of integers $\geq 2$ (not all equal to $2$), that we may assume of exact period $r$.
We denote by $\omega=[[k_0, k_1, \ldots]] = [[\overline{k_0, \ldots, k_{r-1}}]]$ the modified continued fraction
\begin{equation}\label{eqn:continued_fraction}
\omega = k_0 - \frac{1}{k_1 - \frac{1}{\ddots}}.
\end{equation}
Then $\omega$ is a quadratic irrational number, $\omega = a + b \sqrt{d}$, where $a,b \in \nQ_+^*$ are positive rational numbers and $d \in \nN$ is a positive and square-free integer. Moreover $\omega > 1 > \omega' > 0$, where $\omega'$ is the conjugate of $\omega$ in $\nQ(\sqrt{d})=:K$.
Set now
\begin{itemize}
\item $N=N_\omega := \nZ \oplus \omega \nZ$ the lattice of rank $2$ in $\nQ(\sqrt{d})$ generated by $1$ and $\omega$;
\item $\mf{o}^>$ the group of positive units of $K$, i.e., (algebraic) integers $u \in K$ which are invertible and satisfying $u>0$;
\item $\mf{o}^+$ the \emph{totally positive}, i.e., elements $u \in \mf{o}^>$ satisfying $u' > 0$;
\item $\mf{o}_\omega^>$ and $\mf{o}_\omega^+$ the group of positive and totally positive units $u$ satisfying $uN_\omega = N_\omega$.
\end{itemize}
By Dirichlet's unit theorem, $\mf{o}^>$ and $\mf{o}^+$ are infinite cyclic groups, and $\mf{o}^+$ has index either $1$ or $2$ in $\mf{o}^>$.
The same property holds for unities preserving the lattice $N_\omega$. Let $\eps_\omega$ be the generator of $\mf{o}_\omega^+$ satisfying $\eps_\omega > 1$.

The embedding $\Phi \colon K \to \nR^2$ given by $\Phi(\xi) = (\xi, \xi')$ induces a canonical isomorphism $K \otimes_\nQ \nR \to \nR^2$.
Set $e_0=(1,0)$, $e_1=(0,1)$, and recursively $e_n$ for all $n \in \nZ$ by imposing $e_{n-1} + e_{n+1} = k_n e_n$.
Let $L_\omega \colon \nQ^2 \to K$ be defined by $(x,y) \mapsto x+\omega y$, and set $\Phi_\omega = \Phi \circ L_\omega$.
One can check that $L_\omega(e_n)$ are totally positive, and hence for all $n \in \nZ$, we have $\Phi_\omega(e_n) \in (\nR_+^*)^2=:\mc{C}$.
Denote by $\Delta$ boundary of the convex hull of $\Phi(N_\omega) \cap \mc{C}$.
Then $\Phi(N_\omega) \cap \Delta=\{\Phi_\omega(e_n), n \in \nZ\}$.
The fan $\mc{F}$ consisting of the $0$-face ${(0,0)}$, the $1$-faces $\nR_+ \Phi_\omega(e_n)$ and the $2$-faces $\nR_+ \Phi_\omega(e_n) + \nR_+ \Phi_\omega(e_{n+1})$ for all $n \in \nZ$ is a regular fan of the open cone $\mc{C}_0=\{0\} \cup \mc{C}$.
We call $X_\infty$ the (non-compact) toric surface associated to $\mc{F}$.
It contains an infinite chain of compact rational curves $(E_n)_{n \in \nZ}$, with selfintersections $(E_n)^2 = -k_n$.

For any element $\eps \in \mf{o}_\omega^+$, consider the map $g_\eps \colon \mc{C}_0 \to \mc{C}_0$ defined by $g_\eps(x,y)=(\eps x, \eps' y)$.
Since $\eps N_\omega=N_\omega$, the map $g_\eps$ leaves $\Phi(N_\omega)$ invariant, and induces an automorphism $G_\eps \colon X_\infty \to X_\infty$ (the inverse given by $G_{\eps'}$) which (as far as $\eps \neq 1$), acts freely and properly discontinuously.
Hence the quotient $X_\eps=X_\infty / \langle G_\eps\rangle$ is a complex surface, which has a cycle of rational curves.
In fact, assume that $\eps=\eps_\omega^s$, $s \in \nN^*$.
Up to refining the fan $\mc{F}$ by blowing up points in $E_{n+rk} \cap E_{n+1+rk}$ for all $k \in \nZ$, we may assume that $rs \geq 2$.
Then the cycle of rational curves has $rs$ components, of self-intersections $-k_0, \ldots, -k_{rs-1}$, i.e., $s$ copies of the starting sequence $-k_0, \ldots, -k_{r-1}$.
We may contract this cycle in $X_\eps$, obtaining a cusp singularity $(X,x_0)$.

By a result of Laufer's \cite{laufer:taut2dimsing}, cusps singularities are taut, meaning that the dual graph, together with selfintersections of components, determine the analytical type of a cusp singularity.
In particular, all cusps singularities can be constructed as above, up to isomorphisms.

\begin{rmk}\label{rmk:arith_valuations}
Notice that $\Phi_\omega\colon \nQ^2 \to \nR^2$ is $\nQ$-linear, and it can be extended by continuity to a $\nR$-linear map $\Phi_\omega \colon \nR^2 \to \nR^2$.
Analogously, consider the quadratic form $Q(\alpha)=\alpha\alpha'$ on $K$.
The map $Q_\omega =L_\omega \circ Q \colon \nQ^2 \to \nR$ can be extended to a continuous map $Q_\omega\colon \nR^2 \to \nR$.
Notice that $Q_\omega(x,y) > 0$ as far as $\Phi_\omega(x,y) \in \mc{C}$.

If we denote by $\mc{S}_{X}$ the skeleton of $(X,x_0)$, and by $\hat{\mc{S}}_X^*$ the cylinder $\mc{S}_X\times \nR_+^*$, then there is a natural
isomorphism $\tau \colon \hat{\mc{S}}_X \to \mc{C}/\langle g_\eps\rangle=:\mc{C}_{\eps}$.
Denote by $[\cdot]$ the natural projection $\mc{C} \to \mc{C}_\eps$. Then $\tau$ has the following properties :
\begin{itemize}
\item For any exceptional prime $E_n$, $n=0, \ldots, rs-1$, $\tau(\ord_E) = [\Phi_\omega(e_n)]$.
\item For any valuation $\nu \in \hat{\mc{S}}_X^*$ and any $\lambda > 0$, we have $\tau(\lambda \nu) = \lambda \tau(\nu)$.
\item For any intersection point $p_n$ of two exceptional primes $E_n$ and $E_{n+1}$, the monomial valuation $\nu_{r,s}$ at $p_n$ satisfying $Z_{\pi_0}(\nu_{r,s})=r\check{E}_n+s \check{E}_{n+1}$ is sent by $\tau$ to $[\Phi_\omega(re_n + s e_{n+1})]$. 
\end{itemize}

Notice that two valuations $\nu, \mu \in \hat{\mc{S}}_X^*$ are proportional if and only if $\tau(\nu)$ and $\tau(\mu)$ belong to the same ray of $\mc{C}_0$.
A normalization of valuations in $\hat{\mc{S}}_X^*$ corresponds to taking a non-vanishing section of the set of rays of $\mc{C}_0$. We can take for example the section given by $Q \equiv 1$. Notice that since $Q(\eps)=1$, this section is $g_\eps$-invariant, and it induces a section of the rays in $\mc{C}_\eps$.
\end{rmk}

\subsection{Finite endomorphisms}\label{ssec:arith_finitegerm}

We now describe an arithmetic construction which produces finite endomorphisms of cusps $(X,x_0)$ constructed as above (see \cite[\S 2.5]{favre:holoselfmapssingratsurf}).

We denote by $\mf{c}_\omega^+$ the set of (totally) positive elements $\alpha \in K$ so that $\alpha N_\omega \subseteq N_\omega$.
For any $\alpha \in \mf{c}_\omega^+$, the map $g_\alpha \colon \mc{C}_0 \to \mc{C}_0$ induces a morphism $G_\alpha \colon X_\infty \to X_\infty$.
This map will have in general indeterminacy points, and it is not an automorphism as far as $\alpha$ is not a unit.
Since $g_\alpha$ commutes with $g_\eps$, $G_\alpha$ induces a map on the quotient $X_\eps=X_\infty/\langle G_\eps \rangle$.
By contracting the cycle of rational curves, we obtain a finite endomorphism $f_\alpha \colon (X,x_0) \to (X,x_0)$.

\begin{rmk}\label{rmk:Qinteger}
The fact that $\alpha N_\omega \subset N_\omega$ guarantees that $f_\alpha$ can be expressed as a formal (rational) map (with respect to suitable coordinates), while the fact that the cone $\mc{C}$ is invariant assures that $f_\alpha$ defines a holomorphic map at $x_0$. 
Notice also that $\alpha$ needs to be an integer in $K$. In fact, since $\alpha$ leaves $N_\omega$ invariant, we have in particular $\alpha^n \in N_\omega$ for all $n$. If $\alpha$ is not an integer, then $Q(\alpha^n)$ would have unbounded denominators, in contradiction with belonging to $N_\omega$.
Finally, we notice that $f_\alpha$ is finite, of topological degree $Q(\alpha)=\alpha \alpha' \in \nN^*$.
\end{rmk}

The map $\gamma(t)=(\eps^{2t},1)\cdot \nR_+$, where $t \in \nR$, gives a parameterization of the rays in $\mc{C}_0$ (see \cite[p. 413]{favre:holoselfmapssingratsurf}).
The action of $g_\eps$ on the rays of $\mc{C}_0$ corresponds to the translation by $1$ on $\nR$.
By \refrmk{arith_valuations} we deduce that the rays in $\mc{C}_0$, quotiented by the action of $g_\eps$, are in $1$-to-$1$ correspondence with the skeleton $\mc{S}_X$. In fact, the normalization of valuations, which corresponds to taking a section of $\hat{\mc{S}}_X^*$, here correspond to taking a non-vanishing section of the rays of $\mc{C}_0$ which is $g_\eps$-invariant.
We deduce that $\gamma$ induces a parameterization $\gamma_X \colon \nR/\nZ \to \mc{S}_X$.

The action of $g_\alpha$ corresponds, with respect to this parameterization, to a translation by the value
$$
\beta=\frac{\log \alpha - \log \alpha'}{2 \log \eps}.
$$

This construction of finite germs $f_\alpha$ on cusp singularities is quite general. In fact, for any finite germ $f$, there exists a $f_\alpha$ constructed as above so that the actions of $f_\bullet$ and $(f_\alpha)_\bullet$ coincide on the skeleton $\mc{S}_X$. To prove this result, we first need a lemma. 

\begin{lem}\label{lem:Haar_measures}
Let $(X,x_0)$ be a cusp singularity constructed as in \S\ref{ssec:arith_cusp}.
Then there exists a constant $K > 0$ so that, for all $t,u \in \nR/\nZ$, we have that
$$
\rho(\gamma_X(t),\gamma_X(u))=K d_\nZ(t,u),
$$
where $d_\nZ$ denotes the distance induced by the euclidean distance on $\nR/\nZ$.
\end{lem}
\begin{proof}
For any $\alpha \in N_\omega$, consider the finite germ $f_\alpha\colon (X,x_0) \to (X,x_0)$.
We just need to consider a $\alpha$ which gives an irrational rotation on $\mc{S}_X$, whose existence is guaranteed by \refprop{irrational_rotations}.
Since $f_\alpha$ acts as an isometry on $\mc{S}_X$, the measure induced by $\rho$ is the Haar measure of $\mc{S}_X$.
The action induced by $f_\alpha$ on $\nR/\nZ$ is a translation by an irrational number $\beta$, and $d_\nZ$ is the Haar measure of $\nR/\nZ$.
We conclude by uniqueness of the Haar measure up to multiplicative constants.
\end{proof}

\begin{prop}\label{prop:arith_action_on_cycle}
Let $f\colon(X,x_0) \to (X,x_0)$ be a finite germ on a cusp singularity $(X,x_0)$ constructed as in \S\ref{ssec:arith_cusp}.
Suppose that $f_\bullet$ preserves the orientation of the circle $\mc{S}_X$ (it may be always assumed by replacing $f$ by its second iterate).
Then there exists $\alpha \in N_\omega$ so that the action of $f_*$ on $\hat{\mc{S}}_X^*$ corresponds to the action of $g_\alpha$ on $\mc{C}$.
\end{prop}
\begin{proof}
By \reflem{Haar_measures}, the action of $f_\bullet$ corresponds, through the parameterization given by $\gamma_X$, to the translation by some $\beta \in \nR$. Let $\alpha \in \nR$ be so that
$
\eps^{2\beta}=\alpha^2/Q(\alpha).
$
Notice that this defines $\alpha$ up to positive multiplicative constants.

By \reflem{action_monomialweights}, we know that $f_*$ acts on $\mc{C}$ as a piecewise linear map with integer coefficients. This implies that, up to rescaling by a suitable multiplicative constant, $f_*$ acts as the linear map $g_\alpha$.
Since $f$ is holomorphic at $x_0$, we must have $g_\alpha N_\omega \subseteq N_\omega$, and $g_\alpha \mc{C} \subset \mc{C}$, i.e., $\alpha$ is totally positive.
\end{proof}

\refprop{arith_action_on_cycle} states that for any finite germ $f$ on a cusp singularity for which the action on the circle $\mc{S}_X$ is a rotation, then this action coincides with the one of a finite germ $f_\alpha$ constructed arithmetically as above.
Of course, in general $f$ and $f_\alpha$ do not coincide.
In fact, any singularity $(X,x_0)$ admits lots of non-commuting automorphisms, coming from the time-$1$ flow of vector fields tangent to the singularity (see \cite{muller:liegroupsanalyticalgebras}).
We may compose $f_\alpha$ by any such automorphism of $(X,x_0)$, obtaining another germ $f$. Notice that this operation does not change the action induced on $\mc{S}_X$. In particular there exists infinitely many different finite germs whose action on $\mc{S}_X$ coincide.
Nevertheless, we may wonder if all such germs are analytically conjugated one to another.

\begin{quest}
Is any finite germ $f\colon(X,x_0) \to (X,x_0)$ on a cusp singularity $(X,x_0)$, whose action preserves the orientation of the circle $\mc{S}_X$, analytically conjugated to a finite germ of the form $f_\alpha$, for some $\alpha \in N_\omega$ ?
\end{quest}

\newpage

\small
\bibliographystyle{abbrv}
\bibliography{biblio}

	\end{document}